\documentclass{article}

\usepackage[text={6.0in,8.6in},centering]{geometry}
\usepackage{amsthm,amsmath,amsfonts,amssymb}
\usepackage[numbers]{natbib}
\usepackage{graphicx}
\usepackage[ruled,section]{algorithm}
\usepackage{algpseudocode}
\usepackage{subcaption}
\RequirePackage[colorlinks,citecolor=blue,urlcolor=blue]{hyperref}

\newtheorem*{theorem*}{Theorem}
\newtheorem{theorem}{Theorem}[section]

\newtheorem*{rep@theorem}{\rep@title}
\newcommand{\newreptheorem}[2]{%
\newenvironment{rep#1}[1]{%
 \def\rep@title{#2 \ref{##1}}%
 \begin{rep@theorem}}%
 {\end{rep@theorem}}}
\makeatother
\newreptheorem{theorem}{Theorem} 
\newreptheorem{lemma}{Lemma}      
\newreptheorem{proposition}{Proposition} 
\newtheorem{corollary}{Corollary}[section]
\newtheorem{proposition}{Proposition}[section]
\newtheorem{lemma}{Lemma}[section]
\newtheorem{remark}{Remark}[section]
\newtheorem{example}{Example}[section]

\newtheorem{definition}{Definition}[section]

\newcommand{\R}{\ensuremath{\mathbb R}}
\newcommand{\E}{\ensuremath{\mathbb E}}
\newcommand{\bin}{\text{bin}}
\newcommand{\argmax}{\ensuremath{\mathop{\text{\rm arg\,max}}}}
\newcommand{\argmin}{\ensuremath{\mathop{\text{\rm arg\,min}}}}
\newcommand{\tran}{\ensuremath{\text{T}}}
\newcommand{\given}{\ensuremath{\,|\,}}
\newcommand{\rank}{\ensuremath{\mathop{\rm rank}}}

\begin{document}

\begin{center}
{\bf{\Large{Optimal Rates for Community Estimation in the Weighted Stochastic Block Model}}}

\vspace*{.25in}

\begin{tabular}{ccccc}
{\large{Min Xu$^\dagger$}} & \hspace*{.2in} & {\large{Varun Jog$^\ddagger$}} & \hspace*{.2in} & {\large{Po-Ling Loh$^{\ddagger*}$}} \\
{\large{\texttt{mx76@stat.rutgers.edu}}} & & {\large{\texttt{vjog@wisc.edu}}} & & {\large{\texttt{loh@ece.wisc.edu}}}
\end{tabular}

\vspace{.2in}

\begin{tabular}{ccc}
Department of Statistics$^\dagger$ & \hspace{.3in} & Departments of ECE$^\ddagger$ \& Statistics$^*$ \\
Rutgers University && University of Wisconsin - Madison \\ Piscataway, NJ 08854 & & Madison, WI 53706
\end{tabular}
 
\vspace*{.2in}

August 2018

\vspace*{.2in}

\end{center}

\begin{abstract}

Community identification in a network is an important problem in fields such as social science, neuroscience, and genetics. Over the past decade, stochastic block models (SBMs) have emerged as a popular statistical framework for this problem. However, SBMs have an important limitation in that they are suited only for networks with unweighted edges; in various scientific applications, disregarding the edge weights may result in a loss of valuable information. We study a weighted generalization of the SBM, in which observations are collected in the form of a weighted adjacency matrix and the weight of each edge is generated independently from an unknown probability density determined by the community membership of its endpoints. We characterize the optimal rate of misclustering error of the weighted SBM in terms of the Renyi divergence of order 1/2 between the weight distributions of within-community and between-community edges, substantially generalizing existing results for unweighted SBMs. Furthermore, we present a computationally tractable algorithm based on discretization that achieves the optimal error rate. Our method is adaptive in the sense that the algorithm, without assuming knowledge of the weight densities, performs as well as the best algorithm that knows the weight densities.

\end{abstract}

\vskip10pt

\section{Introduction}

The recent explosion of network datasets has created a need for new statistical methodology~\cite{NewEtal06, DavKle10, Jac10, GolEtal10}. One active area of research with diverse scientific applications pertains to community detection and estimation, where observations take the form of edges between nodes in a graph, and the goal is to partition the nodes into disjoint groups based on their relative connectivity~\cite{FieEtal85, HarSha00, PriEtal00, ShiMal00, McS01, NewGir04}.

A standard model assumption in community recovery problems is that---conditioned on the community labels of the nodes of the graph---each edge is generated independently according to a distribution governed solely by the community labels of its endpoints. This is the setting of the stochastic block model (SBM)~\cite{HolEtal83}.
Community recovery may also be viewed as estimating the latent cluster memberships of the nodes a random graph generated by an SBM. The last decade has seen great progress on this problem, beginning with the seminal conjecture of Decelle et al.~\cite{DecEtal11} (see, e.g., the excellent survey paper by Abbe~\cite{Abbe17}). Various algorithms for community recovery have been devised with guaranteed optimality properties, measured in terms of correlated recovery~\cite{MosEtal12, MosEtal13, Mas14}, exact recovery~\cite{abbe2014exact, AbbSan15, abbe2015recovering}, and minimum misclustering error rate~\cite{GaoEtal15, zhangminimax}.

However, an important shortcoming of SBMs is that all edges are assumed to be binary. In contrast, the edges appearing in many real-world networks possess weights reflecting a diversity of strengths or characteristics~\cite{New04, BocEtal06}: Edges in social or cellular networks may quantify the frequency of interactions between pairs of individuals~\cite{Sad72, BloEtal08}. Similarly, edges in gene co-expression networks are assigned weights corresponding to the correlation between expression levels of pairs of genes~\cite{ZhaHor05}; and in brain networks, edge weights may indicate the level of neuronal activity between corresponding regions in the brain~\cite{RubSpo10}. Although an unweighted adjacency matrix could be constructed by disregarding the edge weight data, this might result in a loss of valuable information that could be used to recover hidden communities. 

This motivates the \emph{weighted} stochastic block model, which we study in this paper. Each edge is generated from a Bernoulli$(p)$ or Bernoulli$(q)$ distribution, depending on whether its endpoints lie in the same community, and then each edge is assigned an edge weight generated from one of two arbitrary densities, $p(\cdot)$ or $q(\cdot)$. We study the problem of community estimation based on observations of the edge weights in the network, \emph{without} assuming knowledge of $p$, $q$, $p(\cdot)$ or $q(\cdot)$. Since $p(\cdot)$ and $q(\cdot)$ are allowed to be continuous, our model strictly generalizes the discrete labeled SBMs considered in previous literature~\cite{HeiEtal12, LelEtal13, JogLoh15}, as well as the censored SBM~\cite{AbbEtal14, HajEtal14, HajEtal15}.

We emphasize key differences between the weighted SBM framework and the setting of other clustering problems involving continuous edge weights~\cite{balakrishnan2011noise, hajek2015submatrix}. First, we do not assume that between-cluster edges tend to have heavier weights than within-cluster edges (e.g., in mean-separation models). Such an assumption is critical to many algorithms for weighted networks, since it allows existing algorithms for unweighted SBMs, such as spectral clustering, to be applied in relatively straightforward ways. In contrast, the algorithms in this paper allow us to exploit other potential differences in $p(\cdot)$ and $q(\cdot)$, such as differences in variance or shape. This is crucial to achieve optimal performance. Second, our setting is \emph{nonparametric} in the sense that the densities $p(\cdot)$ and $q(\cdot)$ may be arbitrary and are only required to satisfy mild regularity conditions, whereas previous approaches generally assume that $p(\cdot)$ and $q(\cdot)$ belong to a specific parametric family. Nonparametric density estimation is itself a difficult problem, made even more difficult in the case of weighted SBMs, since we do not know a priori which edge weights have been drawn from which densities.

Our main theoretical contribution is to characterize the \emph{optimal rate of misclustering error} in the weighted SBM. On one side, we derive an information-theoretic lower bound for the performance of any community recovery algorithm for the weighted SBM. Our lower bound applies to all parameters in the parameter space (thus is not minimax) and all algorithms that produce the same output on isomorphic networks---a property that we call \emph{permutation equivariance}. On the other side, we present a computationally tractable algorithm with a rate of convergence that matches the lower bound. Our results show that the optimal rate for community estimation in a weighted SBM is governed by the Renyi divergence of order $\frac{1}{2}$ between two mixed distributions, capturing the discrepancy between the edge probabilities and edge weight densities for between-community and within-community connections. This provides a natural but highly nontrivial generalization of the results in Zhang and Zhou~\cite{zhangminimax} and Gao et al. \cite{GaoEtal15}, which show that the optimal rate of the unweighted SBM is characterized by the Renyi divergence of order $\frac{1}{2}$ between two Bernoulli distributions corresponding only to edge probabilities.

Remarkably, our rate-optimal algorithm is fully \emph{adaptive} and does not require prior knowledge of $p(\cdot)$ and $q(\cdot)$. Thus, even in cases where the densities belong to a parametric family, it is possible---\emph{without making any parametric assumptions}---to obtain the same optimal rate as if one imposes the true parametric form. This is in sharp contrast to most nonparametric estimation problems in statistics, where nonparametric methods usually lead to a slower rate of convergence than parametric methods if a specific parametric form is known. The apparent discrepancy is explained by the simply stated observation that in weighted SBMs, one does \emph{not} need to estimate edge densities well in order to recover communities to desirable accuracy. This intuition is also reflected in the work of Abbe and Sandon~\cite{abbe2015recovering} for the exact recovery problem and  Gao et al.~\cite{GaoEtal15} for the unweighted SBM. Our proposed recovery algorithm hinges on a careful discretization technique: When the edge weights are bounded, we discretize the distribution via a uniformly spaced binning to convert the weighted SBM into an instance of a \emph{labeled} SBM, where each edge possesses a label from a discrete set with finite (but divergent) cardinality; we then perform community recovery in the labeled SBM by extending a coarse-to-fine clustering algorithm that computes an initialization through spectral clustering~\cite{chin2015stochastic, lei2015consistency} and then performs refinement through nodewise likelihood maximization~\cite{GaoEtal15}. When the edge weights are unbounded, we reduce the problem to the bounded case by first applying an appropriate transformation to the edge weight distributions.

The remainder of our paper is organized as follows: Section~\ref{sec:formulation} introduces the mathematical framework of the weighted SBM, defines the community recovery problem, and formalizes the notion of permutation equivariance. Section~\ref{sec:summary} provides an informal summary of our results, later formalized in Section~\ref{sec:rate}. Section~\ref{sec:method} outlines our proposed community estimation algorithm. The key technical components of our proofs are highlighted in Section~\ref{sec:proofs}, and Section~\ref{sec:simulation} reports the results of various simulations. Section~\ref{sec:conclusion} concludes the paper with further implications and open questions.

\paragraph{\textbf{Notation:}} For a positive integer $n$, we write $[n]$ to denote the set $\{1, \dots, n\}$ and $S_n$ to denote the set of permutations of $[n]$. We write $o(1)$ to denote a sequence indexed by $n$ that tends to 0 as $n \rightarrow \infty$, and write $\Theta(1)$ to denote a sequence indexed by $n$ that is bounded away from 0 and $\infty$ as $n \rightarrow \infty$. For two real numbers $a$ and $b$, we write $a \vee b$ to denote $\max(a,b)$ and write $a \wedge b$ to denote $\min(a,b)$.


\section{Model and problem formulation}
\label{sec:formulation}

We begin with a formal definition of the homogeneous weighted SBM and a description of the community recovery problem. 

\subsection{Weighted stochastic block model}

Let $n$ denote the number of nodes in the network and let $K \geq 2$ denote the number of communities. A \emph{clustering} $\sigma$ is a function $[n] \rightarrow [K]$. For each node $u \in [n]$, we refer to $\sigma(u)$ as the cluster of node $u$. 

\begin{definition}
For a positive number $\beta \geq 1$, we define $\mathcal{C}(\beta, K)$ as the set of clusterings with minimum cluster size is at least $\frac{n}{\beta K}$, i.e., $\sigma \in \mathcal{C}(\beta, K)$ if and only if $| \sigma^{-1}(k) | \geq \frac{n}{\beta K}$ for all $k \in [K]$. We refer to $\beta$ as the \emph{cluster-imbalance constant}.
\end{definition}

We first define the homogeneous unweighted SBM, which is
characterized by the following probability distribution over adjacency matrices $A \in \{0,1\}^{n \times n}$:

\begin{definition} [Homogeneous unweighted SBM]
Let $\sigma_0 \in \mathcal{C}(\beta, K)$ and $p,q \in [0,1]$. We say that a random binary-valued matrix $A$ has the distribution $SBM(\sigma_0, p, q)$ if for all $u < v$, the entries of $A$ are generated independently according to 
  \[
A_{uv} \sim \left\{ \begin{array}{cc}
 Ber(p) & \text{ if } \sigma_0(u) = \sigma_0(v), \\
 Ber(q) & \text{ if } \sigma_0(u) \neq \sigma_0(v). 
\end{array} \right.
\]
\end{definition}
Thus, the parameters $p$ and $q$ correspond to the within-cluster and between-cluster edge probabilities. The more general \emph{heterogenous} unweighted SBM is characterized by a matrix $P \in \R^{K \times K}$ of probabilities instead of two scalars $p$ and $q$, and edges are generated independently according to $A_{uv} \sim Ber(P_{\sigma_0(u), \sigma_0(v)})$.

A homogeneous weighted SBM is parametrized by $\sigma_0 \in \mathcal{C}(\beta, K)$, the edge \emph{absence} probabilities $P_0$ and $Q_0$, and the edge weight probability densities $p(\cdot)$ and $q(\cdot)$ supported on $S \subset \mathbb{R}$, where $S$ may be $[0,1]$, $[0, \infty)$, or $\mathbb{R}$. The weighted SBM is then characterized by a distribution over symmetric matrices $A \in S^{n \times n}$ in the following manner:

\begin{definition}
[Homogeneous weighted SBM]
\label{def:weighted_homo_sbm1}
Let $\sigma_0 \in \mathcal{C}(\beta, K)$. We say that a random real-valued matrix $A$ has the distribution $WSBM(\sigma_0, (P_0, p), (Q_0, q))$ if for all $u < v$,
\begin{align}
A_{uv} \sim \left\{ 
   \begin{array}{cc} 
   P_0 \delta_0(\cdot) + (1 - P_0) p(\cdot) & \text{ if } \sigma_0(u) = \sigma_0(v), \\
   Q_0 \delta_0(\cdot) + (1 - Q_0) q(\cdot) & \text{ if } \sigma_0(u) \neq \sigma_0(v). 
   \end{array} \right. \label{eqn:simple_wsbm_defn}
\end{align}
where $P_0 \delta_0(\cdot) + (1 - P_0) p(\cdot)$ denotes a probability distribution whose singular part (with respect to the Lebesgue measure) is a point mass at $0$ with probability $P_0$ and whose continuous part has $(1-P_0)p(\cdot)$ as its Radon-Nikodym derivative with respect to the Lebesgue measure; and $Q_0 \delta_0(\cdot) + (1 - Q_0) q(\cdot)$ is defined analogously. 
\end{definition}

Note that if $p(\cdot)$ and $q(\cdot)$ are Dirac delta masses at $1$, the weighted SBM reduces to the unweighted version. We make a few additional remarks about the definition of the weighted SBM. First, we observe that $\E(A)$ may not exhibit the familiar block structure found in unweighted SBMs, since our model includes the case where $(P_0, p(\cdot))$ and $(Q_0, q(\cdot))$ have the same mean. Second, our definition treats an edge with weight 0 as a missing edge, but it is straightforward to distinguish the two notions by defining $P$ and $Q$ as probability measures over $S \cap \{ * \}$, where the symbol $*$ denotes a missing edge. Lastly, it is possible to generalize the weighted SBM to a \emph{weighted and labeled} SBM with the model
\begin{align*}
  A _{uv} \sim \begin{cases}
      P, \textrm{ if $\sigma_0(u) = \sigma_0(v)$} \\
      Q, \textrm{ if $\sigma_0(u) \neq \sigma_0(v)$}.
    \end{cases}
\end{align*}
where $P$ and $Q$ are general probability distributions over $S$ (and the labels correspond to a discrete part). The theory derived in this paper extends in a straightforward fashion to the cases where the discrete portion of $P$ and $Q$ has finite support.


\subsection{Community estimation}

Given an observation $A \in S^{n \times n}$ generated from a weighted SBM, the goal of community estimation is to recover the true cluster membership structure $\sigma_0$. We assume throughout our paper that the number of clusters $K$ is known.

We evaluate the performance of a community recovery algorithm in terms of its misclustering error.
For a clustering algorithm $\hat{\sigma}$, let $\hat{\sigma}(A) \,:\, [n] \rightarrow [K]$ denote the clustering produced by $\hat{\sigma}$ when provided with the input $A$. We have the following definition:
\begin{definition}
We define the \emph{misclustering error} to be
\[
l(\hat{\sigma}(A), \sigma_0) := \min_{\pi \in S_K} \frac{1}{n} d_H(\pi \circ \hat{\sigma}(A),\,  \sigma_0 ),
\]
where $d_H(\cdot, \, \cdot)$ denotes the Hamming distance. The \emph{risk} of $\hat{\sigma}$ is defined as $R(\hat{\sigma}, \sigma_0) := \E l(\hat{\sigma}(A), \sigma_0)$,
where the expectation is taken with respect to both the random network $A$ and any potential randomness in the algorithm $\hat{\sigma}$.
\end{definition}

The goal of this paper is to characterize the minimal achievable risk for community recovery on the weighted SBM in terms of the parameters $(n, \beta, K, (P_0, p), (Q_0, q))$.

\subsection{Permutation equivariance}
\label{sec:permutation_equivariance}

Since the cluster structure in a network does not depend on how the nodes are labeled, it is natural to focus on estimation algorithms that output equivalent clusterings when provided with isomorphic inputs. We formalize this property in the following definition: 
\begin{definition}
  \label{defn:permutation_equivariance_deterministic}
  For an $n \times n$ matrix $A$ and a permutation $\pi \in S_n$, let $\pi A$ denote the $n \times n$ matrix such that $A_{uv} = [\pi A]_{\pi(u), \pi(v)}$.
Let $\hat{\sigma}$ be a deterministic clustering algorithm. Then $\hat{\sigma}$ is \emph{permutation equivariant} if, for any $A$ and any $\pi \in S_n$,
\begin{align}
 \tau \circ \hat{\sigma}(\pi A) \circ \pi = \hat{\sigma}(A) \,\, \textrm{for some $\tau \in S_K$}. \label{eqn:perm_equiv_key}
\end{align}
\end{definition}

Note that $\hat{\sigma}(\pi A)$ by itself is not equivalent to $\hat{\sigma}(A)$, since the nodes in $\pi A$ are labeled with respect to the permutation $\pi$. It is straightforward to extend Definition~\ref{defn:permutation_equivariance_deterministic} to randomized algorithms by requiring condition~\eqref{eqn:perm_equiv_key} to hold almost everywhere in the probability space that underlies the algorithmic randomness. Permutation equivariance is a natural property satisfied by all the clustering algorithms studied in literature except algorithms that leverage extra side information in addition to the given network. In Section~\ref{sec:lower_bound}, we study permutation equivariance in detail and provide some properties of permutation equivariant estimators.

\section{Overview of main results}
\label{sec:summary}

The difficulty of community recovery depends on the extent to which $(P_0, p)$ and $(Q_0, q)$ are different; it is clearly impossible to have a consistent clustering algorithm if $(P_0, p)$ and $(Q_0, q)$ are equal. We show in this paper that natural measure of discrepancy between $(P_0, p)$ and $(Q_0, q))$ which governs the optimal rate of convergence is the Renyi divergence of order $\frac{1}{2}$.

Given any probability distributions $P$ and $Q$ that are absolutely continuous with respect to each other, the Renyi divergence of order $\frac{1}{2}$ is defined as $I(P,Q) := -2 \log \int \bigl( \frac{dP}{dQ} \bigr)^{1/2} dQ$. For our setting, the Renyi divergence takes on the special form
\begin{align*}
I((P_0, p), (Q_0, q)) =  -2 \log \! \left( \sqrt{P_0 Q_0} + \!\! \int \sqrt{(1-P_0)(1-Q_0)p(x)q(x) } dx \right).
\end{align*}

\noindent If $I((P_0, p), (Q_0, q))$ is bounded above by a universal constant, the Renyi divergence is of the same order as the Hellinger distance (cf.\ Lemma~\ref{lem:renyi_hellinger}):
\begin{align*}
I((P_0, p), (Q_0, q)) &\asymp  (\sqrt{P_0} - \sqrt{Q_0})^2 + \int_S (\sqrt{(1-P_0)p(x)} - \sqrt{(1-Q_0)q(x)} )^2 dx  \\
 &= (\sqrt{P_0} - \sqrt{Q_0})^2 + (\sqrt{1-P_0} - \sqrt{1-Q_0})^2 \\
 & \qquad \qquad + \sqrt{(1-P_0)(1-Q_0)} \int_S (\sqrt{p(x)} - \sqrt{q(x)} )^2 dx.
\end{align*}
Thus, we can think of $I((P_0, p), (Q_0, q))$ as having two components, the first of which captures the divergence between the edge presence probabilities (and also appears in the analysis of unweighted SBM), and the second of which captures the divergence between the edge weight densities.

The presence of the second term illustrates how the weighted SBM behaves quite differently from its unweighted counterpart---in particular, \emph{dense} networks may be interesting in a weighted setting. For example, even if the weighted network is
completely dense in the sense that $1-P_0 = 1- Q_0 = 1$, a nonzero signal $I$ may still exist if $p(\cdot)$ and $q(\cdot)$ are sufficiently different. Our results apply simultaneously to dense and sparse settings; it is important to note that dense weighted networks arise in real-world settings, such as gene co-expression data.

We now provide an informal overview of our main results.
\begin{theorem*} {\bf (Informal statement) }
Let $A$ be generated from a weighted SBM. Under regularity conditions on $((P_0, p), (Q_0, q))$, any permutation equivariant estimator $\hat{\sigma}$ satisfies the lower bound
\[
  \E l(\hat{\sigma}(A), \sigma_0) \geq \exp\left( - (1 + o(1)) \frac{n}{\beta K} I((P_0, p), (Q_0, q)) \right).
\]
\end{theorem*}

\begin{theorem*} {\bf (Informal statement) }
Under regularity conditions on $((P_0, p), (Q_0, q))$, there exists a permutation equivariant algorithm $\hat{\sigma}$ achieving the following misclustering error rate:
  \[
   \lim_{n \rightarrow \infty} P\left( l(\hat{\sigma}(A), \sigma_0) \leq \exp\left( - (1 + o(1)) \frac{n}{\beta K} I((P_0, p), (Q_0, q)) \right) \right) \rightarrow 1.
  \]
  Furthermore, if $\frac{n I}{\beta K \log n} \leq 1$, we have
  \begin{align*}
    \E l(\hat{\sigma}(A), \sigma_0) \leq \exp\left( - (1 + o(1)) \frac{n I}{\beta K} \right).
\end{align*}
\end{theorem*}

Taken together, the theorems imply that in the regime where $\frac{n I}{\beta K \log n} \leq 1$, the optimal risk is tightly characterized by the quantity $\exp \left( - (1 + o(1)) \frac{ n I}{\beta K} \right)$. On the other hand, if $\frac{nI}{\beta K\log n} > 1$, we have $\exp \left( - (1 + o(1)) \frac{n I}{\beta K} \right) < \frac{1}{n}$ for large enough $n$, so $\lim_{n \rightarrow \infty} P\left( l(\hat{\sigma}(A), \sigma_0) = 0 \right) \rightarrow 1$ (since $l(\hat{\sigma}(A),\sigma_0) < \frac{1}{n}$ implies $l(\hat{\sigma}(A),\sigma_0)=0$). Thus, the regime where $\frac{n I}{\beta K \log n} > 1$ is in some sense an easier problem, since we can guarantee perfect recovery with high probability. 

\subsection{Relation to previous work}

Our result generalizes the work of Zhang and Zhou~\cite{zhangminimax}, which establishes the minimax rate of $\exp \left( -(1+o(1)) \frac{ n}{\beta K}I\bigl(Ber(p), Ber(q) \bigr) \right)$ for the unweighted SBM, where
\begin{equation*}
I(Ber(p), Ber(q)) = -2 \log \bigl(\sqrt{pq} + \sqrt{(1-p)(1-q)} \bigr).
\end{equation*}
The optimal algorithm proposed in Zhang and Zhou~\cite{zhangminimax} is intractable, but a computationally feasible version was developed by Gao et al.~\cite{GaoEtal15}; the latter algorithm is a building block for the estimation algorithm proposed in this paper. 

Our result should also be viewed in comparison to Yun and Proutiere~\cite{yun2016optimal}, who studied the optimal risk for the heterogenous labeled SBM with finitely many labels, with respect to a prior on the cluster assignment $\sigma_0$. They characterize the optimal rate under a notion of divergence that reduces to the Renyi divergence of order $\frac{1}{2}$ between two discrete distributions over a fixed finite number of labels in the homogeneous setting (cf.\ Lemma~\ref{lem:information_equivalence}). Since the discussion is somewhat technical, we provide a more detailed comparison of our work to the results of Yun and Proutiere in Section~\ref{sec:labeled_sbm_analysis}.

Jog and Loh~\cite{JogLoh15} proposed a similar weighted block model and show the exact recovery threshold to be dependent on the Renyi divergence. They focus on the setting where the distributions are discrete and known, whereas we consider continuous densities that are unknown. Aicher et al.~\cite{aicher2014learning} introduced a version of a weighted SBM that is a special case of the setting discussed in this paper, where the densities $P$ and $Q$ in equation~\eqref{eqn:simple_wsbm_defn} are drawn from a known exponential family. Notably, the definition of Aicher et al.~\cite{aicher2014learning} cannot incorporate sparsity. The weighted SBM model considered in Hajek et al.~\cite{hajek2017information} is also similar to the one we propose in our paper, except it only involves a single hidden community and assumes knowledge of the distributions $P$ and $Q$. Weighted networks have also received some attention in the physics community~\cite{New04, barrat2004architecture}, and various ad-hoc methods have been proposed; since theoretical properties are generally unknown, we do not explore these connections in our paper. 

\paragraph{\textbf{Other notions of recovery:}} A closely related problem is that of finding the exact recovery threshold. We say that the unweighted SBM has an \emph{exact recovery threshold} if a function $\theta(p, q, n, K, \beta, \sigma_0)$ exists such that exact recovery is asymptotically almost always impossible if $\theta < 1$, and almost always possible if $\theta > 1$. For the homogeneous unweighted SBM, Abbe et al.~\cite{abbe2014exact} show that when $\beta=1, K=2, 1 - P_0 = \frac{a \log n}{n}$, and $1 - Q_0 = \frac{b \log n}{n}$, for some constants $a$ and $b$, the exact recovery threshold is $\sqrt{a} - \sqrt{b}$. This result was later generalized to multiple communities with heterogenous edge probabilities in Abbe and Sandon~\cite{AbbSan15}, where a notion of CH-divergence was shown to characterize the threshold for exact recovery.
A notion of weak recovery, corresponding to a detection threshold, has also been considered~\cite{Mas14, MosEtal14}. 



\section{Estimation algorithm}
\label{sec:method}

A natural approach to community estimation is to first estimate the edge weight densities $p(\cdot)$ and $q(\cdot)$, but this is hindered by the fact that we do not know whether an edge weight observation originates from $p(\cdot)$ or $q(\cdot)$. An alternative approach of applying spectral clustering directly to the weighted adjacency matrix $A$ will also be ineffective if $(P_0, p)$ and $(Q_0, q)$ have the same mean, so $\E(A)$ does not exhibit any cluster structure. A third idea is to output the clustering that maximizes the Kolmogorov-Smirnov distance (or another nonparametric two-sample test statistic) between the empirical CDFs of within-cluster edge weights and the between-cluster edge weights. This idea, though feasible, is computationally intractable, since it involves searching over all possible clusterings. Our approach is appreciably different from the methods suggested above, and consists of combining the idea of discretization from nonparametric density estimation with clustering techniques for unweighted SBMs.

\subsection{Outline of algorithm}
\label{sec:algorithm}

We begin by describing the main components of our algorithm. The key ideas are to convert the edge weights into a finite set of labels by discretization, and then cluster nodes on the labeled network. Our algorithm is summarized pictorially in Figure~\ref{fig:method_pipeline1}.

\begin{enumerate}
\item \textbf{Transformation \& discretization.} We take as input a weighted adjacency matrix $A$ and apply an invertible transformation function $\Phi \,:\, S \rightarrow [0,1]$ (recall $S$ is the support of the edge weights and can be $[0,1]$, $[0, \infty)$, or $\mathbb{R}$) on the nonzero edges to obtain a matrix $\Phi(A)$ with weights between 0 and 1. Next, we divide the interval $[0,1]$  into $L$ equally-spaced subintervals. We replace the real-valued entries of $\Phi(A)$ with categorical labels in $[L]$. We denote the labeled adjacency matrix by $A_L$.

\item \textbf{Add noise.} We perform the following process on every edge of the labeled graph, independently of other edges: With probability $1-\delta$ where $\delta = \frac{2(L+1)}{n}$, keep an edge as it is, and with probability $\delta$, erase the edge and replace it with an edge with label uniformly drawn from the set of labels. We continue to denote the modified adjacency matrix as $A_L$.

\item \textbf{Initialization parts 1 \& 2.} For each label $l$, we create a sub-network by including only edges of label $l$. We then perform spectral clustering on all sub-networks, and output the label $l^*$ that induces the maximally separated spectral clustering. Let $A_{l^*}$ be the adjacency matrix for label $l^*$. For each $u \in \{1, \dots, n\}$, we perform spectral clustering on $A_{l^*} \setminus \{u\}$, which denotes the adjacency matrix with vertex $u$ removed. We output $n$ clusterings $\tilde{\sigma}_1, \dots, \tilde{\sigma}_n$.

\item \textbf{Refinement \& consensus.} From each $\tilde \sigma_u$, we generate a clustering $\hat \sigma_u$ on $\{1, 2, \dots, n\}$ that retains the assignments specified by $\tilde \sigma_u$ for $\{1, 2, \dots, n\} \setminus \{u\}$, and assigns $\hat \sigma_u(u)$ by maximizing the likelihood taking into account only the neighborhood of $u$.  We then align the cluster assignments made in the previous step. 
\end{enumerate}

\begin{figure}[htp]
\centering
\includegraphics[scale=0.38, trim={1in 7in 0 0.4in}]{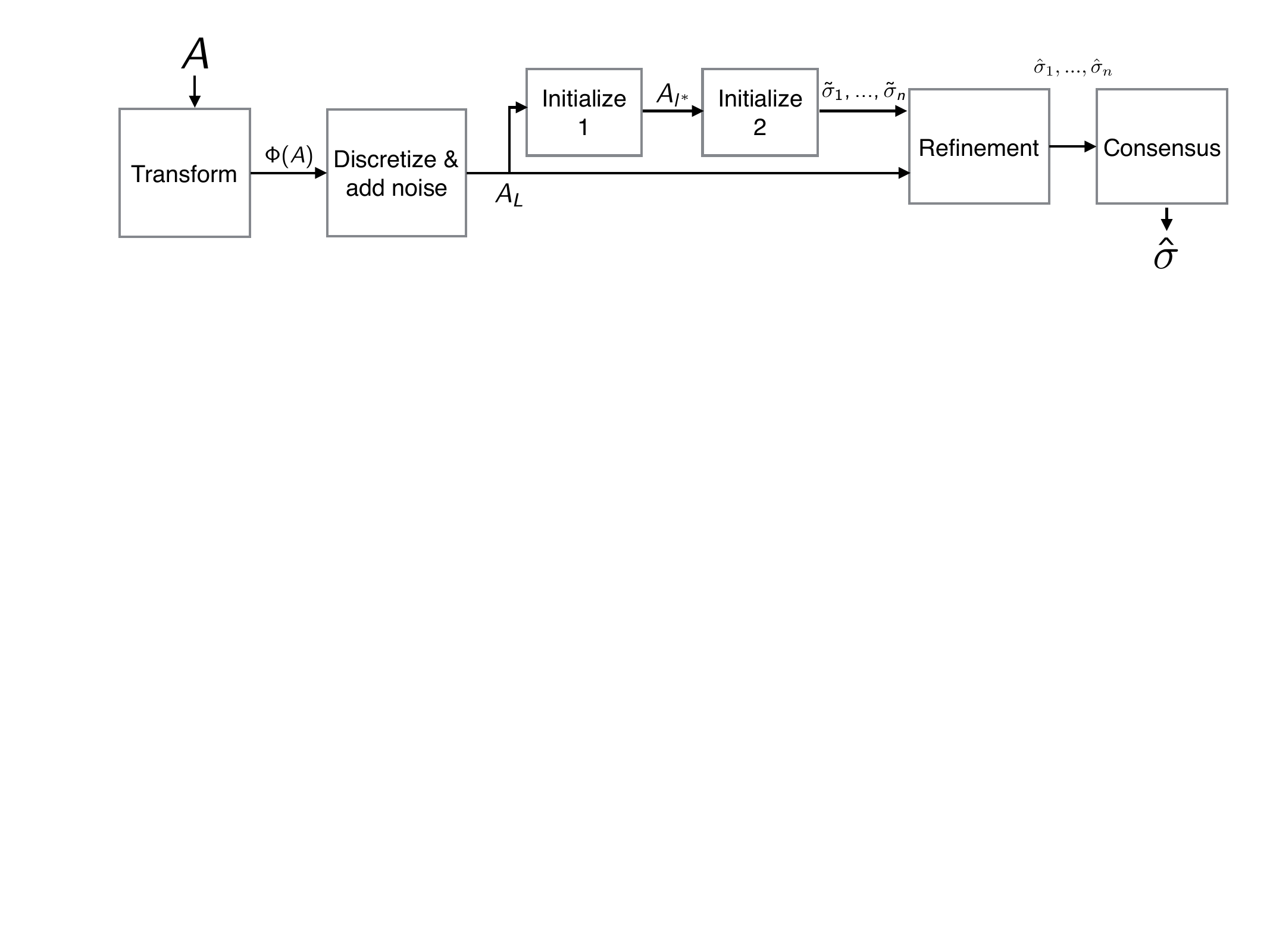}
\caption{Pipeline for our proposed algorithm}
\label{fig:method_pipeline1}
\end{figure}

\subsection{Transformation and discretization}

In the transformation step, we apply an invertible CDF $\Phi \,:\, S \rightarrow [0,1]$ as the transformation function on all the edge weights, so that each entry of $\Phi(A)$ lies in $[0,1]$. In the discretization step, we divide the interval $[0,1]$ into $L$ equally-spaced bins of the form $[a_l, b_l]$, where $a_1 = 0, b_L = 1$, and $b_l - a_l = \frac{1}{L}$. An edge is assigned the label $l$ if the weight of that edge lies in bin $l$. 

\begin{algorithm}
\caption{Transformation and Discretization}
\label{alg:transform_and_discretize}
\begin{flushleft}
 \textbf{Input:} A weighted network $A$, a positive integer $L$, and an invertible function $\Phi \,:\, S \rightarrow [0,1]$\\
 \textbf{Output:} A labeled network $A_L$ with $L$ labels\\
\end{flushleft}
\begin{algorithmic}
\State Divide $[0,1]$ into $L$ bins, labeled $\bin_1, \dots, \bin_L$
\For{every edge $(u,v)$ such that $A_{uv} \neq 0$}
   \State Let $l$ be the bin in which $\Phi(A_{uv})$ falls
   \State Assign the edge $(u,v)$ the label $l$ in the labeled network $A_L$
\EndFor
\end{algorithmic}
\end{algorithm}

\subsection{Add noise}

For technical reasons, we inject noise into the network as a form of regularization. As detailed in the proof of Proposition~\ref{prop:labeled_sbm_rate} in Appendix~\ref{appendix: first}, deliberately forming a noisy version of the graph barely affects the separation between the distributions of the within-community and between-community edge labels, but has the desirable effect of ensuring that all edge labels occur with probability at least $\frac{2}{n}$. This property is crucial to our analysis in subsequent steps of the algorithm. In the description of the algorithm below, we treat the label 0 (i.e., an empty edge) as a separate label, so we have a network with $L+1$ labels.

\begin{algorithm}
\caption{Add noise}
\label{alg:noisify}
\begin{flushleft}
\textbf{Input:} A labeled network $A_L$ with $L+1$ labels \\
\textbf{Output:} A labeled network $A_L$ with $L+1$ labels\\
\end{flushleft}
\begin{algorithmic}
\For{every edge $(u,v)$}
   \State With probability $1-\frac{2(L+1)}{n}$, do nothing
   \State With probability $\frac{2(L+1)}{n}$, replace the edge label with a label drawn uniformly at random from $\{0, 1, 2, \dots, L\}$
\EndFor
\end{algorithmic}
\end{algorithm}

\subsection{Initialization}

The initialization procedure takes as input a network with edges labeled $\{0, 1, \dots, L\}$. The goal of the initialization procedure is to create a rough clustering $\tilde{\sigma}$ that is consistent but not necessarily optimal. As outlined in Algorithm~\ref{alg:initialization1}, the rough clustering is based on a single label $l^*$, selected based on the maximum value of the estimated Renyi divergence between within-community and between-community distributions for the unweighted SBMs based on individual labels.

For technical reasons, we actually create $n$ separate rough clusterings $\{\tilde{\sigma}_u \}_{u = 1, \dots, n}$, where each $\tilde{\sigma}_u \,:\, [n-1] \rightarrow [K]$ is a clustering of a network of $n-1$ nodes with $u$ removed. The clusterings $\{\tilde{\sigma}_u\}$ will later be combined into a single clustering algorithm. In practice, it is sufficient to create a single rough clustering (see Remark~\ref{rem:simple_computation} below).

\begin{remark}
  \label{rem:aggregate_initialization}
The initialization procedure that we propose is based on choosing a single best label $l^*$ and deriving an initial clustering from the unweighted network associated with $l^*$. This is sufficient in theory, but a better initial clustering may be gained in practice by aggregating information from all labels. Such an aggregation must, however, be performed with care, so that uninformative labels do not dilute the information content of the informative labels.  
\end{remark}

\begin{algorithm}[h!]
\caption{Initialization}
\label{alg:initialization1}
\begin{flushleft}
\textbf{Input:} A labeled network $A_L$ with $L$ labels \\
\textbf{Output:} A set of clusterings $\{ \tilde{\sigma}_u \}_{u=1, \dots, n}$, where $\tilde \sigma_u$ is a clustering on $\{1, 2, \dots, n\} \setminus \{u\}$\\
\end{flushleft}
\begin{algorithmic}[1]
  \State Separate $A_L$ into $L$ networks $\{ A_l \}_{l=1, \dots, L}$
  \Comment{Stage 1}
\For{each label $l$}
   \State Perform \textsc{spectral clustering} (Algorithm~\ref{alg:spectral}) with $\tau = 40 K \bar{d}$, where $\bar{d} = \frac{1}{n} \sum_{u=1}^n d_u$ is the average degree, and $\mu = 4 \beta$ to obtain $\tilde{\sigma}_l$
   \State Estimate $\hat{P}_l = 
             \frac{ \sum_{ u\neq v \,:\, \tilde{\sigma}_l(u) = \tilde{\sigma}_l(v) } (A_l)_{uv} }
                  { | \{u \neq v \,:\, \tilde{\sigma}_l(u) = \tilde{\sigma}_l(v) \}| }$ and 
               $\hat{Q}_l = 
            \frac{ \sum_{u \neq v \,:\, \tilde{\sigma}_l(u) \neq \tilde{\sigma}_l(v) } (A_l)_{uv} }
              { |\{u \neq v \,:\, \tilde{\sigma}_l(u) \neq \tilde{\sigma}_l(v) \}| }$
   \State Compute $\hat{I}_l \leftarrow 
               \frac{ (\hat{P}_l - \hat{Q}_l)^2}{\hat{P}_l \vee \hat{Q}_l}$
\EndFor
\State Choose $l^* = \argmax_l \hat{I}_l$
\For{each node $u$}  \Comment{Stage 2}
   \State Create network $A_{l^*} \setminus \{u\}$ by removing node $u$ from $A_{l^*}$
   \State Perform \textsc{spectral clustering} (with the same parameter setting as stage 1) on $A_{l^*} \setminus \{ u \}$ to obtain $\tilde{\sigma}_u$ 
\EndFor 
\State Output the set of clusterings $\{ \tilde{\sigma}_u \}_{u=1, \dots, n}$
\end{algorithmic}
\end{algorithm}

\paragraph{\textbf{Spectral clustering:}} Note that Algorithm~\ref{alg:initialization1} involves several applications of \textsc{spectral clustering}. We describe the spectral clustering algorithm used as a subroutine in Algorithm~\ref{alg:spectral} below. Importantly, note that we may always choose the parameter $\mu$ sufficiently large such that Algorithm~\ref{alg:spectral} generates a set $S$ with $|S| = K$.

\begin{algorithm}[h!]
\caption{\textsc{Spectral clustering}}
\label{alg:spectral}
\begin{flushleft}
\textbf{Input:} An unweighted network $A$ with columns $\{A_u\}$, trim threshold $\tau$, number of communities $K$, and tuning parameter $\mu$ \\
\textbf{Output:} A clustering $\sigma$ 
\end{flushleft}
\begin{algorithmic}[1]
\State For each node $u$ with degree $d_u \geq \tau$, set $A_u = 0$ and $(A^\top)_u = 0$ to obtain $T_{\tau}(A)$
\State Compute $\hat{A} := \argmin_{\tilde{A} \,:\, \rank(\tilde{A}) \leq K} \| \tilde{A} - T_{\tau}(A)\|_2$ by SVD
\State For each node $u$, index the other nodes by $v_{(1)}, \ldots, v_{(n-1)}$ such that
\begin{equation*}
\| \hat{A}_u - \hat{A}_{v_{(1)}}\|_2 \leq \| \hat{A}_u - \hat{A}_{v_{(2)}} \|_2 \leq \ldots \leq \|\hat{A}_u - \hat{A}_{v_{(n-1)}} \|_2,
\end{equation*}
and define $$D(u) := \| \hat{A}_u - \hat{A}_{v_{(\lceil n/\mu K \rceil)}} \|_2 $$ 
\State Initialize $S \leftarrow 0$
\State Select node $u_1 := \argmin_u D(u)$ and add $u_1$ to $S$ as $S[1]$
\For{$i = 2, \dots, K$}
    \State Among all $u$ such that $|D(u)| \leq (1 - \frac{1}{\mu K})\textrm{-quantile}\{ D(v) \,:\, v \in [n]\}$, select 
           $$u_i = \argmax_u \min_{v \in \{S[1], \ldots, S[i-1]\}} \| \hat{A}_u - \hat{A}_v \|_2$$
    \State Add $u_i$ to $S$ as $S[i]$
\EndFor
\For{$u = 1, \dots, n$}
    \State Assign $\sigma(u) = \argmin_i \| \hat{A}_u - \hat{A}_{S[i]} \|$
\EndFor
\end{algorithmic}
\end{algorithm}

\subsection{Refinement and consensus}

This step parallels Gao et al.~\cite{GaoEtal15}. In the refinement step, we use the set of initial clusterings $\{\tilde{\sigma}_u\}_{u=1, \dots, n}$ to generate a more accurate clustering for the labeled network by locally maximizing an approximate log-likelihood for each node $u$. The consensus step resolves any cluster label inconsistencies present after the refinement stage. 

\begin{algorithm}[h!]
\caption{Refinement}
\label{alg:refinement}
\begin{flushleft}
\textbf{Input:} A labeled network $A_L$ and a set of clusterings $\{\tilde{\sigma}_u\}_{u=1, \dots, n}$, where $\tilde \sigma_u$ is a clustering on the set $\{1, 2, \dots, n\} \setminus \{u\}$ for each $u$ \\
\textbf{Output:} A clustering $\hat{\sigma}$ over the whole network
\end{flushleft}
\begin{algorithmic}[1]
\For{each node $u$}
   \State Estimate $\{ \hat{P}_l, \hat{Q}_l\}_{l=0,...,L}$ from $\tilde{\sigma}_u$
   \State Let $\hat{\sigma}_u : [n] \rightarrow [K]$, where 
       $\hat{\sigma}_u(v) = \tilde{\sigma}_u(v)$ for all $v \neq u$ and 
   \[
    \hat{\sigma}_u(u) = \argmax_{k \in [K]} \sum_{v \,:\, \tilde{\sigma}_u(v) = k ,\, v\neq u} 
         \sum_{l=0}^L \log \frac{\hat{P}_l}{\hat{Q}_l} \mathbf{1}(A_{uv} = l) 
     \]    
\EndFor 
\State Let $\hat{\sigma}(1) = \hat{\sigma}_1(1)$  \Comment{Consensus Stage}
\For{each node $u \neq 1$}
\[
\hat{\sigma}(u) = \argmax_{k \in [K]} | \{ v \,:\,  \hat{\sigma}_1(v) = k \} \cap
                                 \{ v \,:\, \hat{\sigma}_u(v) = \hat{\sigma}_u(u) \}|
\]
\EndFor
\State Output $\hat{\sigma}$
\end{algorithmic}

\end{algorithm}

\begin{remark}
  \label{rem:simple_computation}
In our simulation studies, we find that it is sufficient to output a single clustering $\tilde{\sigma}$ on the whole of $A_{l^*}$ in the initialization stage. In the refinement stage, we simply estimate $\{\hat{P}_l, \hat{Q}_l\}_{l \in \{0,\ldots, L\}}$ based on $\tilde{\sigma}$, assign $\hat{\sigma}(u) =  \argmax_{k \in [K]} \sum_{v \,:\, \tilde{\sigma}(v) = k ,\, v\neq u} 
         \sum_{l=0}^L \log \frac{\hat{P}_l}{\hat{Q}_l} \mathbf{1}(A_{uv} = l) $, and then output $\hat{\sigma}$ directly. We also note that one could in practice use a discretization level for the refinement stage that is different from that of the initialization stage (see discussions in Section~\ref{sec:proofs}).
\end{remark}


\section{Optimal misclustering error}
\label{sec:rate}

We analyze the rate of convergence of the estimation algorithm from Section~\ref{sec:method} in Section~\ref{sec:upper_bound}. In Section~\ref{sec:lower_bound}, we provide a matching information-theoretic lower bound. In both sections, we let $\mathcal{P}$ denote the set of probability distributions on $S$ whose singular part is a point mass at 0.

\subsection{Upper bound}
\label{sec:upper_bound}
We begin by stating a condition on the function $\Phi$.
\begin{definition}
  \label{defn:transformation}
Let $S$ be $[0,1]$, $\R$, or $\R^+$. We say that $\Phi \,:\, S \rightarrow [0,1]$ is a \emph{transformation function} if it is a differentiable bijection and $\phi := \Phi'$ satisfies $\left| \frac{\phi'(x)}{\phi(x)} \right| < \infty$. 
\end{definition}

For $S = [0,1]$, we always take $\Phi$ to be the identity. For $S = \R$ or $[0, \infty)$, we choose the function $\Phi$ so that all moments exist and $\phi$ has a subexponential tail. The specific choice of $\Phi$ is not crucial, and we will use the following definitions:
\begin{align}
\phi(x) = \frac{e^{1-\sqrt{x+1}}}{4},  
     \, \textrm{ if $S = [0, \infty)$}, \qquad
\phi(x) = \frac{e^{1-\sqrt{|x|+1}}}{8}, 
   \, \textrm{ if $S = \R$.} \label{eqn:phi_defn}
\end{align}

These expressions are similar to a generalized normal density, modified so that $\left| \frac{\phi'(x)}{\phi(x)} \right|$ is bounded. It is easy to verify that $\Phi(x) = \int_0^x \phi(t) dt$ (respectively, $\Phi(x) = \int_{-\infty}^x \phi(t) dt$) is a valid transformation function. The function $\Phi$ induces a probability measure on $S$, and we let $\Phi \{ \cdot \}$ denote the $\Phi$-measure of a set. 

We describe our regularity conditions by defining an appropriate subset of $\mathcal{P}^2$. For $C \in [1, \infty)$, $c_1, c_2 \in \textrm{int}(S)$, $r > 2$, and $t \in (2/r, 1)$, we define $\mathcal{G}_{\Phi, C, c_1, c_2, r, t} \subset \mathcal{P}^2$ such that $((P_0, p), (Q_0,q)) \in \mathcal{G}_{\Phi, C, c_1, c_2, r, t}$ if and only if

\begin{enumerate}
\item[A0] We have $\frac{1}{C} \leq \frac{1-P_0}{1-Q_0} \leq C$ and $\frac{1}{C} \leq \frac{P_0}{Q_0} \leq C$.
\item[A1] For all $x$ in the interior of $S$, $0 < p(x), q(x) \leq C \phi(x)$.
\item[A2] There exists a quasi-convex $g : S \rightarrow [0, \infty)$ such that $g(x) \geq \bigl| \log \frac{p(x)}{q(x)} \bigr|$ and $\int_S g(x)^r \phi(x) \, dx \leq C$.
\item[A3]   Denoting $\alpha^2 := \int_S (\sqrt{p(x)} - \sqrt{q(x)})^2 \, dx$ and $\gamma(x) := \frac{p(x) - q(x)}{\alpha}$, we have
  \[
    \int_S \biggl( \frac{\gamma(x)}{p(x) + q(x)} \biggr)^r (p(x) + q(x))\, dx \leq C.
    \]
\item[A4] There exists a quasi-convex function $h : S \rightarrow [0, \infty)$ such that
  \[
  h(x) \geq  \frac{1}{\phi(x)} \max \biggl\{ \Big| \frac{\gamma(x)}{p(x) + q(x)} \Big| \; \Big| \frac{\gamma'(x)}{p(x) + q(x)} \Big| \; \Big|\frac{q'(x)}{q(x)} \Big|, \; \Big|\frac{p'(x)}{p(x)}\Big| \biggr\}
  \]
  and $\int_S |h(x)|^{t} \phi(x) dx \leq C.$
\item[A5] We have $(\log p)'(x), (\log q)'(x) \geq (\log \phi)'(x)$ for all $x < c_1$, and $ (\log p)'(x), (\log q)'(x) \leq (\log \phi)'(x)$ for all $x > c_2$.\footnote{If $S = [0,\infty)$ and $g$ is non-decreasing, we only need  $(\log p)'(x), (\log q)'(x) \leq (\log \phi)'(x)$ for all $x > c_2$. \label{note:one_sided}}
\end{enumerate}

The above conditions depend on the choice of $\Phi$, but it generally suffices to choose $\Phi$ such that its derivative $\phi$ is a heavy-tailed density where all moments exist. In particular, we show in Section~\ref{sec:examples_unbounded_support} that choosing $\Phi$ according to equation~\eqref{eqn:phi_defn} allows $\mathcal{G}_\Phi$ to encompass Gaussian, Laplace, and other broad classes of densities. We also provide an intuitive discussion of the regularity conditions in Section~\ref{sec:assumption_discussion} below.

We now state our upper bound. For a given clustering $\sigma_0$ and $((P_0, p), (Q_0, q)) \in \mathcal{P}^2$, let the random network $A$ be distributed according to $WSBM(\sigma_0, (P_0, p), (Q_0, q))$. 
\begin{theorem}
  \label{thm:weighted_sbm_rate}
Let $\sigma_0 \in \mathcal{C}(\beta, K)$. Let $C \geq 1$, $c_1, c_2 \in \mathrm{int}(S)$, $r > 2$, and $t \in (2/r, 1)$, and let $\Phi$ be a transformation function. Define $\mathcal{G}_{\Phi} := \mathcal{G}_{\Phi, C, c_1, c_2, r, t}$. Let $\{I_n, I'_n\}_{n \in \mathbb{N}}$ be arbitrary sequences such that $I_n \rightarrow 0$ and $n I'_n \rightarrow \infty$. Let $L_n$ be a sequence such that $\frac{n I'_n}{ L_n \exp(L_n^{2/r}) } \rightarrow \infty$. Let $\hat{\sigma}_{\Phi, L_n}$ be the algorithm described in Section~\ref{sec:method} with transformation function $\Phi$ and discretization level $L_n$. Then there exists a sequence of real numbers $\zeta_n \rightarrow 0$ such that
\[
\lim_{n \rightarrow \infty} \sup_{ \substack{((P_0, p), (Q_0, q)) \in \mathcal{G}_{\Phi} \,:\,  \\ I'_n \leq I((P_0, p), (Q_0, q)) \leq I_n}}  \mathbb{P}_{\substack{(P_0, p),\\ (Q_0,q)}} \left\{
     l(\hat{\sigma}_{\Phi,L_n}(A), \sigma_0) \leq \exp\left( -  (1 - \zeta_n)\frac{n}{\beta K}I((P_0, p), (Q_0, q)) \right)
    \right\} = 1.
  \]
  Furthermore, if $\frac{nI_n}{\beta K \log n} \leq 1$, we have
  \[
\sup_{\substack{ ((P_0,p), (Q_0,q)) \in \mathcal{G}_{\Phi} \,:\, \\ I'_n \leq I((P_0, p), (Q_0, q)) \leq I_n} } \E_{(P_0, p), (Q_0, q)} \bigl[ l(\hat{\sigma}_{\Phi, L_n}(A), \sigma_0) \bigr] \exp\left((1 - \zeta_n)  \frac{n}{\beta K} I((P_0, p), (Q_0, q)) \right) \leq 1.
  \]
\end{theorem}

We relegate the full proof of Theorem~\ref{thm:weighted_sbm_rate}to Appendix~\ref{AppThmRate} but we provide a proof overview in Section~\ref{sec:proofs}. Since Theorem~\ref{thm:weighted_sbm_rate} involves many technical details, we first make a few high-level remarks to illustrate its implications.

\begin{remark}
It is important to observe that the supremum over $\mathcal{G}_\Phi$ appears \emph{after} the limit. Thus, an equivalent way to understand the theorem is to think of a sequence $((P_{0, n}, p_n), (Q_{0,n}, q_n))$, each term of which is a member of $\mathcal{G}_\Phi$. If $I((P_{0, n}, p_n), (Q_{0,n}, q_n))$ is $o(1)$ but $\omega( L_n \exp(L_n^{2/r}) n^{-1})$, Theorem~\ref{thm:weighted_sbm_rate} states that $\mathbb{P} \left\{
     l(\hat{\sigma}(A), \sigma_0) \leq \exp\left( - (1 + o(1)) \frac{n}{\beta K} I((P_{0, n}, p_n), (Q_{0,n}, q_n)) \right)
    \right\} \rightarrow 1$. Theorem~\ref{thm:weighted_sbm_rate} thus applies to the so-called \emph{sparse} setting where $P_0, Q_0 \rightarrow 1$. In particular, suppose there are constants $a, b > 0$ such that $ P_{0,n} = 1-\frac{a \log n}{n}$ and $Q_0 = 1- \frac{b \log n}{n}$. Then Theorem~\ref{thm:weighted_sbm_rate} states that perfect recovery is achievable if $(\sqrt{a} - \sqrt{b})^2 + \sqrt{ab} \int_S (\sqrt{p_n(x)} - \sqrt{q_n(x)})^2 \, dx > \beta K$; this generalizes the previously known result that perfect recovery for unweighted SBMs when $p = 1 - \frac{a \log n}{n}$ and $q = 1 - \frac{b \log n}{n}$ is possible if $(\sqrt{a} - \sqrt{b})^2 > \beta K$. 
\end{remark}

\begin{remark}
The assumption that there exist sequences $I_n \rightarrow 0$ and $I'_n = \omega(1/n)$ such that $I'_n \leq I((P_0, p), (Q_0, q)) \leq I_n$ is a very mild one. As our information-theoretic lower bound (cf.\ Section~\ref{sec:lower_bound}) shows, estimation consistency is impossible if a sequence $I'_n = \omega(1/n)$ such that $I((P_0, p), (Q_0, q)) \geq I'_n$ does not exist. Moreover, we observe that if $I((P_0, p), (Q_0, q)) > \beta K \frac{\log n}{n}$, then $\mathbb{P}( l(\hat{\sigma}(A), \sigma_0) = 0) \rightarrow 1$, and we are able to perfectly recover the clustering with high probability. Since the estimation problem is intrinsically easier if as $I((P_0, p), (Q_0, q))$ becomes larger, we expect the same perfect recovery guarantee to hold in the case where $I((P_0, p), (Q_0, q))$ is positively bounded away from 0. 
\end{remark}

\begin{remark}
Since $n I'_n \rightarrow \infty$, it is always possible to choose a sequence $L_n \rightarrow \infty$ satisfying the conditions of the theorem. Note that $L_n$ must grow very slowly to satisfy the condition that $\frac{n I'_n}{L_n \exp(L_n^{2/r})} \rightarrow \infty$; indeed, our simulation studies (cf.\ Section~\ref{sec:simulation}) confirm that we should choose the discretization level to be very small in order to achieve good performance. We note that $L_n$ has a second-order effect on the rate and appears in the $\zeta_n$ term. 
\end{remark}

\subsubsection{Additional discussion of the conditions}
\label{sec:assumption_discussion}

It is crucial to note that our algorithm does \emph{not} require prior knowledge of the form of $p(\cdot)$ and $q(\cdot)$; the same algorithm and guarantees apply so long as $((P_0, p), (Q_0, q)) \in \mathcal{G}_{\Phi, C, c_1, c_2, r, t}$ for some universal constants $C, c_1, c_2, r$, and $t$. To aid the reader, we now provide a brief, non-technical interpretation of the regularity conditions described above.

Condition A1 is simple; the last part states that $\phi$ must have a tail at least as heavy as that of $p(\cdot)$ and $q(\cdot)$. Condition A2 requires that the likelihood ratio be integrable. It is analogous to a bounded likelihood ratio condition, but much weaker; we add a mild quasi-convexity constraint for technical reasons related to the analysis of binning. In condition A3, the function $\gamma(\cdot)$ is of constant order in the sense that $\int_S \bigl( \frac{\gamma(x)}{p(x) + q(x)} \bigr)^r (p(x) + q(x)) dx \leq C$. Requirements on $\gamma(\cdot)$ translate into convergence statements on $|p - q|$: For instance, an $L_\infty$-bound on $\gamma$ implies almost uniform convergence (with respect to $\Phi$) of $|p - q|$ to 0. The integrability condition we impose on $\gamma(\cdot)$ in condition A3 is analogous to an $L_\infty$-bound, but much weaker.

Condition A4 controls the smoothness of the derivatives of $\log p(\cdot)$ and $\log q(\cdot)$. Condition A5 is a mild shape constraint on $p(\cdot)$ and $q(\cdot)$. When $S=\R$, this condition essentially requires $p(\cdot)$ and $q(\cdot)$ to be monotonically increasing in $x$ for $x \rightarrow -\infty$, and decreasing in $x$ for $x \rightarrow \infty$.

\subsubsection{Examples for $S = [0,1]$}
\label{sec:examples_bounded_support}

When $S=[0,1]$, we can always take $\Phi$ to be the identity---we do not need a transformation, but we keep the same notation in order to present our results in a unified manner. The simplest example of $\mathcal{G}_\Phi$ that satisfy conditions A1--A5 is when, for all $((P_0, p), (Q_0, q)) \in \mathcal{G}_\Phi$, the densities $p(\cdot)$ and $q(\cdot)$ are bounded above and below by strictly positive universal constants, and when the function $x \mapsto \frac{p(x) - q(x)}{\alpha}$ and its derivative are bounded by universal constants.

\subsubsection{Examples for $S = \R$ or $[0, \infty)$}  
\label{sec:examples_unbounded_support}

We begin with a proposition that characterizes conditions A1--A5 in the setting where $p(\cdot) = e^{f_{\theta_1}(\cdot)}$ and $q(\cdot) = e^{f_{\theta_0}(\cdot)}$, for some parametrized family $\{ f_{\theta} \}_{\theta \in \Theta}$. This result allows us to generate several large classes of examples.

\begin{proposition}
  \label{prop:theta_rate}
  Let $C^{**} \in [1, \infty)$, $c_1, c_2 \in S$, $r > 2$, and $t \in (2/r, 1/2)$. Let $\Theta \subset \mathbb{R}^d$ be compact and suppose $\textrm{diam}(\Theta) < 1 \wedge \frac{1}{2C^{** 2}}$. Let $\{ f_\theta \}_{\theta \in \Theta}$ be a collection of functions such that $e^{f_\theta(\cdot)}$ is a density and:
  \begin{enumerate}
  \item[B1] For all $\theta \in \Theta$ and all $x \in S$, we have $0 < e^{f_\theta(x)} \leq C^* \phi(x)$.
  \item[B2] We have $\inf_{\theta \in \Theta} \lambda_{\min} \bigl( \int_S 2 \nabla f_\theta(x) (\nabla f_\theta(x))^\top \phi(x) \,dx \bigr) \geq C^{*-1}$ and \\
            $ \sup_{\theta \in \Theta}  \int_S \lambda_{\max} \bigl(H (f_\theta)(x)\bigr)^2 \phi(x) \,dx \leq C^*$.
  \item[B3] There exists a quasi-convex function $g^* : S \rightarrow [0,\infty)$ such that $g^*(x) \geq \sup_{\theta} \| \nabla f_\theta(x) \|_2$ and $\int_S g^*(x)^r \phi(x) \,dx \leq C^*$.
  \item[B4] There exists a quasi-convex function $h^* : S \rightarrow [0, \infty)$ such that
    $$h^*(x) \geq \frac{1}{\phi(x)}
    \max\biggl\{ \sup_{\theta \in \Theta} \| \nabla f_\theta(x) \|_2,\;
    \sup_{\theta \in \Theta} \|\nabla f_{\theta}'(x)\|_2,\;
    \sup_{\theta \in \Theta} |f'_{\theta}(x)|
    \biggr\}
    $$
    and $\int_S h^*(x)^{2t} \phi(x) \, dx \leq C^*$.
  \item[B5] For all $x \leq c_1$, we have $\inf_{\theta \in \Theta} f_\theta'(x) \geq (\log \phi)'(x)$, and for all $x \geq c_2$, we have $\sup_{\theta \in \Theta} f_\theta'(x) \leq (\log \phi)'(x)$.\footnote{If $S = [0,\infty)$ and $g^*$ is non-decreasing, we only need  $\sup_{\theta \in \Theta} f_\theta'(x) \leq (\log \phi)'(x)$ for all $x \geq c_2$.  \label{note:one_sided_theta}}
\end{enumerate}

Then there exists $C \in [1, \infty)$ such that for any $\theta_1, \theta_2 \in \Theta$ and any $P_0, Q_0 \in [0,1]$ such that $\frac{1}{C} \leq \frac{P_0}{Q_0}, \frac{1 - P_0}{1-Q_0} \leq C$, we have $((P_0, e^{f_{\theta_1}}), (Q_0, e^{f_{\theta_2}})) \in \mathcal{G}_{\Phi, C, c_1, c_2, r, t}$. 
\end{proposition}

In all the examples below, we take $\Phi$ to be the transformation function defined in equation~\eqref{eqn:phi_defn}. The proofs of all statements in the examples are provided in Section~\ref{sec:appendix_examples}

\begin{example} [Location-scale family over $\R$]
  \label{ex:location_scale}

Let $f \,:\, \mathbb{R} \rightarrow \mathbb{R}$ be a continuously differentiable function such that $\int_{-\infty}^\infty e^{f(x)} \, dx = 1$. Suppose
\begin{itemize}
\item[(a)] $|f^{(k)}(x)|$ is bounded for some $k \geq 2$, and
\item[(b)] there exist $c, M > 0$ such that $f'(x) > M$ for $x < -c$ and $f'(x) < - M$ for $x > c$. 
\end{itemize}
For any $\mu \in \mathbb{R}$ and $\sigma > 0$, define $f_{\mu,\sigma} (x) := f\left( \frac{x - \mu}{\sigma} \right) - \log \sigma$.

Then there exists $C_\mu > 0$ and $c_{\sigma} > 1$ such that, with $\Theta := [- C_{\mu}, C_{\mu}] \times [\frac{1}{c_{\sigma}}, c_{\sigma}]$, the family $\{f_{\mu,\sigma}\}_{(\mu,\sigma) \in \Theta}$ satisfies conditions B1--B5 in Proposition~\ref{prop:theta_rate} with respect to $\phi$ defined in equation~\eqref{eqn:phi_defn}, and some universal constants $C^{**}, c_1, c_2, r$, and $t$. As a direct consequence of Proposition~\ref{prop:theta_rate}, for some universal constant $C > 0$, if we fix any $((\mu_1, \sigma_1), (\mu_0, \sigma_0)) \in \Theta^2$ and define
\begin{align}
  p(x) =  \frac{1}{\sigma_{1}} \exp\left( f\left( \frac{x - \mu_{1}}{\sigma_{1}} \right) \right),  \text{ and } 
  q(x) =  \frac{1}{\sigma_{0}} \exp\left( f\left( \frac{x - \mu_{0}}{\sigma_{0}} \right) \right),   
\end{align}
then $((P_0, p), (Q_0, q)) \in \mathcal{G}_{\Phi, C, c_1, c_2, r, t}$ for any $P_0, Q_0 \in [0,1]$ that satisfy condition A0. 

These assumptions on $f$ are satisfied for \textbf{Gaussian location-scale families}, where the base density is the standard Gaussian density with $f(x) = -x^2 - \frac{1}{2} \log 2 \pi$, and \textbf{Laplace location-scale families}, where the base density is the standard Laplace density with $f(x) = - |x| - \log 2$.
\end{example}

\begin{example} [Scale family over $[0,\infty)$]
  \label{ex:location_scale}

Let $f \,:\, [0,\infty) \rightarrow \mathbb{R}$ be a continuously differentiable function such that $\int_0^\infty e^{f(x)} \, dx = 1$. Suppose
\begin{itemize}
\item[(a)] $|f^{(k)}(x)|$ is bounded for some $k \geq 2$, and
\item[(b)] there exist $c, M > 0$ such that $f'(x) < - M$ for $x > c$. 
\end{itemize}
For any $\sigma > 0$, define $f_{\sigma} (x) := f\left( \frac{x}{\sigma} \right) - \log \sigma$.

Then there exists $c_{\sigma} > 1$ such that, with $\Theta := [\frac{1}{c_{\sigma}}, c_{\sigma}]$, the family $\{f_{\sigma}\}_{\sigma \in \Theta}$ satisfies conditions B1--B5 in Proposition~\ref{prop:theta_rate} with respect to $\phi$ defined in equation~\eqref{eqn:phi_defn}, and some universal constants $C^{**}, c_1, c_2, r$, and $t$. As a direct consequence of Proposition~\ref{prop:theta_rate}, for some universal constant $C > 0$, if we fix any $(\sigma_1, \sigma_0) \in \Theta^2$ and define
\begin{align*}
  p(x) =  \frac{1}{\sigma_{1}} \exp\left( f\left( \frac{x}{\sigma_{1}} \right) \right),  \text{ and } 
  q(x) =  \frac{1}{\sigma_{0}} \exp\left( f\left( \frac{x}{\sigma_{0}} \right) \right),  
\end{align*}
then $((P_0, p), (Q_0, q)) \in \mathcal{G}_{\Phi, C, c_1, c_2, r, t}$ for any $P_0, Q_0 \in [0,1]$ that satisfy condition A0. 

These assumptions on $f$ are satisfied for \textbf{exponential scale families}, where the base density is the standard exponential density with $f(x) = -x$.
\end{example}

Proposition~\ref{prop:theta_rate} also applies to the family of Gamma distributions, see Proposition~\ref{prop:gamma_family} in the appendix. 



  

In this paper, we only study continuous edge weights in detail; in practice, discrete edge weights such as counts are also important. Although Theorem~\ref{thm:weighted_sbm_rate} does not  apply directly to such cases, our analysis is relevant to some instances of SBMs with discrete edge weights. In Appendix~\ref{sec:discrete_weights}, we discuss a crude way to handle count-valued edge weights, with particular attention toward Poisson-distributed edge weights. 


\subsection{Lower bound}
\label{sec:lower_bound}

Our information-theoretic lower bound applies to any permutation equivariant estimators (Definition~\ref{defn:permutation_equivariance_deterministic}). Before stating the result, we define an appropriate subset of $\mathcal{P}^2$ to capture the conditions we need on $((P_0, p), (Q_0, q))$. Let $C^* \in [1, \infty)$, and let $\mathcal{G}^*_{C^*} \subset \mathcal{P}^2$ be such that $((P_0, p), (Q_0, q)) \in \mathcal{G}^*$ if and only if
\begin{enumerate}
\item[$A0^*$] $\frac{1}{C^*} \leq \frac{P_{0}}{Q_{0}} \leq C^*$, and
\item[$A1^*$] $\int_S (p(x) + q(x)) \bigl|\log  \frac{p(x)}{q(x)} \bigr|^2 \, dx \leq C^* \int_S (p(x)^{1/2} - q(x)^{1/2})^2 \, dx$.
\end{enumerate}

Condition A1$^*$ is similar to A2 and A3 in the definition of the set of regular distributions $\mathcal{G}_{{\Phi}, C, c_1, c_2, r, t}$ that appears in the upper bound (Theorem~\ref{thm:weighted_sbm_rate}). In fact, if $\int_S (p^{1/2} - q^{1/2})^2 dx$ is bounded away from 0, then there exists $C^*$ such that A$1^*$ is equivalent to A2. Thus, although $\mathcal{G}^*_{C^*}$ is in general not a superset of $\mathcal{G}_{{\Phi}, C, c_1, c_2, r, t}$, the set $\mathcal{G}^*_{C^*} \cap \mathcal{G}_{{\Phi}, C, c_1, c_2, r, t}$ contains important and interesting examples. For instance, any family that satisfies the conditions of Proposition~\ref{prop:theta_rate} belongs to the intersection, as is verified in the proof (cf.\ Appendix~\ref{sec:theta_rate_proof}).

\begin{theorem}
  \label{thm:lower_bound}
Let $C^* \geq 1$ and let $\sigma_0 \,:\, [n] \rightarrow [K]$ be a clustering such that one cluster is of size $\frac{n}{\beta K}$ and another is of size $\frac{n}{\beta K} + 1$. Let $I'_n$ be any sequence such that $nI'_n \rightarrow \infty$, and let $C = 2\log 2$. Then there exists $\zeta_n \rightarrow 0$ and $c' >0$ such that, for any permutation equivariant algorithm $\hat{\sigma}$,
\begin{align*}
 \inf_{\substack{((P_0, p), (Q_0, q)) \in \mathcal{G}^*_{C^*} \\ I'_n \leq I((P_0, p), (Q_0, q)) \leq C}} \mathbb{E}_{\substack{ (P_{0}, p) \\ (Q_{0}, q) }} \bigl[ \ell\bigl( \hat{\sigma}(A), \sigma_0 \bigr)\bigr] 
     \exp \left( \frac{n}{\beta K}  I\bigl((P_{0}, p), (Q_{0}, q)\bigr) (1 + \zeta_n) \right) \geq c'.
\end{align*}
Furthermore, for any $c > 0$, there exists $c' > 0$ such that for any permutation equivariant algorithm $\hat{\sigma}$,
\begin{align*}
    \inf_{\substack{((P_0, p), (Q_0, q)) \in \mathcal{G}^*_{C^*} \\ I((P_0, p), (Q_0, q)) \leq c/n}}    \mathbb{E}_{\substack{ (P_{0}, p) \\ (Q_{0}, q) }} \bigl[ \ell\bigl( \hat{\sigma}(A), \sigma_0\bigr) \bigr] \geq c'.
\end{align*}

\end{theorem}

Theorem~\ref{thm:lower_bound} shows that if $n I_n \rightarrow \infty$, the misclustering risk of any permutation equivariant algorithm is at least $\exp \left( - (1 + o(1)) \frac{ n I\bigl((P_0, p), (Q_0, q)\bigr)}{\beta K} \right)$. If $n I\bigl((P_0, p), (Q_0, q)\bigr) = O(1)$, any permutation invariant algorithm is inconsistent. 

\begin{remark}
  Rather than being a minimax lower bound that applies to the worst case, Theorem~\ref{thm:lower_bound} applies to \emph{any} parameter $((P_0, q), (Q_0 q)) \in \mathcal{G}^*_{C^*}$; we thus have an infimum over the parameter space rather than a supremum. This is possible because the permutation equivariance condition excludes the trivial case where $\hat{\sigma} = \sigma_0$. 
\end{remark}

\paragraph{\textbf{Proof sketch of Theorem~\ref{thm:lower_bound}:}} The full proof of the theorem is provided in Appendix~\ref{sec:lower_bound_proof}; we highlight key points here. The proof borrows elements from Yun and Proutiere~\citep{yun2016optimal} and Zhang and Zhou~\citep{zhangminimax}. One key difference is that Theorem~\ref{thm:lower_bound} holds for any parameters in the parameter space, rather than adopting a minimax framework, as in Zhang and Zhou~\cite{zhangminimax}, or assuming a prior on $\sigma_0$, as in Yun and Proutiere~\cite{yun2016optimal}.

Crucial to the proof is the notion of a \emph{misclustered} node. Let
\[
S_K[ \hat{\sigma}(A), \sigma_0] := \argmin_{\rho \in S_K} d_H(\rho \circ \hat{\sigma}(A), \sigma_0).
\] 
It is straightforward to define a set of misclustered nodes when $S_K[\hat{\sigma}(A), \sigma_0]$ is a singleton, but care must be taken when $S_K[\hat{\sigma}(A), \sigma_0]$ contains multiple elements.
We define the set of \emph{misclustered nodes} as
\begin{align}
\mathcal{E}[ \hat{\sigma}(A), \sigma_0] \! := \! \Big\{ v: (\rho \circ \hat{\sigma}(A))(v) \neq  \sigma_0(v),
          \text{ for some } \rho \in S_K[\hat{\sigma}(A), \sigma_0]   \Big\}. 
\end{align}
To see an example of the subtlety that arises when $S_K[ \hat{\sigma}(A), \sigma_0]$ is not a singleton, note that the ``$\text{for some } \rho \in S_K[\hat{\sigma}(A), \sigma_0]$'' qualifier in the definition of $\mathcal{E}[\hat{\sigma}(A), \sigma_0]$ cannot be replaced with ``$\text{for all } \rho \in S_K[\hat{\sigma}(A), \sigma_0]$.'' Otherwise, in the case where the clusters in $\sigma_0$ all have the same size, and where $\hat{\sigma}(A)$ is a trivial algorithm that maps all nodes to cluster 1, the set $S_K[\hat{\sigma}(A), \sigma_0]$ equals $S_K$ and $\mathcal{E}[ \hat{\sigma}(A), \sigma_0]$ would be empty. 

With the definition of $\mathcal{E}[ \hat{\sigma}(A), \sigma_0]$, we may formalize symmetry properties of permutation equivariant estimators. For example, if $A$ is distributed according to a weighted SBM with $\sigma_0$ as the true cluster assignment, and if nodes $u$ and $v$ lie in clusters of equal sizes, then $P(u \in \mathcal{E}[\hat{\sigma}(A), \sigma_0]) = P( v \in \mathcal{E}[\hat{\sigma}(A), \sigma_0])$ for any permutation equivariant $\hat{\sigma}$ (cf.\ Corollary~\ref{cor:permutation_equivariance_symmetry_sbm}).

Without loss of generality, let cluster 1 and 2 be the clusters in $\sigma_0$ that have sizes $\frac{n}{\beta K} + 1$ and $\frac{n}{ \beta K}$, respectively, and let node 1 belong to cluster 1. Let $\sigma^*$ be a random cluster assignment where $\sigma^*(u) = \sigma_0(u)$ for all $u \neq 1$, and $\sigma^*(1)$ is $1$ or $2$ with probability $\frac{1}{2}$ each. Let $\mathbb{P}_{SBM}( \cdot \given \sigma^*)$ denote the distribution on $\mathbb{R}^{n \times n}$ induced by the random cluster assignment $\sigma^*$. We perturb $\mathbb{P}_{SBM}( \cdot \given \sigma^*)$ to define a new distribution $\mathbb{P}_{\Psi}(\cdot)$ on $\mathbb{R}^{n \times n}$, where under $P_{\Psi}$, the $A_{uv}$'s are independent; and if $u = 1$ and $v$ is in cluster 1 or 2, then $A_{uv}$ is distributed according to a new probability distribution $Y^*$ instead of $P$ or $Q$. If $u \neq 1$ or if $v$ is not in cluster 1 or 2, then $A_{uv}$ is distributed according to $P_{SBM}(A \given \sigma^*)$. We take $Y^* := \arg\min_{Y} \max \left\{ \int_S \log \frac{dY}{dP} dY, \, \int_S \log \frac{dY}{dQ} dY \right\}$, constructed to be similar to both $P$ and $Q$ (see Lemma~\ref{lem:information_equivalence}). 

Under $\mathbb{P}_{\Psi}(\cdot)$, we can show that it is impossible to consistently cluster node 1 with respect to $\sigma^*$ as the true cluster assignment. We then use the fact that $\mathbb{P}_{\Psi}(A)$ and $\mathbb{P}_{SBM}(A \given \sigma^*)$ are similar to deduce the difficulty of correctly clustering node 1 under $\mathbb{P}_{SBM}(A \given \sigma^*)$. Permutation equivariance translates this into a result on the number of misclustered nodes in a way similar to Lemma~2.1 in Zhang and Zhou~\cite{zhangminimax}---the distinction being that Zhang and Zhou~\cite{zhangminimax} uses a uniform prior over the true clustering to transform arbitrary estimators into permutation equivariant ones, whereas we place no prior on the true clustering, but restrict estimators to be permutation equivariant. Finally, we finish the proof by using permutation equivariance again to show that the misclustering error does not depend on whether $\sigma^*(1) = 1$ or $\sigma^*(1) = 2$. 
\subsection{Adaptivity}

Let $\mathcal{F}^{p.e.}_n$ be the class of permutation equivariant clustering algorithm on networks with $n$ nodes. Theorems~\ref{thm:weighted_sbm_rate} and~\ref{thm:lower_bound} directly imply the following corollary, which sharply characterizes the optimal performance of $\mathcal{F}^{p.e.}_n$:

\begin{corollary}
  \label{cor:optimality_combined}
 Let $\sigma_0 \,:\, [n] \rightarrow [K]$, and suppose one cluster is of size $\frac{n}{\beta K}$ and another is of size $\frac{ n}{\beta K}+1$. Let $C^*, C \geq 1, c_1, c_2 > 0, r > 0$, and $t \in (2/r, 1)$, and let $\Phi$ be a transformation function. Write $\mathcal{G}_{\Phi} := \mathcal{G}_{\Phi, C, c_1, c_2, r, t}$, $\mathcal{G}^* := \mathcal{G}^*_{C^*}$, and $\Lambda := \{\beta, K, C^*, C, c_1, c_2, r, t, \Phi\}$. Let $\bigl( (P_{0,n}, p_n), (Q_{0,n}, q_n) \bigr) \in \mathcal{G}_{\Phi} \cap \mathcal{G}^*$ for every $n \in \mathbb{N}$. 

  \begin{itemize}
  \item[(i)] If $\limsup_n I\bigl( (P_{0,n}, p_n), (Q_{0,n}, q_n) \bigr) \frac{n}{\beta K \log n} \leq 1$,  there exists $\zeta_n \rightarrow 0$, depending only on $\Lambda$, such that
    \[
      \inf_{\hat{\sigma} \in \mathcal{F}^{p.e.}_n} \E_{\substack{ (P_{0,n}, p_n) \\ (Q_{0,n}, q_n)} } \bigl[ l(\hat{\sigma}(A), \sigma_0) \bigl] = \exp\left( - \frac{nI\bigl( (P_{0,n}, p_n), (Q_{0,n}, q_n) \bigr)}{\beta K} (1 + \zeta_n) \right).
    \]
  \item[(ii)] If $\liminf_n I\bigl( (P_{0,n}, p_n), (Q_{0,n}, q_n) \bigr) \frac{n}{\beta K \log n} > 1$,  there exists $\zeta_n \rightarrow 0$, depending only on $\Lambda$, such that $ \inf_{\hat{\sigma} \in \mathcal{F}^{p.e.}_n} \mathbb{P}_{(P_{0,n}, p_n), (Q_{0,n}, q_n)}\bigl( l(\hat{\sigma}(A),\sigma_0) > 0 \bigr) \leq \zeta_n$.
      
    \item[(iii)] If there exists $c > 0$ such that $\limsup_n I\bigl( (P_{0,n}, p_n), (Q_{0,n}, q_n) \bigr) n < c$, there exists $c' > 0$ such that  $\liminf_{n \rightarrow \infty} \inf_{\hat{\sigma} \in \mathcal{F}^{p.e.}_n} \E_{(P_{0,n}, p_n), (Q_{0,n}, q_n)} \bigl[ l(\hat{\sigma}(A), \sigma_0) \bigr] > c'$.
      
\end{itemize}
\end{corollary}

The algorithm $\hat{\sigma}$ described in Section~\ref{sec:algorithm} with discretization level $L_n$ diverging sufficiently slowly achieves the optimal rate in part (i) and (ii) for \emph{any} $((P_{0,n}, p_n), (Q_{0,n}, q_n)) \in \mathcal{G}_\Phi \cap \mathcal{G}^*$. Thus, $\hat{\sigma}$ adapts to the edge probabilities $P_{0,n}$ and $Q_{0,n}$ and the edge weight densities $p_n$ and $q_n$: Although $\hat{\sigma}$ has no knowledge of the parameters $((P_{0,n}, p_n), (Q_{0,n}, q_n))$, it achieves the same optimal rate as if $((P_{0,n}, p_n), (Q_{0,n}, q_n))$ were known. 

In particular, this implies that one does not have to pay a price for taking the nonparametric approach. This seemingly counterintuitive phenomenon arises because the cost of discretization is reflected in the lower-order $\zeta_n$ term in the exponent. As an illustrative example, suppose $1 - P_{0,n} = 1 - Q_{0,n} = a \frac{\log n}{n}$ for some $a > 0$, and the densities $p_n$ and $q_n$ are of $N(\mu_1, \sigma_1^2)$ and $N(\mu_0, \sigma_0^2)$, respectively. Then $I_n = (1 + o(1)) \frac{a \log n}{n} \theta$, where $\theta = 2 \left(1 - \sqrt{ \frac{2 \sigma_1^2 \sigma_0^2}{\sigma_1^2 + \sigma_0^2} } e^{- \frac{1}{4} \frac{(\mu_1 - \mu_0)^2}{\sigma_1^2 + \sigma_0^2}} \right)$, and the optimal rate is $n^{-(1+o(1)) \frac{2 \theta}{\beta K}}$, which is attained by the nonparametric discretization estimator $\hat{\sigma}$.

Similarly, if $1 - P_{0,n} = 1 - Q_{0,n} = \frac{a \log n}{n}$ and the densities $p_n$ and $q_n$ are $\textrm{Exp}(\lambda_1)$ and $\textrm{Exp}(\lambda_0)$, respectively, then $I_n = (1 + o(1)) \frac{\log n}{n} \theta'$, where $\theta' = 2\left( 1 - \sqrt{ \frac{\lambda_1 \lambda_0}{\lambda_1 + \lambda_0}} \right)$. The optimal rate $n^{-(1+o(1)) \frac{2 \theta'}{\beta K}}$ is again achieved by the nonparametric discretization estimator $\hat{\sigma}$.

\section{Proof sketch: Recovery algorithm}
\label{sec:proofs}

A large portion of the Appendix is devoted to proving that our recovery algorithm succeeds and achieves the optimal error rates. We provide an outline of the proofs here.

We divide our argument into propositions that focus on successive stages of our algorithm. A birds-eye view of our method reveals that it contains two major components: (1) convert a weighted network into a labeled network, and then (2) run a community recovery algorithm on the labeled network. The first component involves two steps, transformation and discretization. 
Step (1) comprises the red and green steps in Figure~\ref{fig:method_pipeline1} and outputs an adjacency matrix with discrete edge weights. Step (2) is denoted in blue.

\begin{figure}[htp]
\centering
\includegraphics[scale=0.38, trim={1in 7in 0 0.4in}, clip]{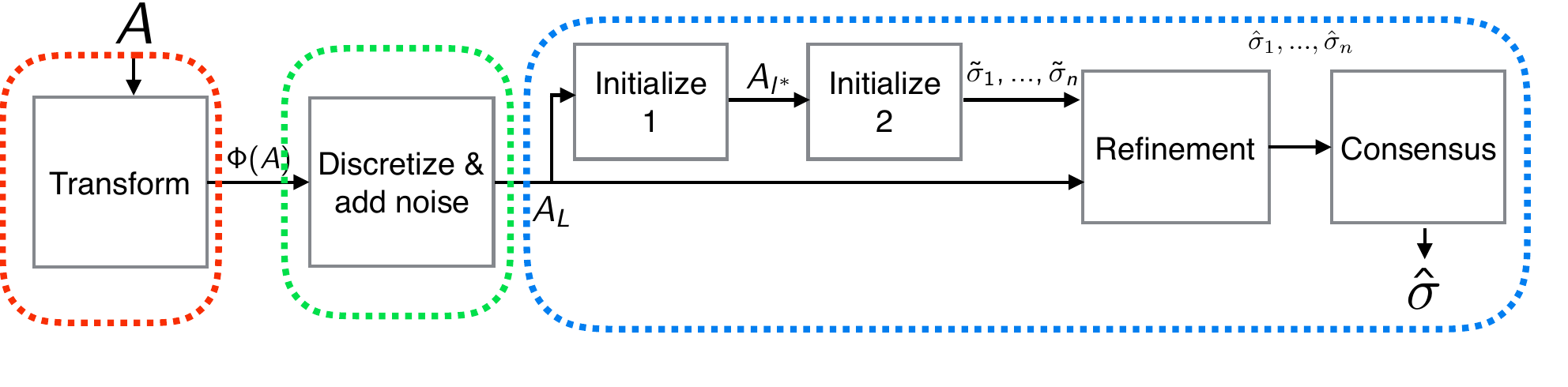}
\caption{Analysis of the right-most blue region is contained in Section~\ref{sec:labeled_sbm_analysis}, of the middle green region in Section~\ref{sec:discretization_analysis}, and of the left-most red region in Section~\ref{sec:transformation_analysis}}.
\label{fig:method_pipeline2}
\end{figure}

In our algorithm, we use a single discretization level $L$ throughout for ease of presentation. In practice, one could use different discretization levels for the initialization stage and for the refinement stage. By comparing Proposition~\ref{prop:labeled_sbm_rate}, Proposition~\ref{prop:discretization1}, and Theorem~\ref{thm:weighted_sbm_rate}, we can see that the bias introduced by discretization is a second-order effect compared to the variance, which is why the discretization level should be small in both stages. The discretization level for the initialization stage can, however, be chosen to be larger than that of the refinement stage, because the initialization stage aims to produce a consistent estimator rather than an optimal one, and can thus tolerate greater variance. More precisely, the theoretical requirements on discretization for the initialization stage are $L \rightarrow \infty$ and $\frac{n I'_n}{L} \rightarrow \infty$, whereas the requirements for the refinement stage are $L \rightarrow \infty$ and $\frac{n I'_n}{L e^{L^{r/2}}}\rightarrow \infty$ (note that $I'_n$ is defined in Theorem~\ref{thm:weighted_sbm_rate}); $L$ is required to be of smaller order to control the ratio $\frac{P_l}{Q_l}$ of the discretized probabilities. 

\subsection{Analysis of community recovery on a labeled network}
\label{sec:labeled_sbm_analysis}

We first examine the second component of our algorithm, which is a subroutine (right-most region in Figure~\ref{fig:method_pipeline1}) for recovering communities in a network where the edges have discrete labels $l=1, \dots, L_n$. The following proposition characterizes the rate of convergence of the output of the subroutine, where within-community edges are assigned edge labels with probabilities $\{P_{l}\}$, and between-community edges are assigned edge labels according to $\{Q_{l}\}$. For convenience, if an edge does not exist between $u$ and $v$, we assign the label 0 to $A_{uv}$, so $P_{0}$ and $Q_{0}$ are the edge absence probabilities.

Formally, for $L \in \mathbb{N}$, define $\mathcal{P}_L := \{ (P_0, ..., P_L) \in [0,1]^{L+1} \,:\, \sum_{l=1}^L P_l = 1 \}$. For a clustering $\sigma_0 \,:\, [n] \rightarrow [K]$ and $(\{P_l\}, \{Q_l\}) \in \mathcal{P}_L^2$, we define a Labeled Stochastic Block Model $LSBM(\sigma_0, \{P_l\}, \{Q_l\})$ as a distribution on $\{0,\ldots,L\}^{n \times n}$ such that if $A \sim LSBM(\sigma_0, \{P_l\}, \{Q_l\})$, then for any $u, v \in [n]$ such that $u > v$, 
\[
A_{uv} \sim \left\{ \begin{array}{cc}
 \{P_l\} & \text{ if } \sigma_0(u) = \sigma_0(v), \\
 \{Q_l\} & \text{ if } \sigma_0(u) \neq \sigma_0(v). 
\end{array} \right.
\]
For $\rho > 1$, let $\mathcal{G}_{L,\rho} \subset \mathcal{P}_L^2$ be such that $( \{P_l\}, \{Q_l\}) \in \mathcal{G}_{L,\rho}$ if and only if $\frac{1}{\rho} \leq \frac{P_l}{Q_l} \leq \rho$ for all $l=0,...,L$. For a pair $(\{P_l\}, \{Q_l\}) \in \mathcal{P}_L$, we define $I\bigl(\{P_l\}, \{Q_l\} \bigr) := -2\log \sum_{l=0}^{L} \sqrt{ P_l Q_l }$.

In the next proposition, for a given clustering $\sigma_0$ and $(\{P_l\}, \{Q_l\}) \in \mathcal{P}_L^2$, we let the random network $A$ have the distribution $LSBM(\sigma_0, \{P_l\}, \{Q_l\})$.

\begin{proposition}
  \label{prop:labeled_sbm_rate}
Let $\sigma_0 \in \mathcal{C}(\beta, K)$. Let $\{ I_n, I'_n, \rho_n, L_n\}_{n \in \mathbb{N}}$ be any sequences such that $I_n \rightarrow 0$, $\rho_n \geq 2$, $L_n \geq 1$, and $\frac{n I'_n}{(L_n + 1) \rho^2_n \log \rho_n} \rightarrow \infty$. Then there exists a sequence $\zeta_n \rightarrow 0$ such that
\[
\lim_{n \rightarrow \infty} \sup_{\substack{ (\{P_l\}, \{Q_l\}) \in \mathcal{G}_{L_n,\rho_n} \\ I'_n \leq I(\{P_l\}, \{Q_l\}) \leq I_n }} \mathbb{P}_{(\{P_l\}, \{Q_l\})} \left( l(\hat{\sigma}(A), \sigma_0) \leq \exp \left( - (1 - \zeta_n) \frac{ n}{ \beta K} I(\{P_l\}, \{Q_l\})  \right) \right) = 1.
\]
Furthermore, if $\frac{n I_n}{\beta K \log n} \leq 1$, then
\[
  \sup_{\substack{  (\{P_l\}, \{Q_l\}) \in \mathcal{G}_{L_n,\rho_n} \\ I'_n \leq I(\{P_l\}, \{Q_l\}) \leq I_n }} \E\bigl[ l(\hat{\sigma}(A), \sigma_0) \bigr] \exp \left( (1 - \zeta_n) \frac{ n }{ \beta K} I\bigl(\{P_l\}, \{Q_l\}\bigr) \right) \leq 1.
\]
\end{proposition}

\begin{remark}
This result resembles that of Yun and Proutiere~\cite{yun2016optimal}, who also study an SBM where the edges carry discrete labels. They state their results using a seemingly different divergence, but it coincides with the Renyi divergence when specialized to our setting (cf.\ Lemma~\ref{lem:information_equivalence}). Proposition~\ref{prop:labeled_sbm_rate} differs critically from Yun and Proutiere~\cite{yun2016optimal} in two respects, however. First, they hold the number of labels $L_n$ to be fixed and assume that the bound $\rho_n$ on the probability ratio $\frac{P_{l,n}}{Q_{l,n}}$ is fixed,  whereas we allow both ${L_n}$ and $\rho_n$ to diverge. Second, they assume that $\sum_{l=1}^{L_n} (P_{l,n} - Q_{l,n})^2$ is sufficiently large when compared to $\max_{l=1, \dots, L_n} P_{l,n}$, whereas we do not make any assumptions of this form. These generalizations are crucial in analyzing the weighted SBM, since in order to achieve consistency for continuous distributions, the discretization level $L_n$ and the bound $\rho_n$ must increase with $n$.
\end{remark}

\subsection{Discretization of the Renyi divergence}
\label{sec:discretization_analysis}

We now analyze the discretization step of the algorithm (green box in Figure \ref{fig:method_pipeline1}). The input to this step is the weighted network $\Phi(A)$ in which all the edge weights are in $[0,1]$. We use $\tilde{p}(z)$ and $\tilde{q}(z)$ for $z \in [0,1]$ to denote the densities of the transformed edge weights; the next section shows the relationship between $\tilde{p}(z)$ and $p(x)$ and $\tilde{q}(z)$ and $q(x)$. The discretization step of the algorithm divides $[0,1]$ into $L_n$ uniform bins, denoted by $[a_l, b_l]$, for $1 \le l \le L_n$. The output is a network $A_{L_n}$, where each edge is assigned label $l=1, \dots, L_n$ with probability either
\begin{align}
  P_l := (1 - P_{0}) \int_{a_l}^{b_l} \tilde{p}(z) dz,  \quad \textrm{ or } \quad Q_l := (1 - Q_{0}) \int_{a_l}^{b_l} \tilde{q}(z) dz.
  \label{eqn:Pl_Ql_defn}
\end{align}
A missing edge is assigned the label 0. It is easy to show that discretization always leads to a loss of information; i.e., $I \bigl(\{P_l\}, \{Q_l\} \bigr) \leq I\bigl((P_0, \tilde{p}), (Q_0, \tilde{q})\bigr)$.

Let $\tilde{\mathcal{P}}$ denote the set of probability distributions on $[0,1]$ whose singular part is a point mass at $0$. Let $\tilde{C} \in (0, \infty)$, $\tilde{c}_1, \tilde{c}_2 \in (0,1/2)$, $r > 2$, and $t > 0$, and define the set $\tilde{\mathcal{G}}_{\tilde{C}, \tilde{c}_1, \tilde{c}_2, r, t} \subset \tilde{\mathcal{P}}^2$ such that $((P_0, \tilde{p}), (Q_0, \tilde{q})) \in \tilde{\mathcal{G}}$ if and only if the following hold:

\begin{enumerate}
\item[C0] We have $\frac{1}{\tilde{C}} \leq \frac{1-P_0}{1-Q_0} \leq \tilde{C}$ and $\frac{1}{\tilde{C}} \leq \frac{P_0}{Q_0} \leq \tilde{C}$.
\item[C1] For all $z \in (0,1)$, we have $0 < \tilde{p}(z), \tilde{q}(z) \leq \tilde{C}$.
  \item[C2] There exists a quasi-convex $\tilde{g} : [0,1] \rightarrow [0, \infty)$ such that $\tilde{g}(z) \geq \bigl| \log \frac{\tilde{p}(z)}{\tilde{q}(z)} \bigr|$ and $\int_0^1 \tilde{g}(z)^r \,dz \leq \tilde{C}$.
  \item[C3]     Denoting $\tilde{\alpha} := \bigl\{ \int_0^1 (\sqrt{\tilde{p}(z)} - \sqrt{\tilde{q}(z)})^2 \, dz \bigr\}^{1/2}$ and
  $\tilde{\gamma}(z) := \frac{\tilde{p}(z) - \tilde{q}(z)}{\tilde{\alpha}}$, we have
  \[
    \int_0^1
    \bigl\{\frac{\tilde{\gamma}(z)}{\tilde{p}(z) + \tilde{q}(z)} \bigr\}^r (\tilde{p}(z) + \tilde{q}(z)) \, dz \leq \tilde{C}.
  \]
\item[C4] There exists a quasi-convex function $\tilde{h} : [0,1] \rightarrow [0, \infty)$ such that
  \[
    \tilde{h}(z) \geq \max \left\{\left| \frac{\tilde{\gamma}(z)}{\tilde{p}(z) + \tilde{q}(z)} \right|, \; \left|\frac{\tilde{p}'(z)}{\tilde{p}(z)} \right|, \; \left| \frac{\tilde{q}'(z)}{\tilde{q}(z)} \right|, \; \left|\frac{\tilde{\gamma}'(z)}{\tilde{p}(z) + \tilde{q}(z)}\right|  \right\}
  \]
  and $\int_0^1 \tilde{h}(z)^t dz < \tilde{C}$.
\item[C5]  We have $\tilde{p}'(z), \tilde{q}'(z) \geq 0$ for all $z < \tilde{c}_1$, and $\tilde{p}'(z), \tilde{q}'(z) \leq 0$ for all $z > 1-\tilde{c}_2$.\footnote{If $\tilde{g}$ is non-decreasing, we need only $\tilde{p}'(z), \tilde{q}'(z) \leq 0$ for all $z > 1-c'_2$. \label{note:one_sided_compact}}
\end{enumerate}

\begin{proposition}
  \label{prop:discretization1}
  Let $\tilde{C} \in (0, \infty)$, $\tilde{c}_1, \tilde{c}_2 \in (0, 1/2)$, $r > 2$, and $t > 0$. For any  $((P_0, \tilde{p}), (Q_0, \tilde{q})) \in \tilde{\mathcal{G}}_{\tilde{C}, \tilde{c}_1, \tilde{c}_2, r, t}$, for any $L \in \mathbb{N}$ such that $L \geq \tilde{c}_1^{-1} \vee \tilde{c}_2^{-1}$, and for $\{P_l, Q_l\}$ defined in equation~\eqref{eqn:Pl_Ql_defn}, we have $\frac{1}{2 \tilde{C} \exp((2 \tilde{C}L)^{1/r})} \leq \frac{P_{l}}{Q_{l}} \leq 2 \tilde{C} \exp((2 \tilde{C} L)^{1/r})$ for all $l \in \{0,\ldots L\}$. Furthermore, 
  \[
    \lim_{L \rightarrow \infty} \sup_{((P_0, \tilde{p}), (Q_0, \tilde{q})) \in \tilde{\mathcal{G}}_{\tilde{C}, \tilde{c}_1, \tilde{c}_2, r, t}}
    \left| 1 - \frac{I(\{P_l\}, \{Q_l\} \bigr)}{I\bigl((P_0, \tilde{p}), (Q_0, \tilde{q})\bigr)} \right| = 0.
  \]
\end{proposition}
We prove Proposition~\ref{prop:discretization1} in Appendix~\ref{sec:proof_of_discretization1}.

\subsection{Analysis of the transformation function}
\label{sec:transformation_analysis}

Proposition~\ref{prop:discretization1} considers densities supported on $[0,1]$. In conjunction with Proposition~\ref{prop:labeled_sbm_rate}, this suffices to obtain Theorem~\ref{thm:weighted_sbm_rate}, because the densities of the transformed edge weights are compactly supported and, importantly, the Renyi divergence is invariant with respect to the transformation function $\Phi$.

To be precise, let $p(\cdot)$ and $q(\cdot)$ denote probability densities on $S$, and for $X \sim p$ and $Y \sim q$, let $\tilde{p}(\cdot)$ and $\tilde{q}(\cdot)$ denote the densities of $\Phi(X)$ and $\Phi(Y)$. We then have $\tilde{p}(z) = \frac{p(\Phi^{-1}(z))}{\phi(\Phi^{-1}(z))}$ and $\tilde{q}(z) = \frac{q(\Phi^{-1}(z))}{\phi(\Phi^{-1}(z))}$ for $z \in [0,1]$. Therefore, via the change of variable $z = \Phi^{-1}(x)$, we have
\begin{align*}
\int_S \sqrt{p(x)q(x)} \, dx &= \int_0^1 \sqrt{\tilde{p}(z)\tilde{q}(z)}\, dz, \,\, \mathrm{and} \\
I\bigl( (P_0, p), (Q_0, q) \bigr) &= I\bigl( (P_0, \tilde{p}), (Q_0, \tilde{q}) \bigr).
\end{align*}


 \section{Simulation studies}
 \label{sec:simulation}
 
 We start with a toy example that illustrates the intuition behind our discretization-based algorithm. In this example, we have $n=1000$ nodes, $K=2$ clusters, and $P_0 = Q_0 = 0.5$. We also set $p(\cdot)$ and $q(\cdot)$ as the normal density $N(0, 1.3^2 +1)$ and mixture of normals $\frac{1}{2} N(-1.3, 1) + \frac{1}{2} N(1.3, 1)$, respectively (see Figure~\ref{fig:PQ}). Observe that $\int_{\mathbb{R}} x dP = \int_{\mathbb{R}} x dQ = 0$ and $\int_{\mathbb{R}} x^2 dP = \int_{\mathbb{R}} x^2 dQ = \frac{1}{2} (1.3^2 + 1)$. The true clustering $\sigma_0$ maps the first 500 nodes to cluster 1 and the rest to cluster 2. 

 \begin{figure}[!htp]
   \centering
   \begin{subfigure}[b]{0.13\textwidth}
     \centering
     \includegraphics[width=\textwidth]{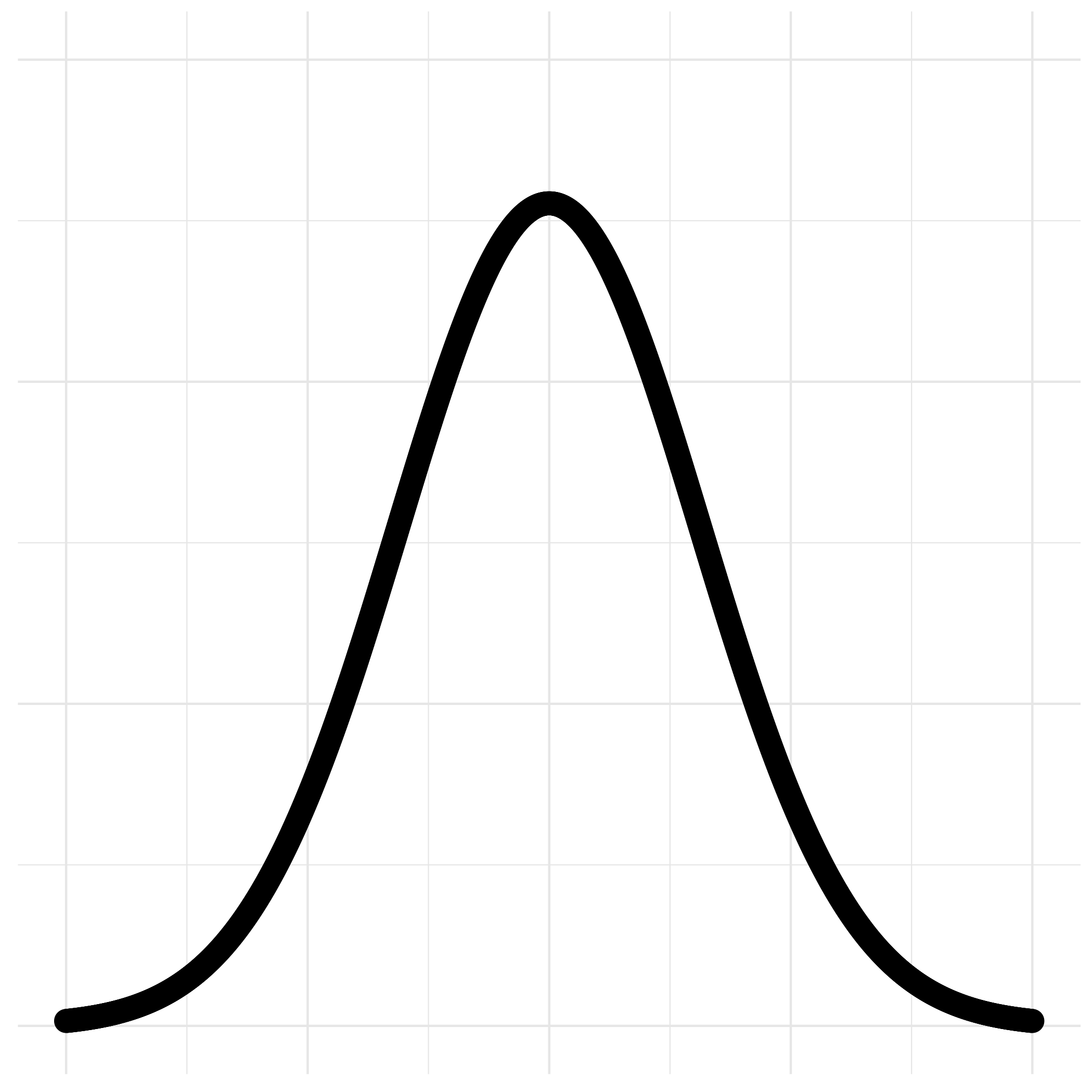}
     \caption{$p(\cdot)$}
   \end{subfigure}
   \hspace{1in}
   \begin{subfigure}[b]{0.13\textwidth}
     \centering
     \includegraphics[width=\textwidth]{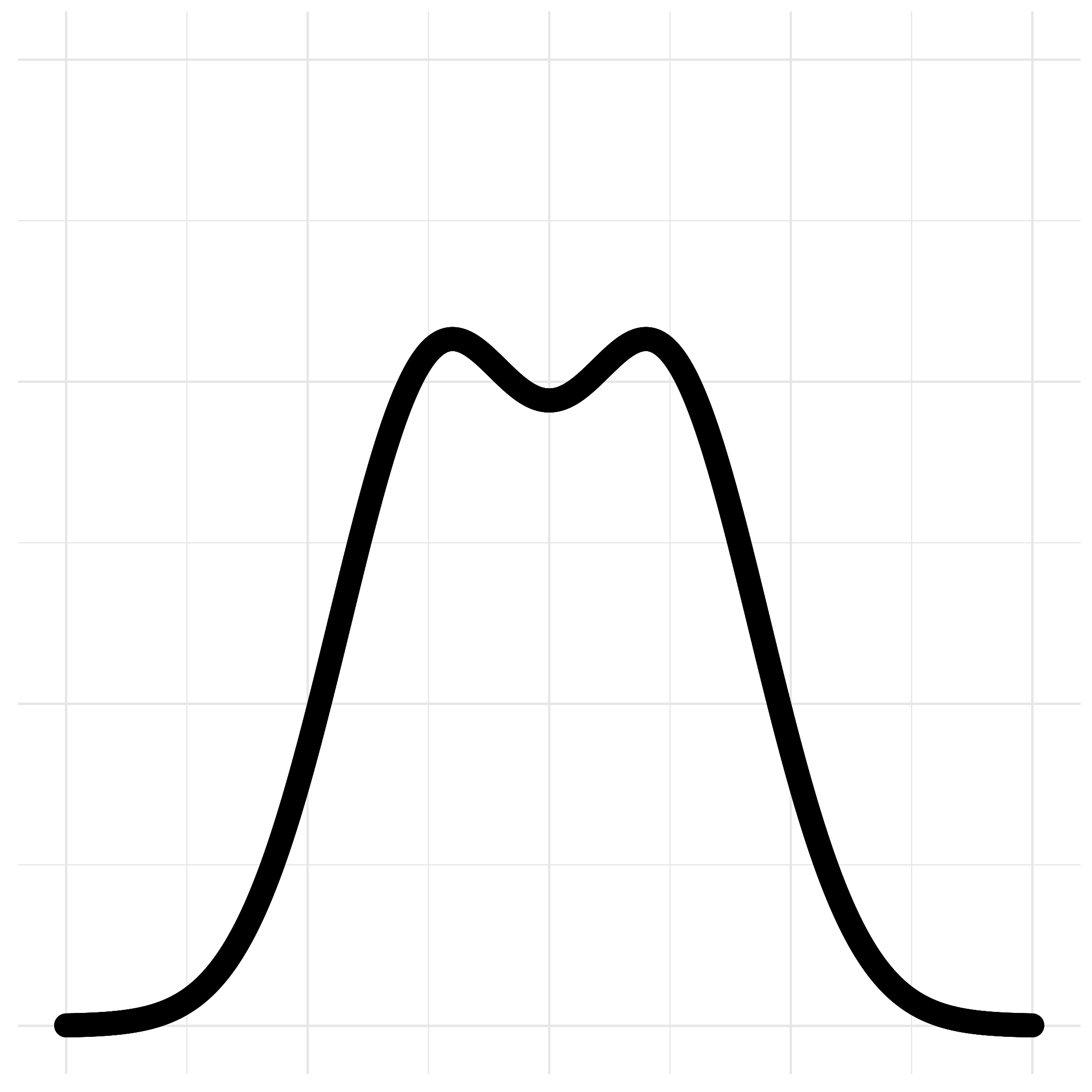}
     \caption{$q(\cdot)$}
   \end{subfigure}
   \caption{}
   \label{fig:PQ}
 \end{figure}
 
 In Figure~\ref{fig:weighted_adj}, we generate a random weighted network $A$ and display the adjacency matrix \emph{without randomly permuting the rows and columns}. It is difficult to discern the block structure because $(P_0, p)$ and $(Q_0, q)$ have equal mean and variance. In Figure~\ref{fig:discrete_adj1},~\ref{fig:discrete_adj2}, and~\ref{fig:discrete_adj3}, we discretize $A$ using the transformation $\Phi(x) = \int_{-\infty}^x \frac{1}{4}e^{-|t|/2}  \,dt$ and $L=3$ bins and show the discretized network $A^1, A^2, A^3$; recall that $A^1$ is a binary adjacency matrix, where $A^1_{uv} = 1$ if $A_{uv} \neq 0$ and $\phi(A_{uv}) \in [0,1/3)$, and $A^1_{uv} = 0$ otherwise, and likewise for $A^2$ and $A^3$. We observe that the block structure is clearly distinguishable in $A^2$ because the densities $p(\cdot)$ and $q(\cdot)$ differ most around the origin; the block structure is somewhat visible in $A^1$ and $A^3$, but to a lesser extent. These figures illustrate why the discretization and initialization stages are useful.
 
 \begin{figure}[!htp]
   \centering
   \begin{subfigure}[b]{0.29\textwidth}
     \centering
     \includegraphics[width=\textwidth, trim=0 57 0 0]{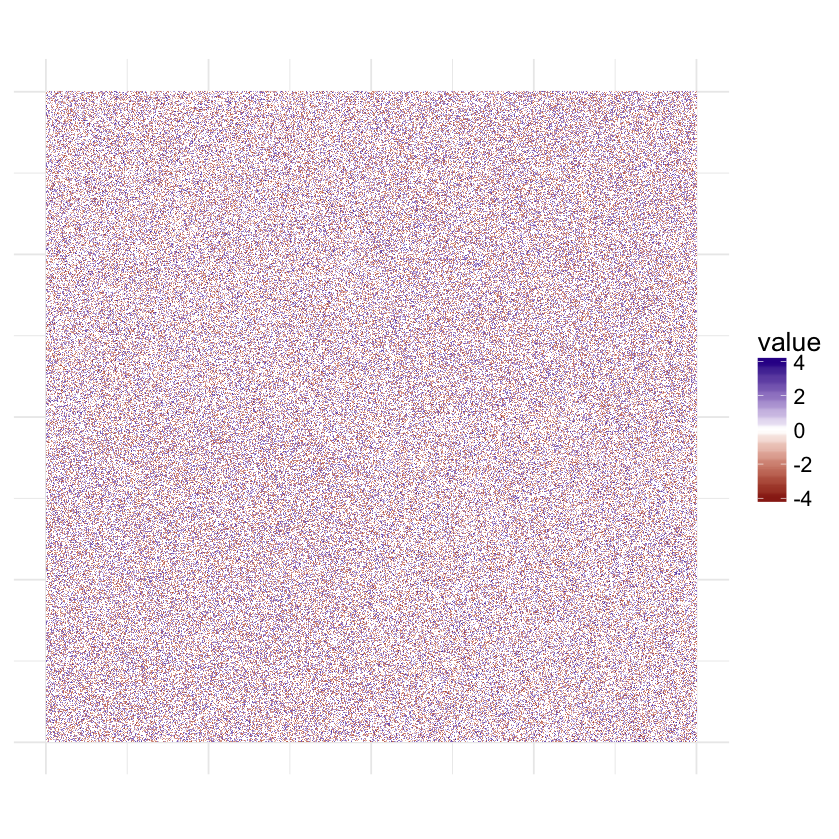}
     \caption{$A$}
     \label{fig:weighted_adj}
   \end{subfigure}
   \begin{subfigure}[b]{0.25\textwidth}
     \centering
     \includegraphics[width=\textwidth]{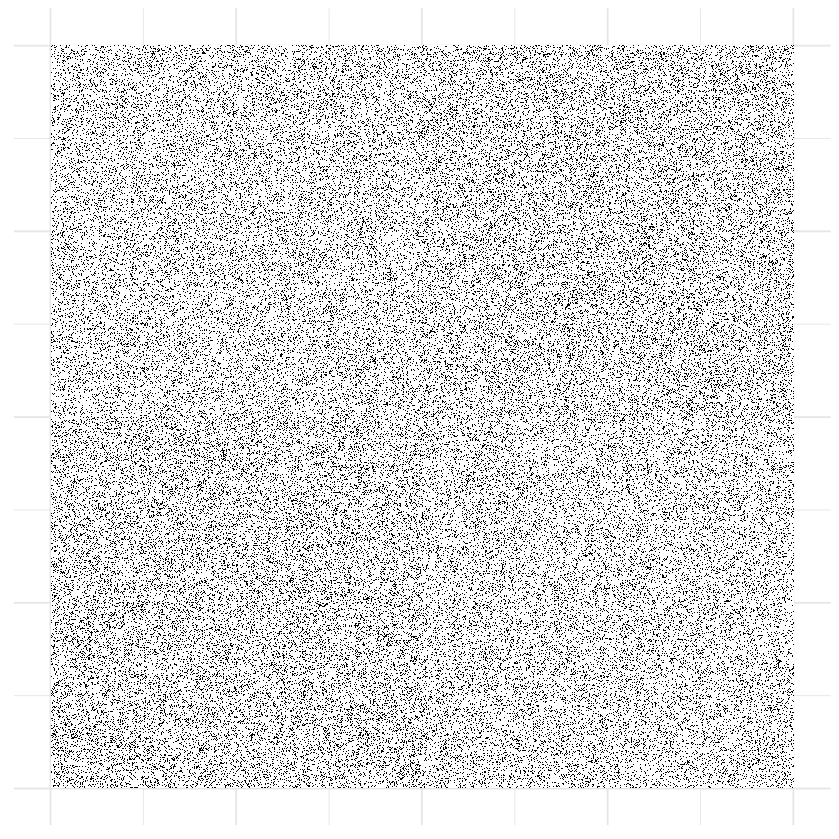}
     \caption{$A^1$}
     \label{fig:discrete_adj1}
   \end{subfigure}\\
   \begin{subfigure}[b]{0.25\textwidth}
     \centering
     \includegraphics[width=\textwidth]{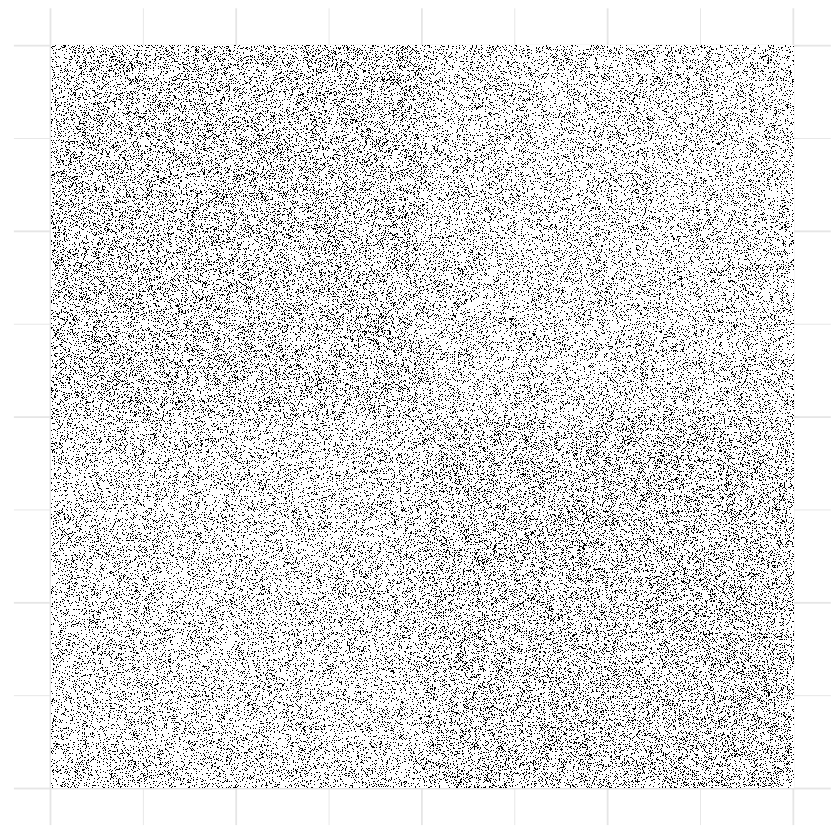}
     \caption{$A^2$}
     \label{fig:discrete_adj2}
   \end{subfigure}
   \hspace{0.2in}
   \begin{subfigure}[b]{0.25\textwidth}
     \centering
     \includegraphics[width=\textwidth]{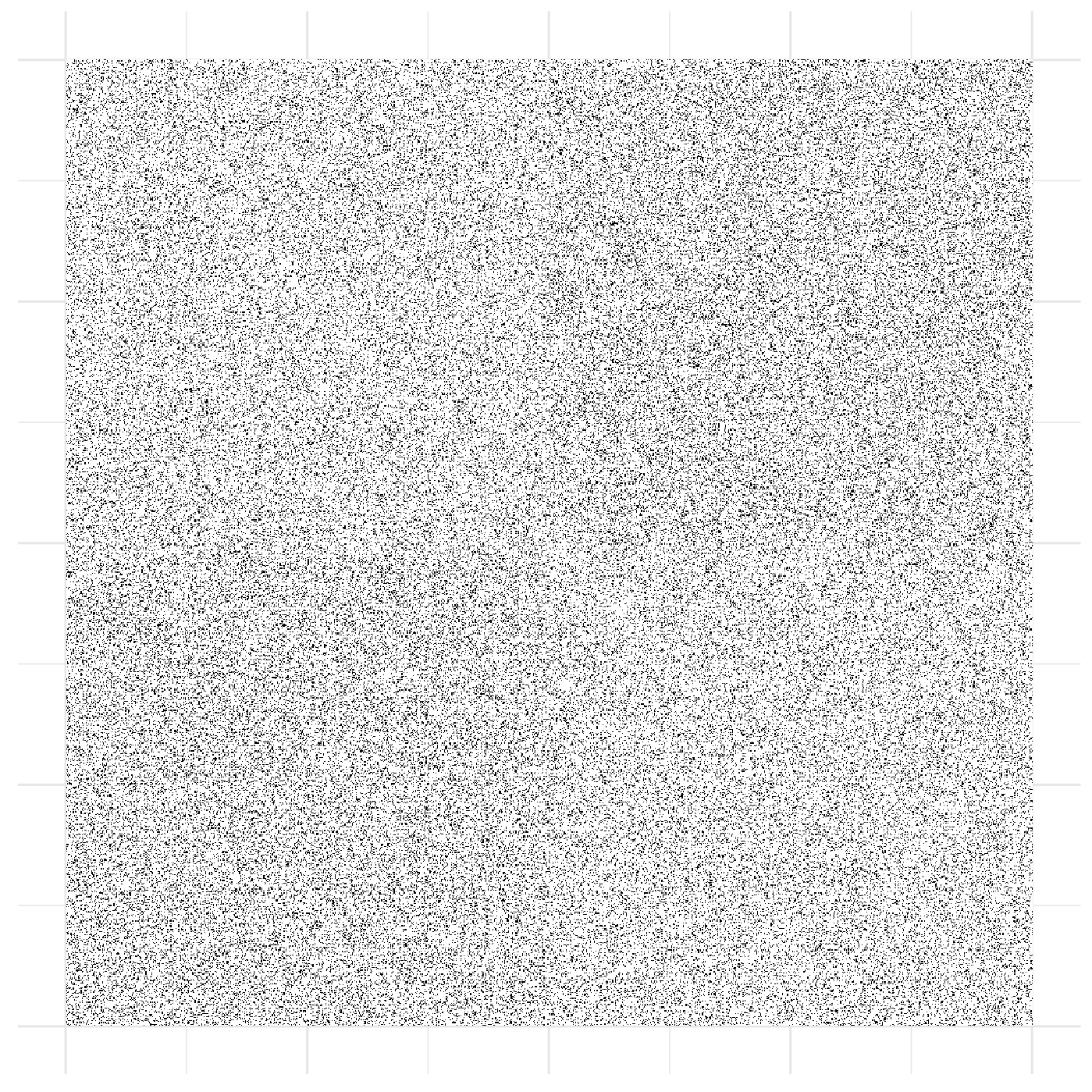}
     \caption{$A^3$}
     \label{fig:discrete_adj3}
   \end{subfigure}
   \caption{}
 \end{figure}

In Figure~\ref{fig:vary_n}, we test how the performance of our algorithm scales with the network size $n$. We use the same setting as our first simulation, except we let $n \in \{400, 600, 800, \ldots, 2000\}$ and $L_n = \lfloor 0.4 (\log(\log n))^4 \rfloor$. For each value of $n$, we perform 100 trials, where we generate a random network $A$, perform our clustering algorithm, and calculate the misclustering error. The misclustering errors are averaged across the 100 random trials and the aggregated medians are shown, with deviations, in Figure~\ref{fig:vary_n}. In Figure~\ref{fig:vary_n}, we observe the same threshold behavior that arises in the unweighted setting: the misclustering error is around $0.5$---equivalent to random guessing---for low $n$, and drops sharply to 0 as the value of $n$ passes a threshold (around $n=1000$ in this case). We note that for this and our next simulation study, we use a simplified version of our algorithm as described in Remark~\ref{rem:simple_computation}; we observed no difference in performance between the full version and the simplified version of the algorithm. 
 
 \begin{figure}[!htp]
   \centering
   \begin{subfigure}[b]{0.37\textwidth}
     \centering
     \includegraphics[width=\textwidth]{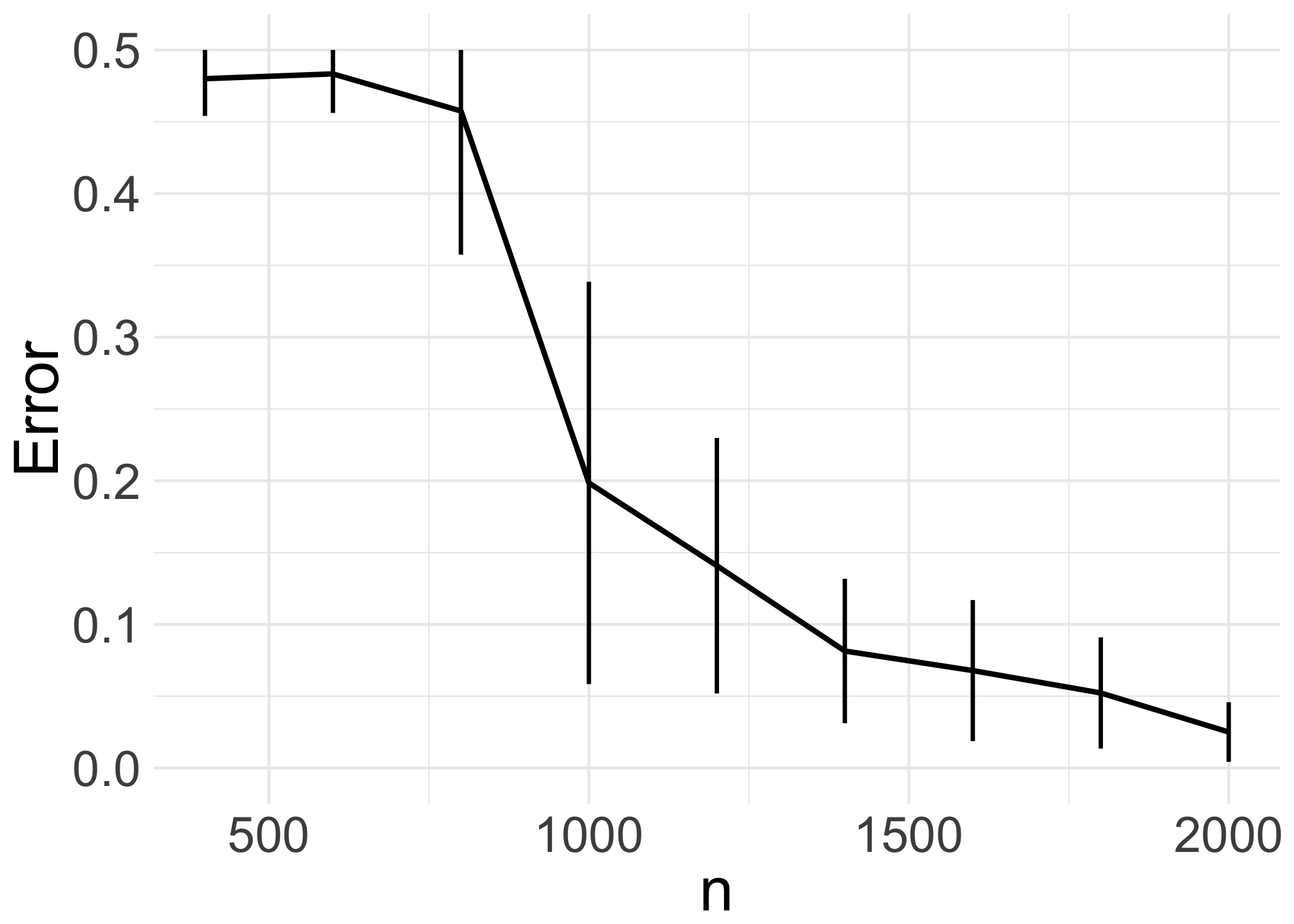}
     \caption{Misclustering error vs.\ $n$}
     \label{fig:vary_n}
   \end{subfigure}
   \hspace{0.5in}
 \begin{subfigure}[b]{0.37\textwidth}
   \centering
   \includegraphics[width=\textwidth]{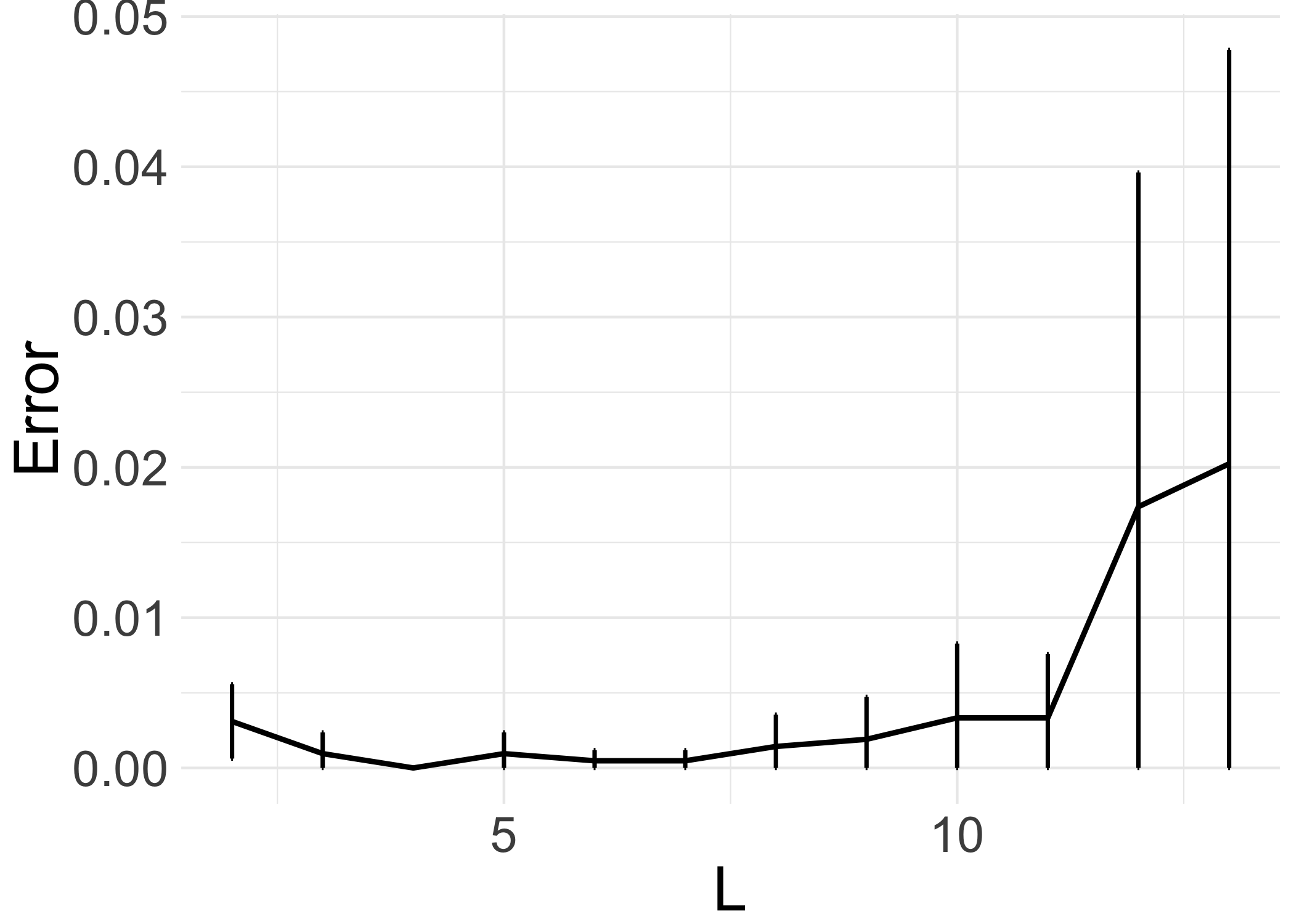}
   \caption{Misclustering error vs.\ $L$}
   \label{fig:vary_L}
 \end{subfigure}
 \caption{}
\end{figure}

\begin{figure}
  \centering
 \begin{subfigure}[b]{0.45\textwidth}
   \centering
   \includegraphics[width=\textwidth]{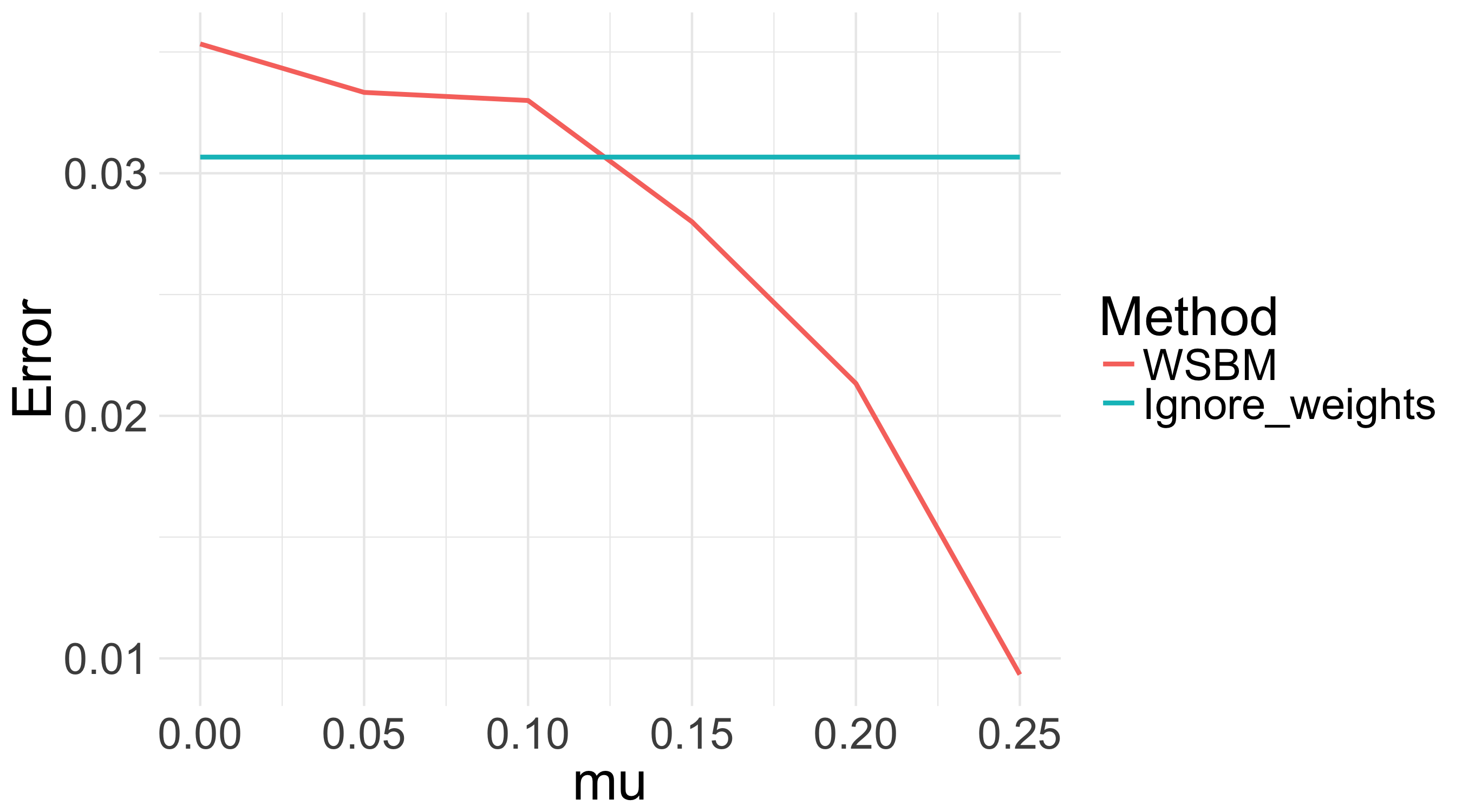}
   \caption{}
   \label{fig:ignore}
 \end{subfigure}
 \begin{subfigure}[b]{0.45\textwidth}
   \centering
   \includegraphics[width=\textwidth]{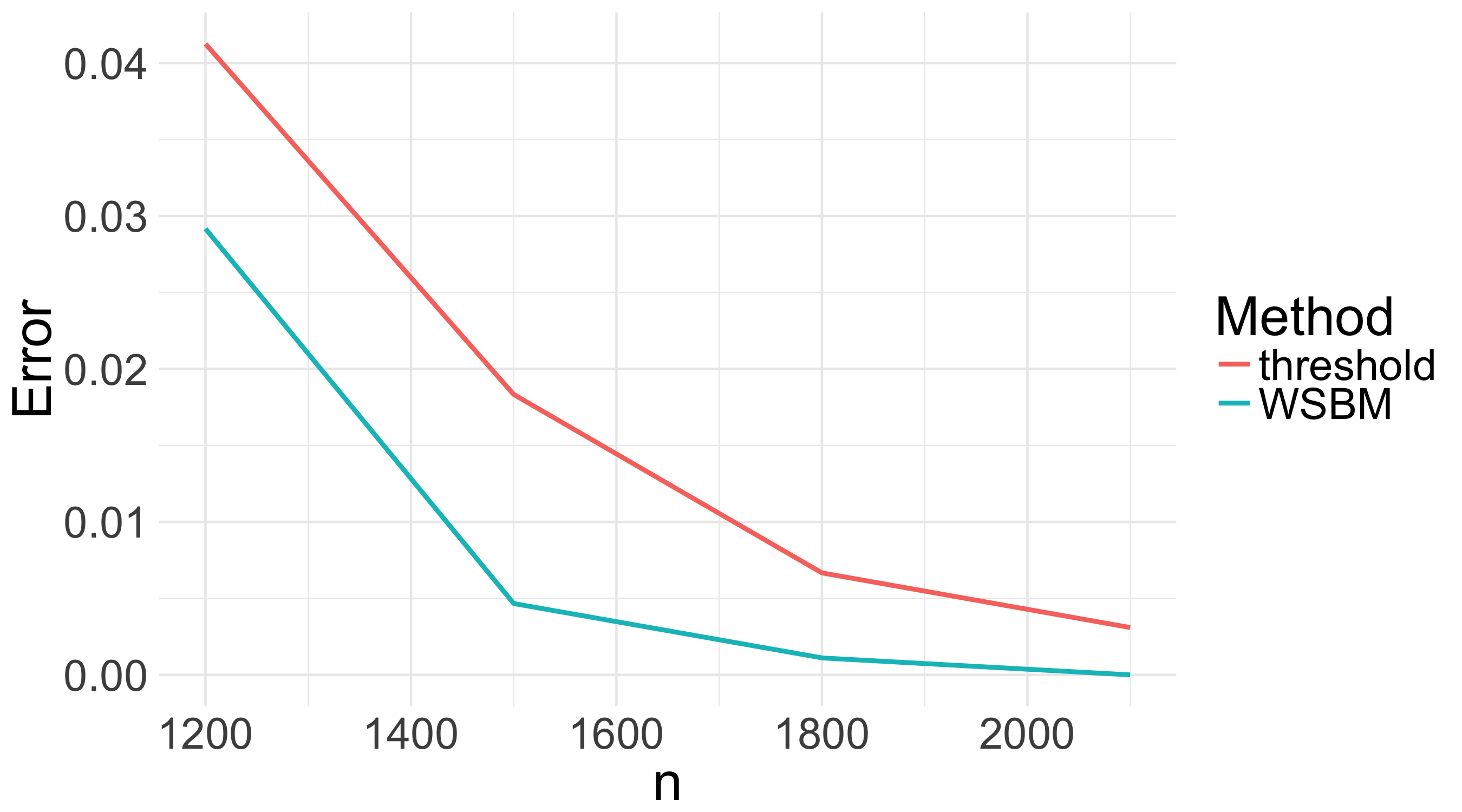}
   \caption{}
   \label{fig:binary}
 \end{subfigure}
 \caption{}
\end{figure}

In Figure~\ref{fig:vary_L}, we study the sensitivity of our algorithm to the choice of discretization level $L$. We let $K=3$, $n=2100$, $P_0 = 0.3$, and $Q_0 = 0.27$, and let $p(\cdot)$ be the density of $N(0.3, 0.8^2)$, and $q(\cdot)$ be the density of $N(0, 1)$. We let $L \in \{1,2,3,\ldots, 12, 13\}$ and, for each setting of $L$, we perform 100 random trials in which we generate a random network $A$, perform our clustering algorithm, and calculate the misclustering error; the results are shown in Figure~\ref{fig:vary_L}; the error for $L=1$, in which we discard the edge weights, exceeds $0.56$ and is thus omitted from the plot. We observe that the algorithm performs best when $L$ is chosen to be small, though not too small, as is suggested by our theoretical analysis.

In Figure~\ref{fig:ignore}, we compare our approach against treating a weighted network as an unweighted one by discarding the edge weights. In this setting, we let $n = 1500$, $P_0 = 0.3, Q_0=0.23$, and $K=3$. We choose $q(\cdot)$ as the density of $N(0,1)$ and $p(\cdot)$ as the density of $N(\mu, 1)$ where we let $\mu \in \{0, 0.05, 0.1, 0.15, 0.2, 0.25\}$. We perform 100 trials and aggregate the result in Figure~\ref{fig:ignore}. In red, we plot the misclustering proportion error incurred by our WSBM clustering algorithm with $L=5$; in blue, we plot the misclustering error incurred by ignoring the edge weights entirely and treating the network as an unweighted one. As we expect, when $\mu$ is close to 0, the edge weights are uninformative and it is better to ignore the edge weights. As $\mu$ increases, however, the advantage of using the weights become significant.

 In Figure~\ref{fig:binary}, we compare our algorithm against clustering an unweighted network formed by optimally thresholding the edge weights. We let $K=3$, $P_0 = 0.3$, and $Q_0=0.27$, and let $p(\cdot)$ be the density of $N(0.3, 0.8)$ and $q(\cdot)$ be the density of $N(0,1)$. For $\tau \in \mathbb{R}$, we define the thresholded network $A_{\tau} \in \{0, 1\}^{n \times n}$ as $A_{\tau,uv} = 1$ if $A_{uv} \neq 0$ and $A_{uv} \geq \tau$, and $A_{\tau, uv} = 0$ if $A_{uv} = 0$ or if $A_{uv} < \tau$. For each $\tau \in \{-2, -1.8, -1.6, \ldots, 1.6, 1.8, 2.0\}$, we form $A_{\tau}$, extract the cluster, and compute the misclustering error. We then report the lowest misclustering error among all $A_{\tau}$ for $\tau \in \{-2, -1.8, -1.6, \ldots, 1.6, 1.8, 2.0\}$ as the red line in Figure~\ref{fig:binary}; this approach is of course impossible to implement in practice, and we use it only for the purpose of comparison. The turquoise line is the misclustering error incurred by our algorithm, using $L_n = \lfloor 0.4 (\log(\log n))^4 \rfloor$. 


\section{Conclusion}
\label{sec:conclusion}

We have provided a rate-optimal community estimation algorithm for the homogeneous weighted stochastic block model. Our algorithm includes a preprocessing step consisting of transforming and discretizing the (possibly) continuous edge weights to obtain a simpler graph with edge weights supported on a finite, discrete set. This approach may be useful for other network data analysis problems involving continuous distributions, where discrete versions of the problem are simpler to analyze.

Our paper provides a step toward understanding the weighted SBM under the same mathematical framework that has been exceptionally fruitful in the case of unweighted models. It is far from comprehensive, however, and many open questions remain. We describe a few here:
\begin{enumerate}
\item An important extension is the \emph{heterogenous} stochastic block model, where edge weight distributions depend on the exact community assignments of both endpoints. In such a setting, Abbe and Sandon~\cite{AbbSan15} and Yun and Proutiere~\cite{yun2016optimal} have shown that a generalized information divergence---the CH divergence---governs the intrinsic difficulty of community recovery. We believe that a similar discretization-based approach should lead to analogous results in the case of a heterogeneous weighted SBM.
\item Real-world networks often have nodes with very high degrees, which may adversely affect the accuracy of recovery methods for the stochastic block model. To solve this problem, degree-corrected SBMs \cite{zhao2012consistency, gao2016community} have been proposed as an effective alternative to regular SBMs. It remains to extend the concept of degree-correction to the weighted SBM.
\end{enumerate}

\section*{Acknowledgments}
The authors would like to thank Zongming Ma for enlightening discussions in the earlier stages of this project. The authors would also like to thank Richard Samworth for several helpful discussions. 

\appendix

\section{Proof of Proposition~\ref{prop:labeled_sbm_rate}}
\label{appendix: first}

We structure the proof according to the flow of our algorithm. Since this proposition addresses the case of discrete labels, we do not need to consider the ``transformation and discretization" step. We will prove the proposition by constructing a sequence $\zeta_n \rightarrow 0$ such that the statements of the proposition are satisfied. 

Since $\frac{n I'_n }{(L_n+1) \rho_n^2 \log \rho_n} \rightarrow \infty$ and $I_n \rightarrow 0$, there exists $N_0 \in \mathbb{N}$ such that for all $n \geq N_0$, we have
\begin{align}
  & \frac{n-1}{\rho_n^2 \beta^2 K^5 \log(\rho_n^2 \beta K) } \frac{I'_n}{L_n+1}
   \geq c_{init}^{-1} C_{spec2} C_{mis},  \label{eqn:large_n_info1}\\
  & \frac{n-1}{\beta^2 K^6} \frac{I'_n}{L_n+1} \geq c_{init}^{-1} C_{spec1},
 \quad \frac{L_n+1}{n} \leq \frac{1}{8},   \quad \textrm{ and } I_n \leq \frac{1}{2}, \label{eqn:large_n_info2}
\end{align}
where $C_{spec1}$ and $C_{spec2}$ are universal constants defined in Proposition~\ref{prop:spectral_analysis}, $c_{init}$ is a universal constant defined in Proposition~\ref{prop:initial_guarantee}, and $C_{mis}$ is a universal constant defined in Proposition~\ref{prop:single_node_error_bound}. For $n \leq N_0$, we define $\zeta_n$ such that $\exp\bigl(- \frac{n I_n}{\beta K} (1 + \zeta_n) \bigr) = 1$. 

Now, suppose $n \geq N_0$, so inequalities~\eqref{eqn:large_n_info1} and~\eqref{eqn:large_n_info2} hold. Let us arbitrarily fix $(\{P_l\}, \{Q_l\}) \in \mathcal{G}_{L_n, \rho_n}$ such that $I_n' \leq I(\{P_l\}, \{Q_l\}) \leq I_n$. We consider Algorithm~\ref{alg:noisify}. Let $\delta = 2 \frac{L+1}{n} $ and define
  \begin{align*}
    P_l' := P_l(1-\delta) + \frac{\delta}{L+1}, \quad \text{and} \quad
    Q_l' := Q_l(1-\delta) + \frac{\delta}{L+1}.
  \end{align*}
  Since the input network $A_L$ has the distribution $LSBM(\sigma_0, (\{P_l\}, \{Q_l\})$, the output $A_L$ of Algorithm~\ref{alg:noisify} has the distribution $LSBM(\sigma_0, (\{P'_l\}, \{Q'_l\}))$. It is then clear that $P'_l, Q'_l \geq \frac{2}{n}$ for all $l \in \{0,\ldots,L+1\}$, and furthermore, by Lemma~\ref{lem:add_noise_bound}, we know that
  \begin{align}
    I( \{P'_l\}, \{Q'_l\} )
    &\geq
      \frac{I(\{P_l\}, \{Q_l\})}{1 + 2 I(\{P_l\}, \{Q_l\})} \bigl( 1 - 4 \frac{L+1}{n} \bigr) 
      \geq \frac{I'_n}{4}, \label{eqn:Iprime_lower_bound} \\
    I( \{P'_l\}, \{Q'_l\} ) &\leq I( \{P_l\}, \{Q_l\}) \leq I_n.
                              \label{eqn:Iprime_upper_bound}
  \end{align}
  
Let $l^* \in \{0,\ldots,L\}$ be the output of the first stage of \textsc{initialization} (Algorithm~\ref{alg:initialization1}). Define $E_1$ as the event that 
\begin{align}
\frac{\Delta_{l^*}^2 }{P_{l^*} \vee Q_{l^*}} 
  \geq c_{init} \frac{I(\{P'_l\}, \{Q'_l\})}{L+1}.
  \label{eqn:informative_lstar_event}
\end{align}

Since
\[
\frac{n I(\{P'_l\}, \{Q'_l\})}{(L+1) \beta^2 K^6} 
\geq \frac{n I'_n}{4 (L+1) \beta^2 K^6} \geq \frac{1}{4} c^{-1}_{init}  C_{spec1} \geq 2 C_{spec1},
\] 
$P'_l \vee Q'_l \geq \frac{1}{n}$, and $\frac{n}{\log n} \geq 2^{15} \vee 30 \beta K$, we may apply Claim 3 of Proposition~\ref{prop:initial_guarantee} to conclude that $\mathbb{P}(E_1) \geq 1 - 12(L+1)^2 n^{-5}$.

Let $\{\tilde{\sigma}_u\}_{u\in [n]}$ be the initial clusterings output by the second stage of \textsc{initialization} (Algorithm~\ref{alg:initialization1}). Define $E_2$ as the event that
 \[
 \textrm{for all $u \in [n]$, }  l_{\setminus \{u\}} (\tilde{\sigma}_u, \sigma_0) \leq C_{spec2} \frac{\beta K^4}{n-1} \frac{ P_{l^*} \vee Q_{l^*}}{\Delta_{l^*}^2},
 \]
where $l_{\setminus \{u\}}(\cdot, \cdot)$ is the misclustering error defined on $[n] \backslash \{u\}$. For $u \in [n]$ and $l \in \{0, \ldots, L\}$, define $\tilde{\sigma}_u^{(l)}$ as the result of applying spectral clustering (with $\mu = 4 \beta$ and $\tau = 40 K \bar{d}$ where $\bar{d}$ is defined with respect to $A_{l, -u}$) on $A_{l, -u} \in \{0,1\}^{(n-1)\times (n-1)}$, that is, the network excluding node $u$ with only the edges whose label is $l$. By Proposition~\ref{prop:spectral_analysis} and a union bound, we have that, with probability at least $1 - Ln(n-1)^{-5}$, 
\[
\max_{u \in [n]} \quad \max_{l \in \{0, \ldots L\} \,:\, \frac{n}{\beta^2 K^6} \frac{(P_l - Q_l)^2}{P_l \vee Q_l} \geq C_{spec1}} l_{\setminus \{u\}} (\tilde{\sigma}_u^{(l)}, \sigma_0) C_{spec2}^{-1} \frac{n}{\beta K^4} \frac{(P_l - Q_l)^2}{P_l \vee Q_l} \leq 1.
\]
 Since, under event $E_1$,
 \begin{align*}
   \frac{n-1}{\beta^2 K^6} \frac{\Delta_{l^*}^2 }{P_{l^*} \vee Q_{l^*}} 
   \geq c_{init} \frac{n I(\{P'_l\}, \{Q'_l\})}{L_n +1} \geq
   c_{init} \frac{n-1}{\beta^2 K^6} \frac{I'_n}{L_n+1} \geq C_{spec1},
 \end{align*}
 since $P_{l^*} \vee Q_{l^*} \geq \frac{1}{n-1}$, and since $n-1 \geq 8 \beta^2 K^2 \vee 2000$, we obtain
 \begin{align}
   \mathbb{P}(E_1 \cap E_2) \geq 1 - 12(L+1)^2n^{-5} - Ln(n-1)^{-5}.
 \end{align}

Note that under event $E_1 \cap E_2$, for all $u \in [n]$, we have
\begin{align}
  l_{\setminus \{u\}} (\tilde{\sigma}_u, \sigma_0)
  &\leq
  C_{spec2} \frac{\beta K^4}{n-1} \frac{ P_{l^*} \vee Q_{l^*}}{\Delta_{l^*}^2} \leq
  c_{init}^{-1} C_{spec2} \frac{\beta K^4}{n-1} \frac{L_n + 1}{I(\{P'_l\}, \{Q'_l\})} \nonumber \\
  & \leq 4 c_{init}^{-1} C_{spec2} \frac{\beta K^4}{n-1} \frac{L_n + 1}{I'_n}
    \stackrel{(a)} \leq \bigl( C_{mis} \rho_n^2 \beta K \log( \rho_n^2 \beta K) \bigr)^{-1}, \label{eqn:initial_errors_small_enough}
\end{align}
where $(a)$ follows from inequality~\eqref{eqn:large_n_info1}.
Since $C_{mis} > 16$ and $\rho_n \geq 1$, we have $l_{\setminus \{u\}} (\tilde{\sigma}_u, \sigma_0) < \frac{1}{16\beta K}$. Recall the definition~\eqref{eqn:Sk_defn} of $S_K[\tilde{\sigma}_u, \sigma_0]$. Since the smallest cluster of $\sigma_0$ is of size at least $\frac{n}{\beta K}$, we have by Lemma~\ref{lem:consensus} that $S_K[\tilde{\sigma}_u, \sigma_0]$ is a singleton; we let $\pi_u$ denote the only element of $S_K[\tilde{\sigma}_u, \sigma_0]$. 

Since $\hat{\sigma}_u = \tilde{\sigma}_u$ on $[n] \backslash \{u\}$, we thus have, for all $u \in [n]$, that 
\begin{align}
  \frac{1}{n} d_H(\pi_u \circ \hat{\sigma}_u, \sigma_0) \leq \frac{1}{n} \bigl\{ d_H(\pi_u \circ \tilde{\sigma}_u, \sigma_0) + 1 \bigr\} \leq 
  \frac{1}{16 \beta K} \frac{n-1}{n} + \frac{1}{n} < \frac{1}{8 \beta K}.
\end{align}
By Lemma~\ref{lem:consensus} again, we know that $\pi_u$ is the only element of $S_K[\hat{\sigma}_u, \sigma_0]$.

Since the smallest cluster of $\sigma_0$ is of size at least $\frac{n}{\beta K}$, the smallest cluster of $\hat{\sigma}_u$ is of size at least $\frac{n}{2 \beta K}$. Furthermore, we have
\begin{align}
  l(\hat{\sigma}_1, \hat{\sigma}_u)
  \leq
  \frac{1}{n} d_H(\pi_1 \circ \hat{\sigma}_1, \pi_u \circ \hat{\sigma}_u)
  \leq \frac{1}{n} \bigl( d_H(\pi_1 \circ \hat{\sigma}_1, \sigma_0) +
    d_H(\pi_u \circ \hat{\sigma}_u, \sigma_0) \bigr) < \frac{1}{4 \beta K}.
\end{align}
Therefore, from Lemma~\ref{lem:consensus}, we conclude that $\pi_1^{-1} \circ \pi_u$ is the only element of $S_K[\hat{\sigma}_u, \hat{\sigma}_1]$ and
\begin{align}
  \hat{\sigma}(u)
  =
  \argmax_{k \in [K]} \bigl|\{ v \in [n] \,:\, \hat{\sigma}_u(v) = \hat{\sigma}_u(u) \} \cap \{ v \in [n] \,:\, \hat{\sigma}_1(v) = k \} \bigr| = 
  (\pi_1^{-1} \circ \pi_u)(\hat{\sigma}_u(u)).
  \label{eqn:sigma_hat_characterization}
\end{align}

Define $\gamma_u = l_{\setminus \{u\}}( \tilde{\sigma}_u, \sigma_0)$, $\eta'_u = 2 \beta K \gamma_u \log \frac{e K}{\gamma_u} + 12 \frac{\log n}{n}$, and $\eta_u = 10 (\sqrt{\eta'_u} + \eta'_u)$. Because inequality~\eqref{eqn:initial_errors_small_enough} holds under $E_1 \cap E_2$ and also $P'_l \vee Q'_l \geq \frac{1}{n-1}$ and $\frac{n}{\log n} \geq 2^8 \rho_n^2$,  we can apply Proposition~\ref{prop:single_node_error_bound}, for each $u \in [n]$, to obtain
\begin{align*}
  \mathbb{P}\bigl\{
  &
  \exists \pi \in S_K[\hat{\sigma}_1, \sigma_0],\, 
  \pi(\hat{\sigma}(u)) \neq \sigma_0(u) \bigr\} \\
  & \leq  \mathbb{P}\bigl\{
    \exists \pi \in S_K[\hat{\sigma}_1, \sigma_0],\,
   \pi(\hat{\sigma}(u)) \neq \sigma_0(u) \given E_1 \cap E_2 \bigr\} + P(E_1^c \cup E_2^c) \\
 &\stackrel{(a)} = \mathbb{P}\bigl\{ \pi_1(\hat{\sigma}(u)) \neq \sigma_0(u) \given E_1 \cap E_2 \bigr\} + P(E_1^c \cup E_2^c) \\
 & \stackrel{(b)} = \mathbb{P}\bigl\{
   \pi_u^{-1}(\sigma_0(u)) \neq \hat{\sigma}_u(u) \given E_1 \cap E_2 \bigr\} + P(E_1^c \cup E_2^c) 
   \\
  &\leq (K-1) \exp\biggl( - (1 - C_{err} \beta K \rho_n \eta_u) \frac{n}{\beta K} I(\{P'_l\}, \{Q'_l\}) \biggr) +  5(L+1)n^{-6} + 12 (L+1)^2 n^{-5} - Ln(n-1)^{-5} \\
  &\leq (K-1) \exp\biggl( - (1 - C_{err} \beta K \rho_n \eta_u) \frac{n}{\beta K} I(\{P'_l\}, \{Q'_l\}) \biggr) + n^{-3},
\end{align*}
where $(a)$ follows because $S_K[\hat{\sigma}_1, \sigma_0]$ is a singleton under $E_1 \cap E_2$, and $(b)$ follows from equation~\eqref{eqn:sigma_hat_characterization}.

Define $\zeta'_n :=  C_{err} \beta K \rho_n \max_{u \in [n]} \eta_u$. By the penultimate statement in inequality~\eqref{eqn:initial_errors_small_enough} and the assumption that $\frac{n I'_n}{(L+1) \rho_n^2 \log \rho_n} \rightarrow \infty$, we have $\beta K \log(\beta K) \max_{u \in [n]} \gamma_u \rightarrow 0$, so $\max_{u \in [n]} \eta_u \rightarrow 0$ and $\zeta'_n \rightarrow 0$.

Observe that
\begin{align*}
  \E\bigl[ l(\hat{\sigma}, \sigma_0) \bigr] 
  & = \E\left[ \min_{\pi \in S_K} \frac{1}{n} \sum_{u=1}^n
    \mathbf{1} \{ (\pi \circ \hat{\sigma})(u) \neq \sigma_0(u) \}  \right] \\
    & \leq \E\left[ \min_{\pi \in S_K[\hat{\sigma}_1, \sigma_0]} \frac{1}{n} \sum_{u=1}^n
    \mathbf{1} \{ (\pi \circ \hat{\sigma})(u) \neq \sigma_0(u) \}  \right] \\
  &\leq
    \E\left[ \frac{1}{n} \sum_{u=1}^n \mathbf{1} \bigl \{ \exists \pi \in S_K[\hat{\sigma}_1, \sigma_0],\, (\pi \circ \hat{\sigma})(u) \neq \sigma_0(u)  \bigr \}  \right] \\
  & \leq   \frac{1}{n} \sum_{u=1}^n
    \mathbb{P}\bigl\{
    \exists \pi \in S_K[\hat{\sigma}_1, \sigma_0],\, 
    \pi(\hat{\sigma}(u)) \neq \sigma_0(u)  \bigr\}\\
  & \leq \exp\biggl( - (1 - \zeta'_n)
    \frac{n}{\beta K} I(\{P'_l\}, \{Q'_l\}) \biggr) +  n^{-3}, \\
  &\leq \exp\biggl( - (1 - \zeta''_n) \frac{n}{\beta K} I(\{P_l\}, \{Q_l\}) \biggr) + n^{-3},
\end{align*}
where in the last inequality, we define $\zeta''_n := 1 - (1 - \zeta'_n)\bigl( 1 - \frac{I(\{P_l\}, \{Q_l\}) - I(\{P'_l\}, \{Q'_l\})}{I( \{P_l\}, \{Q_l\})} \bigr)$. By inequalities~\eqref{eqn:Iprime_upper_bound} and~\eqref{eqn:Iprime_lower_bound}, we have
\[
  0 \leq \frac{I(\{P_l\}, \{Q_l\}) - I(\{P'_l\}, \{Q'_l\})}{I( \{P_l\}, \{Q_l\})} \leq 1 - \frac{1}{1 + 2 I_n} (1 - 4 \frac{L+1}{n}),
\]
so since $I_n \rightarrow 0$ by assumption, we have $\zeta''_n \rightarrow 0$. For the second claim of Proposition~\ref{prop:labeled_sbm_rate}, let $\zeta_n = \zeta''_n$ for all $n \geq N_0$. It is then clear that if $\frac{n I_n}{\beta K \log n} \leq 1$, we have
\[
  \E l(\hat{\sigma}, \sigma_0) \leq \exp \biggl( - (1 - \zeta_n) \frac{n}{\beta K} I(\{P_l\}, \{Q_l\}) \biggr).
\]
Since $(\{P_l\}, \{Q_l\})$ was chosen arbitrarily, the second claim of the proposition follows.   

For the first claim, let us first suppose that $ \exp\biggl( - (1 - \zeta''_n)
\frac{n I(\{P_l\}, \{Q_l\})}{\beta K} \biggr) \geq n^{-2}$. Define $\tilde{\zeta}_n := \zeta''_n + \bigl( \frac{\beta K}{n I_{n'}} \bigr)^{1/2}$. Then

\begin{align*}
 & \mathbb{P} \biggl \{ l(\hat{\sigma}, \sigma_0) >  \exp\biggl( - (1 - \tilde{\zeta}_n)
  \frac{n I(\{P_l\}, \{Q_l\})}{\beta K} \biggr) \biggr\}
  \leq \frac{\E l(\hat{\sigma}, \sigma_0)}{  \exp\biggl( - (1 - \tilde{\zeta}_n)
    \frac{n I(\{P_l\}, \{Q_l\})}{\beta K} \biggr)} \\
  &\qquad \leq  \exp\biggl( (\zeta''_n - \tilde{\zeta}_n)
    \frac{n I(\{P_l\}, \{Q_l\})}{\beta K} \biggr) + \frac{n^{-3}}{ \exp\biggl( - (1 - \tilde{\zeta}_n)
    \frac{n I(\{P_l\}, \{Q_l\})}{\beta K} \biggr)} \\
  &\qquad \leq \exp\biggl( - \bigl( \frac{n I(\{P_l\}, \{Q_l\})}{\beta K} \bigr)^{1/2} \biggr) + n^{-1} \leq
    \exp \biggl( - \bigl( \frac{ n I'_n}{\beta K} \bigr)^{1/2} \biggr) + n^{-1}.
\end{align*}

Let us now suppose that $ \exp\biggl( - (1 - \zeta''_n)
\frac{n I(\{P_l\}, \{Q_l\})}{\beta K} \biggr) < n^{-2}$. Then 
\begin{align*}
 & \mathbb{P} \biggl\{ l(\hat{\sigma}, \sigma_0) \geq
  \exp\biggl( - (1 - \zeta''_n)
   \frac{n I(\{P_l\}, \{Q_l\})}{\beta K} \biggr) \biggr\} \leq
   \mathbb{P}( l(\hat{\sigma}, \sigma_0) > 0) \\
  &\qquad \leq \mathbb{P} \bigl\{  \min_{\pi \in S_K} d_H(\pi \circ \hat{\sigma}, \sigma_0) > 0 \bigr\} \\
  &\qquad \leq \mathbb{P}\bigl\{ \min_{\pi \in S_K[\hat{\sigma}_1, \sigma_0]} d_H(\pi \circ \hat{\sigma}, \sigma_0) > 0 \bigr\} \\
  &\qquad \leq \sum_{u=1}^n
    \mathbb{P}\bigl\{\exists \pi \in S_K[\hat{\sigma}_1, \sigma_0],\, \pi(\hat{\sigma}(u)) \neq \sigma_0(u) \bigr\} \\
  &\qquad \leq n  \exp\biggl( - (1 - \zeta''_n)
\frac{n I(\{P_l\}, \{Q_l\})}{\beta K} \biggr) + n^{-2} \leq 2n^{-1}.
\end{align*}

We now let $\zeta_n = \tilde{\zeta}_n$. Since $\tilde{\zeta}_n \geq \zeta''_n$, we can conclude that
\[
\mathbb{P} \biggl\{ l(\hat{\sigma}, \sigma_0) > \exp\biggl( -(1-\zeta_n) \frac{n}{\beta K} I(\{P_l\}, \{Q_l\}) \biggr) \biggr\} \rightarrow 0.
\]


\section{Supporting results for Proposition~\ref{prop:labeled_sbm_rate}}
\label{appendix: labeled_sbm_rate}

We now provide proofs for the supporting results stated in Appendix~\ref{appendix: first}.

\subsection{Analysis of estimation error of $\hat{P}_l$ and $\hat{Q}_l$}
\label{appendix: hats galore}

We begin with a proposition.

\begin{proposition}
  \label{prop:estimation_consistency}
  Let $\sigma_0 \in \mathcal{C}(\beta, K)$. Let $L \in \mathbb{Z}^+$, let $(\{P_l\}, \{Q_l\}) \in \mathcal{P}_L^2$, and let $A \in \{0,\ldots\,L\}^{n \times n}$ be a random labeled network with the distribution $LSBM(\sigma_0, (\{P_l\}, \{Q_l\}))$. Define $\Delta_l := | P_l - Q_l|$.
For a clustering $\sigma$, define
\begin{align*}
S(\sigma) &:= \{ (u,v) \in [n]^2 \,:\, u\neq v,\, \sigma(u) = \sigma(v) \}, \\
S(\sigma)^c & :=  \{ (u,v) \in [n]^2 \,:\, u\neq v,\, \sigma(u) \neq \sigma(v) \},
\end{align*}
  and also define the estimators
\begin{align}
  \hat{P}_l := \frac{1}{|S(\sigma)|} \sum_{(u,v) \in S(\sigma)} \mathbf{1}\{A_{uv} = l\}, \qquad
  \hat{Q}_l := \frac{1}{|S(\sigma)^c|} \sum_{(u,v) \in S(\sigma)^c} \mathbf{1}\{A_{uv} = l\}. \label{eqn:hatPl_hatQl_definition}
\end{align}
Let $\gamma \in [0, 1]$ and let $\eta := 8 (\sqrt{\eta'} + \eta')$, where $\eta' := \beta K \gamma \log \frac{e K}{\gamma} + 6 \frac{\log n}{n}$. 

Then with probability at least $1 - 4(L+1)n^{-6}$, it holds that for any $\sigma : [n] \rightarrow [K]$ such that $l(\sigma, \sigma_0) \leq \gamma$, we have
\[
  \max\{ |\hat{P}_l - P_l|, |\hat{Q}_l - Q_l| \}
    \leq \Delta_l +
    \eta \biggl( \frac{P_l \vee Q_l}{n} \biggr)^{1/2}.
  \]
  Furthermore, if $\gamma \leq \frac{1}{4 \beta K}$, then
  \begin{enumerate}
  \item for all $l \in \{0,\ldots,L\}$ where $P_l \vee Q_l \geq \frac{1}{n}$ and $\frac{n \Delta_l^2}{P_l \vee Q_l} \geq 1$, we have
    \begin{align*}
      \max\{ | \hat{P}_l - P_l |,\, | \hat{Q}_l - Q_l | \} \leq \eta \Delta_l,
    \end{align*}
  \item and for all $l \in \{0,\ldots,L\}$ where $P_l \vee Q_l \geq \frac{1}{n}$ and $\frac{n \Delta_l^2}{P_l \vee Q_l} \leq 1$, we have
    \begin{align*}
      \max\{ | \hat{P}_l - P_l |,\, | \hat{Q}_l - Q_l | \} \leq \eta \biggl( \frac{P_l \vee Q_l}{n} \biggr)^{1/2}.
    \end{align*}
  \end{enumerate}
\end{proposition}

\begin{proof}
We first fix a clustering $\sigma : [n] \rightarrow [K]$ satisfying $l(\sigma, \sigma_0) \leq \gamma$ and fix a color $l \in \{0,\ldots,L\}$. Then
\begin{align}
\E\hat{P_l} - P_l &= 
 \frac{1}{|S(\sigma)|}  \bigl(  |S(\sigma) \cap S(\sigma_0)| P_l +
                    |S(\sigma) \cap S(\sigma_0)^c| Q_l \bigr) - P_l = \frac{|S(\sigma) \cap S(\sigma_0)^c|}{| S(\sigma)|} (Q_l - P_l), \label{eqn:Pl_bias}
\end{align}
and
\begin{align}
  \E \hat{Q}_l - Q_l &=
 \frac{1}{|S(\sigma)^c|}  \bigl(|S(\sigma)^c \cap S(\sigma_0)| P_l +
                    |S(\sigma)^c \cap S(\sigma_0)^c| Q_l \bigr) - P_l = \frac{|S(\sigma)^c \cap S(\sigma_0)|}{| S(\sigma)^c|} (P_l - Q_l).                     
  \label{eqn:Ql_bias}
\end{align}
Hence, we have
\begin{align}
  \max\{ | \E \hat{P}_l - P_l|,\, |\E \hat{Q}_l - Q_l| \} \leq \Delta_l.
  \label{eqn:bias_crude_bound}
\end{align}
Now suppose $\gamma \leq \frac{1}{4\beta K}$. 
Note that for any $u, v \in [n]$ such that $\sigma(u) = \sigma(v)$ but $\sigma_0(u) \neq \sigma_0(v)$, we either have $\sigma(u) \neq \sigma_0(u)$ or $\sigma(v) \neq \sigma_0(v)$. Therefore,
\begin{align*}
  |S(\sigma) \cap S(\sigma_0)^c| \leq
  |\{ (u,v) \in [n]^2 \,:\, u \neq v,\, \sigma_0(u) \neq \sigma(u) \textrm{ or }  \sigma_0(v) \neq \sigma(v) \}| \leq 2 n^2 \gamma.
\end{align*}
By a symmetric argument, it follows that $|S(\sigma)^c \cap S(\sigma_0)| \leq 2 n^2 \gamma$, as well. Define $\hat{n}_k = |\{ u \in [n] \,:\, \sigma(u) = k\}|$. Then
\begin{align}
  |S(\sigma)| = \sum_{k=1}^K |\{(u,v) \in [n] \,:\, u\neq v,\, \sigma(u) = \sigma(v) = k \} |
  = \sum_{k=1}^K \hat{n}_k^2 \geq \frac{n^2}{K},
  \label{eqn:S_bound_1}
\end{align}
where the inequality holds by observing that $n = \sum_{k=1}^K n_k$ and applying the Cauchy-Schwarz inequality. We also have
\begin{align}
  |S(\sigma)^c| = \sum_{k=1}^K \hat{n}_k \sum_{k' \neq k} \hat{n}_{k'}
 \stackrel{(a)} \geq \sum_{k=1}^K \hat{n}_k (K-1) \frac{n}{2 \beta K} \geq n^2 \frac{(K-1)}{2 \beta K} \geq \frac{n^2}{4 \beta},
  \label{eqn:S_bound_2}
\end{align}
where $(a)$ follows because $l(\sigma, \sigma_0) \leq \gamma \leq \frac{1}{4 \beta K}$, so $\min_{k \in [K]} | \{ u \in [n] \,:\, \sigma(u) = k \} | \geq \frac{n}{2\beta K}$. 

Combining these bounds with equations~\eqref{eqn:Pl_bias} and~\eqref{eqn:Ql_bias}, we conclude that if $\gamma \leq \frac{1}{4 \beta K}$, we have
\begin{align}
  \E \hat{P}_l - P_l \leq K \gamma \Delta_l \quad \textrm{ and } \quad
  \E \hat{Q}_l - Q_l \leq 4\beta \gamma \Delta_l .                     \label{eqn:PlQl_bias_bound}
\end{align}

We now bound the variance. We use the shorthand
\begin{align*}
  T_{1,P} &:= (P_l \vee Q_l) | S(\sigma)|,  \quad T_2 :=  \gamma n \log \frac{e K}{\gamma} + 6 \log n, \quad t_P := \sqrt{7/3} \{ T_{1,P} T_2 \vee T_2^2 \}^{1/2}, \\
  T_{1,Q} &= (P_l \vee Q_l) | S(\sigma)^c |, \quad  \textrm{ and }
            t_Q := \sqrt{7/3} \{ T_{1,Q} T_2 \vee T_2^2 \}^{1/2}.
\end{align*}

We let $E_P(\sigma, l)$ be the event that $| \hat{P}_l - \E \hat{P}_l | \leq \frac{t_P}{\frac{1}{2} | S(\sigma)|}$ and let $E_Q(\sigma, l)$ be the event that $| \hat{Q}_l - \E \hat{Q}_l | \leq \frac{t_Q}{\frac{1}{2} | S(\sigma)^c |}$. For convenience, we also denote $ \tilde{A}_{uv} := \mathbf{1}\{A_{ij} = l\}$. By Bernstein's inequality, we have
\begin{align}
P\bigl(E_P(\sigma, l)^c \bigr) \leq P\biggl( \biggl| \sum_{ \substack{ (u,v) \,:\, u < v, \\ \sigma(u) = \sigma(v)} } & (\tilde{A}_{uv} - \E \tilde{A}_{uv} )  \biggr|  > t_P 
 \biggr) \leq 2 \exp\biggl( 
    - \frac{\frac{1}{2} t_P^2}{ \sum_{ \substack{ (u,v) \,:\,u< v, \\ \sigma(u) = \sigma(v)} } \E \tilde{A}_{uv}  + \frac{1}{3}t_P }
           \biggr) \nonumber \\
  & \leq \exp\biggl( 
    - \frac{\frac{1}{2} t_P^2}{\frac{1}{2} (P_l\vee Q_l) |S(\sigma)|   + \frac{2}{3}t_P } \biggr)
    = \exp \bigg( - \frac{t_P^2}{ T_{1,P} + \frac{1}{3}t_P } \biggr) \label{eqn:bernstein1}.
\end{align}

We consider two cases:
\begin{enumerate}
\item Suppose $T_{1,P} \geq T_2$. Then $t_P^2 = (7/3) T_{1,P}T_2$ and 
\begin{equation*}
2 \exp \left( - \frac{t_P^2}{T_{1,P} + \frac{2}{3} t_P} \right)  \leq 2 \exp \left( - \frac{(7/3) T_{1,P}T_2}{T_{1,P} + \frac{4}{3} T_{1,P}} \right) \leq  2 \exp( - T_2 ).
\end{equation*}
\item Suppose $T_{1,P} \leq T_2$. Then $t_P^2 = (7/3) T_2^2$, and the probability term is at most
\begin{equation*}
2 \exp \left( - \frac{4 T_2^2}{A + \frac{4}{3} T_2} \right) \leq 2 \exp \left( - \frac{(7/3) T_2^2}{ T_2 + \frac{4}{3} T_2} \right) \leq 2 \exp( - T_2).
\end{equation*}
\end{enumerate}

Combining the above with inequality~\eqref{eqn:bernstein1}, we have $P(E_P(\sigma, l)) \geq 1 - 2 \exp( - T_2)$. By an identical argument, we can show that $P(E_Q(\sigma, l)) \geq 1 - 2 \exp(- T_2)$, as well. Now note that
\begin{align*}
    \frac{t_P}{ \frac{1}{2} |S(\sigma)|  } 
  &\leq 
    2\sqrt{\frac{7}{3}} \frac{(T_1 T_2)^{1/2}}{|S(\sigma)|} +
    2\sqrt{\frac{7}{3}} \frac{T_2}{|S(\sigma)|} \\
  &\leq 4 \biggl( \frac{(P_l \vee Q_l) T_2}{|S(\sigma)|} \biggr)^{1/2}
    + 4 \frac{T_2}{|S(\sigma)|} \\
  &\stackrel{(a)} \leq 4 \biggl( \frac{K (P_l \vee Q_l) }{n} \frac{T_2}{n} \biggr)^{1/2} +
    4 \frac{K}{n} \frac{T_2}{n} \\
   &\stackrel{(b)} \leq 4 \biggl( K \frac{T_2}{n} \biggr)^{1/2} \biggl( \frac{P_l \vee Q_l }{n} \biggr)^{1/2} + 4 K \frac{T_2}{n} \biggl( \frac{P_l \vee Q_l }{n} \biggr)^{1/2},
\end{align*}
where $(a)$ follows from inequality~\eqref{eqn:S_bound_1} and $(b)$ holds because $P_l \vee Q_l \geq \frac{1}{n}$. In an identical manner, we can use inequality~\eqref{eqn:S_bound_2} to show that
\begin{align*}
  \frac{t_Q}{2 |S(\sigma)^c| } \leq  4 \biggl( 4 \beta \frac{T_2}{n} \biggr)^{1/2} \biggl( \frac{P_l \vee Q_l }{n} \biggr)^{1/2} + 16 \beta \frac{T_2}{n} \biggl( \frac{P_l \vee Q_l }{n} \biggr)^{1/2}.
\end{align*}
Let $\eta' := 2 \beta K \frac{T_2}{n}$. Using the fact that $4\beta \vee K \leq 2\beta K$, we have
\begin{align}
  \max\{ | \hat{P}_l - \E \hat{P}_l|, |\hat{Q}_l - \E \hat{Q}_l| \}
  \leq 4 ( \sqrt{\eta'} + \eta' )  \biggl( \frac{P_l \vee Q_l}{n} \biggr)^{1/2},\label{eqn:PlQl_variance_bound}
\end{align}
under event $E_P(\sigma, l) \cap E_Q(\sigma, l)$.

We now take a union bound over all colors $l$ and over all clusterings $\sigma$ satisfying $l(\sigma, \sigma_0) \leq \gamma$. There are at most $\binom{n}{\gamma n} K^{\gamma n}$ possible $\sigma$'s satisfying the error bound. Since
\begin{align*}
\log \left(\binom{n}{\gamma n} K^{\gamma n}\right)
 & \leq \log \left( \frac{ n^{\gamma n} e^{\gamma n} }
     { (\gamma n)^{\gamma n} } \frac{1}{\sqrt{2\pi \gamma n}} \right) + \gamma n \log K \\
 & \leq \log \left( \frac{ e^{\gamma n} }{\gamma^{\gamma n}} \right) - \frac{1}{2} \log 2 \pi \gamma n + \gamma n \log K \\
 & \leq \gamma n \log \frac{e}{\gamma}  + \gamma n \log K = \gamma n \log \frac{e K}{\gamma} ,
\end{align*}
we conclude that
\begin{align*}
\mathbb{P} \biggl( \bigcap_{l \in \{0,\ldots, L} \bigcap_{\sigma \,:\, l(\sigma, \sigma_0) \leq \gamma} (E_P(\sigma, l) \cap E_Q(\sigma, l)) \biggr) & \geq 1 - 4 (L+1) \exp( - 6 \log n) \\
& \geq 1 - 4 (L+1) n^{-6}.
\end{align*}
Combining inequalities~\eqref{eqn:bias_crude_bound}, \eqref{eqn:PlQl_bias_bound}, and~\eqref{eqn:PlQl_variance_bound}, we conclude that, with probability at least $1 - 4 (L+1) n^{-6}$, for all clusterings $\sigma$ such that $l(\sigma, \sigma_0) \leq \gamma$ and for all $l \in \{0, \ldots, L\}$, 
\begin{align*}
\max \{ | P_l - \hat{P}_l|,\, |Q_l - \hat{Q}_l| \} \leq \Delta_l + \eta \biggl( \frac{P_l \vee Q_l}{n} \biggr)^{1/2}.
\end{align*}
If also $\gamma \leq \frac{1}{4 K \beta}$, then
\begin{equation*}
\max\left\{|P_l - \hat P_l|, \; |Q_l - \hat Q_l|\right\} \leq \eta \Delta_l + \eta \biggl( \frac{P_l \vee Q_l}{n} \biggr)^{1/2},
\end{equation*}
where $\eta := (2 \beta K \gamma) \vee 4 (\sqrt{\eta'} + \eta') \leq 4 ( \sqrt{\eta'} + \eta')$ and $\eta' = 2 \beta K \gamma \log \frac{e K}{\gamma} + 6 \frac{\log n}{n}$. The statement of the theorem follows immediately.
\end{proof}

\begin{proposition}
  \label{prop:PlQl_lower_bound}
  Let $L, A, (\{\hat{P}_l\}, \{\hat{Q}_l\})$ be defined as in Proposition~\ref{prop:estimation_consistency} and suppose $n \geq 2$. Then with probability at least $1 - (L+1) \exp\bigl( - \frac{n}{5 \beta K} \bigr)$, we have that, for any clustering $\sigma$, and for any $l \in \{0,\ldots,L\}$ such that $P_l \vee Q_l \geq \frac{1}{n}$,
  \[
    \hat{P}_l \vee \hat{Q}_l \geq \frac{1}{\beta K} (P_l \vee Q_l).
  \]
\end{proposition}

\begin{proof}

  Fix a clustering $\sigma$, fix $l \in \{0, \ldots, L\}$, and suppose $P_l \vee Q_l \geq \frac{1}{n}$. Define $\hat{T}_l := \frac{2}{n(n-1)} \sum_{(u,v) \,:\, u < v} (A_l)_{uv}$. Let $S(\sigma)$ and $S(\sigma)^c$ be defined as in the statement of Proposition~\ref{prop:estimation_consistency}. Then
  \begin{align}
    \hat{P}_l \vee \hat{Q}_l \geq
    \frac{|S(\sigma)| \hat{P}_l}{n(n-1)} +
    \frac{|S(\sigma)^c| \hat{Q}_l}{n(n-1)} = \hat{T}_l.
  \end{align}
Note that $\hat{T}_l$ does not depend on $\sigma$. Let $t := \frac{1}{2} \E \hat{T}_l$, and let $E(l)$ be the event that $\hat{T}_l \geq t$. Note that
    \begin{align}
    t = \frac{1}{n(n-1)} ( | S(\sigma_0) P_l + |S(\sigma_0)^c| Q_l)
      \stackrel{(a)} \geq \frac{1}{K} P_l + \frac{1}{2 \beta} Q_l \geq \frac{1}{\beta K} (P_l \vee Q_l). \label{eqn:E_hatT_bound}
    \end{align}
    In the above derivations, $(a)$ follows because we can use similar reason as in inequalities~\eqref{eqn:S_bound_1} and~\eqref{eqn:S_bound_2} to show that $|S(\sigma_0)| \geq \frac{n^2}{K}$ and $| S(\sigma_0)^c | \geq \frac{n^2}{2\beta}$. The last inequality follows because $2 \beta \vee K \leq \beta K$. Thus, under event $E(l)$, for any $\sigma$, we have $\hat{P}_l \vee \hat{Q}_l \geq \hat{T}_l \geq \frac{1}{\beta K} (P_l \vee Q_l)$. By Bernstein's inequality, we have
\begin{align*}
    \mathbb{P}(E(l)^c)
    &\leq \mathbb{P}\biggl( \biggl| \sum_{(u,v) \,:\, u < v} (A_l)_{uv} - \E (A_l)_{uv} \biggr| \geq t \biggr) \\
    &\leq \exp \biggl( \frac{-  \frac{1}{2} t^2}{ \frac{1}{3} t + \sum_{(u,v) \,:\, u < v } \E (A_l)_{uv} } \biggr) \leq \exp \bigl( - \frac{t}{5} \bigr) \stackrel{(a)} \leq \exp \bigl( - \frac{n}{5 \beta K} \bigr),
  \end{align*}
  where $(a)$ follows from inequality~\eqref{eqn:E_hatT_bound}. A union bound over all colors $l \in \{0, \ldots, L\}$ finishes the proof.
\end{proof}


\subsection{Analysis of spectral clustering}
\label{appendix: spectral}

\begin{proposition}
\label{prop:spectral_analysis}
Let $\sigma_0 \in \mathcal{C}(\beta, K)$, let $p,q \in [1/n, 1]$, and let $A \in \{0,1\}^{n \times n}$ be a random matrix with the distribution $SBM(\sigma_0, p, q)$. Let $C_{spec1} := 2^{34}$ and $C_{spec2} := 2^{29}$. Suppose $n \geq 8 \beta^2 K^2\vee 2000$ and
\begin{align}
 \frac{n}{\beta^2 K^6} \frac{(p-q)^2}{p \vee q} \geq C_{spec1}. \label{eqn:ridiculous_constant_bound}
\end{align}
Then the output $\sigma$ of Algorithm~\ref{alg:spectral} with parameters $\mu = 4 \beta$ and $\tau = 40 K \bar{d}$ satisfies
\[
l(\sigma, \sigma_0) \leq C_{spec2} \frac{\beta K^4}{n} \frac{p \vee q}{(p-q)^2},
\]
with probability at least $1 - n^{-5}$.
\end{proposition}

\begin{proof}
Let $E_1$ be the event that $\frac{1}{2K} n (p \vee q) \leq \bar{d} \leq 3 n (p \vee q)$. 
By Proposition~\ref{prop:dbar_bound} and the assumption that $K \leq n^{1/6}$, we have $P(E_1) \geq 1 - \exp\bigl( - \frac{n}{12 K}\bigr) \geq 1 - \exp\bigl( - \frac{n^{5/6}}{12}\bigr)$. Under event $E_1$, we have $\tau = 40 K \bar{d} \geq 20 n (p \vee q)$, so we may apply Lemma~\ref{lem:trimmed_A_bound} to obtain
\begin{equation*}
P(E_2 \cap E_1) \geq 1 - \exp \bigl(-\frac{n^{5/6}}{12} \bigr) - n^{-6},
\end{equation*}
where $E_2$ is defined to be the event that
\[
  \| T_\tau(A) - P \|_2 \leq 2^{11} K \sqrt{n (p \vee q)}.
\]
Note that $n \geq 2000$ implies that $\frac{3}{2} \frac{n}{\log n} \geq 6 + \log(2)$. 

Now suppose the event $E_1 \cap E_2$ holds. Then
\begin{align}
  D^2_{total} := \sum_{u=1}^n \| \hat{A}_u - P_u \|^2
  &\leq \| \hat{A} - P \|_F^2 \leq K \| \hat{A} - P \|_2^2 \nonumber  \\
  &\leq K \bigl( \| T_\tau (A) - P \|_2 + \| \hat{A} - T_{\tau}(A) \|_2 \bigr)^2 \nonumber \\
  &\stackrel{(a)} \leq 4 K \| T_\tau(A) - P \|^2_2 \leq 2^{24} K^3 n (p \vee q). \label{eqn:total_error}
\end{align}
Let $\sigma_0 : [n] \rightarrow [K]$ denote the true clustering function, and let $\{ \mathcal{Z}_1, \ldots, \mathcal{Z}_K\}$ denote the $K$ unique rows of $P$, where we use the indexing convention that $P_u = \mathcal{Z}_k$ if and only if $\sigma_0(u) = k$. We also let $\{ \mathcal{S}_1, \ldots \mathcal{S}_K \}$ denote the set $\{ \hat{A}_u \}_{u \in S}$, where we use the indexing convention that, for any $k \in [K]$, we have $\argmin_{\mathcal{Z}_{k'} \in \{\mathcal{Z}_1, \ldots, \mathcal{Z}_K\}} \| \mathcal{S}_k - \mathcal{Z}_{k'}\|_2 = \mathcal{Z}_k$. For each node $u \in [n]$, we define
\[
  \mathcal{S}^*(u) := \argmin_{\mathcal{S}_k \in \{\mathcal{S}_1,\ldots \mathcal{S}_K\}}
  \| \hat{A}_u - \mathcal{S}_k \|_2, \quad \textrm{and} \quad
   \mathcal{Z}^*(u) := \argmin_{\mathcal{Z}_k \in \{\mathcal{Z}_1,\ldots \mathcal{Z}_K\}}
  \| \hat{A}_u - \mathcal{Z}_k \|_2.
\]
We also define the shorthand $D_{sep}^2 :=  2 \frac{n}{\beta K} (p-q)^2 $. For a node $u \in [n]$, we call $u$ \emph{valid} if $\| \hat{A}_u - \mathcal{Z}^*(u) \|_2 \leq D_{sep}/8$. We now make several claims that we use in the proof.

\begin{enumerate}
  \item[Claim 1:]
For $k, k' \in [K]$ such that $k \neq k'$, we have 
$ \| \mathcal{Z}_k - \mathcal{Z}_{k'} \|_2 \geq D_{sep}$.
\item[Claim 2:] For any $k \in [K]$, we have
  $\bigl| \bigl\{ u \in \mathcal{C}_k \,:\, D(u) < \frac{2 D_{total}}{( \frac{n}{\mu \beta K^2} )^{1/2}},\,
  \| \hat{A}_u - \mathcal{Z}_k \|_2 < \frac{D_{total}}{\left( \frac{n}{\mu \beta K^2} \right)^{1/2}} \bigr\} \bigr|
  \geq |\mathcal{C}_k|(1-1/(\mu K))$.
  Thus, it follows that the $(1-1/(\mu K))\textrm{-quantile}$ of $\{ D(u) \,:\, u \in [n] \}$ is less than  $\frac{2 D_{total}}{\left( \frac{n}{\mu \beta^2 K} \right)^{1/2}}$.
\item[Claim 3:] If a node $u \in S$, then $u$ is valid.
\item[Claim 4:] For any $u, v \in S$ such that $u \neq v$, we have $\mathcal{Z}^*(u) \neq \mathcal{Z}^*(v)$.
\item[Claim 5:] For $k, k' \in [K]$ such that $k \neq k'$, we have $\| \mathcal{S}_k - \mathcal{S}_{k'} \|_2 \geq \frac{3}{4} D_{sep}$. 
\item[Claim 6:] If $u$ is valid, then $\| \hat{A}_u - \mathcal{S}^*(u) \|_2 \leq D_{sep}/4$.
  \end{enumerate}

 Claim 1 follows from the SBM definition and the fact that the smallest cluster has at least $\frac{n}{\beta K}$ elements.  To derive Claim 2, let $k \in [K]$, define $\mathcal{C}_k := \{ u \,:\, \sigma_0(u) = k\}$, and let $C' \subset \mathcal{C}_k$ be such that $|C'| = \lfloor |\mathcal{C}_k|/(\mu K) \rfloor$ and $\max_{u \in \mathcal{C}_k \backslash C'} \| \hat{A}_u - \mathcal{Z}_k \|_2 \leq \min_{u \in C'} \| \hat{A}_u - \mathcal{Z}_k \|_2$. Then
  \begin{align*}
    D^2_{total} \geq \sum_{v \in C'} \| \hat{A}_v - \mathcal{Z}_k \|_2^2
    \geq | C' | \min_{v \in C'} \| \hat{A}_v - \mathcal{Z}_k \|_2^2
    \stackrel{(a)} > \bigl( \frac{n}{\mu \beta K^2}-1 \bigr) \max_{v \in \mathcal{C}_k \backslash C'} \| \hat{A}_v - \mathcal{Z}_k \|_2^2,
  \end{align*}
  where $(a)$ follows because $|\mathcal{C}_k| \geq \frac{n}{\beta K}$. Thus, we have shown that for any $v \in \mathcal{C}_k \backslash C'$, we have $\| \hat{A}_v - \mathcal{Z}_k \|_2 < \frac{D_{total}}{\bigl( \frac{n}{\mu \beta K^2} - 1 \bigr)^{1/2}}$. Since $| \mathcal{C}_k \backslash C' | = \mathcal{C}_k | (1 - 1/(\mu K)) \geq \frac{n}{2 \beta K} \geq \lceil \frac{n}{\mu K} \rceil$, we have that, for any $v \in \mathcal{C}_k \backslash C'$,
  \begin{align*}
    D(v) \leq  \max_{v' \in \mathcal{C}_k \backslash C'} \| \hat{A}_v - \hat{A}_{v'} \|_2
    \leq  \max_{v' \in \mathcal{C}_k \backslash C'} \bigl( \| \hat{A}_v - \mathcal{Z}_k \|_2 + \| \hat{A}_{v'} - \mathcal{Z}_k \|_2 \bigr)
    < \frac{2 D_{total}}{\bigl( \frac{n}{\mu \beta K^2} - 1 \bigr)^{1/2}}.
  \end{align*}
  Claim 2 follows by again noting that $| \mathcal{C}_k \backslash C'| \geq |\mathcal{C}_k | ( 1 - 1/(\mu K))$. 

To argue Claim 3, let $u$ be an arbitrary invalid node and let $N(u) = \{ v \in [n] \,:\, \| \hat{A}_u - \hat{A}_v \| \leq D(u) \}$; it follows from the definition of $D(u)$ that $|N(u)| \geq \lceil \frac{n}{\mu K} \rceil$. Then
  \begin{align*}
    D^2_{total} &\geq \sum_{v \in N(u)} \| \hat{A}_v - \mathcal{Z}^*(v) \|_2^2
     \geq  \sum_{v \in N(u)}  \max\bigl\{ \| \hat{A}_u - \mathcal{Z}^*(v) \|_2 -
    \| \hat{A}_v - \hat{A}_u \|_2,\, 0 \bigr\}^2  \\
                &\geq \frac{n}{\mu K} \max\{ D_{sep}/8 - D(u), 0 \}^2.
  \end{align*}
  Therefore,
\begin{align*}
  D(u) \geq \frac{D_{sep}}{8} - \frac{D_{total} }{( \frac{n}{\mu K} )^{1/2}}
  \stackrel{(a)} \geq  \frac{D_{sep}}{16}
  \stackrel{(b)} \geq \frac{2 D_{total}}{ ( \frac{n}{\mu \beta K^2} -1 )^{1/2}}
  \stackrel{(c)} > (1-1/(\mu K))\textrm{-quantile} \{ D(u) \,:\, u \in [n]\},
\end{align*}
where $(a)$ and $(b)$ follows from hypothesis~\eqref{eqn:ridiculous_constant_bound} of the proposition, and $(c)$ follows from Claim 2. Claim 3 can then be shown from the definition of the set $S$. 
  
We argue Claim 4 by induction. Let $\{u_1, \ldots u_K\} = S$, where $u_1$ is the first node added to $S$, $u_2$ is the second, etc. Let $r$ be a positive integer such that $2 \leq r < K$, and suppose $\mathcal{Z}^*(u_s) \neq \mathcal{Z}^*(u_{s'})$ for all $s, s' \in [r]$. Since $r < K$, there exists $k \in [K]$ such that $\mathcal{Z}^*(u_s) \neq \mathcal{Z}_k$ for all $s \in [r]$. Since $\frac{n}{\mu K} \leq \frac{n}{2\beta K}$, by Claim 2 and the Pigeonhole Principle, there must exist a node
$u' \in \mathcal{C}_k$ such that $\| \hat{A}_{u'} - \mathcal{Z}_k \|_2
\leq \frac{D_{total}}{(\frac{n}{\mu \beta K^2} - 1 )^{1/2}}
\leq D_{sep}/8$,
where the latter inequality is true by hypothesis~\eqref{eqn:ridiculous_constant_bound} of the proposition and the fact that $D(u') \leq (1 - 1/(\mu K))\textrm{-quantile} \{ D(u) \,:\, u \in [n]\}$. Let $u_{r+1} \in S$. Then by the definition of $S$, for any $s \in [r]$, we have
\begin{align*}
  \| \hat{A}_{u_s} - \hat{A}_{u_{r+1}} \|_2
  \geq \| \hat{A}_{u_s} - \hat{A}_{u'} \|_2
  \geq \| \mathcal{Z}_k - \mathcal{Z}^*(u_s) \|_2 -
  \| \hat{A}_{u_s} - \mathcal{Z}^*(u_s) \|_2 -
  \| \hat{A}_{u'} - \mathcal{Z}_k \|_2
  \stackrel{(a)} \geq \frac{3}{4} D_{sep},
\end{align*}
where $(a)$ follows by Claim 1 and Claim 3. At the same time, let $v$ be a valid node such that $\mathcal{Z}^*(v) = \mathcal{Z}^*(u_s)$ for some $s \in [r]$. Then
\[
  \| \hat{A}_v - \hat{A}_{u_s} \|_2
  \leq \| \hat{A}_v - \mathcal{Z}^*(v) \|_2 +
  \| \hat{A}_{u_s} - \mathcal{Z}^*(u_s) \|_2
 \stackrel{(a)} \leq D/4,
\]
where $(a)$ follows again by Claim 3. Therefore, by definition of $u_{r+1}$, it must be that $\mathcal{Z}^*(u_{r+1}) \neq \mathcal{Z}^*(u_s)$ for any $s \in [r]$. Claim 4 follows by induction. 

Claim 5 is true because, for any $k \neq k'$, we have
  \begin{align*}
    \| \mathcal{S}_k - \mathcal{S}_{k'} \|_2
    \geq \|\mathcal{Z}_{k'} - \mathcal{Z}_k \|_2 -
    \| \mathcal{Z}_k - \mathcal{S}_k \|_2 -
    \| \mathcal{Z}_{k'} - \mathcal{S}_{k'} \|_2 \stackrel{(a)} \geq (3/4) D_{sep},
  \end{align*}
where $(a)$ holds because of Claim 1, Claim 3, Claim 4, and the indexing convention for $\{ \mathcal{S}_1, \ldots \mathcal{S}_{K} \}$.

To see that Claim 6 is true, let $u$ be valid and suppose $\mathcal{Z}^*(u) = \mathcal{Z}_k$ for some $k \in [K]$. Then
\begin{align*}
  \| \hat{A}_u - \mathcal{S}^*(u) \|_2
  \leq \| \hat{A}_u - \mathcal{S}_k \|_2
  \leq \|\hat{A}_u - \mathcal{Z}_k \|_2 + \| \mathcal{Z}_k - \mathcal{S}_k \|_2 \leq \frac{D_{sep}}{4}.
  \end{align*}

We have proved all six claims and now proceed to the proof of the proposition. We say that a node $u \in [n]$ is \emph{incorrect} if $u$ is valid and if $\mathcal{S}^*(u) \neq \mathcal{S}_{\sigma_0(u)}$. Suppose $u$ is incorrect and suppose without the loss of generality that $\sigma_0(u) = k$. Then
\begin{align*}
  \| \hat{A}_u - \mathcal{Z}_k \|_2
  &\geq \| \hat{A}_u - \mathcal{S}_k \|_2 -
  \| \mathcal{S}_k - \mathcal{Z}_k\|_2 \\
  &\geq \| \mathcal{S}_k - \mathcal{S}^*(u) \|_2 -
  \| \hat{A}_u - \mathcal{S}^*(u) \|_2 -  \| \mathcal{S}_k - \mathcal{Z}_k\|_2 \\
  &\stackrel{(a)} \geq \frac{3}{4} D_{sep} - \frac{1}{4} D_{sep} - \frac{1}{8} D_{sep} \geq \frac{3}{8} D_{sep},
\end{align*}
where $(a)$ holds because of Claim 5 and Claim 6. Define a permutation $\tau \,:\, [K] \rightarrow [K]$ such that for $i, k \in [K]$, we have $\tau(i) = k$ if $S[i] = \mathcal{S}_k$. Then
\begin{align*}
  l(\sigma_0, \sigma)
  &\leq d_H(\sigma_0, \tau \circ \sigma) \leq \frac{| \{ u \in [n] \,:\, u \textrm{ is invalid or incorrect} \}| }{n} \\
  &\leq \frac{1}{n} \frac{D^2_{total}}{(D_{sep}/8)^2}
    \leq 2^{29} \frac{\beta K^4}{n} \frac{p \vee q}{(p-q)^2}.
\end{align*}
We finish the proof by noting that if $n \geq 2000$, then $P(E_1 \cap E_2) \geq 1 - e^{- n^{5/6} / 12} - n^{-6} \geq n^{-5}$.
\end{proof}

The following Lemma is Lemma 3.3 in \cite{chin2015stochastic} and also as Lemma 5 in \cite{Gaoetal15}. We transcribe the full statement and proof here to make the paper self-contained.

\begin{lemma}
\label{lem:trimmed_A_bound}
Let $P \in [0,1]^{n \times n}$ be a symmetric matrix, let $p_{max} := \max_{u \geq v} P_{uv}$, and suppose $n p_{\max} \geq 1$. Let $A$ be an adjacency matrix such that $A_{uu} = 0$ and $A_{uv} \sim Ber(P_{uv})$ for $u < v$. Let $\tilde{\tau} \geq 20$ and $\tau := \tilde{\tau} n p_{\max}$. Then for any $C' > 0$ and any $n$ such that $\frac{3}{2} \frac{n}{\log n} \geq C' + \log2$, we have
\[
 \| T_\tau(A) - P \|_2 \leq \frac{4}{3}(461 + 160 C' + 16 \tilde{\tau}) \sqrt{n p_{\max}},
\]
with probability at least $1 - n^{-C'}$.
\end{lemma}

\begin{proof}
  Let $S := \{ u \in [n] \,:\, d_u < \tau\}$ and let $A_{SS} := T_{\tau}(A)$. Let $P_{SS}$ be the result of setting row/column $u$ of $P$ to zero for every $u \in S^c$. Observe then that
  \[
    \|T_{\tau}(A) - P \|_2 \leq  \| P_{SS} - P \|_2 + \| A_{SS} - P_{SS} \|_2.
  \]
We first bound $\| P_{SS} - P \|_2$. By Proposition~\ref{prop:high_degree_nodes_bound}, with probability at least $1 - 2\exp(- n/2)$, we have
\begin{align*}
  \| P_{SS} - P \|_2 &\leq \| P_{SS} - P \|_F \leq \bigl| \{(u, v) \,:\, u \in S^c \textrm{ or } v \in S^c \} \bigr|^{1/2} p_{\max} \\
  & \leq 2^{1/2} \frac{n}{\tau^{1/2}} p_{\max}  \stackrel{(a)} \leq \sqrt{n p_{\max}},
\end{align*}
where $(a)$ follows because $\tilde{\tau} \geq 20$. Define $E_1$ to be the event that $\| P_{SS} - P \|_2 \leq \sqrt{n p_{\max}}$. 

Also define $M := A_{SS} - P_{SS}$. Let $\mathcal{C}$ be the minimal $1/8$-covering of $S^{n-1}$; it follows that $| \mathcal{C}| \leq 64^n$. For any $x, y \in S^{n-1}$, let $\tilde{x}, \tilde{y} \in \mathcal{C}$ be such that $\|x - \tilde{x} \|, \| y -\tilde{y} \| \leq 1/8$. Then
\begin{align*}
x^\top M y = \tilde{x}^\top M \tilde{y} + (x - \tilde{x})^\top M y + \tilde{x}^\top M (y - \tilde{y}) \leq \tilde{x}^\top M \tilde{y} + \frac{1}{4} \| M \|_2 .
\end{align*}
Taking the supremum over $x, y \in S^{n-1}$, we have
\begin{align}
  \| M \|_2 = \sup_{x, y \in S^{n-1}} x^\top M y \leq \frac{4}{3} \sup_{\tilde{x}, \tilde{y} \in \mathcal{C}} \tilde{x}^\top M \tilde{y}. \label{eqn:covering}
\end{align}
For any $x, y \in \mathcal{C}$, define $H(x,y) := \{ (u,v) \,:\, |x_u y_v| \geq \sqrt{\frac{p_{\max}}{n}}\}$ and $L(x,y) := \{ (u,v) \,:\, |x_u y_v| < \sqrt{\frac{p_{\max}}{n}}\}$. It follows that for any $x,y \in \mathcal{C}$, we have
\begin{align}
  x^\top M y = \sum_{(u,v) \in H(x, y)} x_u y_v M_{uv} + \sum_{(u,v) \in L(x,y)} x_u y_v M_{uv}.
  \label{eqn:light_heavy_decomposition}
\end{align}
Define $E_2$ to be the event in which $\max_{x,y \in \mathcal{C}} \sum_{(u,v) \in L(x,y)} x_u y_v M_{uv} \leq 12 \sqrt{p_{\max} n}$. By Proposition~\ref{prop:light_set_analysis} and a union bound over all $x,y \in \mathcal{C}$, we have $P(E_2) \geq 1 - e^{-n}$.
For any $x, y \in \mathcal{C}$, $\sum_{(u,v) \in H(x,y)} x_u y_v \leq \sum_{(u,v) \in H(x,y)} x_u^2 y_v^2 \sqrt{\frac{n}{p_{\max}}} \leq \sqrt{n}{p_{\max}}$. Thus, 
\begin{align}
  \sum_{(u,v) \in H(x, y)} x_u y_v (P_{SS})_{uv} \leq \sqrt{ n p_{\max}}. \label{eqn:heavy1}
\end{align}
For any $S',T' \subseteq [n]$, define $e(S', T' | A) :=  \sum_{(u,v) \,:\, (u,v) \in (S' \times T') \cup (T' \times S')} A_{uv}$. Let $E_3$ be the event that, for any $S', T' \subseteq [n]$, either
\begin{equation*}
\frac{e(S', T' | A)}{|S'||T'| p_{\max}} \leq 4
\end{equation*}
or
\begin{equation*}
\frac{e(S', T' | A)}{|S'||T'| p_{\max}} \log \frac{e(S', T' | A)}{|S'||T'| p_{\max}} \leq 3(4 + 2C') \frac{1}{|S'|} \log \frac{en}{|T'|}.
\end{equation*}
Then by Proposition~\ref{prop:heavy_set_prep}, we have $P(E_3) \geq 1 - n^{-2C'}$. For any $S', T' \subseteq [n]$, we have $e(S', T' | A_{SS} ) \leq e(S', T' | A)$, so in the event $E_3$, for any $S', T' \subseteq [n]$, we have either $\frac{e(S', T' | A_{SS})}{|S'||T'| p_{\max}} \leq 4$ or $\frac{e(S', T' | A_{SS})}{|S'||T'| p_{\max}} \log \frac{e(S', T' | A_{SS})}{|S'||T'| p_{\max}} \leq 3(4 + 2C') \frac{1}{|S'|} \log \frac{en}{|T'|}$. By Proposition~\ref{prop:heavy_set_analysis} with $d = n p_{\max}$, $C_2$ set to $8$, $C_3 = 2(4 + 2C')$, and $\tau = \tilde{\tau} n p_{\max}$, it is known that, in the event $E_3$, we have
\begin{align}
  \max_{x,y \in \mathcal{C}} \sum_{(u,v) \in H(x,y)} x_u y_v (A_{SS})_{uv} \leq (448 + 160 C' + 16 \tilde{\tau}) \sqrt{n p_{\max}} \label{eqn:heavy2}.
\end{align}
Under the event $E_1 \cap E_2 \cap E_3$, by combining inequalities~\eqref{eqn:covering}, \ref{eqn:light_heavy_decomposition}, \ref{eqn:heavy1}, \ref{eqn:heavy2}, and the definition of $E_2$, we have
\[
  \| M \|_2 \leq \frac{4}{3}(461 + 160 C' + 16 \tilde{\tau}) \sqrt{n p_{\max}}.
\]
To finish the proof, we take a union bound over the events $E^c_1, E^c_2$, and $E^c_3$, and observe that $2 e^{-n/2} + e^{-n} + n^{-2C'} \leq n^{-C'}$ when $\frac{3}{2} \frac{n}{\log n} \geq C' + \log 2$. 
\end{proof}

\begin{proposition}
  \label{prop:dbar_bound}
  Let $A, P$, and $p_{\max}$ be defined as in Lemma~\ref{lem:trimmed_A_bound}. Let $C_1 > 0$ be defined such that $\sum_{u < v} P_{uv} = C_1 n^2 p_{\max}$. Let $\bar{d} := \frac{1}{n} \sum_{u \neq v} A_{uv}$ be the average degree. Then with probability at least $\exp(- C n)$ where $C := \frac{C_1/2}{1/3 + 2}$, we have
  \[
    C_1 n p_{\max} \leq \bar{d} \leq 3 C_1 n p_{\max}.
  \]
  In the case of $SBM(K, \beta, p, q)$, we have $\frac{1}{2K} \leq C_1 \leq 1$. 
\end{proposition}

\begin{proof}
We use the shorthand $t := \frac{1}{2} C_1 n^2 p_{\max}$. By Bernstein's inequality (Proposition~\ref{prop:bernstein}), we have
  \begin{align*}
    \mathbb{P}\bigl( C_1 n p_{\max} \leq \bar{d} \leq 3 C_1 n p_{\max} \bigr)
    &\leq \mathbb{P}\biggl( 
      \biggl|  \sum_{(u,v) \,:\, u < v} (A_{uv} - \E A_{uv}) \biggr| \geq t
      \biggr) \\
    &\leq \exp\biggl(
      - \frac{ (1/2) t^2}{(1/3) t + C_1 n^2 p_{\max}} \biggr) \\
    &\leq \exp \biggl(
      - \frac{(1/2) C_1 n}{1/3 + 2} \biggr),
  \end{align*}
  where the last inequality follows because $n p_{\max} \geq 1$ by assumption. For the second claim, define $N_w := \sum_{k=1}^K \frac{n_k (n_k-1)}{2} $ and $N_b := \frac{n(n-1)}{2} - N_w$. Then
  \[
    C_1 = \frac{N_w}{n^2} \frac{p}{p_{\max}} + \frac{N_b}{n^2} \frac{q}{p_{\max}}.
  \]
  The lower bound on $C_1$ follows from the fact that $N_w \leq n(n/K - 1)$ and the assumption that $K \leq n/2$. The upper bound on $C_1$ follows from the fact that $N_w + N_b = n(n-1)/2$. 
\end{proof}

\subsection{Supporting results for Lemma~\ref{lem:trimmed_A_bound}}

The following proposition is Lemma 3.1 in \cite{chin2015stochastic} and Lemma 11 in \cite{GaoEtal15}. We transcribe the full statement and the proof here to make the paper self-contained.

\begin{proposition}
  \label{prop:high_degree_nodes_bound}
  Let $A, P$, and $p_{max}$ be defined as in Lemma~\ref{lem:trimmed_A_bound}. For a node $u$, let $d_u := \sum_{v \neq u} A_{uv}$ be the degree. Let $\tilde{\tau} \geq 20$ and let $\tau := \tilde{\tau} n p_{\max}$. Then with probability at least $1 - 2 \exp(- n/2)$, we have
  \[
    \bigl| \{ u \in \{1,\ldots,n\} \,:\, d_u \geq \tau \} \bigr| \leq \frac{n}{\tau}.
   \]
\end{proposition}

\begin{proof}
  Let $S \subseteq [n]$ and define $e(S,V) := \frac{1}{2} \sum_{(u,v) \,:\, u \in S \textrm{ or } v \in S} A_{uv}$. Since
\begin{equation*}
\big| \{ (u,v) \,:\, u \in S \textrm{ or } v \in S\} \bigr| \leq 2|S|n,
\end{equation*}
 we have $\E e(S,V) \leq |S| n p_{\max}$. Therefore, by Chernoff's bound (Proposition~\ref{prop:chernoff_for_bernoulli}), we have
  \begin{align}
    \mathbb{P} \bigl( d_u \geq \tau ,\, \forall u \in S \bigr)
    &\leq \mathbb{P}\bigl( e(S,V) \geq \frac{\tau}{2} |S| \bigr) \nonumber \\
    &\leq \mathbb{P} \biggl( e(S,V) \geq \E e(S,V) + \bigl( \frac{\tilde{\tau}}{2} - 1 \bigr) p_{\max} n |S| \biggr) \nonumber \\
    &\leq \exp \biggl( - \bigl( \frac{\tilde{\tau}}{2} -1 \bigr) p_{\max} n
      \biggl( \log \bigl( 1 + \frac{\bigl(\frac{\tilde{\tau}}{2}-1\bigr)  p_{\max} n}{\E e(S,V)} \bigr) - 1 \biggr) \biggr) \nonumber \\
    &\stackrel{(a)} \leq \exp \bigl( - \bigl( \bigl(\frac{\tilde{\tau}}{2} -1 \bigr) \log 10 \bigr) p_{\max} n \bigr) \stackrel{(b)} \leq \exp \bigl( - \frac{\tilde{\tau}}{4} p_{\max} n |S| \bigr).
      \label{eqn:single_S_bound}
  \end{align}
In the above, $(a)$ follows because $\frac{(\tilde{\tau}/2 -1) p_{\max} n}{\E e_1(S,V)} \geq 9$ and $\log(x) -1 \geq \log(x)/2$ for all $x \geq 10$, and $(b)$ follows because $(\tilde{\tau}/2 - 1) \log 10 \geq \tilde{\tau}$ for all $\tilde{\tau} \geq 20$.

Taking a union bound over all subsets of size greater than $\frac{n}{\tau}$, we have
  \begin{align*}
    \mathbb{P}\bigl( \exists S,\, |S| > \frac{n}{\tau},\, d_u \geq \tau,\, \forall u \in S \bigr)
    &\leq \sum_{l = \lceil n/\tau \rceil}^n \sum_{S \subseteq V,\, |S|=l} \exp \bigl( - \tilde{\tau} p_{\max} n l \bigr) \\
    &\stackrel{(a)} \leq  \sum_{l = \lceil n/\tau \rceil}^n \exp \bigl( - \tilde{\tau} p_{\max} n l + l \log \frac{en}{l} \bigr) \\
    &\stackrel{(b)} \leq  \sum_{l = \lceil n/\tau \rceil}^n  \exp \biggl( - l \bigl( \tilde{\tau} p_{\max} n - \log( e \tilde{\tau} p_{\max} n ) \bigr) \biggr) \\
    &\stackrel{(c)}\leq   \sum_{l = \lceil n/\tau \rceil}^n \exp \bigl( - \frac{\tilde{\tau}}{2} p_{\max} n l \bigr) \stackrel{(d)} \leq 2 \exp( - n/2 ).
  \end{align*}
  In the above derivations, $(a)$ follows from the Stirling approximation; $(b)$ follows because $\frac{n}{l} > \tau$; $(c)$ follows because $\tilde{\tau} - \log(e \tilde{\tau} ) \geq \tilde{\tau}/2$ for all $\tilde{\tau} \geq 20$ and $p_{\max} n \geq 1$; $(d)$ follows because  $ \sum_{l = \lceil n/\tau \rceil}^n \exp \bigl( - \frac{\tilde{\tau}}{2} p_{\max} n l \bigr)$ is a geometric series the largest term of which is bounded above by $\exp( - n/2)$, and the ratio is $\exp( - \frac{\tilde{\tau}}{2} p_{\max} n) \leq e^{-10} \leq 1/2$. 
\end{proof}

\begin{proposition}
  \label{prop:light_set_analysis}
  Let $A, P$, and $p_{\max}$ be defined as in Lemma~\ref{lem:trimmed_A_bound}. Let $x, y \in S^{n-1}$ and let $L := \{ (u,v) \,:\, |x_u y_v| \leq \sqrt{\frac{p_{\max}}{n}} \}$. Then with probability at least $\exp(- 8 n)$,
  \[
    \sup_{S \subseteq [n]} \sum_{(u,v) \in L} x_u y_v (A_{SS} - P_{SS})_{uv} \leq 12 \sqrt{p_{\max} n}.
  \]
\end{proposition}

\begin{proof}
  First fix $S \subseteq [n]$. Since $\max_{(u,v) \in L} |x_u y_v (A_{SS} - P_{SS})_{uv} | \leq \sqrt{ \frac{p_{\max}}{n}}$ and $\sum_{(u,v) \in L} |x_u y_v|^2 \E ((A_{SS} - P_{SS})_{uv}^2 ) \leq p_{\max}$, we may apply Bernstein's inequality (Proposition~\ref{prop:bernstein}) to obtain $\sum_{(u,v) \in L} x_u y_v (A_{SS} - P_{SS})_{uv} \geq 12 \sqrt{p_{\max} n}$, with probability at most $\exp( - 9 n)$. We now take a union bound to obtain
  \begin{align*}
    \mathbb{P}\biggl( \max_{S \subseteq [n]} \sum_{(u,v) \in L} x_u y_v (A_{SS} - P_{SS})_{uv} \geq 12 \sqrt{p_{\max} n} \biggr)
    \leq \exp\bigl( - 9 n + n \log 2 \bigr) \leq \exp(- 8 n).
  \end{align*}
\end{proof}

The following Lemma is Lemma A.3 in \cite{chin2015stochastic}. We transcribe the full statement and proof here to make the paper self-contained.

\begin{proposition}
  \label{prop:heavy_set_prep}
  Let $A, p_{\max}$ be defined as in Lemma~\ref{lem:trimmed_A_bound}. Let $d := n p_{\max}$. Let $C' > 0$. For any $S,T \subseteq [n]$, define $e(S,T) := \frac{1}{2} \sum_{(u,v) \,:\, (u,v) \in (S \times T ) \cup (T \times S)} A_{uv}$. Then with probability at least $1 - n^{-C'}$, for any $S,T \subseteq [n]$, we have
  $$ \frac{e(S, T)}{|S||T| \frac{d}{n}} \leq 8, \qquad \textrm{or} \qquad
  \frac{e(S, T)}{|S||T| \frac{d}{n}} \log \frac{e(S, T)}{|S||T| \frac{d}{n}} \leq 2(4 + C') \frac{1}{|S|} \log \frac{en}{|T|}. $$
\end{proposition}

\begin{proof}

 Fix $C' > 0$. Let us first fix $S, T \subseteq [n]$ and, without loss of generality, suppose that $|T| \geq |S|$. Since $|\{ (u,v) \,:\, u\in S, v \in T \textrm{ or } v \in S, u\in T\}| \leq 2 |S||T|$, we have $\E e(S,T) \leq p_{\max} |S||T|$. By a Chernoff bound (Proposition~\ref{prop:chernoff_for_bernoulli}), for $l \geq 8$, we have $\frac{e(S,T)}{|S||T| p_{\max}} \leq l$ with probability at least $1 - \exp\bigl( - \frac{1}{2} |S||T| p_{\max} l \log l\bigr)$.
 Let $l' > 0$ be a positive real number such that $l' \log l' = \frac{2(4 + C') |T|}{|S||T| p_{\max}} \log \frac{e n}{|T|}$. Let $l = 8 \vee l'$. It is clear that $l \log l \geq  \frac{2(4 + C') |T|}{|S||T| p_{\max}} \log \frac{e n}{|T|}$. Taking a union bound over all subsets, we obtain
 \begin{align*}
   \mathbb{P}\biggl( \exists S, T \subseteq V,\, e(S,T) \geq l p_{\max} |S||T| \biggr)
   &\leq \sum_{(s,t) \in [n]\times [n]\,:\, t \geq s}
     \exp\biggl( - (4 + C') t \log \frac{en}{t} \biggr) \binom{n}{s} \binom{n}{t} \\
   &\leq \sum_{(s,t) \in [n]\times [n],\, t \geq s}  \exp\bigl( - (2 + C') t \log \frac{en}{t} \bigr) \\
   &\stackrel{(a)} \leq  \sum_{(s,t) \in [n]\times [n],\, t \geq s}  n^{-(2 + C')} \leq n^{-C'},
 \end{align*}
where $(a)$ follows because $t \log \frac{en}{t} \geq \log en$ for $1 \leq t \leq n$. The proposition follows.
\end{proof}

The following Lemma is Lemma A.2 in \cite{chin2015stochastic}. We transcribe the full statement and proof here to make the paper self-contained.

\begin{proposition}
  \label{prop:heavy_set_analysis}
  Let $\tau > 0$. Let $A$ be a graph on $n$ nodes such that the maximum degree is $\tau$. For any $S, T \subseteq [n]$, let $e(S,T)$ be defined as in Proposition~\ref{prop:heavy_set_prep}. Let $d > 0$ be a positive number such that, for any $S, T \subset [n]$, we have either
 $$ \frac{e(S, T)}{|S||T| \frac{d}{n}} \leq C_2, \qquad \textrm{or} \qquad
 \frac{e(S, T)}{|S||T| \frac{d}{n}} \log \frac{e(S, T)}{|S||T| \frac{d}{n}} \leq C_3 \frac{1}{|S|} \log \frac{en}{|T|},$$
 where $C_2 \geq e$ and $C_3 > 0$ are positive universal constants. Let $x, y \in S^{n-1}$, and let $H = \{ (u,v) \,:\, |x_u y_v| \geq \sqrt{d}/n \}$. Then
 \[
   \sum_{(u,v) \in H} x_u y_v A_{uv} \leq \bigl(16 C_2 + 40 C_3 + 8 + 16\frac{\tau}{d}\bigr) \sqrt{d}.
   \]
\end{proposition}

\begin{proof}
Fix $x, y \in S^{n-1}$. Since $A_{uv} \geq 0$, we may assume without loss of generality that $x_u, y_v \geq 0$ for all $u,v \in [n]$. For $i, j \in \{1, \ldots, \lceil \log_2 \frac{n}{\sqrt{d}} \rceil \}$, let $\gamma_0 = \sqrt{d/n}$, $\gamma_i := 2^i \gamma_0$, and $\gamma_j := 2^j \gamma_0$. We also define
  \[
    S_i := \bigl \{ u \,:\, \frac{\gamma_{i-1}}{\sqrt{n}} \leq x_u \leq \frac{\gamma_i}{\sqrt{n}} \bigr\}, \qquad
    T_j := \bigl \{ v \,:\, \frac{\gamma_{j-1}}{\sqrt{n}} \leq y_v \leq
      \frac{\gamma_j}{\sqrt{n}} \bigr\}.
    \]
Define the shorthand $s_i := |S_i|$ and $t_j := |T_j|$. Note that since $|x_u y_v| \geq \frac{\sqrt{d}}{n}$ and $x_u, y_v \leq 1$ for all $(u,v) \in H$, we have $H \subseteq \bigcup_{(i,j) \,:\, \gamma_i \gamma_j \geq \sqrt{d}} S_i \times T_j$. Hence,
    \begin{align}
      \sum_{(u,v) \in H} x_u y_v A_{uv} \leq
      \sum_{\substack{ (i,j) : \\\gamma_i \gamma_j \geq \sqrt{d}}} 
      \sum_{u \in S_i,\, v \in T_j} x_u y_v A_{uv} \leq
      \sum_{\substack{(i,j) : \\ \gamma_i \gamma_j \geq \sqrt{d}}}
           \frac{\gamma_i \gamma_j}{n} \frac{e(S_i, T_j)}{s_i t_j \frac{d}{n} } s_i t_j \frac{d}{n}.
      \label{eqn:H_decomposition} 
      \end{align}
Also note that since $x, y \in S^{n-1}$, we have $\sum_i \frac{\gamma_i^2}{n} s_i \leq 4$ and $\sum_j \frac{\gamma_j^2}{n} t_j \leq 4$. Define the shorthand $H^* := \{(i,j) \,:\, \gamma_i \gamma_j \geq \sqrt{d} \}$ and define $H_1 := \bigl\{ (i,j) \in H^* \,:\, \frac{e(S_i, T_j)}{s_i t_j \frac{d}{n}} \leq C_2 \frac{\gamma_i \gamma_j}{\sqrt{d}} \bigr\}$. Then
      \begin{align}
        \sum_{(i,j) \in H_1}  \frac{\gamma_i \gamma_j}{n} \frac{e(S_i, T_j)}{s_i t_j \frac{d}{n} } s_i t_j \frac{d}{n}
        \leq
          C_2\sqrt{d} \sum_{(i,j) \in H_1} \frac{\gamma_i^2}{n} \frac{\gamma_j^2}{n} s_i t_j \leq 16 C_2 \sqrt{d}. \label{eqn:H1_bound}
      \end{align}
Now define $H_2 := \bigl \{(i,j) \in H^* \,:\, \frac{\gamma_i}{\gamma_j} \geq \sqrt{d} \bigr\}$. Note that because the maximum degree is bounded by $\tau$, we have $\frac{e(S_i, T_j)}{s_i t_j \frac{d}{n}} \leq \frac{s_i \tau}{s_i t_j \frac{d}{n}} \leq \frac{\tau}{d} \frac{n}{t_j}$. Thus,
\begin{align}
\sum_{(i,j) \in H_2}  \frac{\gamma_i \gamma_j}{n} \frac{e(S_i, T_j)}{s_i t_j \frac{d}{n} } s_i t_j \frac{d}{n} 
  & \leq \frac{\tau}{d} \sum_{(i,j) \in H_2} \frac{\gamma_i\gamma_j}{n} s_i d \nonumber \\
  &\leq \frac{\tau}{d}  \sum_i \frac{\gamma_i^2}{n} s_i d \sum_{j \,:\, (i,j) \in H_2} \frac{\gamma_j}{\gamma_i} \nonumber \\
  &\stackrel{(a)} \leq
    2 \frac{\tau}{d} \sqrt{d} \sum_i \frac{\gamma_i^2}{n} s_i \leq 8 \frac{\tau}{d} \sqrt{d}, \label{eqn:H2_bound}
\end{align}
where in the above derivations, $(a)$ follows because, for any fixed $i$, the sum $\sum_{(i,j) \in H_2} \frac{\gamma_j}{\gamma_i}$ is a geometric series whose term-to-term ratio is 2. Let $j^* := \max \{ j \,:\, (i,j) \in H_2 \}$. Then $\frac{\gamma_{j^*}}{\gamma_i} \leq \frac{1}{\sqrt{d}}$. Therefore, we have $\sum_{(i,j) \in H_2} \frac{\gamma_j}{\gamma_i} \leq \frac{2}{\sqrt{d}}$.
Now define $H'_2 := \{(i,j) \in H^* \,:\, \frac{\gamma_j}{\gamma_i} \geq \sqrt{d}\}$. By similar reasoning as above, we have
\begin{align}
  \sum_{(i,j) \in H'_2}  \frac{\gamma_i \gamma_j}{n} \frac{e(S_i, T_j)}{s_i t_j \frac{d}{n} } s_i t_j \frac{d}{n} \leq 8 \frac{\tau}{d} \sqrt{d}. \label{eqn:H2prime_bound}
\end{align}
Define $H_3 := \bigl \{ (i,j) \in H^* \,:\,
\frac{e(S_i, T_j)}{s_it_j \frac{d}{n}} \geq C_2 \frac{\gamma_i \gamma_j}{\sqrt{d}},\,
\frac{1}{\sqrt{d}} \leq \frac{\gamma_i}{\gamma_j} \leq \sqrt{d},\,
\frac{e(S_i, T_j)}{s_i t_j \frac{d}{n}} \geq \bigl( \frac{n}{t_j} \bigr)^{1/2} \bigr \}$. Since $\frac{ e(S_i, T_j)}{s_i t_j \frac{d}{n}} \geq C_2 \frac{\gamma_i \gamma_j}{\sqrt{d}} \geq C_2$ for all $(i,j) \in H_3$, we have
$\frac{ e(S_i, T_j)}{s_i t_j \frac{d}{n}} \leq 
\frac{C_3}{\log \frac{ e(S_i, T_j)}{s_i t_j \frac{d}{n}}} \frac{n}{d s_i} \log \frac{n}{t_j} \leq
2 C_3 \frac{n}{ds_i}$. Thus, 
\begin{align}
  \sum_{(i,j) \in H_3} \frac{\gamma_i \gamma_j}{n} \frac{e(S_i, T_j)}{s_i t_j \frac{d}{n} } s_i t_j \frac{d}{n}
  \leq 2 C_3 \sqrt{d} \sum_{(i,j) \in H_3}
       \frac{\gamma_j^2}{n} t_j \frac{\gamma_i}{\gamma_j} \frac{1}{\sqrt{d}} \leq 8 C_3 \sqrt{d}. \label{eqn:H3_bound}
\end{align}
Define $H_4 := \bigl \{ (i,j) \in H^* \,:\,
\frac{e(S_i, T_j)}{s_it_j \frac{d}{n}} \geq C_2 \frac{\gamma_i \gamma_j}{\sqrt{d}},\,
\frac{1}{\sqrt{d}} \leq \frac{\gamma_i}{\gamma_j} \leq \sqrt{d},\,
\frac{e(S_i, T_j)}{s_it_j \frac{d}{n}} \leq \bigl( \frac{n}{t_j} \bigr)^{1/2},\, \gamma_j^2 \leq \frac{n}{t_j \gamma_j^2} \bigr\}$. Note that if $(i,j) \in H_3$, then $\frac{e(S_i, T_j)}{s_it_j \frac{d}{n}} \leq \bigl( \frac{n}{t_j \gamma_j^2} \bigr)^{1/2} \gamma_j \leq \frac{n}{t_j \gamma_j^2}$. Therefore,

\begin{align}
  \sum_{(i,j) \in H_4} \frac{\gamma_i \gamma_j}{n} \frac{e(S_i, T_j)}{s_i t_j \frac{d}{n} } s_i t_j \frac{d}{n}
  \leq \sqrt{d} \sum_{(i,j) \in H_4} \frac{\gamma_i^2}{n} s_i \frac{\sqrt{d}}{\gamma_i \gamma_j} \leq \sqrt{d} \sum_i \frac{\gamma_i^2}{n} s_i \sum_{j \,:\, (i,j)\in H_4} \frac{\sqrt{d}}{\gamma_i \gamma_j} \leq 8 \sqrt{d}. \label{eqn:H4_bound}
\end{align}
The last inequality follows because, for every $i$, the sum $\sum_{j \,:\, (i,j) \in H_4} \frac{\sqrt{d}}{\gamma_i \gamma_j}$ is a geometric series where the term-to-term ratio is 2 and  the largest term is at most 1. Define $H_5 := \bigl \{ (i,j) \in H^* \,:\,
\frac{e(S_i, T_j)}{s_it_j \frac{d}{n}} \geq C_2 \frac{\gamma_i \gamma_j}{\sqrt{d}},\,
\frac{1}{\sqrt{d}} \leq \frac{\gamma_i}{\gamma_j} \leq \sqrt{d},\, \gamma_j^2 \geq \frac{n}{t_j \gamma_j^2} \bigr\}$. For any $(i,j) \in H_5$, we have $4 \log \gamma_j \geq \log \frac{n}{t_j}$. Therefore, $\frac{e(S_i, T_j)}{s_i t_j \frac{d}{n}} \leq 4 C_3 \frac{n}{d s_i} \frac{\log \gamma_j}{\log \bigl( C_2 \frac{\gamma_i \gamma_j}{\sqrt{d}} \bigr) }$, and
\begin{align}
  \sum_{(i,j) \in H_5}  \frac{\gamma_i \gamma_j}{n} \frac{e(S_i, T_j)}{s_i t_j \frac{d}{n} } s_i t_j \frac{d}{n}
  &\leq 4 C_3 \sqrt{d} \sum_{(i,j) \in H_5} \frac{\gamma_j^2}{n} t_j \frac{\gamma_i \gamma_j}{\sqrt{d}} \frac{1}{\gamma_j^2}  \frac{\log \gamma_j}{\log \bigl( C_2 \frac{\gamma_i \gamma_j}{\sqrt{d}} \bigr)} \nonumber \\
  &\leq 4 C_3 \sqrt{d} \sum_j \frac{\gamma_j^2}{n} t_j \sum_{i \,:\, (i,j) \in H_5} \frac{\gamma_i \gamma_j}{\sqrt{d}} \frac{1}{\gamma_j^2}  \frac{\log \gamma_j}{\log \bigl( C_2 \frac{\gamma_i \gamma_j}{\sqrt{d}} \bigr)}
    \nonumber \\
  &\stackrel{(a)} \leq 32 C_3 \sqrt{d} . \label{eqn:H5_bound}
\end{align}
In the above derivations, $(a)$ follows because for all $(i,j) \in H_5$, we have $\gamma_j^2 \geq \frac{\gamma_i \gamma_j}{\sqrt{d}} \geq 1$. Since $C_2 \geq e$ and $x \mapsto \frac{x}{ \log (C_2 x)}$ is non-decreasing for all $x \geq 1$, we have $\frac{\gamma_i \gamma_j}{\sqrt{d}} \frac{1}{\gamma_j^2}
\frac{\log \gamma_j}{\log (C_2 \frac{\gamma_i \gamma_j}{\sqrt{d}} ) } \leq
\frac{\gamma_i \gamma_j}{\sqrt{d}} \frac{1}{\gamma_j^2}
\frac{\log (C_2 \gamma^2_j)}{\log (C_2 \frac{\gamma_i \gamma_j}{\sqrt{d}} ) } \leq 1$. For any $j$, the quantity $\sum_{i \,:\, (i,j) \in H_5} \frac{\gamma_i \gamma_j}{\sqrt{d}} \frac{1}{\gamma_j^2}  \frac{\log \gamma_j}{\log \bigl( C_2 \frac{\gamma_i \gamma_j}{\sqrt{d}} \bigr)}$ is a geometric series with largest term bounded by 1.

Since $H^* = H_1 \cup H_2 \cup H_2' \cup H_3 \cup H_4 \cup H_5$, we may combine inequalities~\eqref{eqn:H_decomposition}, \ref{eqn:H1_bound}, \ref{eqn:H2_bound}, \ref{eqn:H2prime_bound}, \ref{eqn:H3_bound}, \ref{eqn:H4_bound}, and \ref{eqn:H5_bound} to obtain the bound
\begin{align*}
  \sum_{(u,v) \in H} x_u y_v A_{uv} \leq (16 C_2 + 16 \frac{\tau}{d} + 8 + 40 C_3) \sqrt{d}.
  \end{align*}
\end{proof}

\begin{proposition} [Chernoff bound]
  \label{prop:chernoff_for_bernoulli}
  For $i=1,\ldots,n$, let $p_i \in [0,1]$. Let $X_i \sim \textrm{Ber}(p_i)$, and let $p = \frac{1}{n} \sum_{i=1}^n p_i$. Then for any $t > 0$, we have
  \[
    \mathbb{P}\biggl( \sum_{i=1}^n X_i \geq pn + t \biggr) \leq
    \exp \biggl( t - (t + pn) \log\left( 1 + \frac{t}{pn} \right) \biggr).
  \]
\end{proposition}
A bound that we frequently use is $\mathbb{P}\biggl( \sum_{i=1}^n X_i \geq pn + t) \leq \exp \biggl( - (t + pn) \bigl(\log\bigl( 1 + \frac{t}{pn} \bigr) -1 \bigr) \biggr)$.

\begin{proposition} [Bernstein's inequality]
  \label{prop:bernstein}
  Let $X_1,\ldots,X_n$ be real-valued random variables bounded in absolute value by $M > 0$. Suppose $\E X_i = 0$ for all $i=1,\ldots,n$. Then for any $t > 0$, we have
  \[
    \mathbb{P} \biggl( \sum_{i=1}^n X_i  \geq t \biggr)
    \leq
    \exp\biggl( - \frac{(1/2) t^2}{(1/3) M t + \sum_{i=1}^n \E X_i^2 } \biggr).
  \]
\end{proposition}
If $t \geq \frac{3}{M} \sum_{i=1}^n \E X_i^2$, then $\mathbb{P} \biggl(  \sum_{i=1}^n X_i \geq t \biggr) \leq \exp \bigl( - (3/4) \frac{t}{M} \bigr)$.





\subsection{Choosing the label $l^*$}
\label{appendix: starry night}

First,  we show that for sufficiently well-separated labels, $\hat I_l$ is close to $\frac{( P_l -  Q_l)^2}{P_l \vee Q_l}$. If the probabilities are not well-separated, we claim that $\hat I_l$ is negligibly small.

\begin{proposition}
\label{prop:initial_guarantee}
Let $\sigma_0 \in \mathcal{C}(\beta, K)$, let $L \in \mathbb{Z}^+$, $(\{P_l\}, \{Q_l\}) \in \mathcal{P}_L^2$, and let $A \in \{0,\ldots,L\}^{n \times n}$ have the distribution $LSBM(\sigma_0, (\{P_l\}, \{Q_l\})$. For $l=0,\ldots,L$, let $\tilde{A}^l \in \{0,1\}^{n \times n}$ be defined as $\tilde{A}^l_{ij} := \mathbf{1}\{A_{ij} = l\}$. Let $\sigma^l$ be the output of \textsc{spectral clustering} (Algorithm~\ref{alg:spectral} with parameters $\mu = 4 \beta$ and $\tau = 20 \bar{d}$) on $\tilde{A}^l$, and let $\hat{P}_l$ and $\hat{Q}_l$ be estimates of $P_l$ and $Q_l$ constructed from $\sigma^l$. Suppose $\frac{n}{\log n} \geq 2^{15} \vee 30 \beta K$. 

Let $C_{spec1} > 0$ be the universal constant defined in Proposition~\ref{prop:spectral_analysis}. Then the following claims are true with probability at least $1 - 12(L+1)^2 n^{-5}$:

\begin{enumerate} 
\item For all labels $l$ satisfying $P_l \vee Q_l \geq \frac{1}{n}$ and $\frac{ n \Delta_l^2}{P_l \vee Q_l} \geq C_{spec1} \beta^2 K^6$, we have
\begin{align}
\frac{1}{4} \frac{ | P_l - Q_l |}{\sqrt{P_l \vee Q_l}}  \leq \frac{ | \hat{P}_l - \hat{Q}_l| }{\sqrt{ \hat{P}_l \vee \hat{Q}_l }} \leq  
2 \frac{ | P_l - Q_l | }{\sqrt{ P_l \vee Q_l}}. \label{eqn:prop_init_claim1}
\end{align}

\item For all labels satisfying $P_l \vee Q_l \geq \frac{1}{n}$ and $\frac{ n \Delta_l^2}{P_l \vee Q_l} \leq C_{spec1} \beta^2 K^6$, we have
\begin{align}
\frac{ | \hat{P}_l - \hat{Q}_l|}{\sqrt{ \hat{P}_l \vee \hat{Q}_l}} \leq 16 \sqrt{C_{spec1}} \beta^2 K^4 \sqrt{ \frac{1}{n} }.
\label{eqn:prop_init_claim2}
\end{align}

\item Suppose $P_l \vee Q_l \geq \frac{1}{n}$ for all $l \in \{0,\ldots,L\}$ and
  $\frac{n I(\{P_l\}, \{Q_l\})}{L \beta^2 K^6} \geq 2 C_{spec1}$. Let $c_{init} := 2^{-7}$. Then, with $l^* := \argmin_{l \in \{0,\ldots,L\}} \frac{|\hat{P}_l - \hat{Q}_l|}{\sqrt{ \hat{P}_l \vee \hat{Q}_l}}$, we have
\begin{align}
   \frac{\Delta_{l^*}^2}{P_{l^*} \vee Q_{l^*}} \geq c_{init} \frac{I(\{P_l\}, \{Q_l\})}{L}. \label{eqn:l_star_guarantee}
\end{align}

\end{enumerate}
\end{proposition}

\begin{proof}

  Let $C_{spec1} = 2^{34}$. Let us fix $l \in \{0, \ldots, L\}$. Let us first suppose that $P_l \vee Q_l \geq \frac{1}{n}$ and
\begin{align}
  \frac{n \Delta_l^2}{P_l \vee Q_l} \geq C_{spec1} \beta^2 K^6. \label{eqn:strong_signal_test}
\end{align}
Let $\sigma_l$ be the result of performing \textsc{spectral clustering} (Algorithm~\ref{alg:spectral}) on $A_l$, and let $\hat{P}_l$ and $\hat{Q}_l$ be the subsequent estimators of $P_l$ and $Q_l$. Let $E_1(l)$ be the event that $l(\sigma_0, \sigma_l) \leq \gamma$, where $\gamma := \frac{1}{2^{14} \beta K \log (\beta K)}$.
Under assumption~\eqref{eqn:strong_signal_test}, we can verify that the hypothesis of Proposition~\ref{prop:spectral_analysis} holds and thus apply Proposition~\ref{prop:spectral_analysis} to show that $\mathbb{P}(E_1(l)) \geq 1 - n^{-5}$.
Define $E_2(l)$ to be the event that
\begin{align}
\max\{ \hat{P}_l - P_l, \hat{Q}_l - Q_l \} \leq \frac{1}{4} \Delta_l. \label{eqn:est_err_bound_init1}
\end{align}
It is straightforward to show that, under $E_1(l)$ and the hypothesis that $6 \frac{\log n}{n} \leq 2^{-12}$, we can bound $\eta \leq \frac{1}{4}$, where $\eta$ is defined in the statement of Proposition~\ref{prop:estimation_consistency}. Therefore, by Proposition~\ref{prop:estimation_consistency} and a union bound, we have $\mathbb{P}\bigl(E_1(l) \cap E_2(l)\bigr) \geq 1 - 4(L+1)n^{-5}$.
Under $E_1(l) \cap E_2(l)$, we have from inequality~\eqref{eqn:est_err_bound_init1} that
\begin{align*}
| \hat{P}_l - \hat{Q}_l | &\leq |\hat{P}_l - P_l| + |P_l - Q_l| + |\hat{Q}_l - Q_l| \leq 2 \eta \Delta_l + \Delta_l \leq \frac{3}{2} \Delta, \\
| \hat{P}_l - \hat{Q}_l | &\geq  |P_l - Q_l| - |\hat{Q}_l - Q_l| - |\hat{P}_l - P_l| \geq  \Delta_l - 2\eta \Delta_l \geq \frac{1}{2} \Delta_l.
\end{align*}
Furthermore, we can again apply inequality~\eqref{eqn:est_err_bound_init1} to obtain 
\begin{align*}
\hat{P}_l \vee \hat{Q}_l &\leq (P_l \vee Q_l) + \eta \Delta_l \leq (P_l \vee Q_l) + \eta (P_l \vee Q_l) \leq \frac{5}{4} (P_l \vee Q_l), \\
  \hat{P}_l \vee \hat{Q}_l &\ge (P_l \vee Q_l) - \eta \Delta_l \ge \frac{3}{4} (P_l \vee Q_l).
\end{align*}
We can thus prove inequality~\eqref{eqn:prop_init_claim1}:
\[
\frac{1}{\sqrt{5}} \frac{\Delta_l}{\sqrt{P_l \vee Q_l}} \leq \frac{ | \hat{P}_l - \hat{Q}_l | }{\sqrt{ \hat{P}_l \vee \hat{Q}_l}} \leq \frac{3}{\sqrt{3}} \frac{\Delta_l}{\sqrt{P_l \vee Q_l}}.
\]

Now suppose $P_l \vee Q_l \geq \frac{1}{n}$ and $\frac{n \Delta_l^2}{P_l \vee Q_l} < C_{spec1} \beta^2 K^6$. Let $E_1(l)$ denote the entire probability space, and let $E_2(l)$ be the event that
\begin{align}
  \max\{ \hat{P}_l - P_l, \hat{Q}_l - Q_l \} \leq \Delta_l + 4 \beta^2 K^2 \biggl( \frac{P_l \vee Q_l}{n} \biggr)^{1/2},
\quad \textrm{and }
  \hat{P}_l \vee \hat{Q}_l \geq \frac{1}{\beta K} (P_l \vee Q_l).   \label{eqn:est_err_bound_init2}
 \end{align}
Since $\gamma \leq 1$ we have $\eta \leq 4 \beta K^2$, where $\eta$ is defined in the statement of Proposition~\ref{prop:estimation_consistency}. Hence, by Proposition~\ref{prop:estimation_consistency} and Proposition~\ref{prop:PlQl_lower_bound}, we have
 \[
   \mathbb{P}(E_2(l)) \geq 1 - 4(L+1)n^{-6} - (L+1) \exp\bigl( - \frac{n}{5 \beta K} \bigr) \stackrel{(a)} \geq 1 - 8(L+1)n^{-6},
 \]
where $(a)$ follows under the hypothesis that $\frac{n}{\log n} \geq 30 \beta K$. Under $E_2(l) = E_1(l) \cap E_2(l)$, we have
\begin{align*}
  | \hat{P}_l - \hat{Q}_l |
  &\leq \Delta_l + | \hat{P}_l - P_l| + |\hat{Q}_l - Q_l| \\
  &\leq 3 \Delta_l + 8 \beta K^2 \biggl( \frac{P_l \vee Q_l}{n} \biggr)^{1/2} \leq 16 \sqrt{C_{spec1}} \beta K^3  \biggl( \frac{P_l \vee Q_l}{n} \biggr)^{1/2}.
\end{align*}
Therefore, 
\[
  \frac{| \hat{P}_l - \hat{Q}_l| }{\sqrt{\hat{P}_l \vee \hat{Q}_l}} \leq
  16 \sqrt{C_{spec1}} \beta^2 K^4 \biggl( \frac{1}{n} \biggr)^{1/2}.
  \]
A union bound gives us $ \mathbb{P}\biggl( \bigcap_{l=0}^L E_1(l) \cap E_2(l) \biggr) \geq 1 - 12(L+1)^2 n^{-5}$, which proves the first and the second claims of the proposition. For the third claim~\eqref{eqn:l_star_guarantee}, suppose $ \bigcap_{l=0}^L E_1(l) \cap E_2(l)$ holds and let $l' := \argmax_{l \in \{0,\ldots,L\}} \frac{\Delta_{l'}^2}{P_{l'} \vee Q_{l'}}$. Observe that
  \begin{align}
    \frac{n\Delta_{l'}^2}{P_{l'} \vee Q_{l'}} \geq
    \frac{n}{L} \sum_{l=0}^L \frac{\Delta_l^2}{P_l \vee Q_l} \geq
    \frac{n}{L} \sum_{l=0}^L (\sqrt{P_l} - \sqrt{Q_l})^2
    \stackrel{(a)} \geq
    \frac{n}{2L} I(\{P_l\}, \{Q_l\})
    \stackrel{(b)} \geq C_{spec1} \beta^2 K^6,
  \end{align}
  where $(a)$ holds by Lemma~\ref{lem:renyi_hellinger} and $(b)$ follows from the hypothesis of the proposition. Therefore, from the definition of $l^*$ and from inequality~\eqref{eqn:prop_init_claim1}, we have
  \begin{align}
    \frac{| \hat{P}_{l^*} - \hat{Q}_{l^*}|}{\sqrt{\hat{P}_{l^*} \vee \hat{Q}_{l^*}}}
    \geq
    \frac{| \hat{P}_{l'} - \hat{Q}_{l'}|}{\sqrt{\hat{P}_{l'} \vee \hat{Q}_{l'}}}
    \geq 
    \frac{1}{4}
    \frac{\Delta_{l'}}{\sqrt{P_{l'} \vee Q_{l'}}} \geq \sqrt{C_{spec1}} \beta K^3 >
    16 \sqrt{C_{spec1}} \beta^2 K^4 n^{-1/2}.
  \end{align}
We may deduce from inequality~\eqref{eqn:prop_init_claim2} that $\frac{n \Delta_{l^*}^2}{P_{l^*} \vee Q_{l^*}} \geq C_{spec1} \beta^2 K^6$. Hence, by inequality~\eqref{eqn:prop_init_claim1}, we have
  \begin{align}
    \frac{\Delta_{l^*}}{\sqrt{P_{l^*} \vee Q_{l^*}}}
    \geq \frac{1}{2} \frac{| \hat{P}_{l^*} - \hat{Q}_{l^*}|}{\sqrt{\hat{P}_{l^*} \vee \hat{Q}_{l^*}}} \geq \frac{1}{8} \frac{\Delta_{l'}}{\sqrt{P_{l'} \vee Q_{l'}}} \geq \frac{1}{8} \bigl( \frac{I(\{P_l\}, \{Q_l\})}{2L} \bigr)^{1/2}.
  \end{align} 
\end{proof}


\subsection{Analysis of error probability for a single node}
\label{appendix: single node}

\begin{proposition}
  \label{prop:single_node_error_bound}
  Let $(\{P_l\}, \{Q_l\}) \in \mathcal{G}_{L, \rho}$ be such that $P_l \vee Q_l \geq \frac{1}{n-1}$ for all $l \in \{0, \ldots, L\}$. Let $\sigma_0 \in \mathcal{C}(\beta, K)$, and let $A$ be a random labeled matrix taking values in $\{0,\ldots,L\}^{n \times n}$, with the distribution $LSBM(\sigma_0, (\{P_l\}, \{Q_l\}))$. Let $u \in [n]$, and let $\tilde \sigma_u$ be the output clustering on $\{1,\ldots,n\} \backslash \{u\}$ of \textsc{initialization} (Algorithm \ref{alg:initialization1}).
  For $l \in \{0, \ldots, L\}$, let $\hat{P}_l, \hat{Q}_l$ be computed by equation~\eqref{eqn:hatPl_hatQl_definition} on $A_{l^*} \backslash \{u\}$ with respect to $\tilde{\sigma}_l$. 
 Suppose $\pi_u \in S_K$ satisfies
$$l(\sigma_0, \tilde \sigma_u) = \frac{1}{n-1} d_H(\sigma_0, \pi_u \circ \tilde \sigma_u),$$
where both $l$ and $d$ are taken with respect to the set $\{1, 2, \dots, n\} \setminus \{u\}$.
Let $C_{mis} := 2^{15} \log(2^{15})$, let $\gamma_0 := \bigl( C_{mis} \rho^2 K \beta \log (e \rho^2 K \beta) \bigr)^{-1}$, let $\gamma \in [0, \gamma_0]$, and suppose that $l(\sigma_0, \tilde{\sigma}_u) \leq \gamma$. Suppose also that $\frac{n}{\log n} \geq 2^{16} \rho^2$ and $\frac{nI(\{P_l\}, \{Q_l\})}{L+1} \geq 1$. Let $\eta' := 2 \beta K \gamma \log \frac{e K}{\gamma} + 12 \frac{\log n}{n}$ and $\eta := 10 (\sqrt{\eta'} + \eta')$.

Then there exists a universal constant $C > 0$ such that, with probability at least
\begin{equation*}
1 - 5(L+1)n^{-6} - (K-1) \exp \biggl( -  (1 - C \beta K \rho \eta) \frac{n}{\beta K} I(\{P_l\}, \{Q_l\}) \biggr),
\end{equation*}
we have
\[
\pi_u^{-1}(\sigma_0(u)) = \argmax_{k \in [K]} \sum_{v\,: \tilde \sigma_u(v)=k} \sum_{l=0}^L \log \frac{\hat{P}_l}{\hat{Q}_l} \mathbf{1}(A_{uv} = l).
\]
\end{proposition}

\begin{proof}
Define the shorthand $C_{tmp1} := 4 \cdot 40^2 \rho^2 \beta K$ and $C_{tmp2} := e K$. Since $\gamma_0 \leq \bigl( C_{mis} \rho^2 \beta K \log(e \rho^2 \beta K) \bigr)^{-1}$ with $C_{mis} = 2^{15} \log(2^{15})$, we have $\gamma \leq \gamma_0 \leq \bigl(4 C_{tmp1} \log(C_{tmp1} + C_{tmp2} + 4)\bigr)^{-1}$, which implies that
\[
  4 \cdot 40^2 \rho^2 \beta K \gamma \log \frac{e K}{\gamma} =
  C_{tmp1} \gamma \log \frac{C_{tmp2}}{\gamma} \leq 1,
\]
so $2 \beta K \gamma \log \frac{e K}{\gamma} \leq \frac{1}{2} \frac{1}{40^2} \frac{1}{\rho^2}$. Additionally, since $\frac{n}{\log n} \geq 2^{16} \rho^2$ by assumption, we have $\frac{12 \log n}{n} \leq \frac{1}{2} \frac{1}{40^2} \frac{1}{\rho^2}$, as well. Thus, 
  \begin{align}
    \eta' \leq \frac{1}{40^2}{\rho^2} < 1 ,\, \textrm{ and }
    \eta \leq 20 \sqrt{\eta'} \leq \frac{1}{2 \rho}.
    \label{eqn:eta_etaprime_bound}
    \end{align}
Let us assume without loss of generality that $\sigma_0(u) = 1$, and $\pi_u$ is the identity.
Observe that since $2 \beta \geq \frac{n-1}{K} \frac{1}{\frac{n}{\beta K} - 1}$, the minimum cluster size of $\sigma_0$ with respect to $\{1, \ldots, n\} \backslash \{u\}$ is at least $\frac{n}{2 \beta K}$. Define $E_1$ to be the event that for any $l \in \{0, \ldots, L\}$,
  \begin{align}
    \max\{ |\hat{P}_l - P_l|,\, |\hat{Q}_l - Q_l| \}
    \leq \eta \biggl( \Delta_l \vee \biggl( \frac{P_l \vee Q_l}{n} \biggr)^{1/2} \biggr) . \label{eqn:single_node_PlQl_estimation_error}
    \end{align}
Since $l(\tilde{\sigma}_u, \sigma_0) \leq \gamma \leq \frac{1}{8 \beta K}$, we have by Proposition~\ref{prop:estimation_consistency} that
\begin{equation*}
P(E_1) \geq 1 - 4 (L+1)(n-1)^{-6} \geq 1 - 5(L+1)n^{-6}.
\end{equation*}
For $k \in [K]$, define $E_2(k)$ to be the event
\begin{align}
  \sum_{v \in [n]\backslash \{u\} \,:\, \tilde \sigma_u(v)=1} \sum_{l=0}^L \log \frac{\hat{P}_l}{\hat{Q}_l} \mathbf{1} \{ A_{uv} = l \} > 
  \sum_{v \in [n] \backslash \{u\} \,:\, \tilde \sigma_u(v)=k} \sum_{l=0}^L \log \frac{\hat{P}_l}{\hat{Q}_l} \mathbf{1}\{ A_{uv} = l\} .
\end{align}
For any $v \in [n] \backslash \{u\}$, define $\bar{A}_{uv} := \sum_{l=0}^L \log \frac{\hat{P}_l}{\hat{Q}_l} \mathbf{1}(A_{uv} = l)$. Note that $\{\bar{A}_{uv}\}_{v \in [n] \backslash \{u\}}$ is a set of independent random variables. If $v \in \sigma_0^{-1}(1)$, then for $l \in \{0,\ldots,L\}$, we have $\bar{A}_{uv} = \log \frac{\hat{P}_l}{\hat{Q}_l}$ with probability $ P_l$. On the other hand, if $v \notin \sigma_0^{-1}(1)$, then 
for $l \in \{0,\ldots,L\}$, $\bar{A}_{uv} = \log \frac{\hat{P}_l}{\hat{Q}_l} $ with probability $Q_l$. 
Now define the shorthand
\begin{align}
  V_1 &:= \{ v \in [n] \backslash \{u\} \,:\, \sigma_0(v) = 1,\, \tilde{\sigma}_u(v) = 1 \}, \quad
        V_1' :=  \{ v \in [n] \backslash \{u\} \,:\, \sigma_0(v) \neq 1,\, \tilde{\sigma}_u(v) = 1 \}, \\
  V_k &:= \{ v \in [n] \backslash \{u\} \,:\, \sigma_0(v) = 1,\, \tilde{\sigma}_u(v) = k \}, \quad
  V_k' := \{ v \in [n] \backslash \{u\} \,:\, \sigma_0(v) \neq 1,\, \tilde{\sigma}_u(v) = k \}.
\end{align}
Then for any $t \geq 0$, we have
\begin{align*}
  \mathbb{P}&(E_1^c \cup E_2(k)^c) - \mathbb{P}(E_1^c) \\
  &\leq 
    \mathbb{P} \biggl( \sum_{v \in V_k} \bar{A}_{uv} + \sum_{v \in V'_k} \bar{A}_{uv} - \sum_{v \in V_1} \bar{A}_{uv} - \sum_{v \in V'_1} \bar{A}_{uv} \geq 0 \biggr) \\
  &\leq 
    \mathbb{P} \biggl\{ \exp\biggl( t \biggl( \sum_{v \in V_k} \bar{A}_{uv} + \sum_{v \in V'_k} \bar{A}_{uv} - \sum_{v \in V_1} \bar{A}_{uv} - \sum_{v \in V'_1} \bar{A}_{uv}\biggr) \biggr) \geq 1 \biggr\} \\
  &\leq 
    \mathbb{E} \biggl[ \exp\biggl( t \biggl( \sum_{v \in V_k} \bar{A}_{uv} + \sum_{v \in V'_k} \bar{A}_{uv} - \sum_{v \in V_1} \bar{A}_{uv} - \sum_{v \in V'_1} \bar{A}_{uv} \biggr) \biggr)
    \biggr] \\
  &\leq 
    \biggl( \sum_{l=0}^L e^{t \log \frac{\hat{P}_l}{\hat{Q}_l} } P_l \biggr)^{|V_k|}  
    \biggl( \sum_{l=0}^L e^{t \log \frac{\hat{P}_l}{\hat{Q}_l}} Q_l \biggr)^{|V'_k|} \\
  & \qquad \qquad
      \biggl( \sum_{l=0}^L e^{- t \log \frac{\hat{P}_l}{\hat{Q_l}} } P_l \biggr)^{|V_1|}
    \biggl( \sum_{l=0}^L e^{-t \log \frac{\hat{P}_l}{\hat{Q}_l}} Q_l \biggr)^{|V'_1|}. 
\end{align*}

Setting $t = 1/2$, we have

\begin{align}
  \mathbb{P}&(E_1^c \cup E_2(k)^c) - \mathbb{P}(E_1^c)\\
            &\leq
    \left( \sum_{l=0}^L \sqrt{\frac{\hat{P}_l}{\hat{Q}_l} } P_l \right)^{|V_k|}
 \left( \sum_{l=0}^L \sqrt{\frac{\hat{P}_l}{\hat{Q}_l} } Q_l \right)^{|V_k'|}
 \left( \sum_{l=0}^L \sqrt{\frac{\hat{Q}_l}{\hat{P}_l} } P_l \right)^{|V_1|}
              \left( \sum_{l=0}^L \sqrt{\frac{\hat{Q_l}}{\hat{P}_l} } Q_l \right)^{|V_1'|}\nonumber \\
  &= \left( \frac{\sum_{l=0}^L \sqrt{\frac{\hat{P}_l}{\hat{Q}_l} } P_l}
                {\sum_{l=0}^L \sqrt{\frac{\hat{P}_l}{\hat{Q}_l} } Q_l}  \right)^{|V_k|}
 \left( \frac{ \sum_{l=0}^L \sqrt{\frac{\hat{Q}_l}{\hat{P}_l} } Q_l} 
    { \sum_{l=0}^L \sqrt{\frac{\hat{Q}_l}{\hat{P}_l} } P_l} \right)^{|V'_1|}  \nonumber \\
 & \qquad \qquad   \left( \sum_{l=0}^L \sqrt{ \frac{\hat{P}_l}{\hat{Q}_l}} Q_l \right)^{|V_k'| - |V_k|} 
    \left( \sum_{l=0}^L \sqrt{\frac{\hat{Q_l}}{\hat{P}_l} } P_l \right)^{|V_1| - |V_1'|} . \label{eqn:overall_prob_bound}
  \end{align}
  Since $\frac{1}{n-1} d_H(\sigma_0, \tilde{\sigma}_u) \leq \gamma$, we have
  \begin{align}
    \min(|V_1|,\, |V_k'|) &\geq \frac{n}{\beta K} - 1 - \gamma (n-1) \geq \frac{n}{2 \beta K} - \gamma n, \label{eqn:V1Vkprime_lower_bound}\\
    \max(|V_1'|, |V_k|) &\leq \gamma(n-1) \leq \gamma n.
                          \label{eqn:V1primeVk_upper_bound}
  \end{align}
Observe that the bounds~\eqref{eqn:eta_etaprime_bound} and~\eqref{eqn:single_node_PlQl_estimation_error} satisfy the conditions of Lemmas~\ref{lem:bound_ratio_P_Pl} and~\ref{lem:sqrt_ratio_pl_ql_minus_1}. Define $L_1 := \{ l \in \{0, \ldots, L\} \,:\, \frac{n \Delta^2_l}{P_l \vee Q_l} \geq 1 \}$. Then
  \begin{align*}
& \left| 1 -  \frac{\sum_{l=0}^L \sqrt{\frac{\hat{P}_l}{\hat{Q}_l} } P_l}
                {\sum_{l=0}^L \sqrt{\frac{\hat{P}_l}{\hat{Q}_l} } Q_l}  \right|
 = \left| \frac{ \sum_{l=0}^L \sqrt{ \frac{\hat{P}_l}{\hat{Q}_l}} (P_l - Q_l) }
     { \sum_{l=0}^L \sqrt{ \frac{\hat{P}_l}{\hat{Q}_l}} Q_l } \right| \stackrel{(a)}\leq \frac{4}{\sum_{l=0}^L \sqrt{P_l Q_l}} 
     \left| \sum_{l=0}^L \sqrt{ \frac{\hat{P}_l}{\hat{Q}_l} }(P_l - Q_l) \right| \\
& \qquad \stackrel{(b)}\leq 8 \left|  \sum_{l=0}^L \left( \sqrt{ \frac{\hat{P}_l}{\hat{Q}_l} } - 1 \right) (P_l - Q_l)  \right| \\
& \qquad \leq 8 \left| 
     \sum_{l \in L_1} \left( \sqrt{\frac{\hat{P}_l}{\hat{Q}_l}} - 1 \right)(P_l - Q_l) 
     \right| + 8 \sum_{l \notin L_1} \left| \sqrt{ \frac{\hat{P}_l}{\hat{Q}_l}} - 1 \right| |P_l - Q_l| \\
& \qquad \stackrel{(c)}\leq 8(1+ 6 \eta \rho) \sum_{l \in L_1} \frac{\Delta^2_l}{Q_l} + 
       32 \rho \sum_{i \notin L_1} \frac{\Delta_l}{\sqrt{n (P_l \vee Q_l)}} \\
& \qquad \leq 8 (1+ 6 \eta \rho) \rho \sum_{l \in L_1} \frac{\Delta^2_l}{P_l \vee Q_l}  + 
       32 \rho\sum_{l \notin L_1} \frac{\Delta_l}{\sqrt{n (P_l \vee Q_l)}} \\
& \qquad \stackrel{(d)}\leq 2^6 \rho I(\{P_l\}, \{Q_l\}) + 32 \rho \frac{L+1}{n} \leq 2^7 \rho I(\{P_l\}, \{Q_l\}),
\end{align*}
where $(a)$ follows from Claim 3 of Lemma~\ref{lem:bound_ratio_P_Pl}, $(b)$ follows because $\sum_{l=0}^L \sqrt{P_l Q_l} \geq e^{- I(\{P_l\}, \{Q_l\})} \geq \frac{1}{2}$ by assumption, $(c)$ follows by Lemma~\ref{lem:sqrt_ratio_pl_ql_minus_1}, and $(d)$ follows from Lemma~\ref{lem:L1_info_bound} and the fact that $2 \eta \rho \leq 1$. An identical analysis shows that $\left| 1 -  \frac{\sum_{l=0}^L \sqrt{\frac{\hat{Q}_l}{\hat{P}_l} } Q_l}
  {\sum_{l=0}^L \sqrt{\frac{\hat{Q}_l}{\hat{P}_l} } P_l}  \right| \leq 2^7 \rho I(\{P_l\}, \{Q_l\})$.  Using the fact that $|x| \leq \exp(|1-x|)$ for all $x \in \mathbb{R}$, we have
\begin{align}
  \left( \frac{\sum_{l=0}^L \sqrt{\frac{\hat{P}_l}{\hat{Q}_l} } P_l}
                {\sum_{l=0}^L \sqrt{\frac{\hat{P}_l}{\hat{Q}_l} } Q_l}  \right)^{|V_k|}
 \left( \frac{ \sum_{l=0}^L \sqrt{\frac{\hat{Q}_l}{\hat{P}_l} } Q_l} 
  { \sum_{l=0}^L \sqrt{\frac{\hat{Q}_l}{\hat{P}_l} } P_l} \right)^{|V'_1|}
  &\leq \exp\bigl( (|V_k| + |V_1'|) 2^7 \rho I(\{P_l\}, \{Q_l\}) \bigr)
  \nonumber \\
  &\leq \exp\bigl( 2^7 \gamma \rho n I(\{P_l\}, \{Q_l\}) \bigr) \nonumber \\
  &\leq \exp\bigl(  2^7 \rho \beta K \gamma \frac{n}{\beta K}  I(\{P_l\}, \{Q_l\}) 
                  \bigr).
    \label{eqn:extra_final}
  \end{align}
We now use the shorthand $\hat{I} := - \log \biggl\{ \left( \sum_{l=0}^L \sqrt{\frac{\hat{P}_l}{\hat{Q}_l}} Q_l \right) \left( \sum_{l=0}^L \sqrt{\frac{\hat{Q}_l}{\hat{P}_l}} P_l \right) \biggr\} $, and define $m_k := |V_k'| - |V_k|$ and $m_1 := |V_1| - |V_1'|$. Then
  \begin{align}
\biggl( \sum_{l=0}^L &\sqrt{\frac{\hat{P}_l}{\hat{Q}_l} } Q_l \biggr)^{|V'_k| - |V_k|} 
       \biggl( \sum_{l=0}^L \sqrt{\frac{\hat{Q_l}}{\hat{P}_l} } P_l \biggr)^{|V_1| - |V_1'|}  \nonumber \\
       &= \exp\bigl( - \frac{m_1+m_k}{2} \hat{I} \bigr)  \left( \sum_{l=0}^L \sqrt{\frac{\hat{P}_l}{\hat{Q}_l}} Q_l \right)^{\frac{m_k - m_1}{2}} 
 \left( \sum_{l=0}^L \sqrt{\frac{\hat{Q}_l}{\hat{P}_l}} P_l \right)^{\frac{m_1 - m_k}{2}}. \label{eqn:Ihat_overall}
  \end{align}
We now have
\begin{align}
\hat{I} - I(\{P_l\}, \{Q_l\})  &= - \log \frac{ 
     \left( \sum_{l=0}^L \sqrt{\frac{\hat{P}_l}{\hat{Q}_l}} Q_l \right)
     \left( \sum_{l=0}^L \sqrt{\frac{\hat{Q}_l}{\hat{P}_l}} P_l \right)}{ 
          \left( \sum_{l=0}^L \sqrt{P_l Q_l} \right)^2 }.
\end{align}  
Let us first consider the numerator:
 \begin{align*}
 \biggl( \sum_{l=0}^L & \sqrt{ \frac{\hat{P}_l}{\hat{Q}_l}} Q_l \biggr)
\biggl( \sum_{l=0}^L \sqrt{ \frac{\hat{Q}_l}{\hat{P}_l}} P_l \biggr) = \left( \sum_{l=0}^L \sqrt{ P_l Q_l} \sqrt{ \frac{\hat{P}_l}{P_l} \frac{Q_l}{\hat{Q}_l}} \right) 
     \left( \sum_{l=0}^L \sqrt{P_l Q_l} \sqrt{ \frac{P_l}{\hat{P}_l} \frac{\hat{Q}_l}{ Q_l}} \right) \\
&= \sum_{l=0}^L P_l Q_l + 2\sum_{l < l'} \sqrt{P_l Q_l P_{l'} Q_{l'}} + 
   \sum_{l < l'} \sqrt{P_l Q_l P_{l'} Q_{l'}} \left( \sqrt{T_{l,l'}} + \frac{1}{\sqrt{T_{l,l'}}} - 2 \right) \\
&= \left( \sum_{l=0}^L \sqrt{P_l Q_l} \right)^2 + \sum_{l < l'} 
                                                                                                                  \sqrt{P_l Q_l P_{l'} Q_{l'}} \left( \sqrt{T_{l,l'}} + \frac{1}{\sqrt{T_{l,l'}}} - 2 \right),
\end{align*}
where, we write $T_{l,l'} := \frac{\hat{P}_l}{P_l} \frac{Q_l}{\hat{Q}_l} \frac{P_{l'}}{\hat{P}_{l'}} \frac{\hat{Q}_{l'}}{Q_{l'}}.$ Furthermore, since $\sum_{l=0}^L \sqrt{P_l Q_l} \geq 1/2$, we have
\begin{align}
  & \hat{I} - I(\{P_l\}, \{Q_l\}) \nonumber  \\
  &= - \log \biggl\{ 1 + \frac{ \sum_{l<l'} \sqrt{P_l Q_l P_{l'} Q_{l'}} 
    \bigl( \sqrt{T_{l,l'}} + \frac{1}{\sqrt{T_{l,l'}}} - 2 \bigr)}
    { \bigl( \sum_{l=0}^L \sqrt{ P_l Q_l} \bigr)^2 }  \biggr\} \nonumber \\
     &  \geq  - \log \biggl\{ 1 + 4 \sum_{l<l'} \sqrt{P_l Q_l P_{l'} Q_{l'}}  
    \left( \sqrt{T_{l,l'}} + \frac{1}{\sqrt{T_{l,l'}}} - 2 \right)  \biggr\} \nonumber \\
   & \geq - 4 \sum_{l < l'} \sqrt{P_l Q_l P_{l'} Q_{l'}} 
    \left( \sqrt{T_{l,l'}} + \frac{1}{\sqrt{T_{l,l'}}} - 2 \right). \label{eqn:Ihat_Istar2}
\end{align}

We now bound $|T_{l,l'} - 1|$:
\begin{align*}
& |T_{l,l'} - 1| = \left| \frac{\hat{P}_l}{P_l} \frac{Q_l}{\hat{Q}_l} 
      \frac{P_{l'}}{\hat{P}_{l'}} \frac{\hat{Q}_{l'}}{Q_{l'}} - 1 \right| \\
 & \qquad = \left| \left( 1 - \frac{P_l - \hat{P}_l}{P_l} \right)
    \left( 1 - \frac{\hat{Q}_l - Q_l}{\hat{Q}_l} \right)
   \left( 1- \frac{\hat{P}_{l'} - P_{l'}}{\hat{P}_{l'}}\right)
   \left( 1 -  \frac{Q_{l'}- \hat{Q}_{l'}}{Q_{l'}} \right) -1 \right| \\
& \qquad \stackrel{(a)} \leq 4 \left( \frac{|P_l - \hat{P}_l|}{P_l} +  \frac{|\hat{Q}_l - Q_l|}{\hat{Q}_l}
           +   \frac{| \hat{P}_{l'} - P_{l'}|}{\hat{P}_{l'}} +
               \frac{| Q_{l'} - \hat{Q}_{l'} | }{Q_{l'}} \right) \\
& \qquad \stackrel{(b)} \leq 8 \left( \frac{|P_l - \hat{P}_l|}{P_l} +  \frac{|\hat{Q}_l - Q_l|}{Q_l}
           +   \frac{| \hat{P}_{l'} - P_{l'}|}{P_{l'}} +
               \frac{| Q_{l'} - \hat{Q}_{l'} | }{Q_{l'}} \right), 
\end{align*}
where $(a)$ and $(b)$ follow from Lemma~\ref{lem:bound_ratio_P_Pl}. Suppose without loss of generality that for all $l = \{0, \ldots, L-1\}$, we have $\frac{| \hat{P}_l - P_l|}{P_l} + \frac{|\hat{Q}_l - Q_l|}{Q_l} \geq \frac{| \hat{P}_{l+1} - P_{l+1}|}{P_{l+1}} + \frac{|\hat{Q}_{l+1} - Q_{l+1}|}{Q_{l+1}}$. Then for $l, l' \in \{0, \ldots, L\}$ such that $l < l'$, we have, by Proposition~\ref{prop:estimation_consistency}, that
\begin{align}
| T_{l,l'} - 1 | \leq 8
    \left( \frac{|\hat{P}_l - P_l|}{P_l} + \frac{|\hat{Q}_l - Q_l|}{Q_l} \right)
\stackrel{(a)} \leq \begin{cases}
8 \eta \frac{\Delta_l}{P_l \vee Q_l}, & \text{if } l \in L_1, \\
8 \eta \frac{1}{\sqrt{ n (P_l \vee Q_l)}}, & \text{if } l \notin L_1.
\end{cases}
\end{align}
Applying Taylor's theorem on the function $g(x) = x^{1/2} + x^{-1/2} - 2$ and using the fact that $\bigl| \frac{3}{8} x^{-5/2} - \frac{15}{8} x^{-7/2}\bigr | \leq 24$ for all $x \in [1/2, 3/2]$, we have, for some $\tau_{l,l'} \in [-3, 3]$, that $\sqrt{T_{l,l'}} + \frac{1}{\sqrt{T_{l,l'}}} -2 = 
  \tau_{l,l'} (T_{l,l'} - 1)^2$ . Continuing from inequality~\eqref{eqn:Ihat_Istar2}, we obtain
 \begin{align*}
          & \hat{I} - I(\{P_l\}, \{Q_l\}) \\
            &\geq - 4 \sum_{l < l'} \sqrt{P_l Q_l P_{l'} Q_{l'}} 
    \left( \sqrt{T_{l,l'}} + \frac{1}{\sqrt{T_{l,l'}}} - 2 \right) \\
          & \geq - 4 \sum_{l \in L_1} \sum_{l' \,:\, l' > l} \sqrt{P_l Q_l P_{l'} Q_{l'}} 
    \left( \sqrt{T_{l,l'}} + \frac{1}{\sqrt{T_{l,l'}}} - 2 \right) \\
    & \quad \qquad \qquad
     - 4 \sum_{l \notin L_1} \sum_{l' \,:\, l' > l} \sqrt{P_l Q_l P_{l'} Q_{l'}} 
    \left( \sqrt{T_{l,l'}} + \frac{1}{\sqrt{T_{l,l'}}} - 2 \right) \\
  & \geq - 24 \eta \sum_{l \in L_1} \sum_{l' > l} \sqrt{P_l Q_l P_{l'} Q_{l'}} 
              \left( \frac{\Delta_l}{P_l \vee Q_l}  \right)^2 
        - 24 \eta \sum_{l \notin L_1} \sum_{l' > l} \sqrt{P_l Q_l P_{l'} Q_{l'}} 
             \frac{1}{n (P_l \vee Q_l)} \\
&  \geq - 24 \eta \left( \sum_{l \in L_1} \frac{\Delta_l^2 \sqrt{P_l Q_l}}{(P_l \vee Q_l)^2} \right)
 	\left( \sum_{l'=0}^L  \sqrt{P_{l'}Q_{l'}} \right) 
	- 24 \eta \left( \sum_{l \notin L_1} \frac{\sqrt{P_lQ_l}}{n(P_l \vee Q_l)} \right) 
          \left( \sum_{l'=0}^L \sqrt{P_{l'} Q_{l'} } \right) \\
 &  \geq - 24 \eta \left( \sum_{l \in L_1} \frac{\Delta_l^2}{P_l \vee Q_l} \right)
         \left( \sum_{l'=0}^L  \sqrt{P_{l'}Q_{l'}} \right) 
       - 24 \eta \left( \sum_{l \notin L_1} \frac{1}{n} \right) 
          \left( \sum_{l'=0}^L \sqrt{P_{l'} Q_{l'} } \right) \\
 &  \stackrel{(a)}= - 2^7 \eta I(\{P_l\}, \{Q_l\}),
\end{align*}
where $(a)$ follows from the assumptions that $I(\{P_l\}, \{Q_l\}) \leq 2 \log 2$ and $I(\{P_l\}, \{Q_l\}) \geq \frac{L+1}{n}$, and also from Lemma~\ref{lem:L1_info_bound}. Thus, using the fact that $m_k = |V_k'| - |V_k| \geq \frac{n}{2\beta K} - 2 \gamma n$ and $m_1 = |V_1| - |V_1'| \geq \frac{n}{2 \beta K} - 2 \gamma n$, where both inequalities follow from~\eqref{eqn:V1Vkprime_lower_bound} and~\eqref{eqn:V1primeVk_upper_bound}, we have 
\begin{align}
  \exp\bigl( - \frac{(m_1 + m_k)}{2} \hat{I} \bigr)
  &\leq
  \exp\biggl( - \bigl( \frac{n}{\beta K} - \gamma n \bigr)
    I(\{P_l\}, \{Q_l\}) (1 - 2^7 \eta) \biggr) \nonumber \\
  &\leq \exp \biggl( -   (1 - 2^7 \eta)(1 - \beta K \gamma) \frac{n}{\beta K} I(\{P_l\}, \{Q_l\})
   \biggr).
    \label{eqn:Ihat_first_term_final}
\end{align}
Now we bound the last two terms of equation~\eqref{eqn:Ihat_overall}. We first assume that $m_k \geq m_1$. Then
\begin{align}
& \left( \sum_{l=0}^L \sqrt{\frac{\hat{P}_l}{\hat{Q}_l}} Q_l \right)^{\frac{m_k - m_1}{2}} 
 \left( \sum_{l=0}^L \sqrt{\frac{\hat{Q}_l}{\hat{P}_l}} P_l \right)^{\frac{m_1 - m_k}{2}} \nonumber \\
&=  \left( 
   \frac{\sum_{l=0}^L \sqrt{\frac{\hat{P}_l}{\hat{Q}_l}} Q_l}
        {\sum_{l=0}^L \sqrt{\frac{\hat{P}_l}{\hat{Q}_l}} \hat{Q}_l} 
     \right)^{\frac{m_k - m_1}{2}} 
   \left( \frac{\sum_{l=0}^L \sqrt{\frac{\hat{Q}_l}{\hat{P}_l}} P_l}
                                                                                               {\sum_{l=0}^L \sqrt{\frac{\hat{Q}_l}{\hat{P}_l}} \hat{P_l} } \right)^{\frac{m_1 - m_k}{2}} \nonumber \\
  &=
\left( 1 + 
   \frac{\sum_{l=0}^L \sqrt{\frac{\hat{P}_l}{\hat{Q}_l}} (Q_l - \hat{Q}_l)}
        {\sum_{l=0}^L \sqrt{\frac{\hat{P}_l}{\hat{Q}_l}} \hat{Q}_l} 
     \right)^{\frac{m_k - m_1}{2}} 
   \left( 1+ \frac{\sum_{l=0}^L \sqrt{\frac{\hat{Q}_l}{\hat{P}_l}} (\hat{P}_l - P_l)}
         {\sum_{l=0}^L \sqrt{\frac{\hat{Q}_l}{\hat{P}_l}} P_l } \right)^{\frac{m_k - m_1}{2}}. \label{eqn:Ihat_last_term_overall}
\end{align}
By Lemma~\ref{lem:sqrt_ratio_pl_ql_minus_1}, we have $\sum_{l=0}^L \sqrt{\frac{\hat{P}_l}{\hat{Q}_l}} \hat{Q}_l \geq \frac{1}{2} \sqrt{P_l Q_l} \geq \frac{1}{4}$ and $\sum_{l=0}^L \sqrt{\hat{Q}_l}{\hat{P}_l} P_l \geq \frac{1}{4} \sqrt{P_l Q_l} \geq \frac{1}{8}$. We also have 
 \begin{align}
& \left| \sum_{l=0}^L \sqrt{\frac{\hat{P}_l}{\hat{Q}_l}} (Q_l - \hat{Q}_l) \right| = 
  \left|  \sum_{l=0}^L \left( \sqrt{\frac{\hat{P}_l}{\hat{Q}_l}} -1 \right) (Q_l - \hat{Q}_l) 
 \right| \nonumber \\
& \qquad \leq \sum_{l \in L_1} \biggl| \sqrt{\frac{\hat{P}_l}{\hat{Q}_l}} -1 \biggr| |Q_l - \hat{Q}_l| +  
    \sum_{l \notin L_1} \biggl| \sqrt{\frac{\hat{P}_l}{\hat{Q}_l}} -1 \biggr| |Q_l - \hat{Q}_l| \nonumber  \\
   & \qquad \stackrel{(a)} \leq
     \eta(1 + 6 \eta \rho)  \sum_{l \in L_1} \frac{\Delta_l^2}{Q_l}  +
     4 \eta \rho \frac{L+1}{n} \nonumber \\
   & \qquad \stackrel{(b)} \leq \eta \rho ( 1 + 6 \eta \rho) \sum_{l \in L_1} \frac{\Delta_l^2}{Q_l \vee P_l} + 4 \eta \rho I(\{P_l\}, \{Q_l\}) \leq
     12 \eta \rho I(\{P_l\}, \{Q_l\}), \label{eqn:Ihat_last_term_part2}
\end{align}
where $(a)$ follows from Proposition~\ref{prop:estimation_consistency} and Lemma~\ref{lem:sqrt_ratio_pl_ql_minus_1}, $(b)$ follows from the assumption that $\frac{n I(\{P_l\}, \{Q_l\})}{L+1} \geq 1$, and the last inequality follows from Lemma~\ref{lem:L1_info_bound}. By an identical argument, we have
\begin{align}
  \left| \sum_{l=0}^L \sqrt{\frac{\hat{P}_l}{\hat{Q}_l}} (Q_l - \hat{Q}_l) \right| \leq 12 \eta \rho I(\{P_l\}, \{Q_l\}).
  \label{eqn:Ihat_last_term_part3}
  \end{align}
Using inequalities~\eqref{eqn:Ihat_last_term_overall},~\eqref{eqn:Ihat_last_term_part2}, and~\eqref{eqn:Ihat_last_term_part3}, we have 
\begin{align}
 &  \left( \sum_{l=0}^L \sqrt{\frac{\hat{P}_l}{\hat{Q}_l}} Q_l \right)^{\frac{m_k - m_1}{2}} 
 \left( \sum_{l=0}^L \sqrt{\frac{\hat{Q}_l}{\hat{P}_l}} P_l \right)^{\frac{m_1 - m_k}{2}} \nonumber \\
  & \qquad \leq \exp\bigl( (m_k - m_1) \log(1 + 2^7 \eta \rho I(\{P_l\}, \{Q_l\} ) \bigr) \nonumber \\
  &\qquad \leq \exp\bigl( (m_k + m_1) 2^7 \eta \rho I(\{P_l\}, \{Q_l\})  \bigr) \leq \exp \bigl( 2^8 \eta \rho n I(\{P_l\}, \{Q_l\}) \bigr)
    \nonumber \\
  & \qquad \leq \exp \bigl( 2^8 \beta K \rho \eta \frac{n}{\beta K} I(\{P_l\}, \{Q_l\}) \bigr). 
    \label{eqn:Ihat_last_term_final}
\end{align}
If $m_1 \geq m_k$, inequality~\eqref{eqn:Ihat_last_term_final} still holds by an identical argument.

Finally, we combine inequalities~\eqref{eqn:overall_prob_bound},~\eqref{eqn:extra_final},~\eqref{eqn:Ihat_overall},~\eqref{eqn:Ihat_first_term_final}, and~\eqref{eqn:Ihat_last_term_final} to obtain
\begin{align*}
  \mathbb{P}(E_1(k)^c \cup E_2(k)^c)
  &\leq \mathbb{P}(E_1(k)^c) +
  \exp\biggl( - (1 - C \eta \rho) \frac{n}{\beta K} I(\{P_l\}, \{Q_l\}) \biggr)  \\
  &\leq 5(L+1) n^{-6} + \exp\biggl( - (1 - C \beta K \rho \eta) \frac{n}{\beta K} I(\{P_l\}, \{Q_l\}) \biggr).
\end{align*}
A union bound over $\{E_1^c \cup E_2 (k)^c \,:\, k \in [K]\backslash \{1\} \}$ finishes the proof.
\end{proof}


\subsection{Additional lemmas for Proposition~\ref{prop:labeled_sbm_rate}}
\label{appendix: lemmas for labeled_sbm_rate}

\begin{lemma}
  \label{lem:add_noise_bound}
  Let $L \in \mathbb{Z}^+$ and $(\{P_l\}, \{Q_l\}) \in \mathcal{P}_L^2$. Suppose $ I(\{P_l\}, \{Q_l\}) \leq 2 \log 2$. Let $c \in [0, \infty)$, let $\delta := c \frac{L+1}{n}$, and let
  \begin{equation*}
    P_l' := P_l(1-\delta) + \frac{\delta}{L+1}, \quad \text{and} \quad
    Q_l' := Q_l(1-\delta) + \frac{\delta}{L+1}.
  \end{equation*}
Then $P_l', Q_l' \geq \frac{c}{n}$ for all $l \in \{0, \ldots, L\}$, and
  \[
    1 - c \frac{L+1}{n} \geq \frac{I(\{P'_l\}, \{Q'_l\})}{I(\{P_l\}, \{Q_l\}) }
    \geq \frac{1}{1 + 2 I(\{P_l\}, \{Q_l\}) } \biggl( 1 - c \frac{L+1}{n} \biggr).
  \]
\end{lemma}

\begin{proof}
  Define $P''_l := \frac{1}{L+1}$ and $Q''_l := \frac{1}{L+1}$ for all $l \in \{0,\ldots,L\}$. Then $(\{P'_l\}, \{Q'_l\}) = (1 - \delta)
  (\{P_l\}, \{Q_l\}) + \delta (\{P''_l\}, \{Q''_l\})$. Since $I \,:\, \mathcal{P}_L \times \mathcal{P}_L \rightarrow [0, \infty]$ is convex, we have
  \[
    I(\{P'_l\}, \{Q'_l\}) \leq (1-\delta) I(\{P_l\}, \{Q_l\}) + \delta I (\{P''_l\}, \{Q''_l\}) \leq \biggl( 1 - \frac{L+1}{n} \biggr) I(\{P_l\}, \{Q_l\}).
    \]
Define the shorthand $H := \sum_{l=0}^L (\sqrt{P_l} - \sqrt{Q_l})^2$ and $H' :=  \sum_{l=0}^L (\sqrt{P'_l} - \sqrt{Q'_l})^2$. We then have
    \begin{align*}
      H - H'
      & = \sum_{l=0}^L (\sqrt{P_l} - \sqrt{Q_l})^2
        \biggl( 1 - \frac{(\sqrt{P'_l} - \sqrt{Q'_l})^2}{ (\sqrt{P_l} - \sqrt{Q_l})^2 } \biggr) \\
      & = \sum_{l=0}^L (\sqrt{P_l} - \sqrt{Q_l})^2
        \biggl( 1 - \bigl( 1 - c \frac{L+1}{n} \bigr)^2
        \bigl( \frac{\sqrt{P_l} + \sqrt{Q_l}}{\sqrt{P'_l} + \sqrt{Q'_l}} \bigr)^2 \biggr) \\
      &\stackrel{(a)} \leq 2 c \frac{L+1}{n} \sum_{l=0}^L (\sqrt{P_l} - \sqrt{Q_l})^2 \leq 2 c \frac{L+1}{n} H,
    \end{align*}
    where $(a)$ is true because $(\sqrt{P'_l} + \sqrt{Q'_l})^2 \leq (\sqrt{P_l} + \sqrt{Q_l})^2$. Note also that since $ I(\{P_l\}, \{Q_l\}) \leq 2 \log 2$ and $ I(\{P_l\}, \{Q_l\}) = -2 \log( 1 - H/2)$, we have $H \leq 1$, so by Lemma~\ref{lem:log_linearize}, we conclude that$H \leq  I(\{P_l\}, \{Q_l\}) \leq H(1 + 2H)$. 
    
Therefore, by Lemma~\ref{lem:log_linearize}, we have
    \begin{align*}
      I(\{P'_l\}, \{Q'_l\})
      & \geq H' \geq H \bigl( 1 - 2 c \frac{L+1}{n} \bigr) \stackrel{(a)} \geq
        I(\{P_l\}, \{Q_l\}) \frac{1}{1 + 2H} \bigl( 1 - 2 c \frac{L+1}{n} \bigr) \\
       & \stackrel{(b)} \geq I(\{P_l\}, \{Q_l\}) \frac{1}{1 + 2 I(\{P_l\}, \{Q_l\})}  \bigl( 1 - 2 c \frac{L+1}{n} \bigr).
    \end{align*}
\end{proof}

We often use the bound $\frac{1}{2} P \leq \hat{P}_l \leq 2 P_l$, justified in the following lemma:
\begin{lemma}
  \label{lem:bound_ratio_P_Pl}
  Let $L \in \mathbb{Z}^+,\, \rho \in [1, \infty)$, let $(\{P_l\}, \{Q_l\}) \in \mathcal{G}_{L, \rho}$, and suppose $P_l \vee Q_l \geq \frac{1}{n}$ for all $l \in \{0,\ldots,L\}$. Suppose there exists $\eta \in [0, \frac{1}{2 \rho})$ such that, for all $l \in \{0,\ldots,L\}$, we have
\begin{equation*}
\max\{ | \hat{P}_l - P_l| ,\, |\hat{Q}_l - Q_l| \} \leq \eta \bigr(\Delta_l \vee \bigl( \frac{P_l \vee Q_l}{n} \bigr)^{1/2} \bigr).
\end{equation*}
Then for all $l \in \{0, \ldots, L\}$:
  \begin{enumerate}
    \item It holds that
\[
  \max\biggl\{ \frac{|\hat{P}_l - P_l|}{P_l}, \frac{|\hat{Q}_l - Q_l|}{Q_l} \biggr\} \leq \eta \rho.
\]
\item It holds that $\frac{1}{2} P_l \leq \hat{P}_l \leq 2P_l$ and $\frac{1}{2} Q_l \leq \hat{Q}_l \leq 2 Q_l$.
\item It holds that $\bigl( \frac{\hat{P}_l}{\hat{Q}_l} \bigr)^{1/2} Q_l \geq
  \frac{1}{2} \sqrt{\hat{P}_l \hat{Q}_l} \geq\frac{1}{4} \sqrt{P_l Q_l}$.
  \end{enumerate}
\end{lemma}

\begin{proof}
Fix an $l \in \{0,\ldots, L\}$ arbitrarily. We prove the first claim for $P_l$; the same argument applies to $Q_l$. Since $n (P_l \vee Q_l) \geq 1$, we have
\begin{align*}
  \frac{ | \hat{P}_l - P_l |}{P_l} \leq \eta \max\biggl\{ \frac{\Delta_l}{P_l},\, \frac{\sqrt{(P_l \vee Q_l)}}{\sqrt{n} P_l}
  \biggr\} \leq \eta \rho .
\end{align*}
The second claim follows from Claim 1, since $\eta \rho \leq \frac{1}{2}$ by assumption.
The third claim follows because, by Claim 2, we have
\begin{equation*}
\bigl( \frac{\hat{P}_l}{\hat{Q}_l} \bigr)^{1/2} Q_l \geq \sqrt{\hat{P}_l \hat{Q}_l} \frac{Q_l}{\hat{Q}_l} \geq \frac{1}{2} \sqrt{\hat{P}_l \hat{Q}_l} \geq \frac{1}{4} \sqrt{P_l Q_l}.
\end{equation*}
\end{proof}

\begin{lemma}
  \label{lem:sqrt_ratio_pl_ql_minus_1}
  Let $L \in \mathbb{Z}^+,\, \rho \in [1, \infty)$, let $(\{P_l\}, \{Q_l\}) \in \mathcal{G}_{L, \rho}$, and suppose $P_l \vee Q_l \geq \frac{1}{n}$ for all $l \in \{0,\ldots,L\}$. Suppose there exists $\eta \in [0, \frac{1}{2 \rho})$ such that, for all $l \in \{0,\ldots,L\}$, we have
\begin{equation*}
\max\{ | \hat{P}_l - P_l| ,\, |\hat{Q}_l - Q_l| \} \leq \eta \bigr(\Delta_l \vee \bigl( \frac{P_l \vee Q_l}{n} \bigr)^{1/2} \bigr).
\end{equation*}
Then:
\begin{enumerate}
\item 
For all $l$ satisfying  $n \frac{\Delta_l^2}{P_l \vee Q_l} \geq 1$, we have  
\begin{align}
\biggl| \biggl( \frac{\hat{P}_l }{\hat{Q}_l} \biggr)^{1/2}  - 1 \biggr|
  \leq  \frac{\Delta}{Q_l} (1 + 6 \eta \rho) \quad \textrm{and} \quad
  \biggl| \biggl( \frac{\hat{Q}_l }{\hat{P}_l} \biggr)^{1/2}  - 1 \biggr|
  \leq  \frac{\Delta}{P_l} (1 + 6 \eta \rho).
  \label{eqn:sqrt_ratio_case1}
\end{align}

\item
For all $l$ satisfying $n \frac{\Delta_l^2}{P_l \vee Q_l} < 1$, we have
\begin{align}
\biggl| \biggl( \frac{\hat{P}_l}{\hat{Q}_l} \biggr)^{1/2}  - 1 \biggr| \leq
4 \rho \frac{1}{\sqrt{n  (P_l \vee Q_l) } } \quad \textrm{and} \quad
\biggl| \biggl( \frac{\hat{Q}_l}{\hat{P}_l} \biggr)^{1/2} - 1 \biggr| \leq
  4 \rho \frac{1}{\sqrt{n  (P_l \vee Q_l) } }.
  \label{eqn:sqrt_ratio_case2}
\end{align}
\end{enumerate}
\end{lemma}

\begin{proof}
  Fix an $l \in \{0, \ldots, L\}$, and suppose $\Delta_l \geq \bigl( \frac{P_l \vee Q_l}{n} \bigr)^{1/2}$. Define the shorthand
  \begin{align*}
    \eta_1 := \frac{\hat{P}_l - P_l}{P_l - Q_l} ,\;
    \eta_2 := \frac{Q_l - \hat{Q}_l}{P_l - Q_l} ,\;
    \eta_3 := \frac{Q_l}{\hat{Q}_l} - 1.
  \end{align*}
  Then
  \begin{align*}
    \frac{\hat{P}_l}{\hat{Q}_l} - 1
    = \frac{\hat{P}_l - P_l + (P_l - Q_l) + (Q_l - \hat{Q}_l)}{Q_l } \frac{Q_l}{\hat{Q}_l}
    = \frac{P_l - Q_l}{Q_l} ( 1 + \eta_1 + \eta_2) (1 + \eta_3).
  \end{align*}
  Now note that $|\eta_1| \leq \frac{| \hat{P}_l - P_l|}{\Delta_l} \leq \eta$. Likewise, we have $|\eta_2 | \leq \eta$. Finally, we know that $|\eta_3| \leq \frac{|\hat{Q}_l - Q_l|}{Q_l} \frac{Q_l}{\hat{Q}_l} \leq 2 \eta \rho$ by Lemma~\ref{lem:bound_ratio_P_Pl}. Therefore, we have $\bigl|  \frac{\hat{P}_l}{\hat{Q}_l} - 1 \bigr| \leq  \bigl| \frac{P_l - Q_l}{Q_l} \bigr| (1 + 2\eta)(1 + 2\eta \rho)$. We then use the fact that $2 \eta \rho \leq 1$ and apply the inequality $| \sqrt{x} - 1 | \leq |x - 1|$ for all $x \geq 0$ to prove the first inequality of the first case~\eqref{eqn:sqrt_ratio_case1}. The second inequality holds by symmetry.

 Now suppose $\Delta_l < \bigl( \frac{P_l \vee Q_l}{n} \bigr)^{1/2}$. Define the shorthand
  \begin{align*}
    \eta_1 := \frac{\hat{P}_l - P_l}{Q_l} \sqrt{n (P_l \vee Q_l)} ,\;
    \eta_2 := \frac{Q_l - \hat{Q}_l}{Q_l} \sqrt{n (P_l \vee Q_l)} ,\;
    \eta_3 := \frac{Q_l}{\hat{Q}_l} - 1 ,\;
    \tau := \frac{P_l - Q_l}{Q_l} \sqrt{n (P_l \vee Q_l)}.
  \end{align*}
  Then
  \begin{align*}
    \frac{\hat{P}_l}{\hat{Q}_l} - 1
    = \frac{\hat{P}_l - P_l + (P_l - Q_l) + (Q_l - \hat{Q}_l)}{Q_l } \frac{Q_l}{\hat{Q}_l}
    = \frac{1}{\sqrt{n(P_l \vee Q_l)} } (\eta_1 + \eta_2 + \tau) (1 + \eta_3).
  \end{align*}
Observe $|\eta_1| \leq \eta \frac{P_l \vee Q_l}{Q_l} \leq \eta \rho$, and likewise $|\eta_2| \leq \eta \rho$. We use Lemma~\ref{lem:bound_ratio_P_Pl}, the fact that $Q_l n \geq 1$, and $\rho \geq 1$ to bound $|\eta_3 | \leq 2 \eta \frac{\sqrt{P_l\vee Q_l}}{Q_l \sqrt{n}} \leq 2 \eta \rho^{1/2} \leq 2 \eta \rho$. Finally, we have $|\tau| \leq \frac{\Delta_l}{Q_l} \sqrt{n (P_l \vee Q_l)} \leq \rho$. Therefore, we have $\bigl|  \frac{\hat{P}_l}{\hat{Q}_l} - 1 \bigr| \leq  \frac{1}{\sqrt{n(P_l \vee Q_l)} } (2 \eta \rho + \rho)(1 + 2 \eta \rho)$. Since $2 \eta \rho \leq 1 \leq \rho$ and $| \sqrt{x} - 1 | \leq |x - 1|$ for all $x \geq 0$, the first inequality of the second case~\eqref{eqn:sqrt_ratio_case2} holds. The second inequality holds by symmetry.

\end{proof}

\begin{lemma}
  \label{lem:L1_info_bound}
  Let $L \in \mathbb{Z}^+$, let $(\{P_l\}, \{Q_l\}) \in \mathcal{P}_L \times \mathcal{P}_L$, and let $\Delta_l := |P_l - Q_l|$. Let $C \in (0, \infty)$ be such that $\frac{n I(\{P_l\}, \{Q_l\})}{L+1} \geq 4 C$ and suppose $I(\{P_l\}, \{Q_l\}) \leq 2 \log 2$. Define $L_1 := \{ l \in \{0,\ldots, L\} \,:\, \frac{\Delta^2_l}{P_l \vee Q_l} \geq C \}$. Then
  \begin{align}
    \frac{1}{2} \sum_{l \in L_1} \frac{\Delta_l^2}{P_l\vee Q_l}
    \leq I( \{P_l\}, \{Q_l\}) \leq 4\sum_{l \in L_1} \frac{\Delta_l^2}{P_l \vee Q_l}.
    \end{align}
\end{lemma}

\begin{proof}
Observe that
  \[
    I(\{P_l\}, \{Q_l\}) = -2 \log \biggl( 1 - \frac{\sum_{l=0}^L (\sqrt{P_l} - \sqrt{Q_l})^2}{2} \biggr).
  \]
  Since $\log(1 - x) \leq -x$ for all $x \in (-1, \infty)$, we have
  \begin{align*}
    I(\{P_l\}, \{Q_l\}) \geq \sum_{l=0}^L (\sqrt{P_l} - \sqrt{Q_l})^2 \geq \sum_{l=0}^L \frac{\Delta_l^2}{(\sqrt{P_l} + \sqrt{Q_l})^2} \geq
    \frac{1}{2} \sum_{l \in L_1} \frac{\Delta_l^2}{P_l \vee Q_l}.
  \end{align*}
  On the other hand, we know that $\sum_{l=0}^L \sqrt{P_l Q_l} = e^{- \frac{1}{2} I(\{P_l\}, \{Q_l\}) } \geq \frac{1}{2}$. By the fact that $-2x \leq \log(1 - x)$ for all $x \in [0, 1/2]$, we have
  \begin{align*}
    I(\{P_l\}, \{Q_l\}) &\leq 2 \sum_{l=0}^L (\sqrt{P_l} - \sqrt{Q_l})^2 \leq 2 \sum_{l=0}^L \frac{\Delta_l^2}{ P_l \vee Q_l} \\
    &\leq 2 \sum_{l \in L_1} \frac{\Delta_l^2}{P_l \vee Q_l} + 2 \frac{(L+1)C}{n} \leq 2 \sum_{l \in L_1} \frac{\Delta_l^2}{P_l \vee Q_l} + \frac{  I(\{P_l\}, \{Q_l\})}{2}.
    \end{align*}  
\end{proof}

The following lemma slightly expands upon Lemma 4 of Gao et al~\cite{GaoEtal15}. 

\begin{lemma}
  \label{lem:consensus}
  Let $\sigma, \sigma' \,:\, [n] \rightarrow [K]$ be two clusters such that, for some $T \in [n]$, the minimum cluster size of $\sigma$ is at least $T$. Suppose $l(\sigma, \sigma') < \frac{T}{2n}$. Then there is a unique $\xi \in S_K$ such that $l(\sigma, \sigma') = \frac{1}{n} d_H(\xi \circ \sigma, \sigma')$; furthermore, the unique permutation $\xi$ is of the form
  \begin{align}
\xi(k) = \argmax_{k' \in [K]} | \{ v \in [n] \,:\, \sigma(v) = k \} \cap \{ v \in [n] \,:\, \sigma'(v) = k' \} |. \label{eqn:consensus_equation}
\end{align}
\end{lemma}

\begin{proof}
 Suppose $\pi \in S_K$ satisfies $l(\sigma, \sigma') = \frac{1}{n} d_H(\pi \circ \sigma, \sigma') < \frac{T}{2n}$. Fix $k \in [K]$. Then
  \[
    | \{ u \in [n] \,:\, \sigma(u) = k \} \cap
      \{ u \in [n] \,:\, \sigma'(u) \neq \pi(k) \} | \leq
      d_H(\pi \circ \sigma, \sigma') < \frac{T}{2},
  \]
  and
\begin{align*}
  & | \{ u \in [n] \,:\, \sigma(u) = k \} \cap
  \{ u \in [n] \,:\, \sigma'(u) = \pi(k) \} | \\
  &= | \{ u \in [n] \,:\, \sigma(u) = k \}| -
     | \{ u \in [n] \,:\, \sigma(u) = k \} \cap
    \{ u \in [n] \,:\, \sigma'(u) \neq \pi(k) \} | \\
  &\geq T - 
    d_H(\xi \circ \sigma, \sigma') \geq \frac{T}{2}.
\end{align*}
It thus follows that $\pi(k)$ is the unique maximizer of $\argmax_{k' \in [K] } | \{ u \in [n] \,:\, \sigma(u) = k\} \cap \{ u \in [n] \,:\, \sigma'(u) = k' \} |$. Since $k$ was fixed arbitrarily, the lemma follows.
\end{proof}


\section{Proof of Proposition~\ref{prop:discretization1}}
\label{sec:proof_of_discretization1}

\begin{proof}
  Let $\mu$ be the Lebesgue measure on $[0,1]$. Let us arbitrarily fix $((P_0, p), (Q_0, q)) \in \mathcal{G}_{\tilde{C}, \tilde{c}_1, \tilde{c}_2, r, t}$ and define $\tilde{g}$ as in Condition C2. Let $L \in \mathbb{N}$, and suppose $L \geq \tilde{c}_1^{-1} \vee \tilde{c}_2^{-1}$. Let $\{[a_l, b_l]\}_{l \in \{1, \ldots, L\}}$ be a uniformly spaced binning of $[0,1]$, and for each $l \in \{1, \ldots, L\}$, define
  \begin{align}
    \tilde{P}_l := \int_{a_l}^{b_l} \tilde{p}(z) dz \quad \textrm{and} \quad \int_{a_l}^{b_l} \tilde{q}(z) dz.
    \label{eqn:tildePl_Ql_defn}
  \end{align}
  Define
  $R := \{ z \in [0,1] : \tilde{g}(z) \leq (2 \tilde{C} L)^{1/r} \}$. By Markov's inequality, we have
  $\mu(R^c) \leq \frac{1}{2L}$. Since $\tilde{g}(z)$ is quasi-convex, $R$ must be an interval. Thus, only bins $[0, 1/L]$ and $[1-1/L, 1]$ have non-empty intersection with $R^c$. Let $[a_l, b_l]$ be a bin such that $[a_l, b_l] \subset R$. Then
  \[
    \frac{P_l}{Q_l} = \frac{(1-P_0) \int_{a_l}^{b_l} \tilde{p}(z) \,dz}{(1-Q_0) \int_{a_l}^{b_l} \tilde{q}(z)}
    \leq \tilde{C}
    \frac{ \int_{a_l}^{b_l} \frac{\tilde{p}(z)}{\tilde{q}(z)} \tilde{q}(z) \, dz }{ \int_{a_l}^{b_l} \tilde{q}(z) \, dz }
    \leq \tilde{C} \exp( (2 \tilde{C} L)^{1/r} ).
  \]
Likewise, we can show that $\frac{P_l}{Q_l} \geq \tilde{C}^{-1} \exp( - (2\tilde{C} L)^{1/r})$. Now we consider $[0, 1/L]$ and suppose $[0, 1/L] \cap R^c \neq \emptyset$. Define $\tilde{P}_l' := \int_{\bin_l \cap R} \tilde{p}(z) \,dz$ and $\tilde{Q}'_l := \int_{\bin_l \cap R} \tilde{q}(z) \, dz$, and define $\tilde{P}_l'' := \int_{\bin_l \cap R^c} p(z) \, dz$ and $\tilde{Q}''_l := \int_{\bin_l \cap R^c} q(z) \, dz$. We can use the same reasoning as above to show that $\exp(-(2 \tilde{C} L)^{1/r}) \leq \frac{\tilde{P}'_l}{\tilde{Q}'_l} \leq \exp( (2 \tilde{C} L)^{1/r})$. Since $L \geq \tilde{c}_1^{-1}$, both $\tilde{p}(z)$ and $\tilde{q}(z)$ are non-decreasing in $[0, 1/L]$ by Assumption C5. Thus,
\[
\tilde{P}'_l \geq \min_{z \in [0, 1/L] \cap R} \frac{\tilde{p}(z)}{2L} \geq \max_{z \in [0, 1/L] \cap R^c} \frac{\tilde{p}(z)}{2L} \geq \tilde{P}''_l,
\]
where the first inequality follows because $\mu(R^c) \leq \frac{1}{2L}$. With the same reasoning, we know that $Q'_l \geq Q''_l$. Thus,
\begin{align*}
\frac{1}{2} \exp( - (2 \tilde{C} L)^{1/r} )\leq \frac{\tilde{P}'_l}{2 \tilde{Q}'_l} \leq \frac{\tilde{P}_l}{\tilde{Q}_l} \leq \frac{2 \tilde{P}'_l}{ \tilde{Q}'_l} \leq 2 \exp((2 \tilde{C} L)^{1/r}).
\end{align*}
Therefore, we have
\begin{align*}
  \frac{1}{2 \tilde{C}} \exp( - (2 \tilde{C} L)^{1/r} ) \leq \frac{(1-P_0) \tilde{P}_l}{(1-Q_0) \tilde{Q}_0} \leq
  2 \tilde{C}  \exp( (2 \tilde{C} L)^{1/r} ).
\end{align*}
Using identical reasoning, we may obtain the same bounds for $\frac{P_l}{Q_l}$ for $l$ corresponding to the $[1-1/L, 1]$ bin. This proves the first claim of the proposition. Moreover, if $\tilde{g}$ is non-decreasing, then $R^c \subseteq [1-1/L, 1]$. Thus, the condition that $\tilde{p}, \tilde{q}$ are non-increasing in $(1-c_2, 1]$ yields the first claim of the proposition as observed in footnote~\ref{note:one_sided_compact}.

For the second claim, define
\begin{align*}
  \tilde{H} &:= \int_0^1 (\sqrt{\tilde{p}(z)} - \sqrt{\tilde{q}(z)})^2 \, dz, \quad
              \tilde{H}_L := \sum_{l=1}^L (\sqrt{\tilde{P}_l} - \sqrt{\tilde{Q}_l})^2, \\
  H &:= (\sqrt{P_0} - \sqrt{Q_0})^2 + ( \sqrt{1 - P_0} - \sqrt{1 - Q_0})^2 +
  \sqrt{(1-P_0)(1-Q_0)} \tilde{H}, \\
  H_L &:= (\sqrt{P_0} - \sqrt{Q_0})^2 + ( \sqrt{1 - P_0} - \sqrt{1 - Q_0})^2 +
  \sqrt{(1-P_0)(1-Q_0)} \tilde{H}_L.
\end{align*}
Let
$\eta_L := \sup \bigl \{ \bigl| 1 - \frac{\tilde{H}_L}{\tilde{H}} \bigr|
\,:\, ((P_0, \tilde{p}), (Q_0, \tilde{q})) \in \tilde{\mathcal{G}}_{\tilde{C}, \tilde{c}_1, \tilde{c}_2, r, t} \bigr \}$; we know
then by Proposition~\ref{prop:H_HL_convergence1} that $\lim_{L \rightarrow \infty} \eta_L  = 0$. Fix $((P_0, \tilde{p}), (Q_0, \tilde{q})) \in \tilde{\mathcal{G}}_{\tilde{C}, \tilde{c}_1, \tilde{c}_2, r, t}$, and let $\{P_l, Q_l\}$ be the corresponding discretized probabilities. Then
\begin{align*}
  |H - H_L| = \sqrt{(1 - P_0)(1 - Q_0)} |\tilde{H}_L - \tilde{H}| 
              \leq \sqrt{(1 - P_0)(1 - Q_0)} \tilde{H} \eta_L \leq H \eta_L.
\end{align*}
Observe that we also have
\begin{align*}
  I((P_0, \tilde{p}), (Q_0, \tilde{q}))
  = -2 \log \biggl \{ \sqrt{P_0 Q_0} + \sqrt{(1-P_0)(1-Q_0)} \int_0^1 \sqrt{\tilde{p}(z) \tilde{q}(z)} \,dz \biggr\}
  = -2 \log (1 - \frac{1}{2} H ),
\end{align*}
and also $I(\{P_l\}, \{Q_l\}) = -2 \log (1 - \frac{1}{2} H_L )$. Hence,
\begin{align*}
&  I(\{P_l\}, \{Q_l\}) - I((P_0, \tilde{p}), (Q_0, \tilde{q}))
  = -2 \log \frac{ 1 - \frac{1}{2} H_L}{1 - \frac{1}{2} H} 
  = -2 \log \biggl( 1 + \frac{1}{2} \frac{H - H_L}{1 - \frac{1}{2} H} \biggr) \\
&\qquad   \geq - \frac{H - H_L}{1 - \frac{1}{2} H} \geq - 2 (H - H_L) \geq -2 \eta_L H \stackrel{(a)} \geq - 2 \eta_L I((P_0, \tilde{p}), (Q_0, \tilde{q})),
\end{align*}
where $(a)$ follows from Lemma~\ref{lem:log_linearize}. On the other hand, we have by the Cauchy-Schwarz inequality that
\[
  \int_0^1 \sqrt{\tilde{p}(z) \tilde{q}(z)} \, dz \leq
  \biggl\{ \int_0^1 \tilde{p}(z) \, dz \biggr\}^{1/2} \biggl\{ \int_0^1 \tilde{q}(z) \, dz \biggr\}^{1/2}
  = \sqrt{\tilde{P}_l \tilde{Q}_l}.
  \]
Therefore, we have $I((P_0, \tilde{p}), (Q_0, \tilde{q})) \geq I(\{P_l\}, \{ Q_l\})$. Since $((P_0, \tilde{p}), (Q_0, \tilde{q}))$ was chosen arbitrarily, the proposition follows. 
\end{proof}

\begin{proposition}
\label{prop:H_HL_convergence1}
Let $\tilde{C} \in [1, \infty)$, $\tilde{c}_1, \tilde{c}_2 \in (0, 1/2)$, $r > 2$, and $t \in (0, 1)$. For any $((P_0, \tilde{p}), (Q_0, \tilde{q})) \in \tilde{\mathcal{G}}_{\tilde{C}, \tilde{c}_1, \tilde{c}_2, r, t}$ and for any $L \in \mathbb{N}$, let $\tilde{P}_l, \tilde{Q}_l$ be defined as in equation~\eqref{eqn:tildePl_Ql_defn} for $l \in \{1,\ldots,L\}$. 

Then
\[
  \lim_{L \rightarrow \infty}
  \sup_{\substack{ ((P_0, \tilde{p}), (Q_0, \tilde{q})) \\
      \in \tilde{\mathcal{G}}_{\tilde{C}, \tilde{c}_1, \tilde{c}_2, r, t}}}
  \biggl| 1 - \frac{\sum_{l=1}^{L} (\sqrt{\tilde{P}_l} - \sqrt{\tilde{Q}_l} )^2}{\int_0^1 (\sqrt{\tilde{p}(z)} - \sqrt{\tilde{q}(z)})^2 \, dz} \biggr| = 0. 
  \]

\end{proposition}

\begin{proof}
  Let us arbitrarily fix $((P_0, \tilde{p}), (Q_0, \tilde{q})) \in \tilde{\mathcal{G}}_{\tilde{C}, \tilde{c}_1, \tilde{c}_2, r, t}$. Let $\tilde{\alpha} := \int_0^1 (\sqrt{\tilde{p}(z)} - \sqrt{\tilde{q}(z)})^2 \, dz$, and let $\tilde{\gamma}(z) := \frac{\tilde{p}(z) - \tilde{q}(z)}{\tilde{\alpha}}$. Let $L \in \mathbb{N}$ and suppose that $L \geq 4^{ - \frac{r(t+1)}{2 - rt}}$.

For each $l=1,\ldots,L$, we also define $\tilde{\gamma}_l := \int_{a_l}^{b_l} \tilde{\gamma}(z) \, dz$. Since $\tilde{p}$ and $\tilde{q}$ are continuous and bounded, we may define $z_l := \argmax_{z \in [a_l, b_l]} \tilde{p}(z) + \tilde{q}(z)$, $z'_l := \argmax_{z \in [a_l, b_l]}  \tilde{p}(z) $, and $z''_l := \argmax_{z \in [a_l, b_l]} \tilde{q}(z)$. We also define the shorthand 
  \begin{align*}
    &D_l := \int_{a_l}^{b_l} \frac{\tilde{\gamma}(z)^2}{ (\sqrt{\tilde{p}(z)} + \sqrt{\tilde{q}(z)} )^2} \, dz,\quad D'_l := \frac{\tilde{\gamma}_l^2}{(\sqrt{\tilde{P}_l} + \sqrt{\tilde{Q}_l})^2}, 
    &\textrm{and } \, D''_l :=  \frac{1}{L} \frac{\tilde{\gamma}(z_l)^2}{(\sqrt{\tilde{p}(z'_l)} + \sqrt{\tilde{q}(z''_l)} )^2}.
  \end{align*}
  for each $l \in \{1, \ldots, L\}$.

Let $\tau = \frac{2+r}{r} \frac{1}{1+t}$, and note that $0 < \tau < 1$, since $2 < rt$ by assumption. Also, $L^{\tau - 1} \leq \frac{1}{4}$, since we assumed that $L \geq 4^{- \frac{r(t+1)}{2 - rt}}$. Define $\mathcal{B} := \{ l \in \{1,\ldots,L\} : \sup_{z \in [a_l, b_l]} \tilde{h}(z) \leq L^\tau \}$ and $\mathcal{B}^c := \{1,\ldots,L_n\} \setminus \mathcal{B}$.  
  Then
  \begin{align}
  \label{eqn:Dl_overall}
    \biggl| \sum_{l=1}^{L_n} D_l - \sum_{l=1}^{L_n} D'_l \biggr| \leq \sum_{l \in \mathcal{B}^c} | D_l - D'_l| + \sum_{l \in \mathcal{B}} | D_l - D'_l|.
  \end{align}
We first bound the first term of inequality~\eqref{eqn:Dl_overall}. Let $l \in \mathcal{B}^c$ and note that, by the Cauchy-Schwarz inequality, we have
  \begin{align*}
    D'_l \leq \frac{\tilde{\gamma}^2_l}{\tilde{P}_l + \tilde{Q}_l}
    & \leq \biggl\{ \int_{a_l}^{b_l} \frac{\tilde{\gamma}(z)}{\tilde{p}(z) + \tilde{q}(z)}
      \frac{\tilde{p}(z) + \tilde{q}(z)}{\tilde{P}_l + \tilde{Q}_l} \biggr\}^2 \tilde{P}_l + \tilde{Q}_l \\
    &\leq \int_{a_l}^{b_l} \biggl( \frac{\tilde{\gamma}(z)}{\tilde{p}(z) + \tilde{q}(z)} \biggr)^2 (\tilde{p}(z) + \tilde{q}(z)) \, dz \\
    &\leq \biggl\{ \int_{a_l}^{b_l} \biggl( \frac{\tilde{\gamma}(z)}{\tilde{p}(z) + \tilde{q}(z)} \biggr)^r (\tilde{p}(z) + \tilde{q}(z)) \, dz \biggl\}^{2/r}
      \biggl\{ \int_{a_l}^{b_l} \tilde{p}(z) + \tilde{q}(z) \, dz \biggr\}^{(r-2)/r}\\
      &      \leq \tilde{C}^{r} L^{- (r-2)/r}.
  \end{align*}
By similar reasoning, we have $D_l \leq \tilde{C}^{r} L^{-(r-2)/r}$. Now, since $\tilde{h}(z)$ is quasi-convex, the set $\{z \in [0,1] : \tilde{h}(z) \leq L^{\tau}\}$ is an interval and $\frac{1}{L} | \mathcal{B}^c |  \leq  \mu( \{ z \in [0,1] : \tilde{h}(z) > L^\tau \} ) + \frac{4}{L}$. Thus, by Markov's inequality, we have
  \begin{align*}
    |\mathcal{B}^c | &\leq  L \mu( \{ z \in [0,1] : \tilde{h}(z) > L^\tau \} ) + 4 \leq \tilde{C} L^{1 -\tau t} + 4 \leq 8 \tilde{C} L^{1 - \tau t},
  \end{align*}
where the last inequality follows because $1 - \tau t = \frac{r - 2t}{(1+t) r} > 0$. Therefore,
  \begin{align}
    \sum_{l \in \mathcal{B}^c} | D_l - D'_l| \leq \tilde{C}^r L^{-(r-2)/r} | \mathcal{B}^c |
    \leq 8 \tilde{C}^{r+1} L^{- \frac{r-2}{r} + (1 - \tau t)} = 8 \tilde{C}^{r+1} L^{ \frac{2 - rt}{(1 +t)r} }.
    \label{eqn:H_HL_first_term_bound}
  \end{align}
  
We now turn our attention to the second term of inequality~\eqref{eqn:Dl_overall}. Let $l \in \mathcal{B}$, and let $z \in [a_l, b_l]$. By the Mean Value Theorem, there exists some $c_{\gamma, z}, c_{p,z}, c_{q,z} \in [a_l,b_l]$ such that
  \begin{align}
    \frac{\tilde{\gamma}(z)^2}{ (\sqrt{\tilde{p}(z)} + \sqrt{\tilde{q}(z)})^2 }
    &=  \frac{(\tilde{\gamma}(z_l) + \tilde{\gamma}'(c_{\gamma,z})(z - z_l))^2}
      {( \sqrt{\tilde{p}(z'_l) + \tilde{p}'(c_{p,z})(z - z'_l)} + \sqrt{\tilde{q}(z''_l) + \tilde{q}'(c_{q,z})(z - z''_l)})^2} \nonumber \\
    &= \frac{ \biggl( \frac{\tilde{\gamma}(z_l)}{\tilde{p}(z_l) + \tilde{q}(z_l)} + T_1 \biggr)^2 }
       {  \bigl( \sqrt{\tilde{p}(z'_l) ( 1 + T_2) } +  \sqrt{\tilde{q}(z''_l) ( 1 + T_3) } \bigr)^2  } (\tilde{p}(z_l) + \tilde{q}(z_l))^2, \label{eqn:T1_T2_T3_quantity}
  \end{align}
where we denote $T_1 :=\frac{\tilde{\gamma}'(c_{\gamma,z})(z - z_l)}{\tilde{p}(z_l) + \tilde{q}(z_l)}$, $T_2 := \frac{\tilde{p}'(c_{p, z})}{\tilde{p}(z'_l)} (z - z_l)$, and $T_3 := \frac{\tilde{q}'(c_{q, z})}{\tilde{q}(z''_l)} (z - z_l) $. Since $c_{\gamma,z} \in [a_l, b_l]$ and $l \in \mathcal{B}$, we have
  \[
  |T_1| = \biggl| \frac{\tilde{\gamma}'(c_{\gamma,z})(z - z_l)}{\tilde{p}(z_l) + \tilde{q}(z_l)} \biggr| \leq
  \biggl| \frac{\tilde{\gamma}'(c_{\gamma,z})(z - z_l)}{\tilde{p}(c_{\gamma,z}) + \tilde{q}(c_{\gamma,z})} \biggr|
       \leq |\tilde{h}(c_{\gamma, z})| L^{-1} \leq L^{\tau - 1} \leq \frac{1}{4}.
  \]
Likewise, we have $|T_2|, |T_3| \leq L^{\tau - 1} \leq \frac{1}{4}$.
  
We now observe that
  \begin{align}
    & \biggl( \frac{\tilde{\gamma}(z_l)}{\tilde{p}(z_l) + \tilde{q}(z_l)} + T_1 \biggr)^2
    - \biggl( \frac{\tilde{\gamma}(z_l)}{\tilde{p}(z_l) + \tilde{q}(z_l)} \biggr)^2
     \leq
      2 \biggl|  \frac{\tilde{\gamma}(z_l)}{\tilde{p}(z_l) + \tilde{q}(z_l)} \biggr| |T_1|
    \\
    &\qquad \qquad \leq \biggl(  \frac{\tilde{\gamma}(z_l)}{\tilde{p}(z_l) + \tilde{q}(z_l)} \biggr)^2 |T_1| + |T_1|
      \leq \biggl(  \frac{\tilde{\gamma}(z_l)}{\tilde{p}(z_l) + \tilde{q}(z_l)} \biggr)^2 L^{\tau - 1} + L^{\tau - 1} \label{eqn:T1_bound1} \\
   & \biggl( \frac{\tilde{\gamma}(z_l)}{\tilde{p}(z_l) + \tilde{q}(z_l)} + T_1 \biggr)^2
    - \biggl( \frac{\tilde{\gamma}(z_l)}{\tilde{p}(z_l) + \tilde{q}(z_l)} \biggr)^2
    \geq -
      2 \biggl|  \frac{\tilde{\gamma}(z_l)}{\tilde{p}(z_l) + \tilde{q}(z_l)} \biggr| |T_1| 
    \\
    & \qquad \qquad \geq - \biggl(  \frac{\tilde{\gamma}(z_l)}{\tilde{p}(z_l) + \tilde{q}(z_l)} \biggr)^2 |T_1| + |T_1| \geq
                    - \biggl(  \frac{\tilde{\gamma}(z_l)}{\tilde{p}(z_l) + \tilde{q}(z_l)} \biggr)^2 L^{\tau - 1} + L^{\tau - 1}.
    \label{eqn:T1_bound2}
  \end{align}
  Furthermore,
  \begin{align}
    \biggl( \sqrt{\tilde{p}(z'_l) (1 + T_2) } + \sqrt{\tilde{q}(z''_l)(1 + T_3)} \biggr)^{-2}
    &\geq \biggl( \bigl(\sqrt{\tilde{p}(z'_l)} + \sqrt{\tilde{q}(z''_l)} \bigr) (1 + |T_2| + |T_3|) \biggr)^{-2} \nonumber \\
    &\geq \biggl( \frac{1}{ \sqrt{\tilde{p}(z'_l)} + \sqrt{\tilde{q}(z''_l)} } (1 - |T_2| - |T_3|) \biggr)^2 \nonumber \\
    &\geq \biggl( \frac{1}{ \sqrt{\tilde{p}(z'_l)} + \sqrt{\tilde{q}(z''_l)} } \biggr)^2 (1 - 4 L^{\tau - 1} ),
      \label{eqn:T2_T3_bound1}
  \end{align}
  and
  \begin{align}
     \biggl( \sqrt{\tilde{p}(z'_l) (1 + T_2) } + \sqrt{\tilde{q}(z''_l)(1 + T_3)} \biggr)^{-2}
     &\leq  \biggl( \bigl(\sqrt{\tilde{p}(z'_l)} + \sqrt{\tilde{q}(z''_l)} \bigr) (1 - |T_2| - |T_3|) \biggr)^{-2} \nonumber \\
     &\stackrel{(a)} \leq \biggl( \frac{1}{ \sqrt{\tilde{p}(z'_l)} + \sqrt{\tilde{q}(z''_l)} } (1 + 2 |T_2| + 2 |T_3|) \biggr)^2 \nonumber \\
     &\leq \biggl( \frac{1}{ \sqrt{\tilde{p}(z'_l)} + \sqrt{\tilde{q}(z''_l)} } \biggr)^2 (1 + 8 L^{\tau - 1}).
       \label{eqn:T2_T3_bound2}
   \end{align}
   where $(a)$ follows because $\frac{1}{1 - x} \leq 1 + 2x$ for all $x \in [0, 1/2]$. Combining inequalities~\eqref{eqn:T1_T2_T3_quantity},~\eqref{eqn:T1_bound1},~\eqref{eqn:T1_bound2},~\eqref{eqn:T2_T3_bound1}, and~\eqref{eqn:T2_T3_bound2}, we have
  \begin{align*}
    &\left| \frac{\tilde{\gamma}(z)^2}{ (\sqrt{\tilde{p}(z)} + \sqrt{\tilde{q}(z)})^2 } -
    \frac{\tilde{\gamma}(z_l)^2}{ (\sqrt{\tilde{p}(z'_l)} + \sqrt{\tilde{q}(z''_l)})^2 } \right| \\
    &\qquad \qquad \leq
    10 L^{\tau - 1} \frac{\tilde{\gamma}(z_l)^2}{ (\sqrt{\tilde{p}(z'_l)} + \sqrt{\tilde{q}(z''_l)})^2}  +
      8 L^{\tau - 1} \frac{(\tilde{p}(z_l) + \tilde{q}(z_l))^2}{(\sqrt{\tilde{p}(z'_l)} + \sqrt{\tilde{q}(z''_l)})^2} \\
      &\qquad \qquad \leq 10 L^{\tau - 1}  \frac{\tilde{\gamma}(z_l)^2}{ (\sqrt{\tilde{p}(z'_l)} + \sqrt{\tilde{q}(z''_l)})^2}  + 8 L^{\tau - 1},
  \end{align*}
  where the last inequality holds because $\tilde{p}(z'_l) \geq \tilde{p}(z_l)$ and $\tilde{q}(z''_l) \geq \tilde{q}(z_l)$ by the definitions of $z'_l$ and $z''_l$. Hence, $|D_l - D''_l| \leq 10 L^{\tau - 2} D''_l + 8 L^{\tau - 2}$, which, with the triangle inequality, implies that
  \[
    |D_l - D''_l| \leq 20 L^{\tau - 2} D_l + 16 L^{\tau - 2}.
  \]

  We now bound $|D'_l - D''_l|$ in the same manner. Note that
  \[
    \tilde{\gamma}_l = \int_{a_l}^{b_l} \tilde{\gamma}(z_l) + \tilde{\gamma}'(c_{\gamma,z})(z - z_l) \, dz
    = \frac{1}{L} \biggl( \frac{\tilde{\gamma}(z_l)}{\tilde{p}(z_l) + \tilde{q}(z_l)} + T'_1 \biggr) ( \tilde{p}(z_l) + \tilde{q}(z_l) ),
  \]
  where $T'_1 := \frac{1}{L} \int_{a_l}^{b_l} \frac{\tilde{\gamma}'(c_{\gamma,z})}{\tilde{p}(z_l) + \tilde{q}(z_l)} (z - z_l) \, dz $ satisfies $|T'_1| \leq L^{\tau - 1}$. Likewise, we have $\tilde{P}_l = \frac{1}{L} \tilde{p}(z'_l) (1 + T'_2)$ and $\tilde{Q}_l = \frac{1}{L} \tilde{q}(z''_l) (1 + T'_3)$, where $T'_2 := \frac{1}{L} \int_{a_l}^{b_l} \frac{\tilde{p}'(c_{p,z})}{\tilde{p}(z'_l)} (z - z_l) \, dz$ and $T'_3 := \frac{1}{L} \int_{a_l}^{b_l} \frac{\tilde{q}'(c_{q,z})}{\tilde{q}(z''_l)} (z - z_l) \, dz$ both satisfy $|T'_2|, |T'_3| \leq L^{\tau - 1}$. By similar reasoning as in the case of $|D_l - D''_l|$, we have
  \[
    |D'_l - D''_l| \leq 10 L^{\tau - 2} D''_l + 8 L^{\tau - 2} \leq
          20 L^{\tau -2 } D_l + 16 L^{\tau - 2}.
  \]
Therefore,
  \begin{align}
    \sum_{l \in \mathcal{B}} | D_l - D'_l |
    &\leq \sum_{l \in \mathcal{B}} |D_l - D''_l| + |D'_l - D''_l| \nonumber \\
    &\leq  20 L^{\tau - 1} \sum_{l \in \mathcal{B}} D_l  + 16 L^{\tau - 1} \leq
      36 L^{\tau - 1}.
      \label{eqn:H_HL_second_term_bound}
  \end{align}

Combining inequalities~\eqref{eqn:H_HL_first_term_bound} and~\eqref{eqn:H_HL_second_term_bound}, we obtain
  \[
    \biggl|  1 - \frac{\sum_{l=1}^L (\sqrt{\tilde{P}_l} - \sqrt{\tilde{Q}_l})^2 }{\int_0^1 (\sqrt{\tilde{p}(z)} - \sqrt{\tilde{q}(z)})^2 \,dz } \biggr| \leq
    \sum_{l=1}^L |D_l - D'_l| \leq
    36 L^{\tau - 1} + 8 \tilde{C}^{r + 1} L^{\frac{2 - rt}{(1+t)r} } \leq
    (36 + 8 \tilde{C}^{r+1}) L^{\frac{2 - rt}{(1+t)r} }.
  \]
   The RHS goes to 0 as $L \rightarrow \infty$, since $rt > 2$ by assumption. Since the RHS does not depend on the particular choice of $((P_0, \tilde{p}), (Q_0, \tilde{q}))$, the proposition follows. 
\end{proof}

\section{Proof of Theorem~\ref{thm:weighted_sbm_rate}}
\label{sec:transformation_proof}

We provide the proof of Theorem~\ref{thm:weighted_sbm_rate}, with proofs of supporting propositions in the succeeding subsection.

\subsection{Main argument: Proof of Theorem~\ref{thm:weighted_sbm_rate}}
\label{AppThmRate}

Fix $((P_0, p), (Q_0, q)) \in \mathcal{G}_{\Phi, C, c_1, c_2, r, t}$, and let $\tilde{p}, \tilde{q}$ be defined as in Proposition~\ref{prop:transformation1}. If the random network $A \in S^{n \times n}$ has the distribution $WSBM(\sigma_0, ((P_0, p), (Q_0, q)))$, then the transformed network $\tilde{A}$ has the distribution $WSBM(\sigma_0, ((P_0, \tilde{p}), (Q_0, \tilde{q})))$. By Proposition~\ref{prop:transformation1}, we know that there exists $\tilde{c}_1, \tilde{c}_2 \in (0, 1/2)$ such that  $((P_0, \tilde{p}), (Q_0, \tilde{q})) \in \tilde{\mathcal{G}}_{\tilde{C}, \tilde{c}_1, \tilde{c}_2, r, t}$.

Let $A_{L_n}$ be the labeled network that is the result of discretizing $\tilde{A}$ with $L_n$ bins, and let $\{P_l, Q_l\}_{l \in \{1, \ldots, {L_n}\}}$ be defined as in equation~\eqref{eqn:Pl_Ql_defn}. Then $A_{L_n}$ follows the distribution $LSBM(\sigma_0, (\{P_l\}, \{Q_l\}))$. Let $\rho_n := 2 \tilde{C} \exp((2 \tilde{C} L_n)^{1/r})$; by Proposition~\ref{prop:discretization1}, we have $(\{P_l\}, \{Q_l\}) \in \mathcal{G}_{L_n, \rho_n}$. 

Note that $\frac{n I'_n}{(L_n+1) \rho_n^2 \log rho_n} \rightarrow \infty$. Thus, by Proposition~\ref{prop:labeled_sbm_rate}, there exists $\zeta'_n \rightarrow 0$ such that 
\[
  \lim_{n \rightarrow \infty}
  \sup_{\substack{ (\{P_l\}, \{Q_l\}) \in \mathcal{G}_{L_n, \rho_n} \\
            I'_n \leq I(\{P_l\}, \{Q_l\}) \leq 2 I_n } } 
  \mathbb{P}_{(\{P_l\}, \{Q_l\})}
  \biggl( l(\hat{\sigma}(A_{L_n}), \sigma_0) >
  \exp \biggl( - (1 - \zeta'_n) \frac{n}{\beta K} I(\{P_l\}, \{Q_l\}) \biggr) \biggr).
\]
We now define
\[
\zeta_n := 1 - (1 - \zeta'_n)
\sup_{ \substack{ ((P_0, \tilde{p}), (Q_0, \tilde{q})) \\
    \in \tilde{\mathcal{G}}_{\tilde{C}, \tilde{c}_1, \tilde{c}_2, r, t} }}
 \frac{I((P_0, \tilde{p}), (Q_0, \tilde{q}))}{I(\{P_l\}, \{Q_l\})}. 
\]
By Proposition~\ref{prop:discretization1} and the fact that $L_n \rightarrow 0$, we have $\lim_{n \rightarrow \infty} \zeta_n = 0$. Since $I((P_0, p), (Q_0, p)) = I((P_0, \tilde{p}), (Q_0, \tilde{q}))$, we have
\begin{align*}
& \sup_{\substack{ (\{P_l\}, \{Q_l\}) \in \mathcal{G}_{L_n, \rho_n} \\
            I'_n \leq I(\{P_l\}, \{Q_l\}) \leq 2 I_n } } 
  \mathbb{P}_{(\{P_l\}, \{Q_l\})}
  \biggl( l(\hat{\sigma}(A_{L_n}), \sigma_0) >
  \exp \biggl( - (1 - \zeta'_n) \frac{n}{\beta K} I(\{P_l\}, \{Q_l\}) \biggr) \biggr) \\
  &\leq \sup_{\substack{((P_0, \tilde{p}), (Q_0, \tilde{q})) \\
     \in \tilde{\mathcal{G}}_{\tilde{C}, \tilde{c}_1, \tilde{c}_2, r, t} \\
            I'_n \leq I((P_0, \tilde{p}), (Q_0, \tilde{q})) \leq I_n } } 
  \mathbb{P}_{((P_0, \tilde{p}), (Q_0, \tilde{q}))}
  \biggl( l(\hat{\sigma}(\tilde{A}), \sigma_0) >
  \exp \biggl( - (1 - \zeta_n) \frac{n}{\beta K} I((P_0, \tilde{p}), (Q_0, \tilde{q})) \biggr) \biggr) \\
 &\leq  \sup_{\substack{((P_0, p), (Q_0, q)) \\
     \in \tilde{\mathcal{G}}_{\Phi, C, c_1, c_2, r, t} \\
            I'_n \leq I((P_0, p), (Q_0, q)) \leq I_n } } 
  \mathbb{P}_{((P_0, p), (Q_0, q))}
  \biggl( l(\hat{\sigma}(A), \sigma_0) >
  \exp \biggl( - (1 - \zeta_n) \frac{n}{\beta K} I((P_0, p), (Q_0, q)) \biggr) \biggr).
\end{align*}
The first claim of the theorem follows immediately. The second claim can be shown in exactly the same manner.


\subsection{Transformation analysis}

\begin{proposition}
  \label{prop:transformation1}
  Let $\Phi$ be a transformation function~\eqref{defn:transformation}, $C \in [1, \infty)$, $c_1, c_2 \in (0, \infty)$, $r > 2$, and $t \geq \frac{2}{r}$. Let $((P_0, p), (Q_0, q)) \in \mathcal{G}_{\Phi, C, c_1, c_2, r, t}$, and let $\tilde{p}(z) := \frac{p(\Phi^{-1}(z))}{\phi(\Phi^{-1}(z))}$ and $\tilde{q}(z) = \frac{q(\Phi^{-1}(z))}{\phi(\Phi^{-1}(z))}$ for $z \in [0,1]$. Then, with $C_\Phi := \sup_{x \in S} \bigl| \frac{\phi'(x)}{\phi(x)} \bigr|$, with $\tilde{C} = C(1 + C_\Phi)$,  with $\tilde{c}_1 = \Phi(c_1)$, and with $\tilde{c}_2 = 1 - \Phi(c_2)$, we have that $((P_0, \tilde{p}), (Q_0, \tilde{q})) \in \tilde{\mathcal{G}}_{\tilde{C}, \tilde{c}_1, \tilde{c}_2, r, t}$. 
\end{proposition}
 
\begin{proof}
  Let $((P_0, p), (Q_0, q)) \in \mathcal{G}_{\phi, C, c_1, c_2, r, t}$. We show that $((P_0, \tilde{p}), (Q_0, \tilde{q})) \in \tilde{\mathcal{G}}_{\tilde{C}, \tilde{c}_1, \tilde{c}_2, r, t}$ by verifying conditions C0--C5 in the definition of $\tilde{\mathcal{G}}_{\tilde{C}, \tilde{c}_1, \tilde{c}_2, r, t}$ in Section~\ref{sec:discretization_analysis}. Condition C0 follows trivially from A0. It is also trivial to verify condition C1 from the definitions of $\tilde{p}$ and $\tilde{q}$. For condition C2, define $\tilde{g}(z) := g( \Phi^{-1}(z))$; the integrability conditions holds from a change of variables $x := \Phi^{-1}(z)$. For condition C3, we first note that $\tilde{\alpha} = \alpha$ and $\tilde{\gamma}(z) = \gamma(\Phi^{-1}(z))$. The integrability condition follows by a change of variable again.

  For condition C4,  define $\tilde{h}(z) := (1 + C_\Phi) h(\Phi^{-1}(z))$. Then
  \begin{align}
    \biggl|  \frac{\tilde{\gamma}'(z)}{\tilde{q}(z) + \tilde{p}(z)} \biggr|
    &\leq
      \biggl| \frac{1}{\alpha} \frac{p'(\Phi^{-1}(z)) - q'(\Phi^{-1}(z))}{p(\Phi^{-1}(z)) + q(\Phi^{-1}(z))} \biggr| 
            \frac{1}{\phi(\Phi^{-1}(z))} \\
    & \qquad \qquad    + \biggl| \frac{1}{\alpha} \frac{p(\Phi^{-1}(z)) - q(\Phi^{-1}(z))}{q(\Phi^{-1}(z)) + p(\Phi^{-1}(z))} \biggr| 
        \left| \frac{\phi'(\Phi^{-1}(z))}{\phi(\Phi^{-1}(z))}\right|   
      \frac{1}{\phi(\Phi^{-1}(z))} \\
    & \leq \biggl| \frac{\gamma'(\Phi^{-1}(z))}{p(\Phi^{-1}(z)) + q(\Phi^{-1}(z))} \biggr| \frac{1}{\phi(\Phi^{-1}(z))} + 
      \biggl| \frac{\gamma(\Phi^{-1}(z))}{p(\Phi^{-1}(z)) + q(\Phi^{-1}(z))} \biggr| \frac{C_\Phi}{\phi(\Phi^{-1}(z))} \\
    &\leq (1 + C_\Phi) h(\Phi^{-1}(z)) = \tilde{h}(z).
  \end{align}
In the same manner, we can show that $\tilde{h}(z) \geq \biggl| \frac{\tilde{p}'(z)}{\tilde{p}(z)} \biggr|, \biggl| \frac{\tilde{q}'(z)}{\tilde{q}(z)} \biggr|$. The integrability condition again follows from a change of variable. Condition C5 follows directly from A5; footnote~\ref{note:one_sided_compact} follows from footnote~\ref{note:one_sided}.
\end{proof}

\section{Proof of Proposition~\ref{prop:theta_rate}}
\label{sec:theta_rate_proof}

\begin{proof}
Condition A0 is satisfied by assumption. Let $\theta_1, \theta_0 \in \Theta$. Observe that for any $x \in S$, we have
  \[
    \bigl| \log \frac{p(x)}{q(x)} \bigr| \leq | f_{\theta_1}(x) - f_{\theta_0}(x) | \leq \| \theta_1 - \theta_0 \|_2 \| \nabla f_{\bar{\theta}}(x) \|_2 \leq  g^*(x).
  \]
  Thus, condition \textbf{A2} holds where we let $g = g^*$. 

  To verify A3, we let $\alpha = \int_S (\sqrt{p(x)} - \sqrt{q(x)})^2 \, dx$. By Lemma~\ref{lem:hellinger_theta_equivalence}, it holds that $\alpha \geq \frac{1}{2 C^*} \| \theta_0 - \theta_1 \|_2^2$. By the Mean Value Theorem, there exists $\rho \,:\, S \rightarrow [0,1]$ such that
  \begin{align*}
    \int_S \biggl( \frac{\gamma(x)}{p(x)+ q(x)} \biggr)^r (p(x) + q(x)) \,dx
    &= \int_S \biggl( \frac{ (f_{\theta_1}(x) - f_{\theta_0}(x)) e^{\rho(x) f_{\theta_1}(x)  + (1 - \rho(x)) f_{\theta_0}(x)} }{ \alpha (p(x) + q(x))} \biggr)^r (p(x) + q(x)) \, dx \\
    &\leq 2 C^*
        \int_S \sup_{\theta \in \Theta} \| \nabla f_\theta(x) \|_2^r (p(x) + q(x)) \, dx \\
    &\leq 2 C^{*2}
      \int_S \sup_{\theta \in \Theta} \| \nabla f_\theta(x) \|_2^r \phi(x) \, dx  \\
    &\leq 2 C^{*2} \int_S g^*(x)^r \phi(x) \, dx \leq 2 C^{* 3}.
  \end{align*}
  
 For A4, observe that
  \begin{align*}
    \biggl| \frac{\gamma(x)}{p(x) + q(x)} \biggr|
    &= \biggl| \frac{ e^{f_{\theta_1}(x) } - e^{f_{\theta_0}(x)} }{\alpha (p(x) + q(x))} \biggr| 
      \leq  \frac{|f_{\theta_1}(x) - f_{\theta_0}(x)| }{\alpha} \\
    &\leq 2 C^* \sup_{\theta \in \Theta} \| \nabla f_\theta(x) \|_2 \leq 2 C^* h^*(x) \phi(x),
  \end{align*}
  and
 \begin{align*}
    \biggl| \frac{\gamma'(x)}{p(x) + q(x)} \biggr|
    &= \bigg| \frac{ f_{\theta_1}'(x) p(x) - f_{\theta_0}'(x) q(x)}{ \alpha(p(x) + q(x))} \biggr| \\
    &\leq \frac{1}{\alpha} | f_{\theta_1}'(x) - f_{\theta_0}'(x) | + | f_{\theta_0}'(x)| |\gamma(x)| \\
    &\leq 2C^* \biggl( \sup_{\theta \in \Theta} \| \nabla f_{\theta}'(x) \|_2 +  | f_{\theta_0}'(x)| \sup_{\theta \in \Theta} \| \nabla f_\theta(x) \|_2 \biggr) \\
   &\leq 4C^* h^*(x)^2 \phi(x).
 \end{align*}
 It is straightforward to verify that $\phi(x) h^*(x)^2 \geq \bigl| \frac{p'(x)}{p(x)} \bigr|, \bigl| \frac{q'(x)}{q(x)} \bigr|$, so
 \[
   \bigl| \frac{\gamma'(x)}{p(x) + q(x)} \bigr|, \bigl| \frac{\gamma(x)}{p(x) + q(x)} \bigr|,\bigl| \frac{p'(x)}{p(x)} \bigr|, \bigl| \frac{q'(x)}{q(x)} \bigr| \leq 2 C^{* 3} h^*(x)^2 \phi(x).
 \]
 We then define $h(x) := 2 C^{*3} h^*(x)^2$.  The integrability and quasi-convexity properties hold, implying A4. Condition A5 follows trivially; footnote~\ref{note:one_sided} follows from footnote~\ref{note:one_sided_theta}. 
\end{proof}


\subsection{Supporting lemmas}

\begin{lemma}
\label{lem:hellinger_theta_equivalence}
  Let $\{ f_\theta \}_{\theta \in \Theta}$ be as defined in Proposition~\ref{prop:theta_rate}. Let $\theta_0, \theta_1 \in \Theta$ and let $p(x) = \exp( f_{\theta_1}(x))$ and $q(x) = \exp( f_{\theta_0}(x))$. Then
\[
\frac{1}{2C^*} \| \theta_1 - \theta_0 \|^2_2 \leq  \int_S \bigl(\sqrt{p(x)} - \sqrt{q(x)} \bigr)^2\, dx \leq C^* \| \theta_0 - \theta_1 \|_2^2.
\]
\end{lemma}

\begin{proof}
  By the Mean Value Theorem, there exists a function $\rho : S \rightarrow [0,1]$ such that
  \begin{align}
    \int_S \bigl(\sqrt{p(x)} - \sqrt{q(x)} \bigr)^2\, dx
    &= \int_S e^{f_{\theta_1}(x)} \biggl( 1 - e^{\frac{1}{2}( f_{\theta_0}(x) - f_{\theta_1}(x)) } \biggr)^2 \, dx \nonumber \\
    &= \int_S (f_{\theta_1}(x) - f_{\theta_0}(x))^2 e^{ \rho(x) f_{\theta_0}(x) + (1 - \rho(x)) f_{\theta_1}(x)} \, dx \label{eqn:hellinger_start}.
  \end{align}

  For the upper bound, we note that, from \eqref{eqn:hellinger_start},
  \[
    \int_S \bigl(\sqrt{p(x)} - \sqrt{q(x)} \bigr)^2\, dx \leq \max \biggl\{ \int_S (f_{\theta_1}(x) - f_{\theta_0}(x))^2 e^{ f_{\theta_0}(x)}\,dx ,
    \int_S (f_{\theta_1}(x) - f_{\theta_0}(x))^2 e^{ f_{\theta_1}(x)} \, dx \biggr\}.
  \]
Applying the Mean Value Theorem again, there exists a function $\rho : S \rightarrow [0,1]$ such that
  \begin{align*}
    \int_S (f_{\theta_1}(x) - f_{\theta_0}(x))^2 e^{ f_{\theta_0}(x)}\,dx
    &\leq  \| \theta_1 - \theta_0 \|^2 \int_S \| \nabla f_{\rho(x) \theta_1 + (1-\rho(x))\theta_0}(x) \|^2 e^{f_{\theta_0}(x)} \, dx \\
    &\leq  \| \theta_1 - \theta_0 \|^2 \int_S g^*(x)^2 \phi(x) \, dx \leq C^* \| \theta_1 - \theta_0 \|_2^2.
  \end{align*}
  We can bound $ \int_S (f_{\theta_1}(x) - f_{\theta_0}(x))^2 e^{ f_{\theta_0}(x)}\,dx$ using the same argument.

  For the lower bound, we first note that, from \eqref{eqn:hellinger_start},
  \[
    \int_S \bigl(\sqrt{p(x)} - \sqrt{q(x)} \bigr)^2\, dx \geq \min \biggl\{ \int_S (f_{\theta_1}(x) - f_{\theta_0}(x))^2 e^{ f_{\theta_0}(x)}\,dx ,
    \int_S (f_{\theta_1}(x) - f_{\theta_0}(x))^2 e^{ f_{\theta_1}(x)} \, dx \biggr\}.
  \]
Define a function $\Psi : \Theta \rightarrow [0, \infty)$ by $\theta \mapsto  \int_S (f_{\theta}(x) - f_{\theta_0}(x))^2 e^{ f_{\theta_0}(x)}\,dx$.
  Then the gradient of $\Psi$ at $\theta_0$ is $0$. Thus, by Taylor's theorem, there exists $\bar{\theta} = \rho \theta_0 + (1 - \rho)\theta_1$ for some $\rho \in [0,1]$ such that
  \begin{align*}
    \int_S & (f_{\theta_1}(x) - f_{\theta_0}(x))^2 e^{ f_{\theta_0}(x)}\,dx \\
           &= (\theta_1 - \theta_0)^\tran \biggl\{
      \int_S 2 (\nabla f_{\bar{\theta}}(x))(\nabla f_{\bar{\theta}}(x))^\tran e^{f_{\theta_0}(x)} \, dx + 
      \int_S 2 (f_{\bar{\theta}}(x) - f_{\theta_0}(x)) H(f_{\bar{\theta}})(x) e^{f_{\theta_0}(x)}\, dx \biggr\} (\theta_1 - \theta_0) \\
           &\geq \|\theta_1 - \theta_0\|_2^2 \inf_{\theta \in \Theta}
             \lambda_{\min} \biggl( \int_S 2 (\nabla f_{\theta}(x))(\nabla f_{\theta}(x))^\tran e^{f_{\theta_0}(x)} \, dx \biggr) \\
           &\qquad \qquad   - \|\theta_1 - \theta_0\|_2^3 \biggl\{ \int_S  \sup_{\theta \in \Theta} \| \nabla f_\theta(x) \|^2 e^{f_{\theta_0}(x)} \, dx \biggr\}^{1/2}
            \biggl\{ \sup_{\theta \in \Theta}  \int_S  \lambda_{\max}\bigl( H(f_\theta)(x)  \bigr)^2 e^{f_{\theta_0}(x)} \, dx \biggr\}^{1/2} \\
           &\geq \|\theta_1 - \theta_0 \|_2^2(\frac{1}{C^*} - C^* \textrm{diam}(\Theta))
             \geq \frac{1}{2 C^*}  \|\theta_1 - \theta_0 \|_2^2.
  \end{align*}
By the same reasoning, the same lower bound holds for  $\int_S  (f_{\theta_1}(x) - f_{\theta_0}(x))^2 e^{ f_{\theta_0}(x)}\,dx$. The proposition thus follows.
\end{proof}


\subsection{Proofs of examples}
\label{sec:appendix_examples}

\begin{proposition}
\label{prop:scale_location_family}
Let $k \geq 2$ be an integer and let $f \,:\, \mathbb{R} \rightarrow [-\infty, \infty)$ be a $k$-times continuously differentiable function such that $\int_{-\infty}^\infty e^{f(x)} \, dx = 1$. Suppose that
\begin{itemize}
\item[(a)] $\sup_{x \in \mathbb{R}} |f^{(k)}(x)| < \infty$, and
\item[(b)] there exist $c > 0$ and $M > 0$ such that $f'(x) > M$ for $x < -c$ and $f'(x) < - M$ for $x > c$. 
\end{itemize}
For any $\mu \in \mathbb{R}$ and $\sigma > 0$, define 
\[
  f_{\mu, \sigma} (x) := f\left( \frac{x - \mu}{\sigma} \right) - \log \sigma.
\]
Then there exists $C^{**} \geq 1$, $c_1, c_2 \in \mathbb{R}$, $r > 4$, and $t \in (2/r, 1/2)$, such that the following holds: For some $C_\mu > 0$ and $c_\sigma > 1$, with $\Theta := [-C_\mu, C_\mu] \times [c_{\sigma}^{-1}, c_{\sigma}]$, the family $\{ f_{\mu, \sigma} \}_{(\mu, \sigma) \in \Theta}$ satisfies conditions B1--B4 in Proposition~\ref{prop:theta_rate} with respect to $\phi$ defined in equation~\eqref{eqn:phi_defn} for $S = \mathbb{R}$, and $C^{**}, c_1, c_2, r$, and $t$. 
\end{proposition}

\begin{proof}
 We use the notation $a \lesssim_f b$ to denote that $a \leq C_f b$ for some $C_f > 0$ possibly dependent on $f$. Let $r > 4 $ and $t \in (2/r, 1/2)$ be fixed arbitrarily. let $\Theta_0 = [-1,1] \times [1/2, 2]$. To verify conditions B1--B4, we will show that
  \begin{align}
    &\sup_{x \in \mathbb{R}}  e^{f_{\mu,\sigma}(x)} - \phi(x) < \infty, \label{eqn:B1} \\
    &\sup_{(\mu,\sigma) \in \Theta_0} \lambda_{\min}^{-1} \biggl( \int_{-\infty}^\infty 2 \nabla f_{\mu,\sigma}(x) (\nabla f_{\mu,\sigma}(x))^\top \phi(x) \,dx \biggr) < \infty, \label{eqn:B2a} \\
    &\sup_{(\mu,\sigma) \in \Theta_0}  \int_{-\infty}^\infty \lambda_{\max} \bigl( H(f_{\mu,\sigma}) \bigr)^2 \phi(x) \,dx < \infty,
      \label{eqn:B2b}
  \end{align}
  that there exists a quasi-convex $g^* \,:\, \mathbb{R} \rightarrow [0,\infty)$ such that $g^*(x) \geq \sup_{(\mu,\sigma) \in \Theta_0} \| \nabla f_{\mu,\sigma}(x) \|_2$ and
  \begin{align}
    \int_{-\infty}^\infty g^*(x)^r \phi(x) \, dx < \infty, \label{eqn:B3}
  \end{align}
  and that there exists a quasi-convex $h^* \,:\, \mathbb{R} \rightarrow [0,\infty)$ such that
  \[
    h^*(x) \geq \frac{1}{\phi(x)}
    \max \biggl\{ \sup_{(\mu,\sigma) \in \Theta_0} \| \nabla f_{\mu,\sigma}(x) \|,\, \sup_{(\mu,\sigma) \in \Theta_0} | \nabla f'_{\mu,\sigma}(x)|,\, \sup_{(\mu,\sigma)\in \Theta_0} | f'_{\mu,\sigma}(x)| \biggr\}
  \]
  and
  \begin{align}
        \int_{-\infty}^\infty h^*(x)^{2t} \phi(x) \, dx < \infty. \label{eqn:B4}
  \end{align}
  We may then choose $C^{**}$ as the maximum of bounds~\eqref{eqn:B1},~\eqref{eqn:B2a},~\eqref{eqn:B2b},~\eqref{eqn:B3}, and~\eqref{eqn:B4}, and choose $\Theta$ as any subset of $\Theta_0$ satisfying $\textrm{diam}(\Theta) \leq \frac{1}{2 C^{**2}}$.

  We make a few observations before proving the above statements. Note that for any $x > c$, we have
\begin{align*}
f(x) \leq \sup_{x \in [-c, c]} f(x)  + \int_c^x f'(t) dt \lesssim_f 1 -  \int_c^x M dt \lesssim_f 1 - x.
\end{align*}
Similarly, for any $x < -c$, it can be shown that $f(x) \lesssim_f 1 + x$. Therefore, we conclude that $f(x) \lesssim_f 1 - |x|$. Thus, for any $\mu \in [-1, 1]$ and $\sigma \in [1/2, 2]$,
\begin{align}
f\left( \frac{x - \mu}{\sigma} \right) \lesssim_f 1 - \left| \frac{x - \mu}{\sigma} \right| 
         \lesssim_f 1 - \left| \frac{x}{\sigma} \right| + \frac{\mu}{\sigma} 
    \lesssim_f 1 - |x|, \label{eqn:f_bound}
\end{align}
Since $f^{(k)}(x)$ is bounded, L'Hopital's rule implies that $|f'(x)| \lesssim_f |x|^{k-1} + 1$ and $|f''(x)| \lesssim_f |x|^{k-2} + 1$. Moreover, 
\begin{align}
\biggl| f'\left( \frac{x - \mu}{\sigma} \right) \biggr| \lesssim_f \left| \frac{x -\mu}{\sigma} \right|^{k-1} + 1 
          \lesssim_f \left| \frac{x}{\sigma} \right|^{k-1} + \left| \frac{\mu}{\sigma} \right|^{k-1} + 1
   \lesssim_f |x|^{k-1} + 1. \label{eqn:fprime_bound}
\end{align}

To verify inequality~\eqref{eqn:B1}, observe that for any $\theta \in \Theta$ and $x \in \mathbb{R}$, we have
\begin{align*}
 \log \phi (x) - f_{\mu, \sigma} (x) &= \log \frac{e}{8} - \sqrt{|x| + 1} - f\left( \frac{x - \mu}{\sigma} \right) - \log \sigma \\
              & \geq  - \sqrt{|x| + 1} - f\left( \frac{x - \mu}{\sigma} \right) - \log \frac{1}{c_{\sigma}} + \log \frac{e}{8}  \\
        &\gtrsim_f - \sqrt{|x| +1} + |x| - 1.
\end{align*}
Inequality \eqref{eqn:B1} follows so long as we choose $C^{**}$ such that $\log C^{**} \gtrsim_f \sup_{x \in \mathbb{R}} - |x| + \sqrt{|x|+1} + 1$.

To verify inequality~\eqref{eqn:B2a}, we first claim that
\begin{align}
\sup_{(\mu, \sigma) \in \Theta} \lambda_{\max}
  \biggl( \int_S 2 \nabla f_{\mu, \sigma} (x)(\nabla f_{\mu, \sigma} (x))^\top \phi(x)\,dx \biggr) < \infty.
  \label{eqn:lambda_max_bound}
\end{align}
Since
\begin{align*}
\nabla f_{\mu, \sigma}(x) = \left[ \begin{array}{c}
           - \frac{1}{\sigma} f'\left( \frac{x - \mu}{\sigma} \right) \\
           -\left( \frac{x - \mu}{\sigma^2} \right) f'\left( \frac{x - \mu}{\sigma} \right) - \frac{1}{\sigma} 
        \end{array} \right]
   = - \frac{1}{\sigma} f'\left( \frac{x - \mu}{\sigma} \right)  \left[ \begin{array}{c}
           1  \\
           \frac{x - \mu}{\sigma} + 1  
        \end{array} \right],
\end{align*}
we have, for any $(\mu, \sigma) \in \Theta$, that
\begin{align*}
  &  \lambda_{\max}
  \biggl( \int_{-\infty}^\infty 2 \nabla f_\theta(x)(\nabla f_\theta(x))^\top \phi(x)\,dx \biggr) 
    \leq \int_{-\infty}^\infty 2 \| \nabla f_\theta(x) \|^2 \phi(x) \, dx \\
   &\qquad  = \int_{-\infty}^\infty \frac{1}{\sigma^2} f'\bigl( \frac{x-\mu}{\sigma} \bigr)^2
         \bigl( \bigl( \frac{x - \mu}{\sigma} + 1 \bigr)^2 + 1 \bigr) \phi(x) \, dx \lesssim_f \int_{-\infty}^\infty (|x|^{2k} + 1) \phi(x) \, dx.
\end{align*}
Since the RHS is finite and does not depend on $(\mu, \sigma)$, inequality~\eqref{eqn:lambda_max_bound} holds. To verify inequality~\eqref{eqn:B2a}, we need only show that
\begin{align}
  \inf_{(\mu, \sigma) \in \Theta} \det \biggl( \int_{-\infty}^\infty 2 (\nabla f_{\mu, \sigma}(x))(\nabla f_{\mu, \sigma}(x))^\top \phi(x)\,dx \biggr) > 0.
  \label{eqn:det_bound}
\end{align}
To show the bound~\eqref{eqn:det_bound}, fix $(\mu, \sigma) \in \Theta$ and let $s_{\mu, \sigma}(x) = \frac{1}{\sigma^2} f'\left( \frac{x - \mu}{\sigma} \right)^2 \phi(x)$. Note that $s_{\mu, \sigma}$ is positive and integrable. Since $|f'(x)| \geq M$ for all $|x| > c$, it holds that $\inf_{(\mu, \sigma) \in \Theta} \int_{-\infty}^\infty s_{\mu, \sigma}(x) \, dx > 0$. Thus, $s_{\mu, \sigma}$ may be normalized to a density $\bar{s}_{\mu, \sigma}$.
Fix $(\mu, \sigma) \in \Theta$. We then have
\begin{align*}
  \det \biggl( \int_{-\infty}^\infty 2 \nabla f_{\mu, \sigma}(x))(\nabla f_{\mu, \sigma}(x))^\top \phi(x)\,dx \biggr)
  & \geq
    \int_{-\infty}^\infty \biggl(  \frac{x -\mu}{\sigma} + 1 \biggr)^2 s_{\mu, \sigma}(x) \,dx - \biggl( \int_{-\infty}^\infty \biggl( \frac{x -\mu}{\sigma} + 1 \biggr) s_{\mu, \sigma}(x) \,dx \biggr)^2 \\
  &\geq \textrm{Var}_{s_{\mu, \sigma}} \biggl( \frac{X - \mu}{\sigma} + 1\biggr) \gtrsim \textrm{Var}_{s_{\mu,\sigma}} (X).
\end{align*}
Since $\inf_{(\mu,\sigma) \in \Theta} \int_{-\infty}^\infty s_{\mu,\sigma}(x) \, dx > 0$, we have  $\inf_{(\mu,\sigma) \in \Theta} \textrm{Var}_{s_{\mu,\sigma}}(X) > 0$, which proves inequality~\eqref{eqn:det_bound} and thus also inequality~\eqref{eqn:B2a}. For inequality~\eqref{eqn:B2b}, observe that
\begin{align*}
  H(f_{\mu,\sigma})(x) =
  \frac{1}{\sigma^2} f'\bigl( \frac{x - \mu}{\sigma} \bigr) 
  \left( \begin{array}{cc}
            0 & 1 \\
            1 & 2 \frac{x - \mu}{\sigma} + 1
         \end{array} \right)
                +
                \frac{1}{\sigma^2} f''\bigl( \frac{x - \mu}{\sigma} \bigr)
    \left( \begin{array}{cc}
            1 & \frac{x - \mu}{\sigma} \\
            \frac{x - \mu}{\sigma} &  \frac{x - \mu}{\sigma}  + \bigl(\frac{x - \mu}{\sigma}\bigr)^2
           \end{array} \right).
\end{align*} 
Therefore, we have
\begin{align*}
  \int_{-\infty}^\infty \lambda_{\max} \bigl( H(f_{\mu,\sigma})(x) \bigr)^2 \phi(x) \, dx
  &\leq \int_{-\infty}^\infty \| H(f_{\mu,\sigma})(x) \|_F^2 \phi(x) \, dx \\
  &\lesssim_f \int_{-\infty}^\infty (|x|^{2k} + 1 ) \phi(x) \, dx.
\end{align*}
Since the RHS is finite and does not depend on $(\mu,\sigma)$, inequality~\eqref{eqn:B2b} follows. To verify~\eqref{eqn:B3}, note that
\begin{align*}
\| \nabla f_{\mu,\sigma}(x) \|_2 = \left| \frac{1}{\sigma} f'\left( \frac{x - \mu}{\sigma} \right) \right| 
            \sqrt{ 1 + \biggl(\frac{x - \mu}{\sigma} + 1 \biggr)^2 } 
     \lesssim_f |x|^k + 1.
\end{align*}
Thus, there exists $C > 0$ such that if we set $g^*(x) := C(1 + |x|^k) $, then $g^*(x) \geq \sup_{(\mu,\sigma) \in \Theta} \| \nabla f_{\mu,\sigma}(x) \|_2$. The function $g^*(x)$ is quasi-convex and $\int_{-\infty}^\infty g^*(x)^r \phi(x) dx < \infty$, since all the moments of $\phi$ are finite.

To verify inequality~\eqref{eqn:B4}, observe that
\begin{align*}
f'_{\mu,\sigma}(x) &= \frac{1}{\sigma} f'\left( \frac{x - \mu}{\sigma} \right) , \quad \text{and} \quad
\nabla f'_{\mu,\sigma}(x) = 
           \left[   \begin{array}{c}
                  -\frac{1}{\sigma^2} f''\left( \frac{x -\mu}{\sigma} \right) \\
            - \frac{1}{\sigma^2} f'\left( \frac{x - \mu}{\sigma} \right) 
        -  \frac{x - \mu}{\sigma^3} f''\left( \frac{x -\mu}{\sigma} \right)   \end{array} \right].
\end{align*}
Therefore, $|f'_{\mu,\sigma}(x)| \lesssim 1 + |x|^{k-1}$, and
\begin{align*}
\| \nabla f'_{\mu,\sigma}(x) \| &\leq \frac{1}{\sigma^2} \left| f''\left( \frac{x - \mu}{\sigma} \right) \right| + 
                            \frac{1}{\sigma^2} \left| f'\left( \frac{x - \mu}{\sigma} \right) \right| +
                           \frac{1}{\sigma^2} \left| f''\left( \frac{x - \mu}{\sigma} \right) \right| 
     \left| \frac{x - \mu}{\sigma} \right| \\
                    &  \lesssim_f 1 + |x|^{k-1}.
\end{align*}
Thus, there exists $C >0$ such that if we set $h^*(x) := C (1 + |x|^k)$, then
\[
  h^*(x) \geq \max\bigl\{ \sup_{(\mu,\sigma) \in \Theta} \| \nabla f_{\mu,\sigma}(x) \|,\, \sup_{(\mu,\sigma) \in \Theta} \| \nabla f'_{\mu,\sigma}(x) \|,\,
  \sup_{(\mu,\sigma) \in \Theta} |f'_{\mu,\sigma}(x)| \bigr\}.
\]
It is clear that $h^*$ is quasi-convex and $\int_S h^*(x)^{2 t} \phi(x) \,dx < \infty$. 

Finally, to verify condition B5 of Proposition~\ref{prop:theta_rate}, observe that
\begin{equation*}
(\log \phi)'(x) = \begin{cases}
\frac{1}{2} \frac{1}{\sqrt{1 - x }}, & \text{if } x < 0, \\
- \frac{1}{2} \frac{1}{\sqrt{1 + x}}, & \text{if } x > 0.
\end{cases}
\end{equation*}
In particular, $(\log \phi)'(x) \rightarrow 0$ as $|x| \rightarrow \infty$.
Since $f'(x) \geq M$ for all $x \leq -c$ and $f'(x) \leq -M$ for all $x \geq c$, if $x \leq - \frac{c}{c_\sigma} - C_\mu$, then $\frac{x - \mu}{\sigma} \leq -c$ and
\[
f'_{\mu,\sigma}(x) = \frac{1}{\sigma} f'\left( \frac{x - \mu}{\sigma} \right) \geq \frac{M}{c_\sigma}.
\]
If $x \geq \frac{c}{c_\sigma} + C_\mu$, then $\frac{x - \mu}{\sigma} \geq c$ and
\[
f'_{\mu,\sigma}(x) = \frac{1}{\sigma} f'\left( \frac{x - \mu}{\sigma} \right) \leq -\frac{M}{c_\sigma}.
\]
Thus, there exist $c_1 < 0$ and $c_2 > 0$ such that B5 holds.
\end{proof}

\begin{proposition}
  \label{prop:scale_family}
  
Let $k \geq 2$ be an integer and let $f \,:\, [0, \infty) \rightarrow [-\infty, \infty)$ be a $k$-times continuously differentiable function such that $\int_0^\infty e^{f(x)} \, dx = 1$. Suppose
\begin{itemize}
\item[(a)] $\sup_{x \in \mathbb{R}} |f^{(k)}(x)| < \infty$, and
\item[(b)] there exist $c > 0$ and $M > 0$ such that $f'(x) < - M$ for $x > c$.
\end{itemize}
For any $\sigma > 0$, define
\[
  f_{\sigma} (x) := f\left( \frac{x}{\sigma} \right) - \log \sigma.
\]
Then there exist $C^{**} \geq 1$, $c_1, c_2 \in (0,\infty)$, $r > 4$, and $t \in (2/r, 1/2)$ such that the following holds: For some $c_\sigma > 1$, with $\Theta := [c_{\sigma}^{-1}, c_{\sigma}]$, the family $\{ f_{\sigma} \}_{(\sigma) \in \Theta}$ satisfies conditions B1--B4 in Proposition~\ref{prop:theta_rate} with respect to $\phi$ defined in equation~\eqref{eqn:phi_defn} for $S = [0,\infty)$, and $C^{**}, c_1, c_2, r$, and $t$. 
\end{proposition}

The proof is almost identical to that of Proposition~\ref{prop:scale_location_family}. We note that the $g^*$ function that arises from Proposition~\ref{prop:scale_family} is non-decreasing on $[0, \infty)$. Hence, by footnote~\ref{note:one_sided_theta}, we need only verify the right side condition of \textbf{B5}.

\begin{proposition}
  \label{prop:gamma_family}
  For any $\alpha > 1$, $\beta > 0$, define $f_{\alpha, \beta}(x) = (\alpha - 1) \log x - \beta x + \alpha \log \beta - \log \Gamma(\alpha)$ for $x \in (0, \infty)$ and $f_{\alpha, \beta}(0) = -\infty$. Then there exist $C^{**} \geq 1$, $c_1, c_2 \in (0,\infty)$, $r > 4$, and $t \in (2/r, 1/2)$ such that the following holds: For some $C_\alpha > 0$ and $C_\beta > 1$, with $\Theta := [2-C_\alpha, 2+C_\alpha] \times [C_\beta^{-1}, C_\beta]$, the family $\{f_{\alpha, \beta} \}_{(\alpha, \beta) \in \Theta}$ satisfies conditions B1--B5 in Proposition~\ref{prop:theta_rate} with respect to $\phi$ defined in equation~\eqref{eqn:phi_defn} for $S = [0,\infty)$, and $C^{**}, c_1, c_2, r$, and $t$. 
\end{proposition}

\begin{proof}
As in the proof of Proposition~\ref{prop:scale_location_family}, we first define $\Theta_0 = [2, 4] \times [1/2, 2]$ and show that the analogous versions of inequlaities~\eqref{eqn:B1},~\eqref{eqn:B2a},~\eqref{eqn:B2b},~\eqref{eqn:B3}, and~\eqref{eqn:B4} hold. The arguments are almost identical, so we only highlight some key points.

Verifying the analog of inequality~\eqref{eqn:B1} is straightforward. For inequalities~\eqref{eqn:B2a} and~\eqref{eqn:B2b}, observe that
\begin{align*}
\nabla f_{\alpha, \beta}(x) &= \left[ \begin{array}{c}
                   \log x + \log \beta - d_{\alpha} \log \Gamma(\alpha) \\
                   -x + \frac{\alpha}{\beta} 
              \end{array} \right],\\ 
H f_{\alpha, \beta}(x) &=  \left[ \begin{array}{cc}
                 d^2_\alpha \log \Gamma(\alpha) & \frac{1}{\beta} \\
                 \frac{1}{\beta} & - \frac{\alpha}{\beta^2}
              \end{array} \right].
\end{align*}
The argument is then straightfoward using the same technique as in the proof of Proposition~\ref{prop:scale_location_family}.

Since
\begin{align*}
\| \nabla f_{\alpha, \beta}(x) \|_2 &\leq | \log x | + x + |\log \beta| + | d_\alpha \log \Gamma(\alpha) | + \frac{\alpha}{\beta} \\
       & \lesssim |\log x | + x + 1,
\end{align*}
there exists $C' >0 $ such that if we let $g^*(x) = C'(|\log x| + x + 1)$, then $g^*(x) \geq \sup_{\theta \in \Theta_0} \| \nabla f_\theta(x) \|_2$ and 
\begin{align*}
\int g^*(x)^r \phi(x) dx &= \int_0^\infty C' (|\log x| + x + 1)^r \phi(x) dx \\
        &\lesssim \int_0^\infty |\log x|^r \phi(x) + \int_0^\infty x^r \phi(x) dx < \infty,
\end{align*}
thus verifying inequality~\eqref{eqn:B3}. Moreover,
\begin{align*}
f'_{\alpha, \beta}(x) = \frac{(\alpha - 1)}{x} - \beta, \quad \text{and} \quad \nabla f'_{\alpha, \beta}(x) &= \left[ \begin{array}{c}
                  \frac{1}{x} \\
                    -1
           \end{array} \right].
\end{align*}
We thus conclude that $|f'_{\alpha, \beta}(x)| \lesssim 1 + x^{-1}$ and $\| \nabla f'_{\alpha, \beta}(x) \|_2 \lesssim 1 + x^{-1}$. Hence, there exists $C' > 0$ such that $h^*(x) := C'( 1 + x^{-1} + x)$ satisfies
\[
  h^*(x) \geq \max \{  \sup_{(\alpha, \beta) \in \Theta_0} \| \nabla f_{\alpha,\beta}(x) \|_2,\, \sup_{(\alpha,\beta) \in \Theta_0} \| \nabla f'_{\alpha,\beta}(x) \|_2,\, \sup_{(\alpha,\beta) \in \Theta_0} | f'_{\alpha,\beta}(x)| \}.
\]
Note that $h^*$ is clearly quasi-convex and
\[
\int_0^\infty h^*(x)^{2t} \phi(x) dx \lesssim 1 + \int_0^\infty x^{-2t} + x^{2t} \phi(x) dx.
\]
Since $2t < 1$ by assumption, the integral converges.

Condition B5 follows since $\alpha > 1$ for all $(\alpha, \beta) \in \Theta_0$. 
\end{proof}


\section{A discussion of networks with discrete weights}
\label{sec:discrete_weights}

Let $P$ and $Q$ be distributions on $\mathbb{N}$, and suppose 
\begin{align*}
  A \sim \begin{cases}
      P, \textrm{ if $\sigma_0(u) = \sigma_0(v)$,} \\
      Q, \textrm{ if $\sigma_0(u) \neq \sigma_0(v)$}.
    \end{cases}
\end{align*}
We propose a crude truncation scheme to adapt the algorithm described in Section~\ref{sec:method} for this setting: for an input parameter $L \in \mathbb{N}$, we collect all edges with weight greater than $L$, and assign them a new weight of $L$. In other words, we create a new matrix $A_L$, where if $A_{uv} \in [L]$, then $A_{L,uv} = A_{u,v}$; and if $A_{uv} \in \{ L+1, L+2,\ldots \}$, we set $A_{L, uv} = L$. In this way, $A_L$ is a network with $L$ labels. We then proceed with initialization (Algorithm~\ref{alg:initialization1}) and refinement (Algorithm~\ref{alg:refinement}). 

From Proposition~\ref{prop:labeled_sbm_rate}, we know that the truncation scheme achieves a rate of $\exp( - \frac{n}{\beta K} I(P, Q)(1+o(1))$ if the following two criteria are met. For convenience, let $P'$ and $Q'$ represent distributions on $[L]$ that result from truncating $P$ and $Q$ at $L$. 
\begin{enumerate}
\item There exists $\rho_n > 1$ such that $\rho_n^{-1} \leq \frac{P'_l}{Q'_l} \leq \rho_n$ for all $l \in [L]$, and $\frac{n I'}{(L_n +1) \rho^2_n \log \rho_n} \rightarrow \infty$.
\item $\lim_{L_n \rightarrow \infty} |I(P', Q')/I(P, Q) - 1| \rightarrow 0$.
\end{enumerate}
To show that these criteria are not unrealistic, we argue informally that the $P$ and $Q$ were members of the Poisson family $\{ \mathrm{Poi}(\lambda) \}_{\lambda \in [1,2]}$, then they satisfy both criteria if $L_n$ is chosen so that $L_n = o(\log \log( nI')$ and $L_n \rightarrow \infty$ as $n \rightarrow \infty$. To be precise, let $a \in (0, 1]$, let $(\lambda_1, \lambda_0) \in [1,2]$, and for $l > 0$, let $P_l = a \frac{\lambda_1^l e^{-\lambda_1}}{l!}$ and $Q_l = a \frac{\lambda_0^l e^{-\lambda_0}}{l!}$. Let $P_0 = (1 - a) + a e^{-\lambda_1}$ and $Q_0 = (1 - a) + a e^{-\lambda_0}$. Note that we assume a common sparsity factor $a$ for simplicity. 

It is then straightforward to verify that $\rho_n \lesssim 2^{L_n}$ satisfies the first criterion. For the second criterion, we assume that $I(P,Q) \leq 1$ and that $a = 0$; the argument for when $a > 0$ is exactly the same. Observe that
\begin{align*}
  \frac{I(P',Q') - I(P,Q)}{I(P,Q)} &\lesssim \frac{1}{(\lambda_1 - \lambda_0)^2} \sum_{l=L}^{\infty} (\sqrt{P_l} - \sqrt{Q_l})^2 \\
                                   &\lesssim \frac{1}{(\lambda_1 - \lambda_0)^2} \sum_{l=L}^{\infty} \biggl( \frac{\lambda_1^l e^{-\lambda_1}}{l!} - \frac{\lambda_0^l e^{-\lambda_0}}{l!} \biggr)^2 \frac{l!}{\lambda_0^l e^{-\lambda_0}} \\
  &\leq \frac{1}{(\lambda_1 - \lambda_0)^2} \sum_{l=L}^{\infty} \biggl( \biggl(\frac{\lambda_1}{\lambda_0}\biggr)^{l} e^{\lambda_0-\lambda_1} - 1 \biggr)^2 \frac{\lambda_0^l e^{-\lambda_0}}{l!} \rightarrow 0,
\end{align*}
as $L \rightarrow \infty$, as required. Our information-theoretic lower bound can be extended to this setting, as well. However, this truncation scheme may not be sensible for other types of discrete distributions; we leave a thorough investigation of networks with discrete weights for future work. 


\section{Appendix for Theorem~\ref{thm:lower_bound}}
\label{sec:lower_bound_proof}

We begin by defining some notation. Let $\sigma, \sigma_0 \,:\, [n] \rightarrow [K]$ be two clusterings. Let
\begin{align}
S_K[ \sigma, \sigma_0] := \argmin_{\pi \in S_K} d_H(\pi \circ \sigma, \sigma_0), \label{eqn:Sk_defn}
\end{align}
where $d_H(\cdot, \; \cdot)$ denotes the Hamming distance, and define
\begin{align}
\mathcal{E}[ \sigma, \sigma_0] := \Big\{ v \,:\, \pi(\sigma(v)) \neq  \sigma_0(v),\, 
          \text{ for some } \pi \in S_K[\sigma, \sigma_0]   \Big\}. \label{eqn:error_set_defn}
\end{align}
When $S_K[\sigma, \sigma_0]$ is a singleton, the set $\mathcal{E}[\sigma, \sigma_0]$ contains all nodes misclustered by $\sigma$ in relation to $\sigma_0$.
When $S_K[\sigma, \sigma_0]$ contains multiple elements,
we continue to call $\mathcal{E}[\sigma, \sigma_0]$ the set of \emph{misclustered} nodes.

\subsection{Proof of Theorem~\ref{thm:lower_bound}}
\label{sec:actual_lower_bound_proof}
Let $\sigma_0$ be defined as in the statement of Theorem~\ref{thm:lower_bound}. For any clustering $\sigma$, we define
\[
\tilde{l}(\sigma, \sigma_0) := \frac{1}{n} 
     \sum_{v=1}^n \mathbf{1}\{ v \in \mathcal{E}[\sigma, \sigma_0]\},
\]
where $\mathcal{E}[\sigma, \sigma_0]$ is defined in equation~\eqref{eqn:error_set_defn}. In particular, note that if $|S_K[\hat{\sigma}(A), \sigma_0]| = 1$, we have $\tilde{l} = l$. We claim the following:\\

\noindent{\textbf{Claim 1:}} There exists a sequence of real numbers $\zeta_n \rightarrow 0$ and $c \in (0, \infty)$ such that, for any permutation equivariant estimator $\hat{\sigma}$,
\begin{align}
  \inf_{\substack{((P_0, p), (Q_0, q)) \in \mathcal{G}^* \\ I'_n \leq I((P_0, p), (Q_0, q)) \leq C}} \E_{\substack{(P_0, p)  \\ (Q_0, q)}} \bigl[ \tilde{l}(\hat{\sigma}(A), \sigma_0) \bigr] \exp \left( (1 + \zeta_n) \frac{n}{\beta K} I((P_{0}, p), (Q_{0}, q)) \right) \geq c.
  \label{eqn:info_lower_bound_claim1}
\end{align}

\noindent{\textbf{Claim 2:}} For any $c > 0$, there exists $c' > 0$ such that, for any permutation equivariant estimator $\hat{\sigma}$,
\begin{align}
  \inf_{\substack{((P_0, p), (Q_0, q)) \in \mathcal{G}^* \\ I((P_0, p), (Q_0, q)) \leq c/n}} \E_{\substack{(P_0, p)  \\ (Q_0, q)}} \bigl[ \tilde{l}(\hat{\sigma}(A), \sigma_0) \bigr] \geq c'.
  \label{eqn:info_lower_bound_claim2}
\end{align}

\noindent We first prove the theorem from the claims. Let $((P_0, p), (Q_0, p)) \in \mathcal{G}^*$, and let $\hat{\sigma}$ be an arbitrary clustering algorithm. We omit the $((P_0,p), (Q_0,q))$ subscript on the expectation and probability from this point on for simplicity of presentation. Let $\hat{\sigma}$ be an arbitrary permutation equivariant estimator. If $\mathbb{P} \left( l(\hat{\sigma}(A), \sigma_0) \geq \frac{1}{2\beta K} \right) \geq \frac{1}{2} \E \tilde{l}(\hat{\sigma}(A), \sigma_0)$, then by Markov's inequality, we have
\begin{align*}
\E l(\hat{\sigma}(A), \sigma_0) &\geq \frac{1}{2 \beta K} \mathbb{P} \left( l(\hat{\sigma}(A), \sigma_0) \geq \frac{1}{2\beta K} \right) \geq \frac{1}{4 \beta K}  \E \tilde{l}(\hat{\sigma}(A), \sigma_0).
\end{align*}
On the other hand, if $\mathbb{P} \left( l(\hat{\sigma}(A), \sigma_0) \geq \frac{1}{2\beta K} \right) < \frac{1}{2} \E \tilde{l}(\hat{\sigma}(A), \sigma_0)$, we have
\begin{align*}
& \E l(\hat{\sigma}(A), \sigma_0) \geq \E\left[ l(\hat{\sigma}(A), \sigma_0) 
       \; \Big | \; l(\hat{\sigma}(A), \sigma_0) < \frac{1}{2 \beta K} \right] 
         \mathbb{P} \left(  l(\hat{\sigma}(A), \sigma_0) < \frac{1}{2 \beta K} \right) \\
  &\quad  \stackrel{(a)}{=} \E \left[ \tilde{l}(\hat{\sigma}(A), \sigma_0) 
       \; \Big | \; l(\hat{\sigma}(A), \sigma_0) < \frac{1}{2 \beta K}\right]
         \mathbb{P} \left(  l(\hat{\sigma}(A), \sigma_0) < \frac{1}{2 \beta K} \right) \\
  &\quad = \E \tilde{l}(\hat{\sigma}(A), \sigma_0) -  
      \E \left[ \tilde{l}(\hat{\sigma}(A), \sigma_0) 
       \; \Big | \; l(\hat{\sigma}(A), \sigma_0) \geq \frac{1}{2 \beta K} \right] 
         \mathbb{P} \left(  l(\hat{\sigma}(A), \sigma_0) \geq \frac{1}{2 \beta K} \right) \\
  &\quad \geq \E \tilde{l}(\hat{\sigma}(A), \sigma_0) - \frac{1}{2}  \E \tilde{l}(\hat{\sigma}(A), \sigma_0) = \frac{1}{2}  \E \tilde{l}(\hat{\sigma}(A), \sigma_0),
\end{align*}
where $(a)$ holds by Lemma~\ref{lem:consensus}. Hence, we have that, in all cases, $\E l(\hat{\sigma}(A), \sigma_0) \geq \frac{1}{4 \beta K} \E \tilde{l}(\hat{\sigma}(A), \sigma_0)$. 

We now focus on proving the claims. Let us arbitrarily fix a permutation equivariant algorithm $\hat{\sigma}$, fix $((P_0, p), (Q_0, q)) \in \mathcal{G}^*$, and suppose $I((P_0, p), (Q_0, q)) \leq 2 \log 2$. Without loss of generality, suppose cluster 1 has size $\frac{n}{\beta K} + 1$ and cluster 2 has size $\frac{n}{\beta K}$. Let $C_k := \{u \in [n] \,:\, \sigma_0(u) = k\}$ denote the $k^\text{th}$ cluster. We also suppose without loss of generality that $C_1 = \{1, 2, ..., \frac{n}{\beta K}+1 \}$ and $C_2 = \{ \frac{n}{\beta K}+2, ..., 2\frac{n}{\beta K}+1 \}$.


Define $\sigma_0^1 := \sigma_0$ and define $\sigma_0^2 \,:\, [n] \rightarrow [K]$ such that $\sigma_0^2(v) = \sigma_0(v)$ for all $v \neq 1$, and $\sigma_0^2(1) = 2$. Define $\sigma^*$ as a random cluster assignment with the distribution  
\begin{align}
\sigma^* := \left \{
     \begin{array}{cc}
     \sigma^1_0,  &  \text{with probability } \frac{1}{2}, \\
     \sigma^2_0, & \text{with probability } \frac{1}{2}.
    \end{array} \right. \label{eqn:sigma_star_defn}
\end{align}
We note that $\sigma^*(u) = \sigma_0(u)$ for all $u \neq 1$ and $\sigma^*(1)$ is either 1 or 2, each with $\frac{1}{2}$ probability.

Let $\mathbb{P}_\Phi$ denote a probability measure on $\{ \sigma_0^1, \sigma_0^2\} \times \bigl(\mathbb{R}^{\frac{n(n-1)}{2}}, \mathcal{B}(\mathbb{R}^{\frac{n(n-1)}{2}})\bigr)$---where $\{\sigma_0^1, \sigma_0^2\}$ is by default equipped by the discrete $\sigma$-algebra---defined by
\[
\mathbb{P}_{\Phi}(\sigma^*, \mathbf{A}) = \mathbb{P}(\sigma^*) \mathbb{P}_{SBM}(\mathbf{A} \given \sigma^*),
\]
for a Borel-measurable set $\mathbf{A} \in  \mathcal{B}(\mathbb{R}^{\frac{n(n-1)}{2}})$, 
where $\mathbb{P}_{SBM}( \cdot \given \sigma^*)$ is the probability measure $\bigl(\mathbb{R}^{\frac{n(n-1)}{2}}, \mathcal{B}(\mathbb{R}^{\frac{n(n-1)}{2}})\bigr)$ defined by the weighted SBM with respect to $((P_0, p), (Q_0, q))$, treating $\sigma^*$ as the true cluster assignment. Let $\Psi$ denote an alternative probability measure on $\{ \sigma_0^1, \sigma_0^2\} \times (\mathbb{R}^{n \times n}, \mathcal{B}(\mathbb{R}^{n \times n}))$ defined by
\[
\mathbb{P}_{\Psi} (\sigma^*, \mathbf{A}) = P(\sigma^*) \mathbb{P}_{\Psi}(\mathbf{A} \given \sigma^*),
\]
where $\mathbb{P}_{\Psi}( \mathbf{A} \given \sigma^* )$ is a product of $\frac{n(n-1)}{2}$ distributions over $\mathbb{R}$, defined as follows: if $A$ is a random upper-triangular matrix taking value in $\mathbb{R}^{\frac{n(n-1)}{2}}$, whose distribution is $\mathbb{P}_{\Psi}( \cdot \given \sigma^* )$, then
\begin{enumerate}
\item for $(u,v) \in [n]^2$ where $u < u$, if $u\neq 1$ and $v \neq 1$, then $A_{uv}$ is distributed as in $\mathbb{P}_{SBM}( \cdot \given \sigma^*)$;
\item if $u = 1$ and $v \notin C_1 \cup C_2$, then $A_{uv}$ is distributed as in $P_{SBM}( \cdot \given \sigma^*)$;
\item if $u = 1$ and $v \in C_1 \cup C_2$, then $A_{uv}$ is distributed as $Y$, where $Y$ is the distribution on $\mathbb{R}$ that minimizes $D((P_0, p), (Q_0, q))$ in Lemma~\ref{lem:information_equivalence}; i.e., $Y_{0} \propto (P_{0} Q_{0})^{1/2}$ and $(1-Y_{0}) y(x) \propto \sqrt{(1-P_{0})p(x)(1-Q_{0}) q(x)}$.
\end{enumerate}
Note that $P_{\Psi}( \cdot \given \sigma^*)$ does not actually depend on whether $\sigma^* = \sigma_0^1$ or $\sigma_0^2$. Since $Y$ is absolutely continuous with respect to $P$ and $Q$, the distribution $\mathbb{P}_{\Psi}$ is absolutely continuous with respect to $\mathbb{P}_{\Phi}$. Thus, we may define a Borel measurable function $\mathcal{Q} \,:\, \{ \sigma_0^1, \sigma_0^2\} \times \mathbb{R}^{\frac{n(n-1)}{2}} \rightarrow [0, \infty)$ as
\begin{align}
  \mathcal{Q}(\sigma^*, A) := & \log \frac{d \mathbb{P}_{\Psi}}{d \mathbb{P}_{\Phi}}(\sigma^*, A) = \log \frac{ d \mathbb{P}_{SBM}( A \given \sigma^*)}{ d \mathbb{P}_{\Psi} ( A \given \sigma^*)}. \nonumber 
\end{align}
For $\sigma^* \in \{ \sigma_0^1, \sigma_0^2\}$ and $A \in \mathbb{R}^{n \times n}$, we denote $P(A_{u, 1}) := P_0$ if $A_{u, 1} = 0$, and $P(A_{u, 1}) := (1-P_0)p(A_{u, 1})$ if $A_{u, 1} \neq 0$, and similarly for $Q$ and $Y$. Then, from the fact that $\mathbb{P}_{SBM}(\cdot \given \sigma^*)$ and $\mathbb{P}_{\Psi}(\cdot \given \sigma^*)$ are both product measures, we have
\begin{align}
  \mathcal{Q}(\sigma^*, A) = \sum_{\substack{ u \in C_{\sigma^*(1)} \\ u \neq 1}} \log \frac{Y(A_{1,u})}{P(A_{1,u})} + \sum_{u \in C_1 \cup C_2 \backslash C_{\sigma^*(1)}  } \log \frac{Y(A_{1,u})}{Q(A_{1,u})}. \label{eqn:Q_definition}
\end{align}
We now define the event
$$E = \Big\{ \{1\} \notin\mathcal{E}[\hat{\sigma}(A), \sigma^*] \text{ and }
\tilde{l}(\hat{\sigma}(A), \sigma^*) \leq \frac{1}{4\beta K} \Big\}.$$ 
Let  $t := - \log (20 \beta K \E_{\Phi} \tilde{l}(\hat{\sigma}(A), \sigma^*))$. Then
\begin{align}
\mathbb{P}_{\Psi}\bigl( \mathcal{Q}(\sigma^*, A) \leq t \bigr) & = \mathbb{P}_{\Psi}\bigl( \mathcal{Q}(\sigma^*, A) \leq t,\, E^c \bigr) + \mathbb{P}_{\Psi}\bigl( \mathcal{Q}(\sigma^*, A) \leq t,\, E \bigr).
\label{eqn:psi_Q_bound}
\end{align}
We bound the first term of inequality~\eqref{eqn:psi_Q_bound} as follows:
\begin{align}
& \mathbb{P}_{\Psi}( \mathcal{Q}(\sigma^*, A) \leq t, E^c) = \int_{ \{\mathcal{Q} \leq t\} \cap E^c} d \mathbb{P}_{\Psi} = \int_{ \{\mathcal{Q} \leq t\} \cap E^c} \exp(\mathcal{Q}) d \mathbb{P}_{\Phi} \nonumber \\
  & \qquad \leq e^t \mathbb{P}_{\Phi}\bigl( \mathcal{Q}(\sigma^*, A) \leq t, E^c \bigr) \leq e^t \mathbb{P}_{\Phi}(E^c) \nonumber \\
    & \qquad \leq e^t \left( \mathbb{P}_{\Phi}(\{1\} \in \mathcal{E}[\hat{\sigma}(A), \sigma^*]) + 
      \mathbb{P}_{\Phi} \left( \tilde{l}(\hat{\sigma}(A), \sigma^*) \geq \frac{1}{4\beta K}  \right) \right).
      \label{eqn:prob_Q_Ec_bound}
\end{align}
Furthermore,
\begin{align}
 \E_{\Phi}  \tilde{l}(\hat{\sigma}(A), \sigma^*) &= 
    \frac{1}{n} \sum_{v=1}^n P_{\Phi}( v \in \mathcal{E}[\hat{\sigma}(A), \sigma^*]) \geq \frac{1}{n} \sum_{v \in C_{\sigma^*(1)}} P_{\Phi}( v \in \mathcal{E}[\hat{\sigma}(A), \sigma^*]) \nonumber \\
  &\stackrel{(a)}= \frac{|C_{\sigma^*(1)}|}{n} P_{\Phi}(1 \in \mathcal{E}[\hat{\sigma}(A), \sigma^*]) \geq \frac{1}{ \beta K} P_{\Phi}(1 \in \mathcal{E}[\hat{\sigma}(A), \sigma^*]), \label{eqn:expectation_lower_bound1}
\end{align}
where $(a)$ follows from Corollary~\ref{cor:permutation_equivariance_symmetry_sbm}, and
\begin{align}
 \E_{\Phi} \tilde{l}(\hat{\sigma}(A), \sigma^*) &\geq  
         \E_{\Phi}\Big[ \tilde{l}(\hat{\sigma}(A), \sigma^*) \,\Big|\, 
        \tilde{l}(\hat{\sigma}(A), \sigma^*) \geq \frac{1}{4\beta K}\Big] P_\Phi\left( \tilde{l}(\hat{\sigma}(A), \sigma^*) \geq \frac{1}{4\beta K} \right) \nonumber  \\
    & \geq \frac{1}{4\beta K} P_\Phi\left( \tilde{l}(\hat{\sigma}(A), \sigma^*) \geq \frac{1}{4\beta K} \right). \label{eqn:expectation_lower_bound2}
\end{align}
Hence, combining inequalities~\eqref{eqn:prob_Q_Ec_bound},~\eqref{eqn:expectation_lower_bound1},~\eqref{eqn:expectation_lower_bound2}, we obtain
\begin{align*}
P_{\Psi}( \mathcal{Q} \leq t, E^c) &\leq 
     e^t 5 \beta K  \E_{\Phi} \bigl[ \tilde{l}(\hat{\sigma}(A), \sigma^*)\bigr]. 
\end{align*}

We now turn to the second term in equation~\eqref{eqn:psi_Q_bound}. We have
\begin{multline}
\label{EqnSponsor}
P_\Psi(E) = \frac{1}{2} P_{\Psi}\left(1 \notin \mathcal{E}[\hat{\sigma}(A), \sigma_0^1] 
\text{ and }\tilde{l}(\hat{\sigma}(A), \sigma_0^1)  \leq \frac{1}{4\beta K} \right) \\
+ \frac{1}{2} P_{\Psi}\left(1 \notin \mathcal{E}[\hat{\sigma}(A), \sigma_0^2]  
\text{ and }\tilde{l}(\hat{\sigma}(A), \sigma_0^2)  \leq \frac{1}{4\beta K}  \right).
\end{multline}

In the event that $\tilde{l}(\hat{\sigma}(A), \sigma_0^1) \leq \frac{1}{4 \beta K}$, we know, by the fact that $l(\hat{\sigma}(A), \sigma_0^1) \leq \tilde{l}(\hat{\sigma}(A), \sigma_0^1)$ and Lemma~\ref{lem:consensus}, that $S_K[\hat{\sigma}(A), \sigma_0^1]$ contains only one element, which we denote by $\pi$.  Since $d_H(\sigma_0^1, \sigma_0^2) = 1$, we have $\frac{1}{n} d_H(\pi \circ \hat{\sigma}(A), \sigma_0^2) \leq \frac{1}{4 \beta K} + \frac{1}{n} \leq \frac{1}{2 \beta K}$, so we may apply Lemma~\ref{lem:consensus} again to conclude that $\pi$ is the only element of $S_K[\hat{\sigma}(A), \sigma_0^2]$, as well. However, since $\sigma_0^1(1) = 1$ and $\sigma_0^2(1) = 2$, we know that $(\pi \circ \hat{\sigma}(A) )(1) $ cannot be equal to both $\sigma_0^1(1)$ and $\sigma_0^2(1)$. Hence, we cannot simultaneously have $1 \notin \mathcal{E}[\hat{\sigma}(A), \sigma_0^1]$ and $1 \notin \mathcal{E}[\hat{\sigma}(A), \sigma_0^2]$. Thus, the two events in equation~\eqref{EqnSponsor} are disjoint, so
$P_{\Psi}(\mathcal{Q} \leq t, E) 
             \leq P_{\Psi}(E) \leq \frac{1}{2}$.
Plugging back into equation~\eqref{eqn:psi_Q_bound}, we conclude that
\[
P_{\Psi}(\mathcal{Q} \leq t) \leq e^t 5 \beta K \E_{\Phi} \tilde{l}(\hat{\sigma}(A) , \sigma^*) + \frac{1}{2} \stackrel{(a)} \leq \frac{3}{4},
\]
where $(a)$ follows from the definition that $t = - \log ( 20 \beta K \E_{\Phi} \tilde{l}(\hat{\sigma}(A), \sigma^*) )$. By Chebyshev's inequality, we also have
\[
P_{\Psi}\left( \mathcal{Q} \leq \E_{\Psi} \mathcal{Q} + \sqrt{ 5 V_{\Psi}(\mathcal{Q})} \right) \geq 4/5,
\]
where $V_\Psi(Q) := \textrm{Var}_{\Psi}\bigl( \mathcal{Q}(\sigma^*, A) \bigr)$. Hence, $\log \frac{1}{20 \beta K \E_{\Phi} \tilde{l}(\hat{\sigma}(A),\sigma_0)} \leq \E_{\Psi} \mathcal{Q} + \sqrt{ 5 V_{\Psi}(\mathcal{Q})}$, or equivalently,
\begin{align}
  \E_{\Phi} \tilde{l}(\hat{\sigma}(A), \sigma^*) \geq \frac{1}{20 \beta K} \exp\Big( - (\E_{\Psi} \mathcal{Q} + \sqrt{ 5 V_{\Psi}(\mathcal{Q})} ) \Big).
  \label{eqn:expected_loss_lower_bound_overall}
\end{align}
We now compute $\E_{\Psi} \mathcal{Q}$ and $V_{\Psi}(\mathcal{Q})$. Note that
\[
\E_{\Psi} \mathcal{Q} = \frac{1}{2} \E_{\Psi} [ \mathcal{Q} \given \sigma^* = \sigma_0^1 ] + 
              \frac{1}{2} \E_{\Psi} [ \mathcal{Q} \given \sigma^* = \sigma_0^2 ].
\]
By Lemma \ref{lem:information_equivalence}, we have
\begin{align*}
\E_{\Psi} [\mathcal{Q} \given \sigma^* = \sigma_0^1] &= \E_{\Psi} \left[\sum_{u: \, u \neq 1,\, \sigma_0^1(u) = 1} \log \frac{Y(A_{1,u})}{P(A_{1,u})} + \sum_{u :\, \sigma_0^1(u) = 2} \log \frac{Y(A_{1,u})}{Q(A_{1,u})}\right]  \\
     &= \frac{n}{\beta K}  \int \log \frac{dY}{dP} dY + \frac{n}{\beta K} \int \log \frac{dY}{dQ} dY \\
     &= \frac{n}{\beta K} 2 D((P_0, p), (Q_0, q)) = \frac{n}{\beta K} I((P_0, p), (Q_0, q)).
\end{align*}
Similarly, we have $\E_{\Psi}[ \mathcal{Q} \given \sigma^* = \sigma_0^2] = \frac{n}{\beta K} I((P_0, p), (Q_0, q))$, so
\begin{align}
  \E_{\Psi} \mathcal{Q} = \frac{n}{\beta k} I((P_0, p), (Q_0, q)).
  \label{eqn:Q_expectation_bound}
\end{align}
By Lemma \ref{lem:Q_variance}, we also have
\begin{align}
  \sqrt{5 V_{\Psi} (\mathcal{Q})} \leq 20 \sqrt{ \frac{n}{\beta K}   I((P_0, p), (Q_0, q)) }.
  \label{eqn:Q_variance_bound}
\end{align}

Let us define $\zeta_n := 30 \bigl( \frac{n}{\beta K} I((P_0, p), (Q_0, q)) \bigr)^{-1/2}$.  Combining inequalities~\eqref{eqn:expected_loss_lower_bound_overall},~\eqref{eqn:Q_expectation_bound}, and~\eqref{eqn:Q_variance_bound}, we then have
\begin{align}
  \mathbb{E}_{\Phi} \tilde{l}(\hat{\sigma}(A), \sigma^*)
  &\geq \frac{1}{20 \beta K} \exp \biggl( - \frac{n}{\beta K} I((P_0, p), (Q_0, q)) - 30 \biggl( \frac{n}{\beta K} I((P_0,p), (Q_0, q)) \biggr)^{1/2} \biggr)
  \label{eqn:expectation_lower_bound1} \\
  &\geq \frac{1}{20 \beta K} \exp \biggl( - (1 - \zeta_n) \frac{n}{\beta K} I((P_0, p), (Q_0, q)) \biggr).
    \label{eqn:expectation_lower_bound2}
\end{align}
Now suppose $I((P_0, p), (Q_0, q)) \geq I'_n$. Since $\zeta_n \leq 30 \bigl( \frac{n}{\beta K} I_n' \bigr)^{-1/2}$, we have $\zeta_n \rightarrow 0$.
On the other hand, fix $c > 0$ and suppose $I((P_0, p), (Q_0, q)) \leq \frac{c}{n}$. Then from inequality~\eqref{eqn:expectation_lower_bound1}, we have
\begin{align*}
  \mathbb{E}_{\Phi} \tilde{l}(\hat{\sigma}(A), \sigma^*)
  \geq \frac{1}{20 \beta K} \exp\biggl( - \frac{c}{\beta K} - 30 \sqrt{\frac{c}{\beta K}} \biggr) := c'. 
\end{align*}
To finish the proof of the two claims listed at the beginning, we need only show that the same lower bound holds for $\mathbb{E} \tilde{l}(\hat{\sigma}(A), \sigma_0)$. 

Define two probability measures $P_1$ and $P_2$ on $(A, \hat{\sigma}(A))$, as follows:
\begin{align*}
P_1(A, \hat{\sigma}(A)) &= 
    P_{SBM}( A \given \sigma_0^1) P_{alg}( \hat{\sigma}(A) \given A), \\
P_2(A, \hat{\sigma}(A)) &= P_{SBM}( A \given \sigma_0^2) P_{alg}( \hat{\sigma}(A) \given A).
\end{align*}
Note that $\E_{\Phi}[ \tilde{l}(\hat{\sigma}(A), \sigma^*) \given \sigma^* = \sigma_0^1] = \E_1 \tilde{l}(\hat{\sigma}(A), \sigma_0^1) $ and $\E_{\Phi}[ \tilde{l}(\hat{\sigma}(A), \sigma^*) \given \sigma^* = \sigma_0^2] = \E_2 \tilde{l}(\hat{\sigma}(A), \sigma_0^2)$, where $\E_1$ and $\E_2$ are expectations taken with respect to $P_1$ and $P_2$, respectively. We claim that $\E_1 \tilde{l}(\hat{\sigma}(A), \sigma_0^1) = \E_2 \tilde{l}(\hat{\sigma}(A), \sigma_0^2)$, in which case $\E_{\Phi} l(\hat{\sigma}(A), \sigma^*) = \E_1 \tilde{l}(\hat{\sigma}(A), \sigma^1_0) = \E \tilde{l}(\hat{\sigma}(A), \sigma_0)$ and the claims follow.

Define a permutation $\pi \in S_n$ that swaps $\{2, \dots, \frac{n}{\beta K} + 1\}$ with $\{ \frac{n}{\beta K} + 2, \dots, 2 \frac{n}{\beta K} + 1\}$ and satisfies $\pi(u) = u$ for $u = 1$ and $u \geq 2 \frac{n}{\beta K} + 2$. Clearly, $\sigma_0^2 = \tau \circ \sigma_0^1 \circ \pi^{-1}$, where $\tau \in S_K$ swaps cluster labels 1 and 2. Now let $A$ be fixed and let $\rho \in S_K$ be arbitrary. We have
\begin{align*}
d_H( \rho \circ \hat{\sigma}(A), \sigma_0^1) &= 
  d_H( \rho \circ \hat{\sigma}(A), \tau^{-1} \circ \sigma_0^2 \circ \pi)\\
  &= d_H( \rho \circ \hat{\sigma}(A) \circ \pi^{-1}, 
                     \tau^{-1} \circ \sigma_0^2) \\
  &\stackrel{(a)} = d_H( \rho \circ \xi^{-1} \circ \hat{\sigma}(\pi A), 
              \tau^{-1} \circ \sigma_0^2) \\
  &= d_H( \tau \circ \rho \circ \xi^{-1} \circ \hat{\sigma}(\pi A), \sigma_0^2),
\end{align*}
where $(a)$ follows because $\hat{\sigma}(\cdot)$ is permutation equivariant, which implies that there exists $\xi \in S_K$ such that $\hat{\sigma}(A) \circ \pi^{-1} = \xi^{-1} \circ \hat{\sigma}(\pi A)$. 
Thus, $\rho \mapsto \tau \circ \rho \circ \xi^{-1}$ is a bijection between $S_K[\hat{\sigma}(A), \sigma_0^1]$ and $S_K[\hat{\sigma}(\pi A), \sigma_0^2]$. Furthermore, if $v$ satisfies $(\rho \circ \hat{\sigma}(A))(v) \neq \sigma_0^1(v)$, then letting $u = \pi(v)$, we equivalently have
\begin{align*}
(\rho \circ \hat{\sigma}(A))(v) &\neq (\tau^{-1} \circ \sigma_0^2 \circ \pi) (v) \iff (\rho \circ \hat{\sigma}(A) \circ \pi^{-1} )(u) \neq (\tau^{-1} \circ \sigma_0^2 )(u),
\end{align*}
so $(\tau \circ \rho \circ \xi^{-1} \circ \hat{\sigma}(\pi A)) (u) \neq
 \sigma_0^2(u)$. Thus, $v \in \mathcal{E}[\hat{\sigma}(A), \sigma_0^1]$ if and only if $\pi(v) \in \mathcal{E}[\hat{\sigma}(\pi A), \sigma_0^2]$.  Finally, we conclude that
\begin{align*}
\E_1 \tilde{l}(\hat{\sigma}(A), \sigma_0^1) &= 
      \frac{1}{n} \sum_{v = 1}^n P_1 ( v \in \mathcal{E}[\hat{\sigma}(A), \sigma_0^1]) = \frac{1}{n} \sum_{v=1}^n  P_1 ( \pi(v) \in \mathcal{E}[\hat{\sigma}(\pi A), \sigma_0^2]) \\
   &\stackrel{(a)} = \frac{1}{n} \sum_{v=1}^n P_2( \pi(v) \in  \mathcal{E}[\hat{\sigma}(A), \sigma_0^2] ) = \E_2 \tilde{l} (\hat{\sigma}(A), \sigma_0^2),
\end{align*}
where $(a)$ follows because $[\pi A]_{ij} = A_{\pi^{-1}(i), \pi^{-1}(j)}$, implying that if $A$ is distributed according to $P_{SBM}(A \given \sigma_0^1)$, then $\pi A$ is distributed according to $P_{SBM}(A \given \sigma_0^2)$. This concludes the proof.


\subsection{Properties of permutation equivariant estimators}

The following lemma establishes a symmetry property used to prove Theorem~\ref{thm:lower_bound}: 
\begin{lemma}
\label{lem:permutation_equivariance_symmetry}
Let the true clustering $\sigma_0$ be arbitrary. Suppose the weight matrix $A$ is drawn from an arbitrary probability measure and $\hat{\sigma}$ is any permutation equivariant estimator. Let $u$ and $v$ be two nodes such that there exists $\pi \in S_n$ satisfying
\begin{itemize}
\item[(1)] $\pi(u) = v$,
\item[(2)] $\pi$ is measure-preserving; i.e., $A \stackrel{d}{=} \pi A$, and
\item[(3)] $\pi$ preserves the true clustering; i.e., there exists $\tau \in S_K$ such that $\tau \circ \sigma_0 \circ \pi^{-1} = \sigma_0$.
\end{itemize}
Then
\[
P( u \in \mathcal{E}[\hat{\sigma}(A), \sigma_0]) = P( v \in \mathcal{E}[\hat{\sigma}(A), \sigma_0] ).
\]
\end{lemma}

\begin{proof}
Since $\hat{\sigma}(A) \stackrel{d}{=} \hat{\sigma}(\pi A)$, we have
\[
P \Big(v \in \mathcal{E}[\hat{\sigma}(A), \sigma_0 ] \Big) = 
       P \Big(v \in \mathcal{E}[\hat{\sigma}(\pi (A)), \sigma_0] \Big).
\]
We claim that $u \in \mathcal{E}[\hat{\sigma}(A), \sigma_0]$ if and only if $v \in \mathcal{E}[\hat{\sigma}(\pi A), \sigma_0]$, implying the desired result:
\begin{align*}
P\Big( u \in \mathcal{E}[\hat{\sigma}(A), \sigma_0]\Big) &= P\Big(v \in \mathcal{E}[ \hat{\sigma}(\pi A), \sigma_0]\Big) = P\Big(v \in \mathcal{E}[\hat{\sigma}(A), \sigma_0]\Big).
\end{align*}

Consider a fixed matrix $A$, and let $\tau \in S_K$ satisfy $\tau \circ \sigma_0 \circ \pi^{-1} = \sigma_0$. Let $\xi \in S_K$ be the permutation such that $\hat{\sigma}(\pi A) = \xi \circ \hat{\sigma}(A) \circ \pi^{-1}$. For any $\rho \in S_K$, we have
\begin{align*}
d_H(\rho \circ \hat{\sigma}(A), \; \sigma_0) &= 
      d_H( \tau \circ \rho \circ \xi^{-1} \circ \xi \circ \hat{\sigma}(A) \circ \pi^{-1}, \; \tau \circ \sigma_0 \circ \pi^{-1} ) \\
  &= d_H( \tau \circ \rho \circ \xi^{-1} \circ \hat{\sigma}(\pi A), \; \sigma_0).
\end{align*}
Therefore, $\rho \in S_K[\hat{\sigma}(A), \sigma_0]$ if and only if $\tau \circ \rho \circ \xi^{-1} \in S_K[\hat{\sigma}(\pi(A)), \sigma_0]$. In particular, if $v \in \mathcal{E}[\hat{\sigma}(\pi A), \sigma_0]$, we have $\tau \circ \rho \circ \xi^{-1} \circ \hat{\sigma}(\pi A)(v) \neq \sigma_0(v)$ for some $\rho \in S_K[\hat{\sigma}(A), \sigma_0]$. Then
\begin{align*}
\hat{\sigma}(A)(u) &= \hat{\sigma}(A)(\pi^{-1}(v)) = \xi^{-1} \circ \xi \circ \hat{\sigma}(A) \circ \pi^{-1} (v) = \xi^{-1} \circ \hat{\sigma}( \pi A) (v) \\
   &\neq \rho^{-1} \circ \tau^{-1} \circ \sigma_0(v) = \rho^{-1} \circ \tau^{-1} \circ \sigma_0( \pi( u )) = \rho^{-1} (\sigma_0(u)).
\end{align*}
Thus, $u \in \mathcal{E}[ \hat{\sigma}(A), \sigma_0]$. Similar reasoning shows that if $u \in \mathcal{E}[ \hat{\sigma}(A), \sigma_0]$, then $v \in \mathcal{E}[\hat{\sigma}(\pi A), \sigma_0]$.
\end{proof}

\begin{corollary}
\label{cor:permutation_equivariance_symmetry_sbm}
Let the true clustering $\sigma_0$ be arbitrary. Suppose the weight matrix $A$ is drawn from a weighted SBM and $\hat{\sigma}$ is any permutation equivariant estimator.  Let $u$ and $v$ be two nodes lying in equal-sized clusters. Then
\[
P\Big( u \in \mathcal{E}[\hat{\sigma}(A), \sigma_0]\Big) = P\Big( v \in \mathcal{E}[\hat{\sigma}(A), \sigma_0] \Big).
\]
\end{corollary}

\begin{proof}
By Lemma~\ref{lem:permutation_equivariance_symmetry}, it suffices to construct a permutation $\pi \in S_n$ satisfying conditions (1)--(3). First suppose $u$ and $v$ lie in the same cluster. It is easy to see that the conditions are satisfied when $\pi$ is the permutation that swaps $u$ and $v$ and $\tau$ is the identity. If $u$ and $v$ lie in different clusters, suppose without loss of generality that $u$ is in cluster 1 and $v$ is in cluster 2, where clusters 1 and 2 have the same size. Let $\pi$ be the permutation that exchanges all nodes in cluster 1 with all nodes in cluster 2. The conditions are satisfied when $\tau$ is the permutation that transposes cluster labels 1 and 2. 
\end{proof}

\subsection{Properties of Renyi divergence}

We first state a lemma that provides an alternative characterization of the Renyi divergence. The proof of this lemma uses the same technique as Anantharam~\cite{anantharam2017variational}. 
\begin{lemma}
\label{lem:information_equivalence}
Let $P$ and $Q$ be two probability measures on $S \subset \R$ that are absolutely continuous with respect to each other, with point masses $P_0$ and $Q_0$ at zero. Then $I(P, Q) = 2D$, where
\begin{equation*}
D := \inf_{Y \in \mathcal{P}} \max \left\{ \int \log \frac{dY}{dP} dY, \, \int \log \frac{dY}{dQ} dY \right\},
\end{equation*}
and $\mathcal{P}$ denotes the set of probability distributions absolutely continuous with respect to both $P$ and $Q$. Furthermore, the infimum in $D$ is attained by a distribution $Y^*$ whose singular (with respect to the Lebesgue measure) part is $Y^*_0 = \frac{1}{Z} (P_0 Q_0)^{1/2}$ and continuous part is $(1 - Y^*_0) y^*(x) = \frac{1}{Z} \int_S \sqrt{(1 - P_0)(1-Q_0) p(x) q(x)} \, dx$, where $Z =  (P_0 Q_0)^{1/2} + \int_S \sqrt{(1 - P_0)(1-Q_0) p(x) q(x)} \, dx$. 
\end{lemma}

\begin{proof} 
First note that $D$ is finite by choosing $Y = P$. We claim that
\begin{align}
\label{EqnMeteor}
D = \inf_{Y \in \mathcal{P}} \left\{\int \log \frac{dY}{dP} dY: \; \int \log \frac{dP}{dQ} dY = 0\right\}.
\end{align}
This holds because for any $Y \in \mathcal{P}$ such that $\int \log \frac{dP}{dQ} dY \neq 0$, 
we have $\int \log \frac{dY}{dP} dY \neq \int \log \frac{dY}{dQ} dY$. 
Suppose without loss of generality that the first quantity is larger. Then it is possible to take $\tilde{Y} = (1 - \epsilon) Y + \epsilon P$ for $\epsilon$ small enough such that $ \max \left\{ \int \log \frac{d\tilde{Y}}{dP} d\tilde{Y}, \, \int \log \frac{d\tilde{Y}}{dQ} d\tilde{Y} \right\}$ strictly decreases, so the infimum in the definition of $D$ could not have been achieved. 

Now let $Y \in \mathcal{P}$ be such that $\int \log \frac{dP}{dQ} dY = 0$. Since $Y \ll Y^*$, we have
\begin{align*}
  \int \log \frac{dY}{dP} dY
  &= \int \log \frac{dY}{d Y^*} \frac{d Y^*}{d P} dY 
  \geq \int \log \frac{dY}{dY^*} dY + \int \log \frac{dY^*}{dP} dY \\
  &= \frac{1}{2} \int \log \frac{dP}{dQ} dY - \log Z = - \log Z = - \frac{I(P, Q)}{2}.
\end{align*}
We finish the proof by verifying that $\int \log \frac{dP}{dQ} dY^* = 0$, and that in the derivation above, the inequality holds with inequality when $Y = Y^*$. 
\end{proof}


\subsection{Bounding the variance}
\begin{lemma}\label{lem:Q_variance}
  Let $C^*$ and $\mathcal{G}^*$ be defined as in Theorem~\ref{thm:lower_bound}, let $((P_0, p), (Q_0, p) \in \mathcal{G}^*$, and let $\sigma_0$ be a clustering satisfying the hypothesis of Theorem~\ref{thm:lower_bound}. Let $\sigma^*$ be defined as~\eqref{eqn:sigma_star_defn}, and $\mathcal{Q}$ be defined as in equation~\eqref{eqn:Q_definition}. Then
  \[
    V_{\Psi} (\mathcal{Q}) \leq 63 C^{* 2} \frac{n}{\beta K} I((P_0, p), (Q_0, q)) .
  \]
\end{lemma}
\begin{proof}
Observe that since $\E_{\Psi}[ \mathcal{Q} \given \sigma^* = \sigma_0^1] = \E_{\Psi}[ \mathcal{Q} \given \sigma^* = \sigma_0^2]$, we have
\begin{align}
V_{\Psi}(\mathcal{Q}) &= \textrm{Var}( \E_{\Psi}[ \mathcal{Q} \given \sigma^* ] ) + 
                        E[ \textrm{Var}_{\Psi}( \mathcal{Q} \given \sigma^* ) ] =   E[ \textrm{Var}_{\Psi}( \mathcal{Q} \given \sigma^* ) ]
  \nonumber \\
      &=  \frac{1}{2} \textrm{Var}_{\Psi}( \mathcal{Q} \given \sigma^* = \sigma_0^1) + 
          \frac{1}{2} \textrm{Var}_{\Psi}( \mathcal{Q} \given \sigma^* = \sigma_0^2). \label{eqn:variance_decomposition}
          \end{align}
Let us consider the first term of equation~\eqref{eqn:variance_decomposition}. Under the distribution $\mathbb{P}_{\Psi}$, the random variables $A_{1,u}$ for any $u \in C_1 \cup C_2$ are independent and identically distributed according to $Y$. Let $u'$ be an arbitrary node in $C_1 \cup C_2$ such that $u' \neq 1$. We thus have
\begin{align*}
\textrm{Var}_{\Psi}(\mathcal{Q} \given \sigma^* = \sigma_0^1) &= \sum_{ \substack{ u:\, u \neq 1, \\ \sigma_0^1(u) = 1 }} 
         \textrm{Var}_{\Psi} \left( \log \frac{Y(A_{1,u})}{P(A_{1,u})} \right) +
                        \sum_{u :\, \sigma_0^1(u) = 2} \textrm{Var}_{\Psi} \left( \log \frac{Y(A_{1,u})}{Q(A_{1,u})} \right)  \\
   &\leq \frac{n}{\beta K}  \E_{\Psi} \left[ \biggl( \log \frac{Y(A_{1,u'})}{P(A_{1,u'})} \biggr)^2 \right] + 
         \frac{n}{\beta K} \E_{\Psi} \left[ \biggl( \log \frac{Y(A_{1,u'})}{Q(A_{1,u'})} \biggr)^2 \right]. 
\end{align*}
By the same argument, we can show that the second term of equation~\eqref{eqn:variance_decomposition} is upper-bounded by the same quantity, so
\begin{align}
  V_{\Psi}(\mathcal{Q})
  &\leq \frac{n}{\beta K} \frac{n}{\beta K}  \E_{\Psi} \left[ \biggl( \log \frac{Y(A_{1,u'})}{P(A_{1,u'})} \biggr)^2 \right] + 
         \frac{n}{\beta K} \E_{\Psi} \left[ \biggl( \log \frac{Y(A_{1,u'})}{Q(A_{1,u'})} \biggr)^2 \right] \nonumber \\
  &\leq \frac{n}{\beta K} \biggl\{ \int \biggl( \log \frac{dY}{dP} \biggr)^2 dY + \int \biggl( \log \frac{dY}{dQ} \biggr)^2 dY \biggr\}. \label{eqn:variance_integral_bound_overall}
\end{align}
By the definition of $Y$, we have
\begin{align}
 \int \left( \log \frac{dY}{dP} \right)^2 dY 
  = Y_0 \biggl( \log \frac{Y_0}{P_0} \biggr)^2 + (1-Y_0) \int y(x) \biggl( \log \frac{(1-Y_0) y(x)}{(1-P_0)p(x)} \biggr)^2 dx. \label{eqn:lower_bound_var_terms}
\end{align}
For simplicity of presentation, we use the shorthand
\begin{align*}
  I &:= I((P_0, p), (Q_0, q)) \quad \textrm{and} \quad \\
  H &:= (\sqrt{P_0} - \sqrt{Q_0})^2 + (\sqrt{1-P_0} - \sqrt{1-Q_0})^2 + \sqrt{(1-P_0)(1-Q_0)} \int (\sqrt{p(x)} - \sqrt{q(x)})^2 dx,
  \end{align*}
from this point on. Let $Z = \sqrt{P_0 Q_0} + \sqrt{(1-P_0)(1-Q_0)} \int \sqrt{p(x)q(x)} dx$. It is then clear that $I = -2 \log Z$ and $Z \geq \frac{1}{2}$, since $I \leq 2 \log 2$ by assumption. We bound the first term of equation~\eqref{eqn:lower_bound_var_terms}. Note that, since $(x+y)^2 \leq 2 x^2 + 2 y^2$, we have
\begin{align}
  Y_0 \biggl( \log \frac{Y_0}{P_0} \biggr)^2
  \leq \frac{1}{2} Y_0 \biggl( \log \frac{Q_0}{P_0} \biggr)^2 + \frac{1}{2} I^2.
  \label{eqn:var_first_term_decomp}
\end{align}
Suppose $P_0 \leq Q_0$. Then by Lemma~\ref{lem:log_linearize}, we have $0 \leq \log \frac{Q_0}{P_0} \leq \frac{Q_0 - P_0}{P_0}$. Thus, we have 
\begin{align*}
  Y_0 \biggl( \log \frac{Y_0}{P_0} \biggr)^2
  &\leq  C^{*2}  \frac{(Q_0 - P_0)^2}{ P_0 \vee Q_0} + \frac{1}{2} I^2 \\
  &\leq 4 C^{*2} (\sqrt{Q_0} - \sqrt{P_0})^2 + \frac{1}{2} I^2
   \leq  4 C^{*2} H + \frac{1}{2} I^2  \stackrel{(a)} \leq 4 C^{*2} I + \frac{1}{2} I^2,
\end{align*}
using inequality~\eqref{eqn:var_first_term_decomp} and the fact that $((P_0, p), (Q_0,p)) \in \mathcal{G}^*$. Here, $(a)$ follows from the fact that $I = -2 \log (1 - \frac{1}{2} H)$ and Lemma~\ref{lem:log_linearize}. Now suppose $P_0 \geq Q_0$. By Lemma~\ref{lem:log_linearize} again, we have $0 \geq \log \frac{Q_0}{P_0} \geq -(1 + \frac{1}{2} C^*) \bigl( \frac{P_0 - Q_0}{P_0} \bigr) \leq -C^* \bigl( \frac{P_0 - Q_0}{P_0} \bigr)$. Therefore,
\begin{align*}
    Y_0 \biggl( \log \frac{Y_0}{P_0} \biggr)^2
  \leq C^{*2} \frac{(Q_0 - P_0)^2}{ P_0 \vee Q_0} + \frac{1}{2} I^2 \leq 4 C^{*2} I + \frac{1}{2} I^2.
\end{align*}
Thus, we conclude that
\begin{align}
  Y_0 \biggl( \log \frac{Y_0}{P_0} \biggr)^2 \leq 4 C^{*2} I + \frac{1}{2} I^2.
  \label{eqn:var_first_term_bound_final}
\end{align}

Now we turn our attention to the second term in equation~\eqref{eqn:lower_bound_var_terms}. We have
\begin{align}
\label{EqnNoodle}
& (1 - Y_0) \int y(x) \left( \log \frac{1 - Y_0}{1 - P_0} \frac{y(x)}{p(x)} \right)^2 dx \notag  \\
& \qquad \leq 
   (1 - Y_0) \int y(x) \left\{ \frac{1}{2} \left| \log \frac{1-Q_0}{1-P_0} \right| 
              +  \frac{1}{2} \left| \log \frac{q(x)}{p(x)}\right| + \frac{I}{2}\right\}^2 dx \notag \\
  & \qquad \leq \frac{9}{4} (1 - Y_0) \int y(x) \left\{ \left| \log \frac{1-Q_0}{1-P_0} \right|^2 + 
                    \left| \log \frac{q(x)}{p(x)}\right|^2 +  I^2 \right\} dx,
\end{align}
where we have used the fact that $(x + y + z)^2 \leq 9x^2 + 9y^2 + 9z^2$ in the last inequality.
Define
\begin{align*}
\mathcal{A} & := (1 - Y_0) \left| \log \frac{1-Q_0}{1-P_0} \right|^2 \int y(x) dx, \qquad \mathcal{B} := (1 - Y_0) \int y(x) \left| \log \frac{q(x)}{p(x)}\right|^2 dx, \\
\mathcal{C} & := (1 - Y_0) I^2 \int y(x)   dx.
\end{align*}
We bound each term separately, beginning with $\mathcal{A}$. Note that
\begin{align}
  \int y(x) dx \leq 2 \int \sqrt{(1-P_0)(1-Q_0)p(x) q(x)} dx \leq 2.
  \label{eqn:y_density_bound}
\end{align}
Note also that we can use the same reasoning that we applied to derive inequality~\eqref{eqn:var_first_term_bound_final} to show that $(1-Y_0) \left| \log \frac{1-Q_0}{1-P_0} \right|^2 \leq 4 C^{*2} I$. Therefore, we have $\mathcal{A} \leq 8 C^{*2} I$. 
Moving on to $\mathcal{B}$, we have
\begin{align*}
\mathcal{B} 
  & \leq 2 \sqrt{(1 - P_0)(1-Q_0)} \int \sqrt{p(x) q(x)}  \left| \log \frac{q(x)}{p(x)} \right|^2 dx\\
  & \leq \sqrt{(1 - P_0)(1-Q_0)} \int (p(x) + q(x))  \left| \log \frac{q(x)}{p(x)} \right|^2 dx\\
  &\stackrel{(a)} \leq C^* \sqrt{(1 - P_0)(1-Q_0)} \int (\sqrt{p(x)} - \sqrt{q(x)})^2 dx \leq C^* I,
\end{align*}
where $(a)$ follows by the hypothesis of the proposition.
Finally, from inequality~\eqref{eqn:y_density_bound}, we have $\mathcal{C} \leq 8 (1 - Y_0) I^2 \leq (16 \log 2) I $. Substituting back into inequality~\eqref{EqnNoodle}, we therefore obtain
\begin{align*}
(1 - Y_0) \int y(x) \left( \log \frac{1 - Y_0}{1 - P_0} \frac{y(x)}{p(x)} \right)^2 dx \leq \frac{9}{4} (8 C^{*2} + C^* + 16 \log 2) I.
\end{align*}

Substituting inequality~\eqref{eqn:var_first_term_bound_final} and the above inequality back into inequality~\eqref{eqn:lower_bound_var_terms}, we obtain the desired bound on the first term of inequality~\eqref{eqn:variance_integral_bound_overall}. The second term of inequality~\eqref{eqn:variance_integral_bound_overall} can be bounded in exactly the same manner. Thus, we have overall bound
\begin{align*}
  V_{\Psi}(\mathcal{Q}) \leq 2(22 C^{*2} + 3 C^* + 13) \frac{nI}{\beta K} \leq
  63 C^{*2} \frac{nI}{\beta K},
\end{align*}
where we have used the fact that $C^* \geq 1$.
\end{proof}


\section{Additional useful lemmas}

\begin{lemma}
  \label{lem:log_linearize}
  For all $x \in (-\infty, 1)$, we have $\log(1 - x) \leq -x$. Moreover, fix $t \in [0, 1)$. Then for all $x \in [-t, t]$, we have
  \[
    -x \bigl( 1 + \frac{x}{2(1-t)^2} \bigr) \leq \log(1 - x) \leq -x.
  \]
\end{lemma}

\begin{proof}
  Let $x \in [-t, t]$. Then $\log(1 - x) \leq -x$ by concavity of the logarithm. On the other hand, by Taylor's theorem, there exists $x'$ such that $|x'| \leq |x|$ and
  \[
    \log(1 - x) = -x - \frac{1}{2 (1-x')^2} x^2
    \geq -x \bigl( 1 + \frac{x}{2 (1-t)^2} \bigr).
  \]
\end{proof}

\begin{lemma}
  \label{lem:renyi_hellinger}
Let $P$ and $Q$ be two distributions absolutely continuous with respect to each other, and let $H(P,Q) := \int \biggl( \biggl( \frac{dP}{dQ} \biggr)^{1/2} - 1 \biggr) \, dQ$. Then $I(P,Q) \geq H(P, Q)$. If in addition $I(P,Q) \leq 2 \log 2$, then $H(P, Q) \leq 1$ and $I(P, Q) \leq (1 + H(P,Q)) H(P, Q)$. 
\end{lemma}

\begin{proof}
The conclusion follows from the fact that $I(P,Q) = -2 \log (1 - H(P, Q)/2)$ and Lemma~\ref{lem:log_linearize}.
\end{proof}

\bibliographystyle{plain}
\bibliography{refs}

\end{document}